\documentclass{amsart}
\usepackage{amssymb}
\usepackage{graphicx}
\usepackage{comment}
\usepackage{tikz}
\usepackage{tikz-cd}
\usetikzlibrary{positioning}

\usepackage{amsmath, amsthm}
\usepackage{hyperref}
\hypersetup{
    linkcolor=blue
}
\usepackage[capitalise, nameinlink]{cleveref}
\usepackage{thmtools}
\usepackage{zref-clever}

\theoremstyle{definition}

\theoremstyle{remark}

\numberwithin{equation}{section}

\begin{document}

\title[Arithmetic in $(LO, +)$]{full title}
\title{The Additive Arithmetic of Linear Orders}


\author{Garrett Ervin}
\address{Garrett Ervin, E\"otv\"os Lor\'and University, Institute of Mathematics, P\'azm\'any P\'eter stny. 1/C, 1117 Budapest, Hungary}
\curraddr{}
\email{garrette@ttk.elte.hu}
\thanks{The first author is supported by the Momentum MSCA Programme, which is co-funded by the European Commission through the HORIZON-MSCA-2023-COFUND programme and the Secretariat of the Hungarian Academy of Sciences (MTA)}

\author{Eric Paul}
\address{Eric Paul, University of Illinois Urbana-Champaign, Siebel Center for Computer Science, 201 N Goodwin Ave, Urbana, IL 61801, United States}
\curraddr{}
\email{epaul9@illinois.edu}
\thanks{}


\date{\today}

\dedicatory{}

\begin{abstract}
We present a systematic development of the arithmetic of the class of linear orders under the ordered sum $(LO, +)$ and prove a number of new results. Our approach is based on a Euclidean algorithm for pairs of linear orders that almost additively commute. 

Among our results:

\begin{itemize}
    \item[i.] We generalize and give unified proofs of the main classical theorems for $(LO, +)$, including Lindenbaum's division theorem for $(LO, +)$ and a representation theorem for additively commuting pairs of linear orders due to Aronszajn.
    \item[ii.] We solve the following problem, posed by Tarski in 1956: is it true that for every pair of linear orders $A, B$ and quadruple of natural numbers $n, m, k, l \geq 1$, if $nA + mB \cong kB + lA$ then $A + B \cong B + A$? Tarski and Chang showed the answer is yes for certain choices of the coefficients $n, m, k, l$. We show the answer is yes in general. 
    \item[iii.] We prove the following characterization of the additively commuting pairs in $LO$: $A + B \cong B + A$ if and only if $\omega A$ embeds initially in $\omega B$ and $\omega^* A$ embeds finally in $\omega^* B$, or vice versa. We show this can be viewed as a correctly revised version of a refuted conjecture of Tarski.
    \item[iv.] We characterize the commutative semigroups $(S, \oplus)$ that can be represented in $(LO, +)$ and show in particular they are all subsemigroups of naturally totally ordered semigroups in the sense of Clifford. 
\end{itemize}
\end{abstract}

\maketitle

\tableofcontents
\newpage            

\section{Introduction} \label{section:intro}

Let $LO$ denote the class of linear orders. Given two orders $A, B \in LO$, their \textit{ordered sum} $A + B$ is the order obtained by placing a copy of $B$ to the right of $A$. Given a natural number $n \geq 1$, $nA$ denotes the $n$-fold sum $A + A + \cdots + A$.

Many of the fundamental results about sums of linear orders are due to Lindenbaum and Tarski, and were announced in their joint 1926 paper \cite{LindenbaumTarski}. Outstanding among these results are the following two theorems of Lindenbaum. 

\theoremstyle{definition}
\newtheorem{lct}{Theorem}[section]
\begin{lct}\label{lct}
(Lindenbaum's Cancellation Theorem; \cite[3.13]{LindenbaumTarski}) Suppose $A$ and $B$ are linear orders and there is a natural number $n \geq 1$ such that $nA \cong nB$. Then $A \cong B$. 
\end{lct}

\theoremstyle{definition}
\newtheorem{ldt}[lct]{Theorem}
\begin{ldt}\label{ldt}
(Lindenbaum's Division Theorem; \cite[3.14]{LindenbaumTarski}) Suppose $A$ and $B$ are linear orders and there are natural numbers $n, m \geq 1$ with $\gcd(n, m) = 1$ such that $nA \cong mB$. Then there is a linear order $C$ such that $A \cong mC$ and $B \cong nC$. 
\end{ldt}

If we view each natural number $n$ as a finite linear order with $n$-many points, then the ordered sum $+$ agrees with the usual sum on $\mathbb{N}$, and can be viewed as an extension of this operation to the much larger class $LO$. When the orders $A$ and $B$ from the statement of the division theorem are natural numbers, the order $C$ must equal $\gcd (A, B)$. Thus for more general orders, $C$ may be viewed as the order-theoretic greatest common divisor of $A$ and $B$. It follows from the cancellation theorem that $C$ is unique up to isomorphism. 

Lindenbaum's theorems are strikingly general: they apply to an arbitrary pair of linear orders $A$ and $B$ with a common finite multiple $nA \cong mB$. This generality suggests that the class of all linear orders, not just special subclasses such as the class of natural numbers or the class of ordinals, has a reasonable and global additive arithmetic that may be studied as such. More specifically, the theorems suggest that $(LO, +)$ has a Euclidean division structure generalizing that of $(\mathbb{N}, +)$.

The results in Lindenbaum and Tarski's paper \cite{LindenbaumTarski} are presented without proof, and Lindenbaum's own proofs of his theorems were never published. Proofs of the cancellation and division theorems did not appear until 30 years later, after Lindenbaum's death, in Tarski’s book \textit{Ordinal Algebras} \cite{Tarski}.

In \textit{Ordinal Algebras}, Tarski undertakes a systematic study of the arithmetic of the ordered sum by passing from $LO$ to a more abstract context. He axiomatizes a type of algebraic structure he calls an ordinal algebra. Such algebras have the following form:
\[
\mathfrak{A} = (A, +, \sum, {}^*, 0), 
\]
where $A$ is the universe of the algebra, $+$ is a binary ordered sum operation, $\sum$ is an $\mathbb{N}$-ary ordered sum operation extending the binary sum, ${}^*$ is a unary reversal operator, and $0$ is an additive identity. 

These operations generalize the corresponding operations for linear orders: if we distinguish linear orders only up to isomorphism, then 
\[
\mathfrak{A} = (LO, +, \sum, {}^*, \emptyset)
\]
is the prototypical example of an ordinal algebra, and it was Tarski's earlier work with Lindenbaum on the arithmetic of this particular algebra that inspired his abstract definition.  

Tarski's goal was to generalize the study of sums of linear orders to ordered sums of arbitrary binary relations. In the introduction to \textit{Ordinal Algebras}, he writes:

\begin{quote}
An important task in creating the general theory of binary
relations is to develop the \textit{arithmetic of relation types}, i.e., to study operations by means of which the isomorphism types of complicated relations can be obtained from those of simpler ones. This arithmetic may eventually become a powerful instrument which will give us a better insight into the structural variety of relations. 

[...]

The main purpose of this monograph is just the development of the \textit{theory of ordinal addition} for arbitrary [binary] relation types.
\end{quote}
Here, ``ordinal addition" means ``ordered addition" of binary relation types, abstracting from the ordered addition of linear orders. This is in contrast to the (unordered) ``cardinal addition" that Tarski had previously systematized in his book \textit{Cardinal Algebras} \cite{TarskiCard}, abstracting from the addition of cardinals in the absence of the full axiom of choice.

After developing the theory of ordinal algebras from their axioms, Tarski proves that Lindenbaum's division theorem holds in an arbitrary ordinal algebra, and hence in the specific ordinal algebra $(LO, +, \sum, {}^*, \emptyset)$. See \cite[Theorem 1.50]{Tarski}.

Tarski's proof of the division theorem, which he labels ``Euclid's Theorem," is elegant but somewhat opaque, and despite the label not transparently connected to Euclidean division in a familiar arithmetic context such as $(\mathbb{N}, +)$ or $(\mathbb{R}, +)$. We were originally motivated to better understand the theorem, and to see its proof as a proof by division in a sense that clearly generalizes division in the natural numbers. 

This paper is the result of our research in this direction. In it we will develop the arithmetic of $(LO, +)$ on the basis of a Euclidean algorithm for pairs of linear orders $A, B$ that satisfy a one-sided weak commutativity relation. This algorithm is a generalization of the classical Euclidean algorithm on $\mathbb{N}$. Combinatorially, it is equivalent to the dynamical system arising from a two-piece generalized interval exchange transformation on the unit interval
\[
[0, 1) = A B \mapsto B A X = [0, 1),
\]
in which the maps between the corresponding copies of the intervals $A$ and $B$ are required to preserve the orientation of these intervals but not necessarily their lengths, and whose images leave a gap on one side of $[0, 1)$. Readers familiar with interval exchanges or orientation-preserving maps on the circle will recognize the mechanics of such transformations in many of our proofs. 

We will show that the algorithm yields a new proof of Lindenbaum's division theorem that directly generalizes a proof of the theorem for $(\mathbb{N}, +)$ obtained by applying its classical version. We also use it to solve two open problems from \textit{Ordinal Algebras}, one stated explicitly and one implicitly posed, concerning additively commuting pairs of linear orders; see the discussion in Sections \ref{subsect:tarskicommutingsumprob} and \ref{subsect:tarskicommutingpairsconj} below. 

In the second half of the paper, we use the systems of isomorphisms arising from runs of the Euclidean algorithm to prove representation theorems for the linear orders appearing in the course of such a run. In the last section, we use these representations to completely classify the commutative semigroups $(S, \oplus)$ that can be represented as semigroups of linear order types under the ordered sum. See the discussion in Section \ref{subsect:reppinsemigroupsintro} below.  

A version of the algorithm can be applied in an arbitrary ordinal algebra. In a forthcoming companion paper \cite{ErvinPaul}, we will show that as a consequence many of the new arithmetic results established in this paper for $(LO, +)$ hold in an arbitrary ordinal algebra. 

\subsection{Tarski's Commuting Sums Problem}\label{subsect:tarskicommutingsumprob}

Tarski devotes significant effort in \textit{Ordinal Algebras} to the question of characterizing the pairs of elements in a given ordinal algebra that additively commute. He does not explain his interest in this question, but in hindsight it is a natural one: we will see that there is a close connection between commutativity and Euclidean division in $(LO, +)$, and the same is true in an arbitrary ordinal algebra $\mathfrak{A}$. 

On page 43, Tarski poses the following question:

\theoremstyle{definition}
\newtheorem{tsp}[ldt]{Problem}
\begin{tsp}\label{tsp}
(Tarski's Sum Problem; \cite[pg. 43]{Tarski}) Suppose $A$ and $B$ are linear orders and there are natural numbers $n, m, k, l \geq 1$ such that 
\[
nA + mB \cong kB + lA.\]
Does it follow that $A + B \cong B + A$?
\end{tsp}

Tarski shows in Corollary 1.53 that the answer to Problem \ref{tsp} is yes if either $n = l$ or $m = k$, and then notes that the general problem remains open. 

In Appendix A of \textit{Ordinal Algebras}, written by C.C. Chang, Chang extends Tarski's result to get a positive answer for the case when both $n \leq l$ and $m \leq k$, leaving open the case when one of $(n < l) \wedge (m > k)$ or $(n > l) \wedge (m < k)$ holds. Regarding this case, Chang says: 
\begin{quote}
The discussion of this remaining problem seems to present considerable difficulties.
\end{quote}
See \cite[pg. 92]{Tarski}. 

Immediately after posing Problem \ref{tsp}, Tarski poses the following generalization, which asks if the existence of an isomorphism between two arbitrary sums with terms from $\{A, B\}$, one of which begins with $A$ and ends with $B$ and the other of which begins with $B$ and ends with $A$, implies that $A$ and $B$ additively commute.

\theoremstyle{definition}
\newtheorem{tgsp}[ldt]{Problem}
\begin{tgsp}\label{tgsp}
(Tarski's Generalized Sum Problem; \cite[pg. 43]{Tarski}) Suppose $A$ and $B$ are linear orders and there is an isomorphism
\[
C_0 + C_1 + \cdots + C_{n-1} \cong D_0 + D_1 + \cdots + D_{m-1}
\]
such that $n, m \geq 2$, $C_i, D_j \in \{A, B\}$ for $0 \leq i < n$ and $0 \leq j < m$, and $C_0 = D_{m-1} = A$ and $D_0 = C_{n-1} = B$. 

Does it follow that $A + B \cong B + A$?
\end{tgsp}

In Section \ref{section:commutativitylaws}, we solve Tarski's generalized sum problem for $(LO, +)$ in the affirmative, using data accrued from a run of the Euclidean algorithm for the orders $A$ and $B$ appearing in an instance of the problem. In particular, this gives a complete affirmative solution to the original sum problem \ref{tsp}. See Theorem \ref{generalizedsumproblemsoln}.  

Tarski poses Problems \ref{tsp} and \ref{tgsp} for an arbitrary ordinal algebra; the proof given in Section \ref{section:commutativitylaws} solves the restriction of these problems to $LO$. In the forthcoming \cite{ErvinPaul}, we will show that the proof can be generalized to an arbitrary ordinal algebra, giving a complete solution to Tarski's problems.

\subsection{Tarski's Commuting Pairs Conjecture}\label{subsect:tarskicommutingpairsconj}

In unpublished work from the 1930s (see the introduction to \cite{Aronszajn}, and the footnote in \cite[pg. 80]{Tarski}), Tarski made a conjecture to the effect that the only ways that a pair of linear orders $A$ and $B$ can additively commute are the obvious ones.

\theoremstyle{definition}
\newtheorem{tarconj}[ldt]{Conjecture}
\begin{tarconj}\label{tarconj}
(Tarski's Commuting Pairs Conjecture; \cite[pg. 80]{Tarski}) Suppose that $A$ and $B$ are linear orders. Then $A + B \cong B + A$ if and only if one of the following conditions holds:
\begin{itemize}
    \item[i.] There is a linear order $C$ and natural numbers $n, m \geq 1$ such that $A \cong nC$ and $B \cong mC$; 
    \item[ii.] $A + B \cong B + A \cong A$; 
    \item[iii.] $A + B \cong B + A \cong B$. 
\end{itemize}
\end{tarconj}

Tarski showed the conjecture is true if $A$ and $B$ are either countable or scattered. After appropriately generalizing the notion of a scattered element to an arbitrary ordinal algebra, he showed that the conjecture always holds for such elements (\cite[Theorem 1.65]{Tarski}). 

Lindenbaum found a counterexample to Tarski's conjecture (also unpublished). This example was later generalized by Aronszajn, who in \cite{Aronszajn} gives a structural characterization of the additively commuting pairs of linear orders $A$ and $B$. This characterization says roughly that such pairs are obtained by replacing the points in a closed interval of $\mathbb{R}$ with linear orders in a way that respects a group of translations on $\mathbb{R}$; see Theorem \ref{symmetricaronszajnrepnthm} for a precise statement of Aronszajn's result.

Tarski discusses this history in a footnote on page 80 of \textit{Ordinal Algebras}, concluding there:
\begin{quote}
Aronszajn's results, however, cannot be formulated within the arithmetic of ordinal algebras.
\end{quote}

That is, Aronszajn's characterization \ref{symmetricaronszajnrepnthm} of the additively commuting pairs of linear orders cannot be stated in the arithmetic language $\{+, \sum, {}^*, 0\}$ of ordinal algebras, i.e. in terms of sums and reversals of linear orders. This is in contrast to Tarski's conjectured characterization \ref{tarconj}, which \textit{can} be stated in the language of ordinal algebras (though of course not proved, since it fails in $LO$).    

J\'onsson retells this story in his paper \cite{Jonsson} from 1982:

\begin{quote}
Many of the deeper properties of ordered addition were listed without proofs in Lindenbaum and Tarski [1926], and these properties were developed axiomatically in Tarski [1956]. More
specifically, the operations treated there were ordinal addition, as applied to pairs of types and to $\omega$-sequences of types, and conversion. The axioms for ordinal algebras hold for the isomorphism types of arbitrary binary relations, and most of the known properties of these operations are consequences of the axioms. However, some facts about ordinal addition cannot even be formulated within this framework. A notable example is Aronszajn's [1952] characterization of commuting pairs of isomorphism types.
\end{quote}

These remarks raise the question of whether Tarski's conjecture can be revised to get a correct arithmetic characterization of the pairs satisfying $A + B \cong B + A$ in $LO$, i.e. one that is expressible in the language of ordinal algebras. 

We will show the answer is yes. There is a natural weakening of (i.) in Tarski's conjecture which is expressible in the language of ordinal algebras and, when substituted for (i.), yields a correct arithmetic characterization of the additively commuting pairs $A, B \in LO$.  

For a linear order $X$, $\omega X$ denotes the right infinite sum of $X$: 
\[
X + X + \cdots = \sum_{i} X
\]
and $\omega^* X$ denotes left-infinite sum 
\[
\cdots + X + X = \left(\sum_i X^*\right)^*.
\]

Here is the revised version of Tarski's conjecture. 

\theoremstyle{definition}
\newtheorem{revisedtarconj}[ldt]{Theorem}
\begin{revisedtarconj}\label{revisedtarconj}
Suppose that $A$ and $B$ are linear orders. Then $A + B \cong B + A$ if and only if one of the following conditions holds:
\begin{itemize}
    \item[i.] $\omega A \cong \omega B$ and $\omega^*A \cong \omega^*B$; 
    \item[ii.] $A + B \cong B + A \cong A$; 
    \item[iii.] $A + B \cong B + A \cong B$. 
\end{itemize}
\end{revisedtarconj}

Notice that if $A \cong nC$ and $B \cong mC$ as in \ref{tarconj}.(i.), then $A$ and $B$ satisfy the revised condition (i.) from Theorem \ref{revisedtarconj}. In this case, $\omega A \cong \omega B \cong \omega C$ and $\omega^* A \cong \omega^*B \cong \omega^*C$. 

In \cite{ErvinPaul}, we show that this revised arithmetic characterization of the commuting pairs holds in an arbitrary ordinal algebra. 

We also prove the following reformulation of Theorem \ref{revisedtarconj}. Below, $\leqslant_{init}$ means ``embeds onto an initial segment of" and $\leqslant_{fin}$ ``embeds onto a final segment of."

\theoremstyle{definition}
\newtheorem{symmetrizedtarconj}[lct]{Theorem}
\begin{symmetrizedtarconj}\label{symmetrizedtarconj}
$A + B \cong B + A$ if and only if one of the following conditions holds:
\begin{itemize}
    \item[i.] $\omega A \leqslant_{init} \omega B$ and $\omega^* A \leqslant_{fin} \omega^*B$; 
    \item[ii.] $\omega B \leqslant_{init} \omega A$ and $\omega^* B \leqslant_{fin} \omega^* A$. 
\end{itemize}
\end{symmetrizedtarconj}

See Theorems \ref{ABcommuteifOmegasumsinitfin} and \ref{revisedtarconjintext}. 

\subsection{Commutative semigroups of order types}\label{subsect:reppinsemigroupsintro}

The first part of the paper centers around the Euclidean algorithm defined in Section \ref{section:euclideanalgos}, and culminates in the solution to Tarski's Generalized Sum Problem in Section \ref{section:commutativitylaws} and a proof of Lindenbaum's Division Theorem in Section \ref{section:linddivisthm}. 

The second part of the paper begins in Section \ref{section:symboldynamrepnsofdivissystems}. There, we turn from the purely arithmethic results of the first part to a finer structural analysis of linear orders appearing in runs of the Euclidean algorithm. The main results of Section \ref{section:symboldynamrepnsofdivissystems} are representation theorems showing that such orders are obtained by replacing points in a closed interval of $\mathbb{R}$ in a way that respects the forward orbit equivalence relation of a pair of translations on $\mathbb{R}$; see Theorems \ref{leftrealrepnthm}, \ref{rightrealrepnthm}, \ref{symmetricrealrepnthm}. We show in Section \ref{section:aronszajnrepnthm} that these theorems may be viewed as generalizations of Aronszajn's characterization of the commuting pairs in $LO$ from \cite{Aronszajn}. 

In the final section of the paper, we use this analysis to characterize the commutative semigroups that appear as semigroups of linear order types. More precisely, if we distinguish linear orders only up to isomorphism, we may view $(LO, +)$ as a class semigroup, and ask which commutative semigroups appear as subsemigroups of $LO$. These are the semigroups we characterize in Section \ref{section:commutesemigroupsinLOrepn}. In order to state the characterization, we need the following definition. 

\theoremstyle{definition}
\newtheorem{replacementsemigroupdefnintro}[lct]{Definition}
\begin{replacementsemigroupdefnintro}\label{replacementsemigroupdefnintro}
Suppose that $X$ is a linear order and $\{S_x: x \in X\}$ is a collection of semigroups indexed by $X$ with corresponding semigroup operations $\{\oplus_x: x \in X\}$. 

Let $X(S_x)$ denote the set of ordered pairs $(x, s)$ with $x \in X$ and $s \in S_x$, and define a semigroup operation $\oplus$ on $X(S_x)$ by the following rule:
\[
\begin{array}{rcl}
    (x, s) \oplus (y, t) & = & \left\{   \begin{array}{ll}
                                        (x, s) & x > y \\
                                        (y, t) & x < y \\
                                        (x, s \oplus_x t) & x = y.
                                        \end{array}  
                                        \right.
\end{array}
\]
We call $X(S_x)$ the \textit{replacement semigroup} of $X$ by the semigroups $S_x$. 
\end{replacementsemigroupdefnintro}

We show the following. 

\theoremstyle{definition}
\newtheorem{commutativesemigroupsrepableinLOtheoremintro}[lct]{Theorem}
\begin{commutativesemigroupsrepableinLOtheoremintro}\label{commutativesemigroupsrepableinLOtheoremintro}
A commutative semigroup $(S, \oplus)$ appears as a subsemigroup of $(LO, +)$ if and only if $S$ is isomorphic to a replacement semigroup $X(S_x)$ over some linear order $X$ such that for every $x \in X$, one of the following holds:
\begin{itemize}
    \item[i.] There is a strict subgroup $G \leq (\mathbb{R}, +)$ such that $S_x$ is isomorphic to a subsemigroup of the positive cone $G_{>0}$; 
    \item[ii.] $S_x$ is isomorphic to the semigroup on one element. 
\end{itemize}
\end{commutativesemigroupsrepableinLOtheoremintro}

This is established as Theorem \ref{representablesemigroupstheorem}. 

\subsection{Organization of the paper}\label{subsect:organizationofpaper}

In Section \ref{section:notationandterminology}, we introduce the notation and terminology for working with linear orders and their sums that we will use throughout the paper. 

In Section \ref{section:partialconvexselfembeds}, we analyze the elementary dynamics of convex self-embeddings of linear orders. Such embeddings arise in applications of our Euclidean algorithm for $(LO, +)$. 

In Section \ref{section:basicarithmetic}, we establish a number of arithmetic facts concerning sums of linear orders that we will use in our later results. Most of these facts are due to Lindenbaum and Tarski. 

In Section \ref{section:euclideanalgos}, we introduce our Euclidean algorithm for pairs of linear orders $A$ and $B$ that almost additively commute (in a sense specified there), and discuss the arithmetic data about $A$ and $B$ that accrues from a run of the algorithm. In Section \ref{section:divisionsystems}, we study this arithmetic data in an abstract setting, and prove structural results about the linear orders that appear during the course of such a run. 

In Section \ref{section:commutativitylaws}, we solve Tarski's Generalized Sum Problem \ref{tgsp} for $(LO, +)$ in the affirmative, and prove some related arithmetic results. 

In Section \ref{section:linddivisthm}, we use the Euclidean algorithm to give a new proof of Lindenbaum's Division Theorem \ref{ldt}. 

In Section \ref{section:symboldynamrepnsofdivissystems}, we establish representation theorems for the linear orders appearing as arithmetic terms in the course of a run of the Euclidean algorithm. These representation theorems are the basis for the remaining work in the paper. 

In Section \ref{section:aronszajnrepnthm}, we use these theorems to prove one-sided generalizations of Aronszajn's representation from \cite{Aronszajn} for commuting pairs of linear orders. 

In Section \ref{section:revisedtarskiconj}, we use the representations to prove our revised version \ref{revisedtarconj} of Tarski's Commuting Pairs Conjecture \ref{tarconj}. 

In Section \ref{section:continuationandvaluation}, we develop an arithmetic theory of archimedean dominance and archimedean equivalence in $(LO, +)$, and show that these notions induce global valuations on $LO$ analogous to valuations in abelian orderable groups in terms of their archimedean classes. 

Section \ref{section:coloringsoflinearorders} is a preparatory section for a technical aspect of the construction in Section 14. Finally, in Section \ref{section:commutesemigroupsinLOrepn} we prove our characterization of the commutative semigroups that can be represented in $(LO, +)$.

\section*{AI use statement} AI tools were not used in the writing of this paper, nor in the formulation and proofs of its results, which are the work of the authors. 

\section*{Acknowledgments} We thank Alekos Kechris for several helpful comments.

\section{Notation and basic terminology} \label{section:notationandterminology}

A \textit{linear order} is a pair $(X, <)$ where $X$ is a set and $<$ is a strict, total order on $X$. The class of linear orders is denoted $LO$. When not otherwise specified, capital letters $A, B, X, Y, \ldots$ appearing in the text and the statements of propositions below stand for fixed linear orders.

The expression $x \leq y$ abbreviates the assertion ``$x < y$ or $x = y$."

A \textit{suborder} of $X$ is a subset $Y \subseteq X$ equipped with the inherited ordering relation from $X$. 

We write $X^*$ for the \textit{reverse} of $X$. The orders $X$ and $X^*$ share the same underlying set of points, with $x < y$ in $X^*$ if and only if $y < x$ in $X$. Note that $X^{**} = X$.

An order $X$ is \emph{dense} (or \emph{densely ordered}) if $X$ contains at least two points and whenever $x, y \in X$ and $x < y$, there is $z \in X$ with $x < z < y$. 

If $x < y$ and there is no $z$ such that $x < z < y$, then $x$ is the \textit{predecessor} of $y$ and $y$ is the \textit{successor} of $x$. We sometimes say $x$ and $y$ form a \textit{jump pair}. Thus $X$ is dense if and only if $X$ contains no jump pairs. 

A suborder $Y \subseteq X$ is \emph{dense in $X$} (or, $Y$ is a \emph{dense suborder of $X$}) if whenever $x, y \in X$ and $x < y$, then either $x$ and $y$ both belong to $Y$ or there is $z \in Y$ such that $x < z < y$.

\subsection{Intervals and cuts}\label{subsection:intervalsandcuts}

An \textit{interval} or \textit{convex subset} of $X$ is a suborder $I \subseteq X$ such that whenever $x < z < y$ and $x, y \in I$ then $z \in I$. In particular, singleton suborders $\{x\} \subseteq X$ are intervals of $X$. 

Given two intervals $I, J \subseteq X$, we write $I < J$ if $x < y$ for all $x \in I$ and $y \in J$. This defines a (strict) partial ordering on the collection of intervals of $X$ that we call the \textit{induced ordering}.

An \textit{initial segment} of $X$ is a subset $I \subseteq X$ such that whenever $y < x$ and $x \in I$ then $y \in I$. A \textit{final segment} of $X$ is a subset $J \subseteq X$ such that whenever $x < y$ and $x \in J$, then $y \in J$. Observe that initial and final segments are intervals, and that $I$ is an initial segment of $X$ if and only if $X \setminus I$ is a final segment of $X$.

A \textit{cut} in $X$ is a pair $c = (I, J)$ where $I$ is an initial segment of $X$ and $J = X \setminus I$ is the corresponding final segment. We think of $c$ as the space between $I$ and $J$. $\mathsf{Cut}(X)$ denotes the set of cuts in $X$. 

The ordering $<$ on $X$ naturally extends to an ordering on $X \cup \mathsf{Cut}(X)$. Given $x \in X$ and a cut $c = (I, J)$ in $\mathsf{Cut}(X)$, define $x < c$ if $x \in I$ and $c < x$ if $x \in J$. Given another cut $c' = (I', J') \in \mathsf{Cut}(X)$, define $c < c'$ if there is $x \in X$ such that $c < x < c'$, or equivalently, if $I \subsetneq I'$. It is straightforward to check that this extension of $<$ linearly orders $X \cup \mathsf{Cut}(X)$.

We will sometimes informally quantify over $X \cup \mathsf{Cut}(X)$ by using clauses of the form ``point or cut $c \in X$" instead of ``$c \in X \cup \mathsf{Cut}(X)$."

Suppose $K \subseteq X$ is an interval. The \textit{cut at the left} of $K$ is $c = (I, J)$, where 
\[
I = \{x \in X: \forall y \in K, x < y\}.
\]
The \textit{cut at the right} of $K$ is $d = (I', J')$ where 
\[
J' = \{x \in X: \forall y \in K, x > y\}.
\]
The cuts at the left and right of $K$ are the \textit{endcuts} of $K$. By convention, the cuts at the left and right of $X$ are $(\emptyset, X)$ and $(X, \emptyset)$ respectively. 

Cuts in $K$ may be identified with cuts in $X$ in the natural way. If $c = (I, J)$ is a cut in $K$, we identify it with $c' = (I', J') \in \mathsf{Cut}(X)$, where 
\[
I' = \{x \in X: \exists y \in I, x \leq y\}.
\]
The cuts at the left and right of $K$ are identified with the cuts at the left and right of $K$ in $X$. Under this identification, $K \cup \mathsf{Cut}(K)$ is an interval in $X \cup \mathsf{Cut}(X)$.

More generally, we may define cuts at the left and right of an arbitrary suborder of $X$. Given $Y \subseteq X$, the \textit{convex closure} of $Y$ is the interval 
\[
\textrm{conv}(Y) = \{x \in X: \exists a, b \in Y, a \leq x \leq b\}.
\]
The cuts at the left and right of $Y$ are defined as the cuts at the left and right of $\textrm{conv}(Y)$. 

We will use both points in $X$ and cuts in $X$ to define intervals in $X$. Given $a, b \in X \cup \mathsf{Cut}(X)$ with $a \leq b$, define
\begin{itemize}
    \item $[a, b] = \{x \in X: a \leq x \leq b\}$,
    \item $[a, b) = \{x \in X: a \leq x < b\}$,
    \item $(a, b] = \{x \in X: a < x \leq b\}$,
    \item $(a, b) = \{x \in X: a < x < b\}$.
\end{itemize}
Thus we are viewing the usual endpoint notation for intervals in $X \cup \mathsf{Cut}(X)$ as denoting the restriction of these intervals to $X$.

If $K$ is an interval of $X$, then 
\[
K = [c, d] = [c, d) = (c, d] = (c, d)
\]
where $c$ and $d$ are the cuts at the left and right of $K$, respectively. We will usually use closed interval notation when defining intervals by their endcuts. The phrase ``Label $K = [c, d]$" means ``Let $c$ and $d$ denote the cuts at the left and right of $K$, respectively."

For $a, b \in X \cup \mathsf{Cut}(X)$ not necessarily satisfying $a \leq b$, we write $[\{a, b\}]$ to mean $[a, b]$ when $a \leq b$ and $[b, a]$ when $a \geq b$. Analogously, we write $[\{a, b\})$, $(\{a, b\}]$, and $(\{a, b\})$ to denote the interval of points between $a$ and $b$ that excludes $b$, excludes $a$, and excludes both $a$ and $b$, respectively.

\subsection{Convex embeddings}\label{subsection:convexembeddings} 

An \textit{embedding} of $X$ into $Y$ is a map $f: X \rightarrow Y$ such that $x < y$ in $X$ implies $f(x) < f(y)$ in $Y$. Since the ordering relations that we consider are strict, embeddings are automatically injective. 

An \textit{isomorphism} is a surjective embedding. We say $X$ and $Y$ are \textit{isomorphic} and write $X \cong Y$ if there is an isomorphism $f: X \rightarrow Y$. An \textit{automorphism} of $X$ is an isomorphism $f: X \rightarrow X$. 

An \emph{order type} is an isomorphism class of linear orders. Though we will work with linear orders as opposed to order types throughout, often we are interested in a given linear order $X$ only up to isomorphism, and most of the definitions and theorems we present could be rephrased in terms of order types.

A \textit{partial embedding} of $X$ into $Y$ is an embedding $f: K \rightarrow Y$ whose domain $K$ is a suborder of $X$. Given such an embedding, $f^{-1}$ is a partial embedding of $Y$ into $X$ with domain $f[K]$. 

An embedding $f: X \rightarrow Y$ is \textit{convex} if $f[X]$ is an interval in $Y$. We write $X \leqslant_{conv} Y$ if there exists a convex embedding of $X$ into $Y$. Note that isomorphisms are convex embeddings. 

A convex embedding $f: X \rightarrow Y$ is \textit{initial} if $f[X]$ is an initial segment of $Y$ and \textit{final} if $f[X]$ is a final segment of $Y$. We write $X \leqslant_{init} Y$ if there is an initial embedding of $X$ into $Y$, and $X \leqslant_{fin} Y$ if there is a final embedding of $X$ into $Y$. Observe that all three relations $\leqslant_{conv}, \leqslant_{init}$, and $\leqslant_{fin}$ are transitive on $LO$.

If $f: X \rightarrow Y$ is a convex embedding and $c = (I, J)$ is a cut in $X$, then $(f[I], f[J])$ is a cut in $f[X]$, which as above we may view as a cut in $Y$. Conversely, every cut in $f[X]$ is of the form $(f[I], f[J])$ for some $(I, J) \in \mathsf{Cut}(X)$. Thus $f$ may be uniquely extended to a convex embedding of $X \cup \mathsf{Cut}(X)$ into $Y \cup \mathsf{Cut}(Y)$ by defining, for each $c = (I, J) \in \mathsf{Cut}(X)$,
\[
f(c) = (f[I], f[J]).
\]
We identify $f$ with this extension. 

Note then that if $K = [c, d]$ is an interval in $X$ labeled by its endcuts, the convexity of $f$ implies $f[K] = [f(c), f(d)]$. We express this by saying that the image of an interval under a convex embedding is determined by the images of its endcuts. 

\subsection{Discrete orders}\label{subsection:discreteorders}

We use both $\omega$ and $\mathbb{N}$ to denote the set of natural numbers $\{0, 1, 2, \ldots\}$ equipped with its usual ordering relation. We identify each $n \in \omega$ with the set of its predecessors $\{0, 1, \ldots, n-1\}$, and view $n$ as a suborder of $\omega$. In particular, $0$ denotes the empty order. 

The \textit{discrete orders} are the finite orders $n$, along with $\omega$, $\omega^*$, and $\mathbb{Z}$. 

\subsection{Replacements}\label{subsection:replacements}

Suppose that $\{I_x: x \in X\}$ is a collection of linear orders indexed by $X$. Define
\begin{equation*}
    X(I_x) = \{(x, i): x \in X, i \in I_x\}
\end{equation*}
and order $X(I_x)$ lexicographically by the rule $(x, i) < (y, j)$ if $x < y$ (in $X$), or $x = y$ and $i < j$ (in $I_x = I_y$). Equipped with this ordering, $X(I_x)$ is a linear order that we call the \textit{replacement} of $X$ by the orders $I_x$.

For a fixed $x \in X$, the set of pairs $\{(x, i) \in X(I_x): i \in I_x\}$ is an interval in $X(I_x)$ that is isomorphic to $I_x$. We sometimes informally refer to this interval also by $I_x$. Observe that $I_x < I_y$ in $X(I_x)$ if and only if $x < y$ in $X$.

Given a suborder $K \subseteq X$, we write $K(I_x)$ for the restriction \[
\{(x, i) \in X(I_x): x \in K\}
\]
of $X(I_x)$ to $K$. Observe that if $K$ is an interval in $X$, then $K(I_x)$ is an interval in $X(I_x)$. 

We explicitly allow the replacement orders $I_x$ to be empty. Thus $X(I_x) = K(I_x)$, where $K = \{x \in X: I_x \neq \emptyset\}$.

\subsection{Replacements up to an equivalence relation}\label{subsection:replacementsuptoE}

Suppose that $E$ is an equivalence relation on $X$. In practice, $E$ will often be the orbit equivalence relation of a collection of partial convex self-embeddings of $X$. For a given $x \in X$, the $E$-class of $x$ is denoted $[x]_E$, or simply $[x]$ when $E$ is clear from context.

Suppose $\{I_{[x]}: x \in X\}$ is a collection of linear orders indexed by the $E$-classes of $X$. We write $X(I_{[x]})$ for the replacement $X(I_x)$ of $X$ in which $I_x = I_{[x]}$ for every $x \in X$. Thus for every $x, y \in X$ with $x E y$, we have 
\[
I_x = I_y = I_{[x]} = I_{[y]}.
\]
A replacement of the form $X(I_{[x]})$ is a \emph{replacement of $X$ up to $E$}. 

More generally, we call a replacement $X(I_x)$ a replacement up to $E$, and denote it by $X(I_{[x]})$, even if we only have $x E y \Rightarrow I_x \cong I_y$ (instead of $x E y \Rightarrow I_x = I_y$) for all $x, y \in X$.

\subsection{Sums}\label{subsection:sums}

A replacement $X(I_x)$ is sometimes called an \emph{ordered sum} and denoted $\sum_{x \in X} I_x$. We usually use replacement notation; however, when $X$ is isomorphic to a discrete order we often use summation notation instead.

When $X = 2 = \{0, 1\}$, the replacement $X(I_x)$ is called the \emph{sum} of the orders $I_0$ and $I_1$ and denoted $I_0 + I_1$. We write $(LO, +)$ for the class of linear orders equipped with the sum. 

Similarly, when $X = n = \{0, 1, \ldots, n-1\}$, we write $I_0 + I_1 + \cdots + I_{n-1}$ for the replacement $n(I_x)$ and call such a replacement an \emph{$n$-sum}. We also write
\begin{equation*}
\begin{array}{rclcl}
\omega(I_x) & = & \sum_{x \in \omega} I_x & = & I_0 + I_1 + \cdots \\
\omega^*(I_x) & = & \sum_{x \in \omega^*} I_x & = & \cdots + I_1 + I_0 \\
\mathbb{Z}(I_x) & = & \sum_{x \in \mathbb{Z}} I_x & = & \cdots + I_{-1} + I_0 + I_1 + I_2 + \cdots 
\end{array}
\end{equation*}
and call replacements of these forms \emph{$\omega$-sums}, \emph{$\omega^*$-sums}, and \emph{$\mathbb{Z}$-sums}, respectively. To distinguish replacements of discrete orders from more general replacements $X(I_x)$, we refer to them as \textit{discrete sums}. 

Sums are closely related to cuts. If $c = (I, J)$ is a cut in $X$, then $X \cong I + J$. In the other direction, in an order of the form $I_0 + I_1$ there is a cut between the initial segment corresponding to $I_0$ and final segment corresponding to $I_1$. We call this cut the \emph{cut at the $+$ sign}. Similarly, we may refer to cuts at $+$ signs in longer discrete sums as well. 

We have $(I_0 + I_1) + I_2 \cong I_0 + (I_1 + I_2) \cong I_0 + I_1 + I_2$ for all $I_0, I_1, I_2 \in LO$. The sum is an associative operation in this sense.

More generally, we have the associative law 
\begin{equation*}
X(I_x)(J_{(x, i)}) \cong X(I_x(J_{(x, i)})).
\end{equation*}
That is, if we replace the points in a replacement $X(I_x)$ by orders $J_{(x, i)}$, then the resulting order $X(I_x)(J_{(x, i)})$ is isomorphic to the order obtained by first forming, for each $x \in X$, the replacement $I_x(J_{(x, i)})$, and then forming the replacement $X(I_x(J_{(x, i)}))$.

If $I_x = Y$ for all $x \in X$, then the replacement $X(I_x)$ is called the \emph{lexicographic product} of $X$ and $Y$ and denoted $XY$. Discrete products of the form $nY$, $\omega Y$, $\omega^* Y$, and $\mathbb{Z}Y$ will arise frequently in our study of $(LO, +)$.

Products distribute over sums on the right: 
\[
(A + B)C \cong AC + BC
\]
for all $A, B, C \in LO$. More generally, products distribute over replacements on the right: given a replacement $X(I_x)$ and an order $A$, we have 
\[
X(I_x)A \cong X(I_xA).
\]
This follows from the general associativity law above. Products do not in general distribute over replacements (or sums) on the left.

\subsection{Condensations}\label{subsection:condensations}

An equivalence relation $\sim$ on $X$ is a \emph{condensation} if every $\sim$-class is an interval of $X$.

We reserve the symbol $\sim$ for condensations, and use $E$ for more general equivalence relations. Instead of $[x]_{\sim}$ or $[x]$, we write $\gamma_{\sim}(x)$ or $\gamma(x)$ for the $\sim$-class of a given $x \in X$. The set of $\sim$-classes is denoted both by $X/{\sim}$ and $X/{\gamma}$, depending on context. We view $\gamma$ as a map $\gamma: X \rightarrow X /{\sim}$ called the \textit{condensation map}. 

Since $X/{\sim}$ is a collection of pairwise disjoint intervals in $X$, the induced order on $X/{\sim}$ is a linear order. We view $X/{\sim}$ as equipped with this order. Then the condensation map $\gamma$ is a surjective order-homomorphism of $X$ onto $X/{\sim}$, i.e. a surjective map satisfying $x < y \Rightarrow \gamma(x) \leq \gamma(y)$. 

Conversely, any surjective order-homomorphism $\gamma: X \rightarrow L$ from $X$ onto an order $L$ induces a condensation of $X$ defined by
\[
\textrm{$x \sim_{\gamma} y$ if $\gamma(x) =\gamma(y)$}
\]
and we have $X/{\gamma} \cong L$.

Given a replacement $X(I_x)$, the map $\gamma: X(I_x) \rightarrow X$ defined by $\gamma(x, i) = x$ is a surjective order-homomorphism. It induces the \textit{replacement condensation} on $X(I_x)$ defined by the rule 
\[
\textrm{$(x, i) \sim (y, j)$ if $x = y$.}
\]
The replacement condensation condenses each of the orders $I_x$ back to a point, yielding $X(I_x)/{\gamma} \cong X$. 

\section{Partial convex self-embeddings of linear orders}\label{section:partialconvexselfembeds}

Arithmetic identities over $(LO, +)$ are often naturally represented by systems of partial convex embeddings on a linear order $X$. In this section, we examine the basic dynamical structure of such embeddings.

\theoremstyle{definition}
\newtheorem{partialconvselfembeddef}[lct]{Definition}
\begin{partialconvselfembeddef}\label{partialconvselfembeddef}
A \textit{partial convex self-embedding of $X$} is a convex embedding $f: K \rightarrow X$ whose domain $K$ is an interval in $X$. 
\end{partialconvselfembeddef}

For the remainder of this section, let $f: K \rightarrow X$ denote a fixed partial convex self-embedding of $X$ with domain $K$. Then $f^{-1}$ is a partial convex self-embedding of $X$ with domain $f[K]$. 

Label $K = [c, d]$ by its endcuts and identify $f$ with its unique extension to $K \cup \mathsf{Cut}(K)$. We view $f$ as defined at $c$ and $d$, so that $f[K] = [f(c), f(d)]$. 

The \textit{extent of $f$}, denoted $\textrm{ext}(f)$, is the convex closure of $K \cup f[K]$. 

If $K \cap f[K] \neq \emptyset$, then $\textrm{ext}(f) = K \cup f[K]$, in which case we say $f$ is \textit{overlapping}. In practice, most of the partial convex self-embeddings we consider are overlapping.

\theoremstyle{definition}
\newtheorem{pushcompexpslidedefn}[lct]{Definition}
\begin{pushcompexpslidedefn}\label{pushcompexpslidedefn}
\phantom{.}
\begin{itemize}
    \item[i.] If $f[K] < K$ (i.e. $f(d) \leq c$), then $f$ is a \textit{left push}; if $f[K] > K$ (i.e. $f(c) \geq d$), then $f$ is a \textit{right push}.
    \item[ii.] If $f[K] \subseteq K$ (i.e. $f(c) \geq c$ and $f(d) \leq d$), then $f$ is a \textit{compression}.
    \item[iii.] If $K \subseteq f[K]$ (i.e. $f(c) \leq c$ and $f(d) \geq d$), then $f$ is an \textit{expansion}.
    \item[iv.] If $f(c) < c < f(d) < d$, then $f$ is a \textit{left slide}; if $c < f(c) < d < f(d)$, then $f$ is a \textit{right slide}. 
\end{itemize}
\end{pushcompexpslidedefn}

Notice that $f$ is overlapping exactly when $f$ is a compression, expansion, or slide.

\subsection{Orbits}\label{subsect:orbits}

Fix a point or cut $x \in K \cup f[K]$. For an integer $n \geq 0$, $f^n(x)$ denotes the $n$-fold iterate $(ff \cdots f)(x)$, whenever this iterate is defined. For $n < 0$, $f^n(x)$ denotes $(f^{-1})^{(-n)}(x)$.

The \textit{orbit of $x$ under $f$}, denoted $o_f(x)$, is the set of defined $f$-iterates of $x$: 
\begin{equation*}
    o_f(x) = \{f^n(x): n \in \mathbb{Z}, \textrm{$f^n(x)$ is defined}\}.
\end{equation*}
Notice $o_f(y) = o_f(x)$ if and only if $y \in o_f(x)$, if and only if $y = f^n(x)$ for some $n \in \mathbb{Z}$. Note also that $o_f(x) = o_{f^{-1}}(x)$.

Since $x \in K \cup f[K]$, at least one of $f(x)$ and $f^{-1}(x)$ is defined and thus included in $o_f(x)$. Either there exists a least $N \geq 0$ such that $f^{N+1}(x)$ is undefined, or $f^n(x)$ is defined for all $n \geq 0$. Symmetrically, there is either a least $M \geq 0$ such that $f^{-(M+1)}(x)$ is undefined, or $f^m(x)$ is defined for all $m \leq 0$. Thus $o_f(x)$ has one of the following forms:
\begin{itemize}
    \item[\phantom{}] $\{\ldots, f^{-1}(x), x, f(x), f^2(x), \ldots\}$,
    \item[\phantom{}] $\{f^{-M}(x), f^{-M+1}(x), \ldots\}$,
    \item[\phantom{}] $\{\ldots, f^{N-1}(x), f^N(x)\}$,
    \item[\phantom{}] $\{f^{-M}(x), f^{-M+1}(x), \ldots, x, \ldots, f^{N-1}(x), f^N(x)\}$.
\end{itemize}

If $f(x) > x$, then $f$ is \textit{increasing} at $x$. In this case, since $f$ is order-preserving it follows inductively that $f^n(x) < f^{n+1}(x)$ for every $n \in \mathbb{Z}$ for which $f^n(x)$ and $f^{n+1}(x)$ are both defined. Thus $o_f(x)$ has one of the following forms:
\begin{itemize}
    \item[\phantom{}] $\ldots < f^{-1}(x) < x < f(x) < f^2(x) < \ldots$,
    \item[\phantom{}] $f^{-M}(x) < f^{-M+1}(x) < \ldots$, 
    \item[\phantom{}] $\ldots < f^{N-1}(x) < f^N(x)$,
    \item[\phantom{}] $f^{-M}(x) < f^{-M+1}(x) < \ldots < x < \ldots < f^{N-1}(x) < f^N(x)$.
\end{itemize}
We describe these situations respectively by saying $o_f(x)$ is an increasing $\mathbb{Z}$-\textit{orbit}, $\omega$-\textit{orbit}, $\omega^*$-\textit{orbit}, or \textit{finite} orbit. Orbits of the latter three forms are \textit{truncated}, on the left, right, and both sides, respectively.

If $f(x) < x$, then $f$ is \textit{decreasing} at $x$. The possible forms for $o_f(x)$ in this case are obtained from the increasing forms by replacing $<$ with $>$.

If $f(x) = x$, then $f$ \textit{fixes} $x$. Then $o_f(x) = \{x\}$ is a \textit{singleton orbit}. By convention, singleton orbits are non-truncated. 

\subsection{Orbitals of endcuts}\label{subsect:orbsofends}

\theoremstyle{definition}
\newtheorem{fjumps}[lct]{Definition}
\begin{fjumps}\label{fjumps}
Suppose $x \in \mathsf{Cut}(K)$. Define
\[
A_{x, f} = [\{x, f(x)\}].
\]
We call $A_{x, f}$ the \textit{$f$-jump at the cut $x$}.
\end{fjumps}

When $x$ is a point, $A_{x, f}$ is not defined. 

Any two cuts from the same orbit have isomorphic jumps. More precisely, given $x \in \mathsf{Cut}(K)$ and $n \in \mathbb{Z}$, if $f^n(x)$ and $f^{n+1}(x)$ are both defined we have 
\[
f^n[A_{x, f}] = [\{f^n(x), f^{n+1}(x)\}] = A_{f^{n}(x), f}.
\]
Since $f^n[A_{x, f}] \cong A_{x, f}$, we have $A_{x, f} \cong A_{f^n(x), f}$.

Our next goal is to define, for each point or cut $x \in K \cup f[K]$, the \textit{orbital} $O_f(x)$ of $x$ under $f$. 

Abstractly, $O_f(x)$ is the smallest convex subset of $\textrm{ext}(f)$ containing $o_f(x)$ that is invariant under both $f$ and $f^{-1}$. Often, it is just $\textrm{conv}(o_f(x))$, but for points or cuts near the endcuts of $K$, it may be a strict superset of $\textrm{conv}(o_f(x))$. 

We will take a more concrete approach to the definition by first defining the orbitals $O_f(c)$ and $O_f(d)$ of the endcuts of $K$ directly, and then working inward. An important arithmetic feature of orbitals is that they may be decomposed as discrete sums of $f$-jumps.

We begin with the left endcut $c$. The definition of $O_f(c)$ depends on the structure of the orbit $o_f(c)$. 

\underline{Case (1.)}: $f(c) = c$. In this case, define $O_f(c) = o_f(c) = \{c\}$. 

\underline{Case (2.)}: $f(c) > c$. 

$\rightarrow$ \underline{Case (2.1)}: $o_f(c)$ is finite. In this case, let $N \geq 1$ be least such that $f^{N+1}(c)$ is undefined. Since $f$ is increasing at $c$ and $c$ is the left endcut of the domain of $f$, $f^{-1}(c)$ is undefined. Thus $o_f(c)$ consists of the points
\[
c < f(c) < f^2(c) < \ldots < f^{N-1}(c) < f^N(c).
\]
We claim $f^{N-1}(c) \leq d < f^N(c)$. Indeed, if $f^N(c) \leq d$, then $f^N(c) \in [c, d] = K$ and $f^{N+1}(c)$ would be defined; and if $d < f^{N-1}(c)$, then $f^N(c)$ would lie strictly above $f(d)$ and hence outside of $f[K]$, both impossible. Thus $f^N(c) \leq f(d)$ as well. 

Tracking backward, we have $f^{N-1-k}(c) \leq f^{-k}(d) < f^{N-k}(c) \leq f^{-k + 1}(d)$ for $0 \leq k \leq N-1$. it follows that $o_f(d)$ consists of the points
\[
f^{-(N-1)}(d) < f^{-(N-2)}(d) < \ldots < d < f(d).
\]
Since $f$ is increasing at both $c$ and $d$, $\textrm{ext}(f) = [c, f(d)]$. 

The cut sequence 
\[
c < f(c) < f^2(c) < \ldots < f^{N-1}(c) < f^N(c) \leq f(d)
\]
spans $\textrm{ext}(f) = [c, f(d)]$ and yields the sum decomposition
\begin{equation}\label{finitaryleftorbitaldecomp}
\begin{array}{rcl}
\textrm{ext}(f) & \cong & [c, f(c)] + [f(c), f^2(c)] + \cdots + [f^N(c), f(d)] \\
& \cong & A_{c, f} + A_{f(c), f} + \cdots + A_{f^{N-1}(c), f} + B \\
& \cong & NA_{c, f} + B
\end{array}
\end{equation}
where $B = [f^N(c), f(d)]$. Observe that $B$ is isomorphic to the initial segment $[c, f^{-(N-1)}(d)]$ of $A_{c, f}$. 

We have that $f$ is overlapping when $N > 2$, or when $N = 2$ and $B \neq \emptyset$. The case when $N = 2$ and $B = \emptyset$ occurs when $f(c) = d$, i.e. when $K$ and $f[K]$ are directly adjacent but non-intersecting intervals in $X$. 

We view $f$ as acting on the sum decomposition \ref{finitaryleftorbitaldecomp} by mapping each copy $A_{f^n(c), f}$ of $A_{c, f}$ onto the copy $A_{f^{n+1}(c), f}$ to its right for $n < N-1$, and mapping the initial segment $f^{-1}[B] = [f^{N-1}(c), d]$ of $A_{f^{N-1}(c), f}$ onto $B$. 

Notice that \ref{finitaryleftorbitaldecomp} restricts to a decomposition of $K = \textrm{dom}(f) = [c, d]$:
\begin{equation}\label{finitaryleftorbitaldecompdomain}
\begin{array}{rcl}
[c, d] & \cong & [c, f(c)] + \cdots + [f^{N-2}(c), f^{N-1}(c)] + [f^{N-1}(c), d] \\
& \cong & A_{c, f} + A_{f(c), f} + \cdots + A_{f^{N-2}(c), f} + f^{-1}[B] \\
& \cong & (N-1)A_{c, f} + B,
\end{array}
\end{equation}

Now let $C = [d, f^N(c)]$, so that 
\[
A_{d, f} = [d, f(d)] \cong [d, f^N(c)] + [f^N(c), f(d)] = C + B.
\]

Observe
\[
A_{f^{N-1}(c), f} \cong [f^{N-1}(c), d] + [d, f^N(c)] = f^{-1}[B] + C \cong B + C. 
\]
Hence $A_{c, f} \cong B + C$ as well. Thus we obtain from the sum decomposition \ref{finitaryleftorbitaldecomp} the alternate decomposition in terms of the jump $A_{d, f}$:
\[
\begin{array}{rcl}
\textrm{ext}(f) & \cong & NA_{c,f} + B \\
& \cong & A_{c, f} + A_{c,f} + \cdots + A_{c,f} + B \\
& \cong & (B + C) + (B + C) + \cdots (B + C) + B \\
& \cong & B + (C + B) + \cdots + (C + B) + (C + B) \\
& \cong & B + NA_{d, f}.
\end{array}
\]
This decomposition can also be obtained directly from the cut sequence
\[
c < f^{-N+1}(d) < f^{-N+2}(d) < \ldots < d < f(d)
\]
consisting of the left endcut $c$ and the cuts in $o_f(d)$. 

In this case, we define the orbital of $c$ to be
\[
O_f(c) = \textrm{ext}(f) = [c, f(d)],
\]
and call the orbital \textit{finitary}.

Further, for every point or cut $x \in K \cup f[K]$, we define $O_f(x) = O_f(c)$. In particular, $O_f(d) = O_f(c)$. From above, we have two sum representations of this orbital in terms of the jumps $A_{c, f}$ and $A_{d, f}$:
\begin{equation}\label{finitaryorbitaldecompsboth}
\begin{array}{rcl}
O_f(c) & \cong & NA_{c,f} + B \\
& \cong & B + NA_{d, f} \\
& \cong & O_f(d).
\end{array}
\end{equation}

$\rightarrow$ \underline{Case (2.2)}: $o_f(c)$ is infinite. In this case, $o_f(c)$ is an increasing $\omega$-orbit consisting of the points
\[
c < f(c) < f^2(c) < \ldots.
\]
In this case we define the orbital of $c$ to be the convex closure of its orbit,
\[
O_f(c) = \textrm{conv}(o_f(c))
\]
and call $O_f(c)$ an \textit{$\omega$-orbital}. Further, for every point or cut $x \in O_f(c)$ we also define $O_f(x) = O_f(c)$.  

Since $f^n(c)$ is defined for every $n \in \omega$, $o_f(c)$ is a subset of the domain $K = [c, d]$ of $f$. Thus $O_f(c)$ is initial not only in $\textrm{ext}(f)$ but in $K$. 

We have that $O_f(c)$ decomposes as an $\omega$-sum of the initial jump $A_{c, f}$:
\begin{equation}\label{omegaorbitaldecomposition}
\begin{array}{rcl}
O_f(c) & \cong & [c, f(c)] + [f(c), f^2(c)] + [f^2(c), f^3(c)] +  \cdots \\
& \cong & A_{c, f} + A_{c, f(c)} + A_{c, f^2(c)} + \cdots \\ 
& \cong & \omega A_{c, f}.
\end{array}
\end{equation}
We view $f$ as acting on this decomposition by mapping each copy of $A_{c, f}$ onto the copy to its right. Observe that $f \upharpoonright O_f(c)$ is a final self-embedding of $O_f(c)$ onto its final segment beginning at $f(c)$. 

\underline{Case (3.)}: $f(c) < c$. In this case, write $f' = f^{-1}$ and $c' = f(c)$. Then $f'$ is a partial convex self-embedding of $X$ on the domain $K' = f[K]$ with left endcut $c'$, and $f'(c') > c'$. Use the conventions above to define $O_{f'}(c')$. Then define $O_f(c) = O_{f'}(c') = O_{f^{-1}}(f(c))$.  

There are two cases as before, and they yield sum decompositions of $O_f(c)$ of the same forms, the difference being that now $f$ is decreasing $O_f(c)$ instead of increasing. 

If $o_f(c)$ is finite, $O_f(c)$ is finitary and we again have decompositions of the form
\[
O_f(c) = O_f(d) = \textrm{ext}(f) = [f(c), d] \cong NA_{c, f} + B \cong B + NA_{d, f}.
\]

If $o_f(c)$ is infinite, $O_f(c)$ is an $\omega$-orbital and we again have
\[
O_f(c) \cong \omega A_{c, f}.
\]

We now turn to the right endcut $d$. The conventions for defining the orbital $O_f(d)$ are symmetric with the definition of $O_f(c)$.

If $o_f(d) = \{d\}$, then $O_f(d) = \{d\}$.

If $o_f(d)$ is finite, then $O_f(d) = \textrm{ext}(f) = O_f(c)$. We say the orbital is \textit{finitary} and we have sum decompositions of the form
\begin{equation*}
\begin{array}{rcl}
O_f(d) & \cong & B + NA_{d,f} \\
& \cong & NA_{c, f} + B \\
& \cong & O_f(c).
\end{array}
\end{equation*}

If $o_f(d)$ is an $\omega^*$-orbit, we define $O_f(d) = \textrm{conv}(o_f(d))$. We say $O_f(d)$ is an \textit{$\omega^*$-orbital} and we have the decomposition
\[
O_f(d) \cong \omega^*A_{d, f}.
\]

We view $f$ as acting on these sum decompositions for $O_f(d)$ by shifting each copy (or partial copy) of $A_{d, f}$ in the domain $K$ onto the copy to its left (when $f$ is decreasing) or right (when $f$ is increasing). 

It remains to define the orbitals $O_f(x)$ for points or cuts $x \in K \cup f[K]$ not belonging to $O_f(c)$ or $O_f(d)$. 

For such an $x$ to exist, we must have $O_f(c) < O_f(d)$. Suppose this is the case and let $c'$ denote the cut at the right of $O_f(c)$ (which coincides with $c$ if $O_f(c)$ is a singleton), and let $d'$ denote the cut at the left of $O_f(d)$.

Label $O_f(c) = [c, c']$ by its endcuts. From the definition of $O_f(c)$, we have $f[O_f(c)] = [f(c), c']$, that is, $f[[c, c']] = [f(c), c']$. Since $f$ is convex, we also have $f[[c, c']] = [f(c), f(c')]$, and it follows $f(c') = c'$. That is, $f$ fixes the cut at the right of $O_f(c)$. Symmetrically, $f(d') = d'$. More generally, $f(b) = b$ for any cut $b$ on a non-truncated side of an $f$-orbit. 

Fix a point or cut $x \in K \cup f[K]$ with $O_f(c) < x < O_f(d)$. Then $c' \leq x \leq d'$. Hence $f^n(c') \leq f^n(x) \leq f^n(d')$ for any $n \in \mathbb{Z}$, which gives $c' \leq f^n(x) < d'$ for any $n \in \mathbb{Z}$. In particular, $f^n(x)$ is defined for every $n \in \mathbb{Z}$, and it follows that $o_f(x)$ is either a $\mathbb{Z}$-orbit or a singleton.

In either case, we define $O_f(x) = \textrm{conv}(o_f(x))$. Notice that $O_f(y) = O_f(x)$ if and only if $y \in \textrm{conv}(o_f(x))$, if and only if $\textrm{conv}(o_f(y)) = \textrm{conv}(o_f(x))$. 

If $o_f(x)$ is an increasing $\mathbb{Z}$-orbit, then it consists of the sequence
\[
\ldots < f^{-1}(x) < x < f(x) < f^2(x) < \ldots.
\]
If $x$ is a cut, this sequence yields the $\mathbb{Z}$-sum decomposition
\[
\begin{array}{rcl}
O_f(x) & \cong & \cdots + [f^{-1}(x), x] + [x, f(x)] + [f(x), f^2(x)] + \cdots \\
& \cong & \cdots + A_{f^{-1}(x), f} + A_{x, f} + A_{f(x), f} + \cdots \\
& \cong & \mathbb{Z}A_{x, f}.
\end{array}
\]
We view $f$ as acting on this sum by shifting each copy of $A_{x, f}$ onto the copy to its right. The situation is symmetric if $o_f(x)$ is a decreasing $\mathbb{Z}$-orbit. 

In summary, $\textrm{ext}(f)$ decomposes as a pairwise disjoint collection of orbitals: an initial orbital $O_f(c)$, which is either a singleton orbital, finitary orbital, or $\omega$-orbital, a final orbital $O_f(d)$, which is either a singleton, finitary, or an $\omega^*$-orbital, and a (possibly empty) collection of intermediate orbitals $O_f(x)$, each of which are either singletons or $\mathbb{Z}$-orbitals. 

Further, we have $O_f(c) = O_f(d)$ if and only if one (equivalently, both) of these orbitals is finitary, in which case $O_f(c) = O_f(d) = \textrm{ext}(f)$ and $f$ is either increasing or decreasing on all of $K$. 

From this, we get the following relationships between the orbital decomposition of $f$ and its designation from Definition \ref{pushcompexpslidedefn}:
\begin{itemize}
    \item[i.] If $f$ is a push, then $O_f(c) = O_f(d)$.
    \item[ii.] If $f$ is an expansion or compression, then $f$ is non-increasing at one of the endcuts $c, d$ and non-decreasing at the other; hence $O_f(c) < O_f(d)$. 
    \item[iii.] If $f$ is a slide, then it may be that either $O_f(c) = O_f(d)$ or $O_f(c) < O_f(d)$. 
\end{itemize}

\section{Basic arithmetic in $(LO, +)$}\label{section:basicarithmetic}

Here we collect the basic arithmetic facts about discrete sums in $(LO, +)$ that we will need in later sections. Many of these results are due to one or both of Lindenbaum \cite{LindenbaumTarski} and Tarski \cite{Tarski}. 

Perhaps the fundamental observation about sums of linear orders is Lindenbaum's Theorem \ref{XcongAXBiffABladder}, along with its alternate form \ref{XcongAXBiffcongAXandXB}. In the language of the previous section, it says that if $f: K \rightarrow X$ is a compression that does not fix either endcut of its domain $K = [c, d]$, then $O_f(c)$ is an $\omega$-orbital and $O_f(d)$ is an $\omega^*$-orbital. 

\textit{Cut conventions}: When working with isomorphisms or convex embeddings of discrete sums, we will often label cuts in the target order corresponding to cuts at the $+$ signs in the embedded sum. 

For example, suppose $X = [c,d]$ is a linear order labeled by its endcuts, and $X \cong A_0 + A_1 + A_2$. A cut sequence in $X$ witnessing this isomorphism is a sequence of cuts of the form 
\begin{equation*}
c = a_0 < a_1 < a_2 < a_3 = d
\end{equation*}
where $[a_i, a_{i+1}] \cong A_i$ for $0 \leq i \leq 2$. Sometimes we double label such cuts. For example, if $X \cong A + B$ we may write 
\begin{equation*}
c = a_l < a_r = b_l < b_r = d
\end{equation*} 
for a sequence witnessing this decomposition, i.e. for which we have $[a_l, a_r] \cong A$ and $[b_l, b_r] \cong B$. When there is no danger of confusion, we sometimes informally identify such orders with their corresponding segments in $X$. For example, given a map $f: X \rightarrow Y$, we may write $f[A]$ for the image of $[a_l, a_r]$ under $f$.

\theoremstyle{definition}
\newtheorem{leqslantstosums}[lct]{Proposition}
\begin{leqslantstosums}\label{leqslantstosums} \phantom{.}
\begin{itemize}
    \item[i.] $X \leqslant_{init} Y$ if and only if $Y \cong X + A$ for some $A \in LO$;
    \item[ii.] $X \leqslant_{fin} Y$ if and only if $Y \cong A + X$ for some $A \in LO$;
    \item[iii.] $X \leqslant_{conv} Y$ if and only if $Y \cong A + X + B$ for some $A, B \in LO$. 
\end{itemize}
\end{leqslantstosums}
\begin{proof}
Clear.
\end{proof}

\theoremstyle{definition}
\newtheorem{bothinitbothfin}[lct]{Proposition}
\begin{bothinitbothfin}\label{bothinitbothfin}
\phantom{.}
\begin{itemize}
    \item[i.] If $X \leqslant_{init} Z$ and $Y \leqslant_{init} Z$, then either $X \leqslant_{init} Y$ or $Y \leqslant_{init} X$. 
    \item[ii.] If $X \leqslant_{fin} Z$ and $Y \leqslant_{fin} Z$, then either $X \leqslant_{fin} Y$ or $Y \leqslant_{fin} X$. 
\end{itemize}
\end{bothinitbothfin}
\begin{proof}
Clear.
\end{proof}

\theoremstyle{definition}
\newtheorem{BinitCimpliesABinitAC}[lct]{Proposition}
\begin{BinitCimpliesABinitAC}\label{BinitCimpliesABinitAC}
\phantom{.}
\begin{itemize}
    \item[i.] If $B \leqslant_{init} C$ then $A + B \leqslant_{init} A + C$.
    \item[ii.] If $B \leqslant_{fin} C$, then $B + A \leqslant_{fin} C + A$. 
\end{itemize}
\end{BinitCimpliesABinitAC}
\begin{proof}
For (i.): $B \leqslant_{init} C$ implies $C \cong B + D$ for some $D$ by Proposition \ref{leqslantstosums}. Hence $A + C \cong A + B + D$, and so $A + B \leqslant_{init} A + C$ by the same proposition.

The proof for (ii.) is symmetric. 
\end{proof}

\theoremstyle{definition}
\newtheorem{AplusBconvXplusY}[lct]{Proposition}
\begin{AplusBconvXplusY}\label{AplusBconvXplusY}
If $A + B \leqslant_{conv} X + Y$, then either $A \leqslant_{conv} X$ or $B \leqslant_{conv} Y$.  
\end{AplusBconvXplusY}

\begin{proof}
Label $X + Y = [c, d]$ by its endcuts. Let $p$ denote the cut at the $+$ sign in $X + Y$, so that $[c, p] \cong X$ and $[p, d] \cong Y$. Let 
\[
a_l < a_r = b_l < b_r
\]
be a sequence of cuts in $X + Y$ witnessing $A + B \leqslant_{conv} X + Y$, i.e. such that $[a_l, a_r] \cong A$ and $[b_l, b_r] \cong B$. If $a_r \leq p$ then $[a_l, a_r] \subseteq [c, p]$, which witnesses $A \leqslant_{conv} X$. Symmetrically, if $a_r \geq p$, then $B \leqslant_{conv} Y$. 
\end{proof}

\theoremstyle{definition}
\newtheorem{sumAiconvsumBi}[lct]{Corollary}
\begin{sumAiconvsumBi}\label{sumAiconvsumBi}
Fix a natural number $n \geq 1$. If \[
A_0 + A_1 + \cdots + A_{n-1} \leqslant_{conv} B_0 + B_1 + \cdots + B_{n-1},
\]
then $A_i \leqslant_{conv} B_i$ for some $i < n$.
\end{sumAiconvsumBi}
\begin{proof}
By induction on $n$. The claim for $n = 1$ is immediate. 

Assume the result for $n$ and suppose $A_0 + A_1 + \cdots + A_n \leqslant_{conv} B_0 + B_1 + \cdots + B_n$. Writing this as 
\[
(A_0 + \cdots + A_{n-1}) + A_n \leqslant_{conv} (B_0 + \cdots + B_{n-1}) + B_n,
\]
Proposition \ref{AplusBconvXplusY} gives that either $A_0 + \cdots + A_{n-1} \leqslant_{conv} B_0 + \cdots + B_{n-1}$, which yields the claim by induction, or $A_n \leqslant_{conv} B_n$, which also yields the claim. 
\end{proof}

\theoremstyle{definition}
\newtheorem{directedrefinementpost}[lct]{Proposition}
\begin{directedrefinementpost} \label{directedrefinementpost}
(Tarski's directed refinement postulate; cf. \cite[1.1.III]{Tarski}): 
\begin{itemize}
    \item[i.] If $\sum_{i \in \omega} A_i \leqslant_{conv} X + Y$, either $\sum_{i \in \omega} A_i \leqslant_{conv} X$ or $\sum_{i\in \omega} A_{i+k_0} \leqslant_{conv} Y$ for some $k_0 \in \omega$. 
    \item[ii.] If $\sum_{i \in \omega^*} A_i \leqslant_{conv} X + Y$, either $\sum_{i \in \omega^*} A_i \leqslant_{conv} Y$ or $\sum_{i\in \omega^*} A_{i+k_0} \leqslant_{conv} X$ for some $k_0 \in \omega$. 
\end{itemize}
\end{directedrefinementpost}

\begin{proof}
For (i.): Label $X + Y = [c, d]$ and let $p$ be the cut at the $+$ sign. Fix an increasing sequence of cuts 
\[
a_0 < a_1 < \ldots
\]
in $X + Y$ such that $[a_i, a_{i+1}] \cong A_i$ for all $i \in \omega$, and let $b$ denote the cut at the right of this sequence. If $b \leq p$, then the sequence $a_i$ lies entirely in $X$ and witnesses $\sum_{i \in \omega} A_i \leqslant_{conv} X$. If $b > p$, let $k_0$ be least such that $a_{k_0} \geq p$. Then the tail-sequence $a_{i + k_0}$ lies in $Y$ and witnesses $\sum_{i\in \omega} A_{i+k_0} \leqslant_{conv} Y$.

The proof for (ii.) is symmetric.
\end{proof}

Induction yields the following $n$-sum version of Proposition \ref{directedrefinementpost}. 

\theoremstyle{definition}
\newtheorem{directedrefinementfornsums}[lct]{Corollary}
\begin{directedrefinementfornsums}\label{directedrefinementfornsums}
Fix a natural number $n \geq 1$. 
\begin{itemize}
    \item[i.] If $\sum_{i \in \omega} A_i \leqslant_{conv} X_0 + X_1 + \cdots + X_{n-1}$, then there is a natural number $k_0$ and index $i_0 < n$ such that $\sum_{i \in \omega} A_{i+k_0} \leqslant_{conv} X_{i_0}$. 
    \item[ii.] If $\sum_{i \in \omega^*} A_i \leqslant_{conv} X_0 + X_1 + \cdots + X_{n-1}$, then there is a natural number $k_0$ and index $i_0 < n$ such that $\sum_{i \in \omega^*} A_{i+k_0} \leqslant_{conv} X_{i_0}$.
\end{itemize}
\end{directedrefinementfornsums}

Since for any natural number $k_0$ we have $\omega A =\sum_{i \in \omega} A_i = \sum_{i \in \omega} A_{i+k_0}$, where $A_i = A$ for all $i \in \omega$ (and symmetrically, for $\omega^* A$), Corollary \ref{directedrefinementfornsums} yields the following. 

\theoremstyle{definition}
\newtheorem{directedrefinementforomegaA}[lct]{Corollary}
\begin{directedrefinementforomegaA}\label{directedrefinementforomegaA}
Fix a natural number $n \geq 1$. 
\begin{itemize}
    \item[i.] If $\omega A \leqslant_{conv} X_0 + X_1 + \cdots + X_{n-1}$, then $\omega A \leqslant_{conv} X_{i_0}$ for some $i_0 < n$.
    \item[ii.] If $\omega^* A \leqslant_{conv} X_0 + X_1 + \cdots + X_{n-1}$, then $\omega^* A \leqslant_{conv} X_{i_0}$ for some $i_0 < n$.
\end{itemize}
\end{directedrefinementforomegaA}

\theoremstyle{definition}
\newtheorem{AplusomegaA}[lct]{Lemma}
\begin{AplusomegaA}\label{AplusomegaA}
The following identities hold:
\begin{itemize}
    \item[i.] $A + \omega A \cong \omega A$;
    \item[ii.] $\omega^*B + B \cong B$. 
\end{itemize}
\end{AplusomegaA}
\begin{proof}
For (i.): Informally, we have
\begin{equation*}
A + \omega A \cong A + (A + A + \cdots) \cong A + A + A + \cdots = \omega A.
\end{equation*}

Formally, by definition we have 
\[
\begin{array}{rcl}
\omega A & = & \{(n, a): n \in \omega, a \in A\} \\
A + \omega A & = & \{(0, a): a \in A\} \cup \{(1, (n, a)): n \in \omega, a \in A\}
\end{array}
\]
both with the lexicographic ordering. The map $f: A + \omega A \rightarrow \omega A$ defined by $f(0, a) = (0, a)$ and $f(1, (n, a)) = (n+1, a)$ is an isomorphism. 

The proof for (ii.) is symmetric. 
\end{proof}

One way to justify the informal associativity argument in the proof above is to show directly the visually obvious fact that $1 + \omega \cong \omega$, and then quote right distributivity of the product to write $A + \omega A \cong (1 + \omega)A \cong \omega A$. 

We will sometimes make use of similarly basic (and visually clear) additive and multiplicative absorption properties of the discrete orders $\omega, \omega^*$, and $\mathbb{Z}$, sometimes without explicit justification. For example, for any natural number $n \geq 1$ we have $\omega n \cong \omega$, $\omega^* n \cong \omega^*$, and $\mathbb{Z} n \cong \mathbb{Z}$, which yields the following proposition. 

\theoremstyle{definition}
\newtheorem{omeganXcongomegaX}[lct]{Proposition}
\begin{omeganXcongomegaX}\label{omeganXcongomegaX}
For any $n \in \omega$, the following identities hold:
\begin{itemize}
    \item[i.] $\omega(nX) \cong \omega X$;
    \item[ii.] $\omega^*(nX) \cong \omega^*X$;
    \item[iii.] $\mathbb{Z}(nX) \cong \mathbb{Z}X$.
\end{itemize}
\end{omeganXcongomegaX}

\theoremstyle{definition}
\newtheorem{XcongAXBiffABladder}[lct]{Theorem}
\begin{XcongAXBiffABladder}\label{XcongAXBiffABladder}
(Lindenbaum; \cite[3.1]{LindenbaumTarski}) $X \cong A + X + B$ if and only if there is a linear order $C$ such that $X \cong \omega A + C + \omega^* B$. 
\end{XcongAXBiffABladder}

\begin{proof}
The backward direction follows from Lemma \ref{AplusomegaA} and the associativity of the sum. So assume $X \cong A + X + B$ and label $X = [c, d]$ by its endcuts. We assume $A$ and $B$ are non-empty; the argument is similar when one or both of these orders is empty. 

Fix a sequence of cuts
\begin{equation*}
c = a_l < a_r = x_l < x_r = b_l < b_r = d
\end{equation*}
in $X$ witnessing the isomorphism $X \cong A + X + B$, i.e. such that $[a_l, a_r] \cong A$, $[x_l, x_r] \cong X$, and $[b_l, b_r] \cong B$. 

Fix an isomorphism $f: X \rightarrow [x_l, x_r]$. Then $f$ is a convex self-embedding of $X$ with 
\begin{equation*}
A_{c, f} = [c, f(c)] = [a_l, a_r] \cong A,
\end{equation*}
and similarly $A_{d, f} \cong B$. Hence $O_f(c) \cong \omega A$ and $O_f(d) \cong \omega^* B$, and these orbitals are disjoint since $f$ is increasing at $c$ and decreasing at $d$. Letting $c_l, c_r$ denote the cuts at the right of $O_f(c) = [c, c_l]$ and left of $O_f(d) = [c_r, d]$ respectively, and letting $C = [c_l, c_r]$, we have 
\begin{equation*}
\begin{array}{rcl}
X & \cong & [c, c_l] + [c_l, c_r] + [c_r, d] \\
& \cong & \omega A + C + \omega^* B,
\end{array}
\end{equation*}
as claimed.
\end{proof}

\theoremstyle{definition}
\newtheorem{absorbsleftrightdef}[lct]{Definition}
\begin{absorbsleftrightdef}\label{absorbsleftrightdef}
\phantom{.}
\begin{itemize}
    \item[i.] $X$ \textit{absorbs $A$ on the left} if $A + X \cong X$;
    \item[ii.] $X$ \textit{absorbs $A$ on the right} if $X + A \cong X$. 
\end{itemize}
\end{absorbsleftrightdef}

\theoremstyle{definition}
\newtheorem{absorbsleftrightcoro}[lct]{Corollary}
\begin{absorbsleftrightcoro}\label{absorbsleftrightcoro}
(cf. \cite[1.25]{Tarski})
\begin{itemize}
    \item[i.] $X$ absorbs $A$ on the left if and only if $\omega A \leqslant_{init} X$;
    \item[ii.] $X$ absorbs $A$ on the right if and only if $\omega^* A \leqslant_{fin} X$. 
\end{itemize}
\end{absorbsleftrightcoro}

\begin{proof}
For (i.): Since $\omega A \leqslant_{init} X$ implies $X \cong \omega A + C$ for some linear order $C$, the forward direction follows from Lemma \ref{AplusomegaA}. For the backward direction, apply Theorem \ref{XcongAXBiffABladder} with $B = \emptyset$. 

The proof for (ii.) is symmetric. 
\end{proof}

\theoremstyle{definition}
\newtheorem{XcongAXBiffcongAXandXB}[lct]{Theorem}
\begin{XcongAXBiffcongAXandXB}\label{XcongAXBiffcongAXandXB}
(Lindenbaum; \cite[3.2]{LindenbaumTarski}) $X \cong A + X + B$ if and only if $X \cong A + X$ and $X \cong X + B$. 
\end{XcongAXBiffcongAXandXB}

\begin{proof}
If $X \cong A + X + B$, then $\omega A \leqslant_{init} X$ and $\omega^*B \leqslant_{fin} X$ by Theorem \ref{XcongAXBiffABladder}. By Corollary \ref{absorbsleftrightcoro}, $A + X \cong X + B \cong X$.

Conversely, if $A + X \cong X$ and $X + B \cong X$, then $X \cong A + (X + B) \cong A + X + B$.
\end{proof}

\theoremstyle{definition}
\newtheorem{CSBLO}[lct]{Corollary}
\begin{CSBLO} \label{CSBLO}
(Lindenbaum; \cite[3.3]{LindenbaumTarski}) If $X \leqslant_{init} Y$ and $Y \leqslant_{fin} X$ then $X \cong Y$. 
\end{CSBLO}

\begin{proof}
The hypotheses give $Y \cong X + B$ and $X \cong A + Y$ for some orders $A$ and $B$. Hence $X \cong A + X + B$, which gives $X \cong X + B$ by Theorem \ref{XcongAXBiffcongAXandXB}, and so $X \cong Y$.  
\end{proof}

\theoremstyle{definition}
\newtheorem{XinitXfinYconv}[lct]{Corollary}
\begin{XinitXfinYconv}\label{XinitXfinYconv}
If $X \leqslant_{init} Y$, $X \leqslant_{fin} Y$, and $Y \leqslant_{conv} X$, then $X \cong Y$. 
\end{XinitXfinYconv}

\begin{proof}
The hypotheses $X \leqslant_{fin} Y$ and $Y \leqslant_{conv} X$ give that $Y \cong A + X$ and $X \cong C + Y + D$ for some $A, C, D$. Thus $X \cong C + A + X + D$, which gives $X \cong C + A + X$ by Theorem \ref{XcongAXBiffcongAXandXB}. Hence $X \cong C + Y$, and so $Y \leqslant_{fin} X$. But then since $X \leqslant_{init} Y$, we have $X \cong Y$ by Corollary \ref{CSBLO}.
\end{proof}

\theoremstyle{definition}
\newtheorem{XcongAXMXB}[lct]{Corollary}
\begin{XcongAXMXB} \label{XcongAXMXB}
If $X \cong A + X + M + X + B$ then $X \cong X + M + X$. 
\end{XcongAXMXB}

\begin{proof}
Twice applying Theorem \ref{XcongAXBiffcongAXandXB}, we have
\begin{equation*}
    \begin{array}{rcl}
    X \cong A + X + (M + X + B) & \Rightarrow & X \cong X + (M + X + B)\\
    & & \phantom{X} \cong (X + M) + X + B, \\
    & \Rightarrow & X \cong (X + M) + X \\
    & & \phantom{X} \cong X + M + X,
    \end{array}
\end{equation*} 
as claimed. 
\end{proof}

\theoremstyle{definition}
\newtheorem{2XconvXiff2XcongX}[lct]{Corollary}
\begin{2XconvXiff2XcongX} \label{2XconvXiff2XcongX}
$2X \cong X$ if and only if $2X \leqslant_{conv} X$. 
\end{2XconvXiff2XcongX}

\begin{proof}
As isomorphisms are convex embeddings, the forward direction is immediate. If $2X \leqslant_{conv} X$, then $X \cong A + X + X + B$ for some $A, B$. Now apply Corollary \ref{XcongAXMXB} with $M = \emptyset$. 
\end{proof}

\theoremstyle{definition}
\newtheorem{splittingdefn}[lct]{Definition}
\begin{splittingdefn}\label{splittingdefn}
$X$ is \textit{splitting} if $2X \cong X$. 
\end{splittingdefn}

Corollary \ref{2XconvXiff2XcongX} says that $X$ is splitting if and only if $2X \leqslant_{conv} X$. Observe that $X$ is splitting if and only if $nX \cong mX$ for every pair of integers $n, m \geq 1$. 

\theoremstyle{definition}
\newtheorem{splittingdichthm}[lct]{Theorem}
\begin{splittingdichthm} \label{splittingdichthm}
(Tarski; \cite[1.47]{Tarski}) The following are equivalent:
\begin{itemize}
    \item[i.] There exist natural numbers $n < m$ such that $mX \leqslant_{conv} nX$.
    \item[ii.] For every pair of natural numbers $n, m \geq 1$, $nX \cong mX$ (i.e., $X$ is splitting). 
\end{itemize}
\end{splittingdichthm}

\begin{proof}
Assume (i.). Since $mX \cong nX + (m-n)X$ we have $nX \leqslant_{init} mX$. Symmetrically, $nX \leqslant_{fin} mX$. Since also $mX \leqslant_{conv} nX$, Corollary \ref{XinitXfinYconv} implies $mX \cong nX$. Writing this as $(m-n)X + nX \cong nX$, it follows from Theorem \ref{XcongAXBiffABladder} that $\omega (m-n)X \leqslant_{init} nX$, which gives $\omega X \leqslant_{init} nX$. 

From this and Corollary \ref{directedrefinementforomegaA} it follows $\omega X \leqslant_{conv} X$. Since $2X \leqslant_{init} \omega X$, we have $2X \leqslant_{conv} X$. Hence $2X \cong X$ by Corollary \ref{2XconvXiff2XcongX}, which gives (ii.).

The reverse direction is immediate. 
\end{proof}

Theorem \ref{splittingdichthm} immediately yields the following, which says that the finite multiples of a given linear order are either pairwise distinct or all isomorphic. 

\theoremstyle{definition}
\newtheorem{splittingdichcor}[lct]{Corollary}
\begin{splittingdichcor}\label{splittingdichcor}
(The splitting dichotomy) Exactly one of the following holds:
\begin{itemize}
    \item[i.] For every pair of natural numbers $n \neq m$, $nX \not\cong mX$; 
    \item[ii.] For every pair of natural numbers $n, m \geq 1$, $nX \cong mX$ (i.e., $X$ is splitting). 
\end{itemize}
\end{splittingdichcor}

The following theorem gives left and right-sided finite cancellation laws for $LO$.

\theoremstyle{definition}
\newtheorem{leqslantinitfincancellationthm}[lct]{Theorem} 
\begin{leqslantinitfincancellationthm}\label{leqslantinitfincancellationthm}
(Tarski; \cite[1.45]{Tarski}) Fix a natural number $n \geq 1$. 
\begin{itemize}
    \item[i.] If $nA \leqslant_{init} nB$, then $A \leqslant_{init} B$;
    \item[ii.] If $nA \leqslant_{fin} nB$, then $A \leqslant_{fin} B$.
\end{itemize}
\end{leqslantinitfincancellationthm}

\begin{proof}
For (i.): By induction on $n$. The claim is immediate for $n = 1$. Assume the claim for $n$, and suppose $(n+1) A \leqslant_{init} (n+1)B$. 

Write $(n+1)A = A_0 + A_1 + \cdots + A_n$ and $(n+1)B = B_0 + B_1 + \cdots + B_n$ where $A_i = A$ and $B_i = B$ for all $i \leq n$. Fix an initial embedding $f: (n+1)A \rightarrow (n+1)B$. Then either $f[A_0]$ is initial in $B_0$, which witnesses $A \leqslant_{init} B$ and we are done, or $B_0$ is initial in $f[A_0]$, which witnesses $B \leqslant_{init} A$. 

In the second case $f[A_1 + \cdots + A_n]$ is convex in $B_1 + \cdots + B_n$, which witnesses $nA \leqslant_{conv} nB$. Then Corollary \ref{sumAiconvsumBi} implies $A \leqslant_{conv} B$. Thus in this second case we have both that $A \cong B + X$ and $B \cong Y + A + Z$ for some $X, Y, Z \in LO$. Substituting, this gives $A \cong Y + A + Z + B$, which yields $A \cong Y + A$ by \ref{XcongAXBiffcongAXandXB}. Hence $B \cong A + Z$, and so $A \leqslant_{init} B$, and we are again done. 

The proof for (ii.) is symmetric. 
\end{proof}

Lindenbaum's cancellation theorem \ref{lct} follows from Theorem \ref{leqslantinitfincancellationthm}. 

\theoremstyle{definition}
\newtheorem{lctintext}[lct]{Theorem}
\begin{lctintext}\label{lctintext}
Suppose there is a natural number $n \geq 1$ such that $nA \cong nB$. Then $A \cong B$. 
\end{lctintext}

\begin{proof}
Since $nA \cong nB$ implies $nA \leqslant_{init} nB$, we have $A \leqslant_{init} B$ by Theorem \ref{leqslantinitfincancellationthm}. On the other hand, $nA \cong nB$ also implies $nB \leqslant_{fin} nA$, which gives $B \leqslant_{fin} A$. Now $A \cong B$ follows from \ref{CSBLO}.
\end{proof}

Central to our study of $(LO, +)$ will be various notions of ``archimedean comparability" between orders $A$ and $B$. A reasonable candidate for such a notion would be to define $A$ and $B$ as comparable if $A$ is convexly embeddable in some finite multiple $nB$ of $B$, and $B$ is in turn convexly embeddable in some finite multiple $mA$ of $A$. 

For the problems we consider, it turns out to be more useful to view $A$ and $B$ as comparable when a connected \textit{pair} of $A$'s can be embedded in some finite multiple of $B$, and vice versa. The following proposition says that for orders that are comparable in this sense, one is splitting if and only if the other is splitting.

\theoremstyle{definition}
\newtheorem{2AconvnB2BconvmAsplittingiff}[lct]{Lemma}
\begin{2AconvnB2BconvmAsplittingiff} \label{2AconvnB2BconvmAsplittingiff}
Suppose $2A \leqslant_{conv} nB$ and $2B \leqslant_{conv} mA$ for some natural numbers $n, m \geq 1$. Then $A$ is splitting if and only if $B$ is splitting. 
\end{2AconvnB2BconvmAsplittingiff}

\begin{proof}
Suppose first that $A$ is splitting. Then $2A \cong mA$, and so $2B \leqslant_{conv} A$. We also have $2A \cong nA$, and so $nA \leqslant_{conv} nB$, which yields $A \leqslant_{conv} B$ by \ref{sumAiconvsumBi}. By transitivity of $\leqslant_{conv}$, $2B \leqslant_{conv} B$. Then $B$ is splitting by \ref{2XconvXiff2XcongX}, as desired.  

The argument for the converse is symmetric.  
\end{proof}

\theoremstyle{definition}
\newtheorem{simconvdefn}[lct]{Definition}
\begin{simconvdefn}\label{simconvdefn}
Define $A \sim_{conv} B$ if there exist natural numbers $n, m \geq 1$ such that $2A \leqslant_{conv} nB$ and $2B \leqslant_{conv} mA$. 
\end{simconvdefn}

Thus Lemma \ref{2AconvnB2BconvmAsplittingiff} says that whenever $A \sim_{conv} B$, $A$ is splitting if and only if $B$ is splitting. 

\section{Euclidean algorithms for almost commuting pairs}\label{section:euclideanalgos}

We define Euclidean algorithms for pairs of linear orders $A, B$ that, in a one-sided sense, almost additively commute. The left-sided version of the algorithm applies to pairs satisfying the relation 
\[
B + A \leqslant_{init} A + B.
\]
The symmetric right-sided version applies to pairs satisfying 
\[
B + A \leqslant_{fin} A + B.
\]
These algorithms are the basis for the subsequent work in the paper.

\subsection{Division setups}\label{subsect:divissetups}

Fix $A, B \in LO$ such that $B + A \leqslant_{init} A + B$. Label $A + B = [c, d]$ by its endcuts. Fix a cut sequence
\[
c = a_{0l} < a_{0r} = b_{0l} < b_{0r} = d
\]
where $a_{0r} = b_{0l}$ is the cut at the $+$ sign in $A + B$, so that $[a_{0l}, a_{0r}] \cong A$ and $[b_{0l}, b_{0r}] \cong B$. Denote these intervals by $A_0$ and $B_0$, respectively. 

Since $B + A \leqslant_{init} A + B$, there is also a cut sequence
\[
c = b^{1l} < b^{1r} = a^{1l} < a^{1r} \leq d
\]
in $A + B$ where $[b^{1l}, b^{1r}] \cong B$ and $[a^{1l}, a^{1r}] \cong A$. We write $B_1$ and $A_1$ for these intervals. Let $X$ denote the extra segment $[a^{1r}, d]$. Then this labeling witnesses $A + B \cong B + A + X$. 

Fix isomorphisms $\tau: A_0 \rightarrow A_1$ and $\rho: B_1 \rightarrow B_0$, i.e. $\tau: [a_{0l}, a_{0r}] \rightarrow [a^{1l}, a^{1r}]$ and $\rho: [b^{1l}, b^{1r}] \rightarrow [b_{0l}, b_{0r}]$. Then $\tau$ and $\rho$ are partial convex self-embeddings of $A+B$ that are both increasing at the left endcut $c$. 

Observe $\tau(c) = \tau(a_{0l}) = a^{1l} = b^{1r}$ so that $B_1 = [c, \tau(c)]$. Hence $B_0 = \rho[B_1] = [\rho(c), \rho\tau(c)]$. Similarly, $A_0 = [c, \rho(c)]$ and hence $A_1 = [\tau(c), \tau\rho(c)]$. The fact that $B_1 + A_1$ is initial in $A_0 + B_0$ is expressed by the inequality $\tau\rho(c) \leq \rho\tau(c) = d$. 

Either $\tau(c) \leq \rho(c)$ or $\tau(c) > \rho(c)$ depending on whether $B_1$ sits initially in $A_0$ or vice versa. In the first case, we say $B$ is the \textit{(left) divisor} of the \textit{(left) dividend} $A$ in the relation $B + A \leqslant_{init} A + B$, with the terminology reversed in the second case. 

\theoremstyle{definition}
\newtheorem{leftsetupdefn}[lct]{Definition}
\begin{leftsetupdefn}\label{leftsetupdefn}
Suppose that $I = [c, d]$ is a linear order labeled by its endcuts, and $\tau$ and $\rho$ are partial convex self-embeddings of $I$ satisfying the following \textit{left setup conditions}:
\begin{itemize}
    \item[i.] $c \in \textrm{dom}(\rho) \cap \textrm{dom}(\tau)$,
    \item[ii.] $\rho$ and $\tau$ are increasing at $c$,
    \item[iii.] $\rho(c) \in \textrm{dom}(\tau)$ and $\tau(c) \in \textrm{dom}(\rho)$,
    \item[iv.] $\tau\rho(c) \leq \rho\tau(c) = d$. 
\end{itemize}
The triple $(I =[c, d]; \rho, \tau)$ is a \textit{left division setup}.

If $\tau(c) \leq \rho(c)$, the setup is in \textit{Case (I.)}. 

If $\tau(c) \geq \rho(c)$, the setup is in \textit{Case (II.)}.
\end{leftsetupdefn}

Fix a left division setup $(I = [c, d]; \rho, \tau)$. Labeling $A = [c, \rho(c)]$ and $B = [c, \tau(c)]$, we have $\tau[A] = [\tau(c), \tau\rho(c)] \cong A$ and $\rho[B] = [\rho(c), \rho\tau(c)] \cong B$. The cut sequence $c < \rho(c) < \rho\tau(c) = d$ decomposes $I$ as a sum:
\begin{equation*}\label{AB}
\begin{array}{rcl}
I & \cong & [c, \rho(c)] + [\rho(c), \rho\tau(c)] \\
& \cong & A + \rho[B] \\
& \cong & A + B. 
\end{array}
\end{equation*}

Label $X = [\tau \rho(c), \rho \tau(c)]$. Then the sequence $c < \tau(c) < \tau\rho(c) \leq \rho\tau(c) = d$ yields the decomposition
\begin{equation*}\label{BAX}
\begin{array}{rcl}
I & \cong & [c, \tau(c)] + [\tau(c), \tau\rho(c)] + [\tau\rho(c), \rho\tau(c)] \\
& \cong & B + \tau[A] + X \\
& \cong & B + A + X. 
\end{array}
\end{equation*}

Thus the left setup $(I; \rho, \tau)$ witnesses a relation of the form $B + A \leqslant_{init} A + B$. Conversely, by the preceding discussion such a relation (along with a choice of cuts and isomorphisms witnessing it) determines a pair of functions $\rho, \tau$ that satisfy the left setup conditions on $I = A + B$. 

We will also call a relation of the form $B + A \leqslant_{init} A + B$ (or equivalently, $A + B \cong B + A + X$) a left division setup, with the caveat that an actual left setup is not determined by such a relation until a pair of cut sequences and isomorphisms witnessing the relation are fixed. 

If they are fixed and labeled as above, then the resulting left setup is in Case (I.) if $B$ is the divisor of the setup and $A$ the dividend, and Case (II.) if the roles are reversed. Said another way, Case (I.) corresponds to the dividend being adjacent to the extra segment $X$ in the isomorphism $A + B \cong B + A + X$, and Case (II.) to the divisor being adjacent to $X$. 

Right division setups are defined symmetrically. 

\theoremstyle{definition}
\newtheorem{rightsetupdefn}[lct]{Definition}
\begin{rightsetupdefn}\label{rightsetupdefn}
Suppose that $I = [c, d]$ is a linear order labeled by its endcuts, and $\tau$ and $\rho$ are partial convex self-embeddings of $I$ satisfying the following \textit{right setup conditions}:
\begin{itemize}
    \item[i.] $d \in \textrm{dom}(\rho) \cap \textrm{dom}(\tau)$,
    \item[ii.] $\rho$ and $\tau$ are decreasing at $d$,
    \item[iii.] $\rho(d) \in \textrm{dom}(\tau)$ and $\tau(d) \in \textrm{dom}(\rho)$,
    \item[iv.] $c = \rho\tau(d) \leq \tau\rho(d)$. 
\end{itemize}
The triple $(I =[c, d]; \rho, \tau)$ is a \textit{right division setup}.

If $\tau(d) \geq \rho(d)$, the setup is in \textit{Case (I.)}. 

If $\tau(d) \leq \rho(d)$, the setup is in \textit{Case (II.)}.
\end{rightsetupdefn}

If we label the $\rho$-jump $[\rho(d), d]$ (now beginning from the right) by $A$ and the $\tau$ jump [$\tau(d), d]$ by $B$, then a right division setup corresponds to a relation of the form $A + B \leqslant_{fin} B + A$, or equivalently $X + A + B \cong B + A$. 

If $\rho(d) \leq \tau(d)$, then $B$ is the \textit{right divisor} and $A$ the \textit{right dividend} of the system, and the roles reverse if $\tau(d) \leq \rho(d)$. Again, Case (I.) corresponds to the dividend being adjacent to the extra segment $X$, and Case (II.) the divisor. 

A \textit{symmetric setup} is a left setup $(I = [c, d]; \rho, \tau)$ satisfying $\tau\rho(c) = \rho\tau(c) = d$. Symmetric setups correspond to commuting isomorphisms $A + B \cong B + A$. Since there is no gap segment $X$ for symmetric setups, the distinction between Cases (I.) and (II.) vanishes. 

\subsection{Stage of the left Euclidean algorithm} \label{subsect:leftalgostage}
We define a stage of the left-sided Euclidean algorithm and then walk through an instance of such a stage. A stage for the right-sided algorithm is defined symmetrically.  

\theoremstyle{definition}
\newtheorem{leftalgostage}[lct]{Definition}
\begin{leftalgostage}\label{leftalgostage}
A stage of the \textit{left-sided Euclidean algorithm} is executed as follows: 

INPUT: a left setup $(I = [c, d]; \rho, \tau)$. 
\begin{itemize}
    \item[--] CASE (I.):
        \begin{itemize}
            \item[--] Subcase (I.i): $\tau^n(c) < \rho(c)$ for every $n \geq 0$. STOP.
            \item[--] Subcase (I.ii): $\tau^N(c) = \rho(c)$ for some $N \geq 1$. STOP.
            \item[--] Subcase (I.iii): There exists $N \geq 1$ such that $\tau^N(c) < \rho(c) < \tau^{N+1}(c)$.
            \begin{itemize}
                \item[$\rightarrow$] Let $c' = \tau^N(c)$;
                \item[$\rightarrow$] Let $I' = [\tau^N(c), d] = [c', d]$;
                \item[$\rightarrow$] Let $\rho' = \rho\tau^{-N}$ and $\tau' = \tau$. 
                \item[$\rightarrow$] GOTO next stage on input $(I' = [c', d]; \rho', \tau')$.
            \end{itemize}
        \end{itemize}
    \item[--] CASE (II.): 
        \begin{itemize}
            \item[--] Subcase (II.i): $\rho^n(c) < \tau(c)$ for every $n \geq 0$. STOP.
            \item[--] Subcase (II.ii): $\rho^N(c) = \tau(c)$ for some $N \geq 1$. STOP.
            \item[--] Subcase (II.iii): There exists $N \geq 1$ such that $\rho^N(c) < \tau(c) < \rho^{N+1}(c)$. 
            \begin{itemize}
                \item[$\rightarrow$] Let $c' = \rho^N(c)$;
                \item[$\rightarrow$] Let $I' = [\rho^N(c), d] = [c', d]$;
                \item[$\rightarrow$] Let $\tau' = \tau \rho^{-N}$ and $\rho' = \rho$. 
                \item[$\rightarrow$] GOTO next stage on input $(I' = [c', d]; \rho', \tau')$.
            \end{itemize}
        \end{itemize}
\end{itemize}
\end{leftalgostage}

Consider an input left setup $(I = [c, d]; \rho, \tau)$. Departing slightly from our conventions above, let $A = [c, \rho(c)]$ and $B = [c, \tau(c)]$. Let $X = [\tau\rho(c), \rho\tau(c)] = [\tau\rho(c),d]$. Suppose first the setup is in Case (I.). 

In Subcase (I.i), the cut sequence
\[
c < \tau(c) < \tau^2(c) < \ldots
\]
traverses the $\omega$-orbital $O_{\tau}(c) \cong \omega [c, \tau(c)] = \omega B$. Since the iterates $\tau^n(c)$ lie below $\rho(c)$, this orbital is initial in $A$. Hence $\omega B \leqslant_{init} A$, and so $B + A \cong A$. In this case we say the algorithm \textit{terminates in an absorbed factor}. 

We remark that Subcase (I.i) must occur if $\tau\rho(c) \leq \rho(c)$ (i.e. if $\tau$ is non-increasing at $\rho(c)$, i.e. if $\rho[B]$ is a final segment of $X$) as then $\tau$ is a compression of $A$. 

In Subcase (I.ii), the cut sequence
\[
c < \tau(c) < \ldots < \tau^N(c) = \rho(c)
\]
yields the decomposition
\begin{equation*}
\begin{array}{rcl}
A & = & [c, \rho(c)] \\
& \cong & [c, \tau(c)] + [\tau(c), \tau^2(c)] + \cdots + [\tau^{N-1}(c), \tau^N(c)] \\
& = & B + \tau[B] + \tau^2[B] + \cdots + \tau^{N-1}[B] \\
& \cong & NB,
\end{array}
\end{equation*}
and we say the algorithm \textit{terminates in a divisor}. 

In Subcase (I.iii), let $A' = [\tau^N(c), \rho(c)] = [c', \rho(c)]$. We view $A'$ as the remainder in the division of $A$ by $B$ from the left. Notice that we have 
\begin{equation*}
\begin{array}{rcl}
A & = & [c, \rho(c)] \\
& \cong & [c, \tau(c)] + [\tau(c), \tau^2(c)] + \cdots + [\tau^{N-1}(c), \tau^N(c)] + [\tau^N(c), \rho(c)] \\
& \cong & NB + A'.
\end{array}
\end{equation*}

Now let $B' = \tau^N[B] = [\tau^N(c), \tau^{N+1}(c)]$, and notice that we have $A' = [c', \rho'(c')]$ and $B' = [c', \tau'(c')]$. Observe that $I'$ is traversed by the cut sequence
\[
\tau^N(c) < \rho(c) < \tau^{N+1}(c) < \tau\rho(c) \leq \rho\tau(c) = d.
\]
which is the same as the cut sequence
\[
c' < \rho'(c') < \tau'(c') < \tau'\rho'(c') \leq \rho'\tau'(c') = d.
\]
Thus the input $(I'=[c', d]; \rho', \tau')$, on which we GOTO the next stage, is a left setup in Case (II.).

To express this arithmetically, observe that the subsequence $c'< \rho'(c') < \rho'\tau'(c')$ yields the decomposition
\begin{equation*}
\begin{array}{rcl}
I' & \cong & [c', \rho'(c')] + [\rho'(c'), \rho'\tau'(c')] \\
& = & A' + \rho'[B'] \\
& \cong & A' + B'. 
\end{array}
\end{equation*}
whereas the subsequence $c' < \tau'(c') < \tau'\rho'(c') \leq \rho'\tau'(c')$ yields the decomposition
\begin{equation*}
\begin{array}{rcl}
I' & \cong & [c', \tau'(c')] + [\tau'(c'), \tau'\rho'(c')] + [\tau'\rho'(c'), \rho'\tau'(c')] \\
& = & B' + \tau'[A'] + X \\
& \cong & B' + A' + X.
\end{array}
\end{equation*}
Since $A'$ is initial in $B'$, we are set up to divide $B'$ by $A'$ from the left. This is a Case (II.) setup since $A'$ is adjacent to the extra segment $X$. 

Suppose now we begin in Case (II.). The discussion is almost symmetric with Case (I.). 

In Subcase (II.i), $\rho$ witnesses that $B$ has an initial segment isomorphic to $\omega A$. We note that since $\rho$ is increasing at $\tau(c)$ (as $\rho\tau(c) = d > \tau(c)$), Subcase (II.i) is never forced to occur in the same sense as is possible for Subcase (I.i).

In Subcase (II.ii), $\rho$ witnesses $B \cong NA$.

In Subcase (II.iii), define the remainder $B' = [\rho^N(c), \tau(c)] = [c', \tau(c)]$. Then
\begin{equation*}
B = [c, \tau(c)] \cong [c, \rho^N(c)] + [\rho^N(c), \tau(c)] \cong NA + B'. 
\end{equation*}

It is straightforward to check that $(I' = [c', d]; \rho', \tau')$, the input data on which we GOTO the next stage, is a left setup in Case (I.). To express this arithmetically, let $A' = [c', \rho'(c')] \cong A$. Then on one hand we have 

\begin{equation*}
\begin{array}{rcl}
I' & = & [c', \rho'\tau'(c')] \\
& \cong & [c', \rho'(c')] + [\rho'(c'), \rho'\tau'(c')] \\
& = & A' + \rho[B'] \\
& \cong & A' + B',
\end{array}
\end{equation*}
and on the other,
\begin{equation*}
\begin{array}{rcl}
I' & \cong & [c', \tau'(c')] + [\tau'(c'), \tau'\rho'(c')] + [\tau'\rho'(c'), \rho'\tau'(c')] \\
& = & B' + \tau'[A'] + X \\
& \cong & B' + A' + X.
\end{array}
\end{equation*}
Since $B'$ is initial in $A'$, it is set up to divide $A'$ from the left, which is a Case (I.) setup. 

\subsection{Run of the left algorithm}\label{subsect:leftalgorun}

Beginning at Stage 0, on an input left setup $(I; \rho, \tau)$ the Euclidean algorithm runs until termination at some finite stage, or for infinitely many stages. We describe the arithmetic information yielded by the algorithm in both cases. 

Since they will feature in our later results, including the revised version \ref{revisedtarconj} of Tarski's conjecture \ref{tarconj}, we will be especially interested in the one-sided infinite discrete sums $\omega A$ and $\omega^* A$ of orders $A$ appearing as an additive term at some stage of the algorithm.

Suppose the algorithm does not terminate before Stage $k$. Let $I_k$ denote the interval from the input setup at Stage $k$, and let $A_k$ and  $A_{k+1}$ respectively denote the dividend and divisor at Stage $k$. Then either
\[
I_k \cong A_k + A_{k+1} \cong A_{k+1} + A_k + X 
\]
or
\[
I_k \cong A_k + A_{k+1} + X \cong A_{k+1} + A_k
\]
depending on whether the input setup at Stage $k$ is in Case (I.) or (II.). 

\underline{Non-termination at Stage $k$}: If the algorithm does not terminate at Stage $k$, define $I_{k+1} = I_k'$ in the notation above. Then for some $N_k \geq 1$ we have
\[
A_k \cong N_k A_{k+1} + A_{k+2}
\]
where $A_{k+2} = A_k'$ in the notation above, and either
\[
I_{k+1} \cong A_{k+2} + A_{k+1} \cong A_{k+1} + A_{k+2} + X
\]
or 
\[
I_{k+1} \cong A_{k+2} + A_{k+1} + X \cong A_{k+1} + A_{k+2}
\]
depending on whether the output setup at Stage $k$ is in Case (II.) or (I.). 

\underline{Termination at Stage $k$}: Suppose now the algorithm terminates at Stage $k$. From the previous stages we have the system
\begin{equation}\label{euclideanisosfinite}
\begin{array}{rcl}
A_0 & \cong & N_0 A_1 + A_2 \\
A_1 & \cong & N_1 A_2 + A_3 \\
& \vdots &  \\
A_{k-1} & \cong & N_{k-1}A_k + A_{k+1}.
\end{array}
\end{equation}

$\rightarrow$ \underline{Termination in a divisor}: If the algorithm terminates in a divisor, we have $A_k \cong N_k A_{k+1}$ for some $N_k \geq 1$. We note that if $k > 0$, then since the input setup at Stage $k$ does not have $\rho(c) = \tau(c)$, termination in a divisor implies that in fact $N_k \geq 2$.

By inductively back substituting into the isomorphisms from the system \ref{euclideanisosfinite} (just as in the usual Euclidean algorithm), we may write each $A_i$ for $i \leq k$ as a multiple $A_i \cong M_i A_{k+1}$ of $A_{k+1}$. Thus for every $i \leq k$ we have
\[
\begin{array}{rcl}
\omega A_i & \cong & \omega (M_i A_{k+1}) \\
& \cong & \omega A_{k+1}.
\end{array}
\]
Likewise
\[
\begin{array}{rcl}
\omega^* A_i & \cong & \omega^* A_{k+1}, \\
\mathbb{Z} A_i & \cong & \mathbb{Z}A_{k+1}.
\end{array}
\]
In particular, $\omega A_i \cong \omega A_j$, $\omega^*A_i \cong \omega^* A_j$ and $\mathbb{Z}A_i \cong \mathbb{Z}A_j$ for all $i, j \leq k$. 

$\rightarrow$ \underline{Termination in an absorbed factor}: If the algorithm terminates in an absorbed factor, we have $A_{k+1} + A_k \cong A_k$. 

If $k = 0$, then $A_1 + A_0 \cong A_0$, which gives $\omega A_1 \leqslant_{init} A_0$ by \ref{absorbsleftrightcoro}. 

If $k > 0$, then we claim that for each $0 \leq i \leq k$, there is a positive integer $M_i$ such that 
\begin{itemize}
    \item[i.] if $i$ is odd, then $A_{k-i} \cong M_i A_k + A_{k+1}$;
    \item[ii.] if $i$ is even, then $A_{k-i} \cong M_i A_k$. 
\end{itemize}
Indeed, if $i = 0$, the claim holds with $M_0 = 1$. If $i = 1$, the last isomorphism in the system \ref{euclideanisosfinite} gives the claim by taking $M_1 = N_{k-1}$. 

Suppose $i > 1$ and the claim is true for $j < i$. If $i$ is even, then by induction we have
\[
\begin{array}{rcl}
A_{k-(i-1)} & \cong & M_{i-1} A_k + A_{k+1}, \\
A_{k-(i-2)} & \cong & M_{i-2} A_k. 
\end{array}
\]
Then since $A_{k - i} \cong N_{k - i} A_{k - (i-1)} + A_{k - (i-2)}$ we have
\[
\begin{array}{rcl}
A_{k - i} & \cong & N_{k-i} A_{k - (i-1)} + A_{k - (i-2)} \\
& \cong & N_{k-i}(M_{i-1} A_k + A_{k+1}) + M_{i-2} A_k \\
& \cong & \underbrace{M_{i-1} A_k + A_{k+1} + M_{i-1} A_k + A_{k+1} + \cdots + M_{i-1} A_k + A_{k+1}}_{\textrm{$N_{k-i}$ times}} + M_{i-2} A_k \\
& \cong & M_i A_k,
\end{array}
\]
where $M_i = N_{k-i}M_{i-1} + M_{i-2}$ and in passing from the third to fourth line we have used $A_{k+1} + A_k \cong A_k$ repeatedly to absorb the $A_{k+1}$ terms into the $A_k$ terms at their right.  

If $i$ is odd, then we have
\[
\begin{array}{rcl}
A_{k-(i-1)} & \cong & M_{i-1} A_k, \\
A_{k-(i-2)} & \cong & M_{i-2} A_k + A_{k+1}. 
\end{array}
\]
Then
\[
\begin{array}{rcl}
A_{k - i} & \cong & N_{k-i} A_{k - (i-1)} + A_{k - (i-2)} \\
& \cong & N_{k-i}M_{i-1} A_k + M_{i-2} A_k + A_{k+1} \\
& \cong & M_i A_k + A_{k+1},
\end{array}
\]
where again $M_i = N_{k-i}M_{i-1} + M_{i-2}$. By induction, the claim is proved. 

Thus for $0 \leq i \leq k$, if $i \equiv k \pmod 2$ we have
\[
\begin{array}{rcl}
\omega A_i & \cong & \omega (M_i A_k) \\
& \cong & \omega A_k; \\
\omega^* A_i & \cong & \omega^*(M_i A_k) \\
& \cong & \omega^* A_k; \\
\mathbb{Z}A_i & \cong & \mathbb{Z}(M_i A_k) \\
& \cong & \mathbb{Z}A_k.
\end{array}
\]
If $i \equiv k + 1 \pmod 2$ we have
\[
\begin{array}{rcl}
\omega A_i & \cong & \omega (M_i A_k + A_{k+1}) \\
& \cong & M_i A_k + A_{k+1} + M_i A_k + A_{k+1} + \cdots \\
& \cong & M_i A_k + M_i A_k + \cdots \\
& \cong & \omega A_k; \\
\omega^* A_i & \cong & \omega^*(M_i A_k + A_{k+1}) \\
& \cong & \cdots + M_i A_k + A_{k+1} + M_i A_k + A_{k+1} \\
& \cong & \cdots + M_i A_k + M_i A_k + A_{k+1} \\
& \cong & \omega^*A_k + A_{k+1}; \\
\mathbb{Z}A_i & \cong & \mathbb{Z}(M_i A_k + A_{k+1}) \\
& \cong & \cdots + M_i A_k + A_{k+1} + M_i A_k + A_{k+1} + \cdots \\
& \cong & \cdots + M_i A_k + M_i A_k + \cdots \\
& \cong & \mathbb{Z}A_k. 
\end{array}
\]

In summary, for all $0 \leq i, j \leq k$ we have $\omega A_i \cong \omega A_j \cong \omega A_k$ and $\mathbb{Z}A_i \cong \mathbb{Z}A_j \cong \mathbb{Z}A_k$. If $i \equiv j \equiv k \pmod 2$ then also $\omega^*A_i \cong \omega^* A_j \cong \omega^*A_k$. If $i+1 \equiv j \equiv k \pmod 2$, then $\omega^*A_i \cong \omega^* A_k + A_{k+1} \cong \omega^*A_j + A_{k+1}$.

\underline{Non-termination}: Finally, suppose that the algorithm does not terminate at any stage. Then after infinitely many stages we have a system of isomorphisms:
\begin{equation}\label{turkish}
\{A_k \cong N_k A_{k+1} + A_{k+2}: k \in \omega\}.
\end{equation}
We will study such systems in later sections. In Section \ref{section:symboldynamrepnsofdivissystems}, we will show that frequently (though not always) the infinite discrete sums of the terms appearing in such systems satisfy the same isomorphisms as in the terminating cases. 

More specifically, we will show that if the system \ref{turkish} satisfies a sufficiency hypothesis, then we have $\omega A_i \cong \omega A_j$ and $\mathbb{Z}A_i \cong \mathbb{Z}A_j$ for all $i, j \geq 0$; and $\omega^*A_i \cong \omega^*A_j$ for all $i, j \geq 0$ with $i \equiv j \pmod 2$. See Theorem \ref{sufficiencyimpliesIsosforomegaprods}.

\subsection{Run of the right algorithm}\label{subsect:runofrightalgo} 

The right Euclidean algorithm on the input of a right setup $(I = [c, d]; \rho, \tau)$ is defined symmetrically to the left algorithm. We will not write it out explicitly. Instead, we simply record the analogous arithmetic data from a run of the algorithm. 

Suppose the right algorithm is run on an input right setup $(I = [c, d]; \rho, \tau)$. If the algorithm terminates at Stage $k$, then from the previous stages we have a system
\[
\{A_i \cong A_{i+2} + N_i A_{i+1}: 0 \leq i < k\},
\]
where $A_i$ and $A_{i+1}$ denote the dividend and divisor at Stage $i$ of the run. 

If the algorithm terminates in a divisor, then each $A_i$ for $i \leq k$ is a finite multiple of $A_{k+1}$. Then we have $\omega A_i \cong \omega A_j \cong \omega A_{k+1}$, $\omega^* A_i \cong \omega^* A_j \cong \omega^* A_{k+1}$ and $\mathbb{Z}A_i \cong \mathbb{Z}A_j \cong \mathbb{Z}A_{k+1}$ for all $i, j \leq k + 1$. 

If it terminates in an absorbed factor, we have $A_k + A_{k+1} \cong A_k$. If $k = 0$, this gives $A_0 + A_1 \cong A_0$ and hence $\omega^*A_1 \leqslant_{fin} A_0$. If $k > 0$, then for all $i \leq k$, there is a positive integer $M_i$ such that
\begin{itemize}
    \item[i.] if $i \equiv k \pmod 2$, then $A_i \cong M_i A_k$;
    \item[ii.] if $i \not\equiv k \pmod 2$, then $A_i \cong A_{k+1} + M_i A_k$.
\end{itemize}
It follows that $\omega^*A_i \cong \omega^*A_j \cong \omega^*A_k$ and $\mathbb{Z}A_i \cong \mathbb{Z}A_j \cong \mathbb{Z}A_k$ for all $i, j \leq k$. If $i \equiv j \equiv k \pmod 2$ then $\omega A_i \cong \omega A_j \cong \omega A_k$, whereas if $i + 1 \equiv j \equiv k \pmod 2$, then $\omega A_i \cong A_{k+1} + \omega A_k \cong A_{k+1} + \omega A_j$. 

If the algorithm does not terminate at any stage, then it yields a system 
\[
\{A_k \cong A_{k+2} + N_k A_{k+1}: k \in \omega\}.
\]

\subsection{Symmetric runs}\label{subsect:runofsymmetricalgo}

Finally, we record the arithmetic information yielded by a run of the left algorithm on a symmetric setup $(I; \rho, \tau)$. It is straightforward to check that a run of the right algorithm on the same input yields the same data. 

If the algorithm does not terminate before Stage $k$, then the input interval $I_k$ at Stage $k$ satisfies
\[
I_k \cong A_k + A_{k+1} \cong A_{k+1} + A_k.
\]
Likewise, from the previous stages $k' \leq k$, we have the commuting isomorphisms $A_{k'} + A_{k'+1} \cong A_{k'+1} + A_{k'}$.

We claim that in fact we have $A_i + A_j \cong A_j + A_i$ for all $i, j \leq k + 1$. Indeed, suppose we have $i, j$ with $i < j-1 \leq k+1$. By back substitution in the isomorphisms $A_{k'} \cong N_{k'} A_{k'+1} + A_{k'+2}$, we may write $A_i$ as a sum of $A_j$ and $A_{j-1}$. Since $A_j$ commutes with $A_j$ and $A_{j-1}$, it follows $A_j$ commutes with $A_i$, as claimed. 

Returning to Stage $k$ of the algorithm: if the algorithm terminates in a divisor at Stage $k$, then the arithmetic information yielded is the same as in the asymmetric case: we have $\omega A_i \cong \omega A_j \cong \omega A_{k+1}$, $\omega^* A_i \cong \omega^* A_j \cong \omega^* A_{k+1}$ and $\mathbb{Z}A_i \cong \mathbb{Z}A_j \cong \mathbb{Z}A_{k+1}$ for all $i, j \leq k + 1$.

If it terminates in an absorbed factor, then $A_{k+1} + A_k \cong A_k + A_{k+1} \cong A_k$. When $k = 0$, this gives $A_1 + A_0 \cong A_0 + A_1 \cong A_0$, and hence $\omega A_1 \leqslant_{init} A_0$ and $\omega^* A_1 \leqslant_{fin} A_0$. 

If $k > 0$, then from the previous stage we have the isomorphism 
\[
A_{k-1} \cong N_{k-1}A_k + A_{k+1}.
\]
Since $A_k + A_{k+1} \cong A_k$, this gives $A_{k-1} \cong N_{k-1}A_k$. Now back substitution into the previous identities $A_i \cong N_i A_{i+1} + A_{i+2}$ yields the same arithmetic data as if the algorithm terminated in the divisor $A_k$ at the previous stage: $\omega A_i \cong \omega A_j \cong \omega A_k$, $\omega^* A_i \cong \omega^* A_j \cong \omega^* A_k$ and $\mathbb{Z}A_i \cong \mathbb{Z}A_j \cong \mathbb{Z}A_k$ for all $i, j \leq k$.

For this reason, in practice we often assume that a terminating run of the algorithm on a symmetric setup terminates in a divisor. 

Finally, if the algorithm does not terminate, we finish with a system
\[
\{A_k \cong N_k A_{k+1} + A_{k+2}: k \in \omega\}
\]
with the property that the terms $A_k$ appearing in this system pairwise additively commute. We will see later that in this case we always have $\omega A_i \cong \omega A_j$ and $\omega^* A_i \cong \omega^* A_j$ (and hence $\mathbb{Z}A_i \cong \mathbb{Z}A_j$) for all $i, j \in \omega$. 

\section{Division systems}\label{section:divisionsystems}

In this section we study the systems of isomorphisms $A_k \cong N_k A_{k+1} + A_{k+2}$ that result from runs of the Euclidean algorithm, as well as analogous ``rational" systems arising from repeated divisions without remainders. In proofs of arithmetic results that involve runs of the algorithm, it is usually only the abstract properties of these resulting systems that are relevant.

In the definition below we use an indexing variable $M$ that can take the value $\omega$. We adopt the convention that when $M = \omega$ and $n \in \omega$, then $n + M = M = \omega$.  

\theoremstyle{definition}
\newtheorem{leftdivissystemdefn}[lct]{Definition}
\begin{leftdivissystemdefn}\label{leftdivissystemdefn}
Fix $M$, $1 \leq M \leq \omega$. A \textit{left division system} $\mathcal{D}$ (of \textit{length $M$}) consists of a sequence $\{N_k: 0 \leq k < M\}$ of natural numbers $N_k \geq 1$, and a sequence of linear orders $\{A_k: 0 \leq k < 2+M\}$ such that
\begin{itemize}
    \item[i.] for all $k < M$, $A_k \cong N_k A_{k+1} + A_{k+2}$; 
    \item[ii.] for all $k < 1 + M$, $A_k \neq \emptyset$. 
\end{itemize}
We write $\mathcal{D} = \{A_k; N_k; M\}$.

The system is \textit{terminating} if $M < \omega$. In this case, the terms $A_M$ and $A_{M+1}$ from the last isomorphism $A_{M-1} \cong N_{M-1}A_M + A_{M+1}$ satisfy
\[
A_{M+1} + A_M \cong A_M.
\]
If $A_{M+1} \neq \emptyset$, then the system \textit{terminates in an absorbed factor}. If $A_{M+1} = \emptyset$, then the system \textit{terminates in a divisor}.

The system is \textit{non-terminating} if $M = \omega$. 
\end{leftdivissystemdefn}

The notion of a \textit{right division system} is defined symmetrically. Such a system consists of isomorphisms of the form 
\[
A_k \cong A_{k+2} + N_kA_{k+1}.
\]

A \textit{symmetric division system} is a left division system in which $A_i + A_j \cong A_j + A_i$ for all terms $A_i, A_j$ appearing in the system. 

As we observed in Section \ref{subsect:runofsymmetricalgo}, if a symmetric system terminates in an absorbed factor, its final isomorphism is of the form 
\[
A_{M-1} \cong N_k A_M + A_{M+1},
\]
where $A_{M+1} + A_M \cong A_M + A_{M+1} \cong A_M$. By absorbing its $A_{M+1}$ term, the final isomorphism yields $A_{M-1} \cong N_k A_M$. Thus by replacing the final isomorphism with its absorbed form if necessary, we may always assume that a terminating symmetric system terminates in a divisor.

Unless it terminates at Stage $0$ in an absorbed factor, a run of the left or right Euclidean algorithm yields a corresponding left or right division system, and a run from a symmetric setup yields a symmetric system. It can be shown, however, that there exist non-terminating division systems that do not arise from runs of the division algorithm. 

\theoremstyle{definition}
\newtheorem{rationaldivissystemdefn}[lct]{Definition}
\begin{rationaldivissystemdefn}\label{rationaldivissystemdefn}
Fix $M$, $1 \leq M \leq \omega$. A \textit{rational division system} $\mathcal{D}$ (of \textit{length $M$}) consists of a sequence $\{N_k: k < M\}$ of natural numbers $N_k \geq 2$, and a sequence of non-empty linear orders $\{A_k: k < 1 + M\}$ such that for all $k < M$ we have 
\[
A_k \cong N_k A_{k+1}.
\]

The system is \textit{terminating} if $M < \omega$, and \textit{non-terminating} if $M = \omega$. 
\end{rationaldivissystemdefn}

The term \textit{division system} and notation $\{A_k; N_k; M\}$ refers to either a left system, right system, or rational system of some length $M \leq \omega$.

In a left system, whenever $A_k$ and $A_{k+1}$ are both defined we have that either $A_k \cong N_k A_{k+1} + A_{k+2}$ or $A_{k+1} + A_k \cong A_k$ (and hence $\omega A_{k+1} \leqslant_{init} A_k$). In either case, $A_{k+1}$ is initial in $A_k$. Further, whenever $A_{k+2}$ is defined we have $A_{k+2} \leqslant_{fin} A_k$, again as $A_k \cong N_k A_{k+1} + A_{k+2}$. 

Symmetrically, in a right system, $A_{k+1}$ is final in $A_k$ and $A_{k+2}$ is initial in $A_k$ whenever these terms are defined. From this observation the proposition below follows by induction. 

\theoremstyle{definition}
\newtheorem{initfinaltermsindivissys}[lct]{Proposition}
\begin{initfinaltermsindivissys}\label{initfinaltermsindivissys}
Suppose $\mathcal{D} = \{A_k; N_k; M\}$ is a division system and $n < 2 + M$. 
\begin{itemize}
    \item[i.] If $\mathcal{D}$ is a left system, then $A_n \leqslant_{init} A_k$ whenever $k \leq n$, and $A_n \leqslant_{fin} A_k$ whenever $k \leq n$ and $n \equiv k \pmod 2$. 
    \item[ii.] If $\mathcal{D}$ is a right system, then $A_n \leqslant_{fin} A_k$ whenever $k \leq n$, and $A_n \leqslant_{init} A_k$ whenever $k \leq n$ and $n \equiv k \pmod 2$. 
    \item[iii.] If $\mathcal{D}$ is a symmetric or rational system, then $A_n \leqslant_{init} A_k$ and $A_n \leqslant_{fin} A_k$ whenever $k \leq n$. 
\end{itemize}
\end{initfinaltermsindivissys}

The next proposition is an arithmetic expression of the heuristic that consecutive terms from a division system are almost additively commuting. 

\theoremstyle{definition}
\newtheorem{AkplusAkplusone}[lct]{Proposition}
\begin{AkplusAkplusone}\label{AkplusAkplusone}
Suppose $\mathcal{D} = \{A_k; N_k; M\}$ is a division system. 
\begin{itemize}
    \item[i.] If $\mathcal{D}$ is a left system, then for all $k < M$ we have
    \[
    \begin{array}{rcl}
    A_k + A_{k+1} & \cong & N_k A_{k+1} + A_{k+2} + A_{k+1} \\ 
    A_{k+1} + A_k & \cong & N_k A_{k+1} + A_{k+1} + A_{k+2}.
    \end{array}
    \]
    \item[ii.] If $\mathcal{D}$ is a right system, then for all $k < M$ we have
    \[
    \begin{array}{rcl}
    A_{k+1} + A_k & \cong & A_{k+1} + A_{k+2} + N_k A_{k+1} \\ 
    A_k + A_{k+1} & \cong & A_{k+2} + A_{k+1} + N_k A_{k+1}.
    \end{array}
    \]
\end{itemize}
\end{AkplusAkplusone}

\begin{proof}
For (i.): The first isomorphism is immediate from $A_k \cong N_k A_{k+1} + A_{k+2}$. For the second, we have
\[
A_{k+1} + A_k \cong A_{k+1} + N_k A_{k+1} + A_{k+2} \cong N_k A_{k+1} + A_{k+1} + A_{k+2},
\]
as claimed. The proof for (ii.) is symmetric. 
\end{proof}

If $\mathcal{D} = \{A_k; N_k; M\}$ is a left or right system, then from the isomorphisms in Proposition \ref{AkplusAkplusone} we can simulate a run of the left or right Euclidean algorithm.

For example, suppose $\mathcal{D}$ is a left system. Let $I_0 = A_0 + A_1$ and label $I_0 = [c, d]$ by its endcuts. Then from the proposition we have
\[
I_0 \cong N_0A_1 + A_2 + A_1
\]
Let $c_1$ be a cut in $I_0$ corresponding to the left $+$ sign in this expression, and let $I_1 = [c_1, d]$. Then 
\[
I_1 = [c_1, d] \cong A_2 + A_1 \cong N_1 A_2 + A_2 + A_3
\]
Let $c_2$ be a cut in $I_1$ corresponding to the left $+$ sign, and let $I_2 = [c_2, d] \cong A_2 + A_3$. Now iterate. If $M < \omega$, we obtain a cut sequence $c < c_1 < c_2 < \ldots < c_M < d$ witnessing
\[
I_0 \cong N_0 A_1 + N_1 A_2 + \cdots + N_{M-1}A_M + R,
\]
where $R \cong A_M + A_{M+1}$ or $R \cong A_{M+1} + A_M$ depending on the parity of $M$. 

Similarly, beginning with $I_0' = A_1 + A_0 = [c', d']$, we obtain a cut sequence $c' < c_1' < c_2' < \ldots < c_M' < d'$ witnessing
\[
I_0' \cong N_0 A_1 + N_1 A_2 + \cdots + N_{M-1}A_M + R',
\]
where $R' \cong A_{M+1} + A_M$ if $R \cong A_M + A_{M+1}$ and $R' \cong A_M + A_{M+1}$ if $R \cong A_{M+1} + A_M$.

If instead $M = \omega$, then we obtain cut sequences $c < c_1 < c_2 < \ldots$ and $c' < c_1' < c_2' < \ldots$ with limits $b \leq d$ and $b' \leq d'$ witnessing
\[
\begin{array}{rcl}
I_0 = A_0 + A_1 & \cong & \sum_{k \in \omega} N_kA_{k+1} + L \\
I_0' = A_1 + A_0 & \cong & \sum_{k \in \omega} N_kA_{k+1} + L'
\end{array}
\]
where $L = [b, d]$ and $L' = [b', d']$. It need not be that $L \cong L'$, but if this is the case then these decompositions show $A_0 + A_1 \cong A_1 + A_0$. 

Taking tails of the sums $\sum_{k < M} N_k A_{k+1}$ corresponding to the intervals $I_k$ and $I_k'$ in the decompositions above yields analogous decompositions for pairs $A_k + A_{k+1}$ and $A_{k+1} + A_k$. Again, we may interpret these expressions as showing that consecutive terms in a left division system are almost pairwise commuting, up to a difference term on the right. A symmetric discussion holds for right systems. 

\subsection{A splitting dichotomy theorem for division systems}

The following lemma is useful for proving that terms in a given division system are splitting. We will use it to prove a splitting dichotomy for division systems in Proposition \ref{splittingdichthmfordivsystems} below. 

\theoremstyle{definition}
\newtheorem{splittinglemmadivissystems}[lct]{Lemma}
\begin{splittinglemmadivissystems}\label{splittinglemmadivissystems}
Suppose $\mathcal{D} = \{A_k; N_k; M\}$ is a division system of length $M$, and fix $k < M$. 

If $\mathcal{D}$ is a left system, then
\begin{itemize}
    \item[i.] $A_k \leqslant_{init} (N_k + 1)A_{k+1}$;
    \item[ii.] $A_k \leqslant_{init} A_{k+1} + A_k$;
    \item[iii.] $2A_{k+1} \leqslant_{init} 2A_k$; 
    \item[iv.] $2A_{k+2} \leqslant_{init} A_k$. 
\end{itemize}

Symmetrically, if $\mathcal{D}$ is a right system, then 
\begin{itemize}
    \item[i.] $A_k \leqslant_{fin} (N_k + 1)A_{k+1}$;
    \item[ii.] $A_k \leqslant_{fin} A_k + A_{k+1}$;
    \item[iii.] $2 A_{k+1} \leqslant_{fin} 2A_k$; 
    \item[iv.] $2A_{k+2} \leqslant_{fin} A_k$. 
\end{itemize}
\end{splittinglemmadivissystems}

\begin{proof}
Suppose first $\mathcal{D}$ is a left system and $M < \omega$. Since $k < M$, we have $A_k \cong N_k A_{k+1} + A_{k+2}$. 

For (i.): Since $A_{k+2} \leqslant_{init} A_{k+1}$, we have
\[
A_k \cong N_k A_{k+1} + A_{k+2} \leqslant_{init} N_k A_{k+1} + A_{k+1} \cong (N_k + 1)A_{k+1},
\]
which gives $A_k \leqslant_{init} (N_k+1) A_{k+1}$.

For (ii.): We have 
\[
A_{k+1} + A_k \cong A_{k+1} + (N_k A_{k+1} + A_{k+2}) \cong (N_k+1)A_{k+1} + A_{k+2}.
\]
Since $A_k \leqslant_{init} (N_k + 1)A_{k+1}$ by (i.), it follows $A_k \leqslant_{init} A_{k+1} + A_k$.

For (iii.): This is immediate if $N_k \geq 2$, as then $2A_{k+1} \leqslant_{init} A_k$. Otherwise $A_k \cong A_{k+1} + A_{k+2}$, and so
\begin{equation}\label{ayoo}
2A_k \cong A_{k+1} + A_{k+2} + A_{k+1} + A_{k+2}.
\end{equation}
If $k = M - 1$, then $A_{k+2} = A_{M+1}$ is left absorbed by $A_{k+1} = A_M$, and we deduce $2A_k \cong 2A_{k+1} + A_{k+2}$, which gives $2A_{k+1} \leqslant_{init} 2A_k$. If $k < M-1$, applying (ii.) to the index $k + 1$, we have $A_{k+1} \leqslant_{init} A_{k+2} + A_{k+1}$. Along with \ref{ayoo} this yields the chain
\[
A_{k+1} + A_{k+1} \leqslant_{init} A_{k+1} + A_{k+2} + A_{k+1} \leqslant_{init} 2A_k,
\]
which again gives $2A_{k+1} \leqslant_{init} 2A_k$. 

For (iv.): If $k = M-1$, then $A_{k+2} + A_{k+1} \cong A_{k+1}$. Hence $\omega A_{k+2} \leqslant_{init} A_{k+1}$ by \ref{absorbsleftrightcoro}, which certainly implies $2A_{k+2} \leqslant_{init} A_k$. 

If $k < M-1$, we have $A_{k+1} \cong N_{k+1}A_{k+2} + A_{k+3}$. If $N_{k+1} > 1$, then $2A_{k+2} \leqslant_{init} A_{k+1}$ and hence $2A_{k+2} \leqslant_{init} A_k$. So suppose $N_{k+1} = 1$. Then $A_{k+1} \cong A_{k+2} + A_{k+3}$. 

If now $N_k > 1$, then $2A_{k+1} \leqslant_{init} A_k$. By (iii.) applied to the index $k+1$, we have $2A_{k+2} \leqslant_{init} 2A_{k+1}$; hence $2A_{k+2} \leqslant_{init} A_k$ as well. So suppose finally that $N_k = 1$. Then 
\begin{equation}\label{yo}
A_k \cong A_{k+1} + A_{k+2} \cong A_{k+2} + A_{k+3} + A_{k+2}.
\end{equation}

If $k = M-2$, then $A_{k+3} = A_{M+1}$ is absorbed on the left by $A_{k+2} = A_M$, and we deduce $A_k \cong A_{k+2} + A_{k+2} \cong 2A_{k+2}$, which implies $2A_{k+2} \leqslant_{init} A_k$. Otherwise $k < M - 2$, and then by (ii.) applied to the index $k + 2$ we have $A_{k+2} \leqslant_{init} A_{k+3} + A_{k+2}$. Now along with \ref{yo}, we get the chain
\[
2A_{k + 2} \cong A_{k+2} + A_{k+2} \leqslant_{init} A_{k+2} + A_{k+3} + A_{k+2} \cong A_k,
\]
which yields $2A_{k + 2} \leqslant_{init} A_k$ in this case as well. This concludes the proof of (iv.).

The proofs when $M=\omega$ are the same, without the casework involving absorbed factors. 

The corresponding facts for right systems are proved symmetrically.    
\end{proof}

The next two propositions are splitting dichotomy theorems for division systems. They say essentially that the terms $A_k$ appearing in a division system are either all splitting or all non-splitting. This is true outright when the system is non-terminating, and also for arbitrary symmetric and rational systems. For terminating one-sided systems, the situation is slightly more delicate. 

\theoremstyle{definition}
\newtheorem{splittingdichthmfordivsystems}[lct]{Proposition}
\begin{splittingdichthmfordivsystems}\label{splittingdichthmfordivsystems}
Suppose $\mathcal{D} = \{A_k; N_k; M\}$ is a left or right division system. 
\begin{itemize}
    \item[i.] If $M < \omega$, then $A_k$ is splitting for some $k < M$ if and only if $A_k$ is splitting for all $k \leq M$. 
    \item[ii.] If $M = \omega$, then $A_k$ is splitting for some $k \in \omega$ if and only if $A_k$ is splitting for all $k \in \omega$. 
\end{itemize} 
\end{splittingdichthmfordivsystems}

\begin{proof}
For (i.): Suppose that $\mathcal{D}$ is a left division system.

We first make the following claim:
\begin{equation}\label{ktokplusone}
\textrm{If $k < M$ and $A_k$ is splitting, then $A_{k+1}$ is splitting.}
\end{equation}
Indeed, assume $k < M$ and $A_k \cong 2A_k$. By Lemma \ref{splittinglemmadivissystems}.(i.), $A_k \leqslant_{init} (N_k+1)A_{k+1}$, and hence $2A_k \leqslant_{init} (N_k+1)A_{k+1}$. On the other hand, $2A_{k+1} \leqslant_{init} 2A_k$ by \ref{splittinglemmadivissystems}.(iii.). Now \ref{2AconvnB2BconvmAsplittingiff} implies $A_{k+1}$ is splitting, as claimed. 

We next claim:
\begin{equation}\label{kplusonetominusone}
\textrm{If $1 \leq k < M$ and both $A_k$ and $A_{k+1}$ are splitting, then $A_{k-1}$ is splitting.}
\end{equation}
Assume the hypothesis of the claim. We will show that in fact $A_{k-1} \cong A_{k+1}$. Since $k \geq1$, we have $A_{k-1} \cong N_{k-1}A_k + A_{k+1}$, and so $A_{k+1} \leqslant_{fin} A_{k-1}$. On the other hand, we have 
\[
A_k \leqslant_{init} (N_k + 1)A_{k+1} \cong A_{k+1}
\]
where the isomorphism holds because $A_{k+1}$ is splitting. Similarly, 
\[
A_{k-1} \leqslant_{init} (N_{k-1} + 1)A_k \cong A_k
\]
as $A_k$ is also splitting. Thus we have
\[
A_{k-1} \leqslant_{init} A_k \leqslant_{init} A_{k+1}
\]
which yields $A_{k-1} \leqslant_{init} A_{k+1}$. Now \ref{CSBLO} gives $A_{k-1} \cong A_{k+1}$. In particular, $A_{k-1}$ is splitting, as claimed.

Now we may finish the proof of (i.). Only the forward direction is non-trivial, so suppose $A_k$ is splitting for some $k < M$. Then from the claim \ref{ktokplusone} we get by induction that $A_n$ is splitting for all $k \leq n \leq M$, and in particular $A_{k+1}$ is splitting. If $k = 0$, we are done. If $k > 0$, then \ref{kplusonetominusone} gives that $A_{k-1}$ is splitting. By induction, $A_n$ is splitting for all $0 \leq n \leq k$, and hence for all $0 \leq n \leq M$. 

When $M = \omega$, essentially the same argument applies and yields (ii.). For right systems, the argument is symmetric. 
\end{proof}

\theoremstyle{definition}
\newtheorem{splittingdichthmrationalsystems}[lct]{Proposition}
\begin{splittingdichthmrationalsystems}\label{splittingdichthmrationalsystems}
Suppose $\mathcal{D} = \{A_k; N_k; M\}$ is a rational system of length $M$. Then $A_k$ is splitting for some $k$ if and only if $A_k$ is splitting for all $k$.  
\end{splittingdichthmrationalsystems}

\begin{proof}
The isomorphism $A_k \cong N_k A_{k+1}$ gives $2A_k \cong 2N_k A_{k+1}$. It follows that each pair of consecutive terms $A_k, A_{k+1}$ from the system satisfies the hypotheses of Lemma \ref{2AconvnB2BconvmAsplittingiff}. Thus $A_k$ is splitting if and only if $A_{k+1}$ is splitting, and the proposition follows by induction.  
\end{proof}

\theoremstyle{definition}
\newtheorem{divisionsystemsplittingdefn}[lct]{Definition}
\begin{divisionsystemsplittingdefn}\label{divisionsystemsplittingdefn}
A division system $\{A_k; N_k; M\}$ is \textit{non-splitting} if $A_k$ is non-splitting for all $0 \leq k < M$.
\end{divisionsystemsplittingdefn}

Equivalently, by Propositions \ref{splittingdichthmfordivsystems} and \ref{splittingdichthmrationalsystems}, a system is non-splitting if its initial term $A_0$ is non-splitting. If a system is not non-splitting, it is \textit{splitting}; equivalently, it is splitting if $A_k$ is splitting for every $k < M$. 

The proofs of Propositions \ref{splittingdichthmfordivsystems} and \ref{splittingdichthmrationalsystems} show that the terms $A_k$ from a splitting system are pairwise nearly isomorphic. We record this observation in the following proposition. 

\theoremstyle{definition}
\newtheorem{isosoftermsinsplittingsystems}[lct]{Proposition}
\begin{isosoftermsinsplittingsystems}\label{isosoftermsinsplittingsystems}
Suppose that $\{A_k; N_k; M\}$ is a splitting division system. 
\begin{itemize}
    \item[i.] If it is a left system, then for all $k, k' < M$ we have $A_k \leqslant_{init} A_{k'}$, and if $k \equiv k' \pmod 2$, then $A_k \cong A_{k'}$.
    \item[ii.] If it is a right system, then for all $k, k' < M$ we have $A_k \leqslant_{fin} A_{k'}$, and if $k \equiv k' \pmod 2$, then $A_k \cong A_{k'}$.
    \item[iii.] If it is a symmetric system, then for all $k, k' < M$ we have $A_k \cong A_{k'}$.
    \item[iv.] If it is a rational system, then for all $k, k' \leq M$ we have $A_k \cong A_{k'}$.
\end{itemize}
\end{isosoftermsinsplittingsystems}
\begin{proof}
The proof of (i.) is contained in the proof of Proposition \ref{splittingdichthmfordivsystems}, and (ii.) is symmetric. 

Since a symmetric system $\{A_k\}$ may be viewed as both a left and right system, (iii.) follows from \ref{CSBLO} combined with (i.) and (ii.), as these conditions imply $A_k \leqslant_{init} A_{k'}$ and $A_{k'} \leqslant_{fin} A_k$ for all $k, k' < M$. 

For splitting rational systems it is immediate that $A_k \cong A_{k'}$ for all $k, k' \leq M$, as one of these terms may always be written as a multiple of the other using the isomorphisms from the system. Thus (iv.) holds. 
\end{proof}

\subsection{Division condensations}\label{subsect:simDcondensation}

To every division system, we associate a global binary relation on $LO$.

\theoremstyle{definition}
\newtheorem{simDdefn}[lct]{Definition}
\begin{simDdefn}\label{simDdefn}
Suppose $\mathcal{D} = \{A_k; N_k; M\}$ is a division system. Given $X \in LO$ and $x, y \in X$, define 
\[
\textrm{$x \sim_{\mathcal{D}} y$ if $A_k \not \leqslant_{conv} [\{x, y\}]$ for all $k < M$.}
\]
\end{simDdefn}

If $\mathcal{D}$ is non-terminating, then the associated relation $\sim_{\mathcal{D}}$ defines a condensation on every order $X$.

\theoremstyle{definition}
\newtheorem{simDisacondpropn}[lct]{Proposition}
\begin{simDisacondpropn}\label{simDisacondpropn}
Suppose $\mathcal{D}$ is a non-terminating division system and $X \in LO$. Then:

\begin{itemize}
    \item[i.] $\sim_{\mathcal{D}}$ is a condensation of $X$.
    \item[ii.]  $\sim_{\mathcal{D}}$ is invariant under convex embeddings. That is, if $f: X \rightarrow Y$ is a convex embedding, then $x \sim_{\mathcal{D}} y$ in $X$ if and only if $f(x) \sim_{\mathcal{D}} f(y)$ in $Y$.
    \item[iii.] The condensed order $X/{\sim_{\mathcal{D}}}$ is either a singleton or densely ordered. 
\end{itemize}
\end{simDisacondpropn}

\begin{proof}
For (i.): We must show the relation $\sim_{\mathcal{D}}$ is a convex equivalence relation on $X$, that is, an equivalence relation with convex classes. It is clearly reflexive and symmetric. Since $A_k \not \leqslant_{conv} [\{x, y\}]$ implies $A_k \not \leqslant_{conv} [\{x, z\}]$ for any $z \in [\{x, y\}]$, its classes are convex. 

To show it is transitive, fix $x, y, z \in X$ with $x \sim_{\mathcal{D}} y$ and $y \sim_{\mathcal{D}} z$. If $z \in [\{x, y\}]$ then by the convexity of $\sim_{\mathcal{D}}$ we have $x \sim_{\mathcal{D}} z$. 

So suppose $z \not\in [\{x, y\}]$. Without loss of generality, we may assume $x < y < z$. Suppose there is $k$ such that $A_k \leqslant_{conv} [x, z]$. Fix a subinterval $I = [c, d]$ of $[x, z]$ that is isomorphic to $A_k$. 

We must have $c < y < d$. Otherwise, $A_k$ convexly embeds in either $[x, y]$ or $[y, z]$, contradicting $x \sim_{\mathcal{D}} y$ or $y \sim_{\mathcal{D}} z$. If the system is a left division system, then since it is non-terminating we have $A_k \cong N_k A_{k+1} + A_{k+2}$. Fix a cut $c'$ with $c < c' < d$ corresponding to the cut at the $+$ sign in this expression. If $c' \leq y$ then $A_{k+1} \leqslant_{conv} [x, y]$, and if $c' \geq y$ then $A_{k+2} \leqslant_{conv} [y, z]$, a contradiction in either case. 

Hence $x \sim_{\mathcal{D}} z$. The argument is similar in the cases when the system is right or rational. Thus $\sim_{\mathcal{D}}$ is transitive, and hence a condensation of $X$.

For (ii.): By the convexity of $f$ we have $A_k \leqslant_{conv} [\{x, y\}]$ if and only if $A_k \leqslant_{conv} [\{f(x), f(y)\}]$.

For (iii.) Let $\delta$ denote the condensation map for $\sim_{\mathcal{D}}$ on $X$, and suppose there are $x, y \in X$ with $\delta(x) < \delta(y)$. Then $x \not\sim_{\mathcal{D}} y$, and so $A_k \leqslant_{conv} [x, y]$ for some $k$. Fix $I = [c, d] \cong A_k$ with $x \leq c < d \leq y$. 

If the system is a left system, we have
\[
A_k \cong N_k A_{k+1} + A_{k+2} \cong N_k A_{k+1} + N_{k+2}A_{k+3} + A_{k+4}. 
\]
Fix $c' < c''$ with $c < c' < c'' < d$ corresponding to the cuts at the $+$ signs in the expression on the right. Then for any $z \in [c', c'']$,  we have $A_{k+1} \leqslant_{conv} [x, z]$ and $A_{k+4} \leqslant_{conv} [z, y]$ which gives $x \not\sim_{\mathcal{D}} z$ and $z \not\sim_{\mathcal{D}} y$. Hence $\delta(x) < \delta(z) < \delta(y)$. Since $x, y$ were arbitrary, $X/{\sim_{\mathcal{D}}}$ is dense. 

The arguments for right systems and rational systems are similar. 
\end{proof}

As in the proof of Proposition \ref{simDisacondpropn}, we will write $\delta_{\mathcal{D}}$, or simply $\delta$, for the condensation map of a non-terminating system $\mathcal{D}$. 

Part (ii.) of Proposition \ref{simDisacondpropn} can be reformulated as follows.

\theoremstyle{definition}
\newtheorem{simDclassesinvarunderconvfcoro}[lct]{Corollary}
\begin{simDclassesinvarunderconvfcoro}\label{simDclassesinvarunderconvfcoro}
Suppose $\mathcal{D}$ is a non-terminating division system and $X \in LO$. If $f: X \rightarrow Y$ is a convex embedding, then for every $x \in X$ we have
\[
f[\delta(x)] = \delta(f(x)) \cap f[X].
\]
In particular, if $\delta(f(x)) \subseteq f[X]$, then $f[\delta(x)] = \delta(f(x))$. 
\end{simDclassesinvarunderconvfcoro}

It follows from Corollary \ref{simDclassesinvarunderconvfcoro} that if $\mathcal{D}$ is non-terminating and $f: X \rightarrow Y$ is an isomorphism, then $f$ lifts to an isomorphism $f: X/{\sim_{\mathcal{D}}} \rightarrow Y/{\sim_{\mathcal{D}}}$ defined by (and well-defined by) the rule $f[\delta(x)] = \delta(f(x))$. 

\subsection{$\mathcal{D}$-branchings}\label{subsect:Dbranchings}

Given a non-terminating division system $\mathcal{D}$, we will be interested in computing the condensation classes $\delta(x)$ of $\sim_{\mathcal{D}}$ in a given order $X$ via countable nested intersections of copies of the division terms $A_k$. This turns out to always be possible when $\mathcal{D}$ is non-splitting.

\theoremstyle{definition}
\newtheorem{Dbranchingdefn}[lct]{Definition}
\begin{Dbranchingdefn}\label{Dbranchingdefn}
Suppose $\mathcal{D} = \{A_k\}$ is a non-terminating division system. A \textit{$\mathcal{D}$-branching} in $X$ consists of a decreasing sequence of intervals $\{I_n: n \in \omega\}$
\[
    X \supseteq I_0 \supseteq I_1 \supseteq \ldots
\]
along with two sequences of natural numbers $\{k_n: n \in \omega\}$ and $\{l_n: n \in \omega\}$, the first of which is strictly increasing, such that for all $n \in \omega$ we have
\begin{itemize}
    \item[i.] $I_n \cong A_{k_n}$;
    \item[ii.] $A_{l_n} \leqslant_{conv} (I_n \setminus I_{n+1}$).
\end{itemize}
We denote the $\mathcal{D}$-branching by $\{I_n; k_n; l_n\}$, or simply $\{I_n\}$.
\end{Dbranchingdefn}

When $\mathcal{D}$ is non-splitting, $\mathcal{D}$-branchings are the means by which we will zoom in on $\sim_{\mathcal{D}}$-classes $\delta(x)$ in a given order $X$. They are typically constructed by a version of the following procedure: 
\begin{itemize}
    \item Start with an interval $I_0 \subseteq X$ that is isomorphic to $A_k = A_{k_0}$ for some term $A_k$ appearing in $\mathcal{D}$ (a left system, say);
    \item Decompose $I_0$ to witness $A_k \cong N_k A_{k+1} + A_{k+2}$;
    \item Choose a subinterval $I_1$ corresponding to one of the terms $A_{k_1}$ in this expression;
    \item Iterate. 
\end{itemize}
All interval sequences $\{I_n\}$ constructed this way are branchings. Definition \ref{Dbranchingdefn} allows for slightly more general constructions. 

Suppose that $\{I_n\}$ is a $\mathcal{D}$-branching. Since the $I_n$ are intervals, each difference $I_n \setminus I_{n+1}$ consists of an initial segment $L_n$ and final segment $R_n$ of $I_n$. 

The branching is \textit{rightward} if for infinitely many $n$ we have $A_{l_n} \leqslant_{conv} L_n$ and \textit{leftward} if for infinitely many $n$ we have $A_{l_n} \leqslant_{conv} R_n$. Condition (ii.) in the definition implies that any branching is either leftward or rightward. A \textit{middle branching} is a branching that is both leftward and rightward.

We will be interested in taking intersections of branchings. In general, if $\{I_n\}$ is a descending sequence of intervals and we label each $I_n = [c_n, d_n]$ by its endcuts, then letting $c$ denote the cut at the right of the non-decreasing sequence $c_0 \leq c_1 \leq \ldots$ and $d$ denote the cut at the left of the non-increasing sequence $\ldots \leq d_1 \leq d_0$, we have $\bigcap_n I_n = [c, d]$. In the case when $c = d$ this intersection is formally empty, but we sometimes view it as the cut singleton $\{c\}$.

\theoremstyle{definition}
\newtheorem{Dbranchingintersectionpropn}[lct]{Proposition}
\begin{Dbranchingintersectionpropn}\label{Dbranchingintersectionpropn}
Suppose $\mathcal{D} = \{A_k\}$ is a non-terminating and non-splitting division system and $\{I_n; k_n; l_n\}$ is a $\mathcal{D}$-branching in $X$. Fix a point $x \in \bigcap_n I_n$. 
\begin{itemize}
    \item[i.] If the branching is rightward, $\bigcap_n I_n$ is an initial segment of $\delta(x)$;
    \item[ii.] If the branching is leftward, $\bigcap_n I_n$ is a final segment of $\delta(x)$;
    \item[iii.] If the branching is middle, $\bigcap_n I_n = \delta(x)$.
\end{itemize}
\end{Dbranchingintersectionpropn}

\begin{proof}: For (i.): As above, label each $I_n = [c_n, d_n]$ by its endcuts, as well as $\bigcap_n I_n = [c, d]$. For each $n_0$ we have $\bigcap_n I_n \leqslant_{conv} I_{n_0} \cong A_{k_{n_0}}$. Since the indices $k_n$ are strictly increasing, this implies $\bigcap_n I_n \leqslant_{conv} A_k$ for all $k$.  

We first show $\bigcap_n I_n \subseteq \delta(x)$. Fix $y \in \bigcap_n I_n$. If $y \not\sim x$, then for some $k$ we have $A_k \leqslant_{conv} [\{y, x\}]$. Since $[\{y, x\}] \leqslant_{conv} \bigcap_n I_n \leqslant_{conv} A_{k+2}$, this gives $A_k \leqslant_{conv} A_{k+2}$. By Lemma \ref{splittinglemmadivissystems}, it follows that $2A_{k+2} \leqslant_{conv} A_{k+2}$. Thus $A_{k+2}$ is splitting by \ref{2XconvXiff2XcongX}, and so $\mathcal{D}$ is splitting as well, a contradiction. Thus $y \sim x$ and $\bigcap_n I_n \subseteq \delta(x)$ as claimed. 

We now show $\bigcap_n I_n$ is initial in $\delta(x)$. Fix $x' \in \delta(x)$ and suppose $x' < y$ for some $y \in \bigcap_n I_n$. If $x' \not \in \bigcap_n I_n$, then for all sufficiently large $n$ we have $x' < c_n$. Since the branching is rightward, there is $N$ such that $x' < c_N < c_{N+1}$ and $A_{l_N} \leqslant_{conv} [c_N, c_{N+1}]$. It follows $A_{l_N} \leqslant_{conv} [x', y]$, and so $x' \not\sim y$, i.e. $x' \not\in \delta(y)$. Since $y \in \bigcap_n I_n \subseteq \delta(x)$ we have $\delta(y) = \delta(x)$. Hence $x' \not\in \delta(x)$, a contradiction, and it follows $\bigcap_n I_n$ is initial in $\delta(x)$, as claimed. 

A symmetric argument gives (ii.). Then (iii.) is immediate by combining (i.) and (ii.). 
\end{proof}

In Section \ref{section:symboldynamrepnsofdivissystems}, we will construct symbolic representations of the orders $A_k$ in a given non-terminating system $\mathcal{D}$. We will use Proposition \ref{Dbranchingintersectionpropn} to show that for non-splitting $\mathcal{D}$, these representations are canonically expressed in terms of the $\sim_{\mathcal{D}}$-classes $\delta(x)$.  

\section{New commutativity laws for $(LO, +)$}\label{section:commutativitylaws}

In this section, we solve Tarski's Generalized Sum Problem \ref{tgsp} for $LO$ in the affirmative; see Theorem \ref{generalizedsumproblemsoln} below. Along the way, we establish several arithmetic conditions on pairs $A, B \in LO$ that either imply or are implied by the isomorphism $A + B \cong B + A$. These results will be used in the proof of the revised version \ref{revisedtarconj} of Tarski's Commuting Pairs Conjecture \ref{tarconj} in Section \ref{section:continuationandvaluation}.

Our fundamental result of this section is the following theorem.

\theoremstyle{definition}
\newtheorem{BplusAinitandfinAplusB}[lct]{Theorem}
\begin{BplusAinitandfinAplusB}\label{BplusAinitandfinAplusB}
If $B + A \leqslant_{init} A + B$ and $B + A \leqslant_{fin} A + B$ then $A + B \cong B + A$.
\end{BplusAinitandfinAplusB}

\begin{proof}
Fix a left setup $(I = [c, d]; \rho, \tau)$ witnessing $B + A \leqslant_{init} A + B$. Let $A_0, A_1$ respectively denote the dividend and divisor among $A, B$ in this setup. Run the left Euclidean algorithm on input $(I; \rho, \tau)$.  

Suppose first the algorithm terminates at a finite stage $k$. From the previous stages, we have the division system
\[
\begin{array}{rcl}
A_0 & \cong & N_0A_1 + A_2 \\
A_1 & \cong & N_1A_2 + A_3 \\
& \vdots & \\
A_{k-1} & \cong & N_{k-1} A_k + A_{k+1}.
\end{array}
\]

If the algorithm terminates in a divisor, we have $A_k \cong N_k A_{k+1}$. By back substitution, $A_0 \cong M A_{k+1}$ and $A_1 \cong N A_{k+1}$ for some integers $M, N \geq 1$. Hence $A_0 + A_1 \cong A_1 + A_0$, i.e. $A + B \cong B + A$, and we are done in this case. 

If the algorithm terminates in an absorbed factor, then $A_{k+1} + A_k \cong A_k$, and hence $\omega A_{k+1} \leqslant_{init} A_k$. If $k > 0$, then from the discussion in Section \ref{subsect:leftalgorun}, there are integers $N, N' \geq 1$ such that one of $A$, $B$ is isomorphic to $NA_k + A_{k+1}$ and the other is isomorphic to $N'A_k$, depending on the parity of $k$. And if $k = 0$ (i.e. the algorithm terminates at Stage $0$ in an absorbed factor, so that $A_1 + A_0 \cong A_0$), then the same is true taking $N = 0$ and $N' = 1$. 

In either case, let $M = N + N'$. Then $M \geq 1$ and one of the following holds:
\begin{itemize}
    \item[i.] $A + B \cong MA_k$ and $B + A \cong MA_k + A_{k+1}$;
    \item[ii.] $A + B \cong MA_k + A_{k+1}$ and $B + A \cong MA_k$. 
\end{itemize}

Suppose (i.) holds. Then since $MA_k \leqslant_{init} MA_k + A_{k+1}$, it follows $A + B \leqslant_{init} B + A$. By hypothesis we have $B + A \leqslant_{fin} A + B$. Therefore $A + B \cong B + A$ by Corollary \ref{CSBLO}, and we are done in case (i.).  

Now suppose (ii.) holds. Then from $B + A \leqslant_{fin} A + B$, we get
\[
MA_k \leqslant_{fin} MA_k + A_{k+1}. 
\]
Since $A_k \leqslant_{fin} MA_k$, it follows $A_k \leqslant_{fin} MA_k + A_{k+1}$. Since also $A_k + A_{k+1} \leqslant_{fin} MA_k + A_{k+1}$, by \ref{bothinitbothfin} one of the following holds:
\begin{itemize}
    \item[a.] $A_k \leqslant_{fin} A_k + A_{k+1}$;
    \item[b.] $A_k + A_{k+1} \leqslant_{fin} A_k$. 
\end{itemize}

Suppose (ii.b.) holds. Then $A_k \cong Y + A_k + A_{k+1}$ for some $Y$. Then by \ref{XcongAXBiffcongAXandXB} we have $A_k + A_{k+1} \cong A_k$. Then
\[
A + B \cong MA_k + A_{k+1} \cong MA_k \cong B + A,
\]
and we are done in case (ii.b.).

Now suppose (ii.a.) holds. Label $A_k + A_{k+1} = [l, r]$ and let $p$ denote the cut at the $+$ sign in this expression, so that $[l, p] \cong A_k$ and $[p, r] \cong A_{k+1}$. Let $q$ be a cut in $A_k + A_{k+1}$ witnessing $A_k \leqslant_{fin} A_k + A_{k+1}$, i.e. such that $[q, r] \cong A_k$. Then one of the following holds:
\begin{itemize}
    \item[x.] $q \geq p$;
    \item[y.] $q < p$.
\end{itemize}

If (ii.a.x.) holds, then $A_k \leqslant_{fin} A_{k+1}$. Since $A_{k+1} \leqslant_{init} A_k$, this gives $A_k \cong A_{k+1}$ by \ref{CSBLO}. Since $\omega A_{k+1} \leqslant_{init} A_k$, we have the chain
\[
2A_{k+1} \leqslant_{init} \omega A_{k+1} \leqslant_{init} A_k \cong A_{k+1}
\]
which gives $2A_{k+1} \leqslant_{init} A_{k+1}$. Hence $A_{k+1}$ is splitting by \ref{2XconvXiff2XcongX}, and so $A_k$ is splitting as well. Now we have
\[
A + B \cong MA_k + A_{k+1} \cong MA_k + A_k \cong (M+1)A_k \cong A_k \cong MA_k \cong B + A, 
\]
and we are done in case (ii.a.x.). 

Suppose (ii.a.y.) holds, and fix an isomorphism $f: [q, r] \rightarrow [l, p]$. Since $q < p$, $f$ is an overlapping partial convex self-embedding, decreasing at the right endcut $r$ of its domain. The jumps in the orbital $O_f(r)$ are isomorphic to $[f(r), r] = [p, r] \cong A_{k+1}$. There are two further possibilities:
\begin{itemize}
    \item[0.] $O_f(r)$ is a finitary orbital;
    \item[1.] $O_f(r)$ is an $\omega^*$-orbital. 
\end{itemize}

If (ii.a.y.0.), from the discussion in Section \ref{subsect:orbsofends} (see equation \ref{finitaryleftorbitaldecomp}) we have that $O_f(r)$ decomposes $\textrm{ext}(f) = [l, r] \cong A_k + A_{k+1}$ as a sum $C + NA_{k+1}$ for some final segment $C$ of $A_{k+1}$. Viewing $C + NA_{k+1}$ as a final segment of $(N+1)A_{k+1}$, we have the chain
\[
\omega A_{k+1} \leqslant_{init} A_k \leqslant_{init} A_k + A_{k+1} \cong C + NA_{k+1} \leqslant_{fin} (N+1)A_{k+1}.
\]
It follows $\omega A_{k+1} \leqslant_{conv} (N+1)A_{k+1}$, and hence $\omega A_{k+1} \leqslant_{conv} A_{k+1}$ by \ref{directedrefinementforomegaA}. Thus $A_{k+1}$ is splitting. From this and the hypothesis $A_k \leqslant_{fin} A_k + A_{k+1}$, we get
\[
A_k \leqslant_{fin} A_k + A_{k+1} \leqslant_{fin} (N+1)A_{k+1} \cong A_{k+1}.
\]
Again, this yields $A_k \cong A_{k+1}$, which as in case (ii.a.x.) yields $A + B \cong B + A$, concluding (ii.a.y.0.). 

Finally, suppose (ii.a.y.1) holds. Then $O_f(r) \cong \omega^*A_{k+1}$. Let $b$ denote the cut at the left of $O_f(r)$. If $b \geq q$, then $[b, r]$ is final in $[q, r]$, which witnesses $\omega^* A_{k+1} \leqslant_{fin} A_k$. This gives $A_k + A_{k+1} \cong A_k$, and then again we have
\[
A + B \cong MA_k + A_{k+1} \cong MA_k \cong B + A.
\]
If instead $b < q$, then $[q, r]$ is final in $[f^n(r), r]$ for some $n$, witnessing $A_k \leqslant_{fin} nA_{k+1}$. As before, this implies $A_k \cong A_{k+1} \cong 2A_{k+1} \cong 2A_k$, which once again gives $A + B \cong B + A$, and we are done in case (ii.a.y.1.). This concludes case (ii.), and hence concludes the case when the algorithm terminates at a finite stage. 

Now suppose the algorithm is non-terminating. Then it yields a non-terminating left division system
\[
\mathcal{D} = \{A_k \cong N_k A_{k+1} + A_{k+2}: k \in \omega\}.
\]
This system is either splitting or non-splitting. Suppose first the system is splitting. Then by \ref{splittingdichthmfordivsystems} both $A_0$ and $A_1$ are splitting. 

We always have $A_1 \leqslant_{init} A_0$, and since $A_0 \leqslant_{init} (N_0 + 1) A_1 \cong A_1$ by \ref{splittinglemmadivissystems}, we have $A_0 \leqslant_{init} A_1$ as well. On the other hand, since $B + A \leqslant_{fin} A + B$, either $A \leqslant_{fin} B$ or $B \leqslant_{fin} A$. That is, either $A_0 \leqslant_{fin} A_1$ or $A_1 \leqslant_{fin} A_0$. In either case, we conclude from \ref{CSBLO} that $A_0 \cong A_1$. In particular $A_0 + A_1 \cong A_1 + A_0$, i.e $A + B \cong B + A$, and we are done in this case. 

Now suppose the system is non-splitting. From the discussion following Proposition \ref{AkplusAkplusone}, the isomorphisms in the system $\mathcal{D}$ yield the decompositions
\[
\begin{array}{rclcl}
A_0 + A_1 & \cong & N_0A_1 + N_1A_2 + \cdots + L & = & \sum_{i \in \omega} N_i A_{i+1} + L \\
A_1 + A_0 & \cong & N_0A_1 + N_1A_2 + \cdots + L' & = & \sum_{i \in \omega} N_i A_{i+1} + L'
\end{array}
\]
To show $A_0 + A_1 \cong A_1 + A_0$, it suffices to show $L \cong L'$. 

Label $A_0 + A_1 = [c, d] = I_0$, and fix a sequence of cuts 
\[
c = c_0 < c_1 < c_2 < \ldots < b \leq d
\]
witnessing the decomposition for $A_0 + A_1$, so that $c_n$ corresponds to the cut at the $+$ sign to the left of the term $N_n A_{n+1}$, and $b$ to the cut at the $+$ sign preceding $L$. Label $I_n = [c_n, d] \cong \sum_i N_{n+i} A_{n+i+1} + L$. 

It follows from the discussion after Proposition \ref{AkplusAkplusone} that for $n$ even we have $I_n \cong A_n + A_{n+1}$ and for $n$ odd we have $I_n \cong A_{n+1} + A_n$. Each difference $I_n \setminus I_{n+1} = [c_n, c_{n+1}]$ is isomorphic to $N_n A_{n+1}$. We also have $[b, d] = \bigcap_n I_n \cong L$. 

Similarly, label $A_1 + A_0 = [c', d'] = I_0'$ and fix a sequence of cuts
\[
c' = c_0' < c_1' < c_2' < \ldots < b' \leq d'
\]
witnessing the decomposition for $A_1 + A_0$. Label $I_n' = [c_n', d']$. Then for $n$ even we have $I_n' \cong A_{n+1} + A_n$ and for $n$ odd we have $I_n' \cong A_n + A_{n+1}$. Further, $[c_n', c_{n+1}'] \cong N_n A_{n+1}$ for all $n$, and $[b', d'] = \bigcap I_n' \cong L'$. 

Suppose for concreteness $A_0 = A$ and $A_1 = B$; the argument for the reverse case is symmetric. Fix an embedding $f: B + A \rightarrow A + B$ witnessing $B + A \leqslant_{fin} A + B$, i.e. a final embedding $f: A_1 + A_0 \rightarrow A_0 + A_1$. Then $f(d') = d$. 

We claim $f(b') = b$. If this holds, then $f[L'] = [f(b'), f(d')] = [b, d] = L$, witnessing $L' \cong L$, and we are done. 

So suppose the claim is false. Then either $f(b') < b$ or $f(b') > b$. If $f(b') < b$ then for all sufficiently large $k$ we have $f(b') < c_k$. Fix an even such $k$. We have $[c_k, c_{k+1}] \cong N_k A_{k+1}$ is a subinterval of $[f(b'), d] = f[L'] \cong L'$. It follows $A_{k+1} \leqslant_{conv} L'$. Since $2A_{k+3} \leqslant_{init} A_{k+1}$ by \ref{splittinglemmadivissystems}, it follows $2A_{k+3} \leqslant_{conv} L'$. 

Since $L'$ is final in $I_n'$ for all $n'$, in particular $L' \leqslant_{fin} I_{k+5}'$. Since $k$ is even, $I_{k+5}' \cong A_{k+5} + A_{k+6}$. Thus we have the chain
\[
2A_{k+3} \cong A_{k+3} + A_{k+3} \leqslant_{conv} L' \leqslant_{conv} A_{k+5} + A_{k+6}.
\]
From \ref{AplusBconvXplusY} it follows either $A_{k+3} \leqslant_{conv} A_{k+5}$ or $A_{k+3} \leqslant_{conv} A_{k+6}$. Since $2A_{k+5} \leqslant_{init} A_{k+3}$ and $2A_{k+6} \leqslant_{init} A_{k+4} \leqslant_{init} A_{k+3}$, we get either $2A_{k+5} \leqslant_{conv} A_{k+5}$ or $2A_{k+6} \leqslant_{init} A_{k+6}$. Hence either $A_{k+5}$ or $A_{k+6}$ is splitting, which implies that $\mathcal{D}$ is splitting by \ref{splittingdichthmfordivsystems}, a contradiction. 

The case when $f(b') > b$ is similar. Thus $f(b') = b$, as claimed. As noted, it follows $A_0 + A_1 \cong A_1 + A_0$, i.e. $A + B \cong B + A$. We are done. 
\end{proof}

One way we will apply Theorem \ref{BplusAinitandfinAplusB} is through the following corollary. 

\theoremstyle{definition}
\newtheorem{AplusBBplusAinitfinXandY}[lct]{Corollary}
\begin{AplusBBplusAinitfinXandY}\label{AplusBBplusAinitfinXandY}
Suppose $A + B \leqslant_{init} X$ and $B + A \leqslant_{init} X$, and $A + B \leqslant_{fin} Y$ and $B + A \leqslant_{fin} Y$. Then $A + B \cong B + A$. 
\end{AplusBBplusAinitfinXandY}

\begin{proof}
Since $A + B \leqslant_{init} X$ and $B + A \leqslant_{init} X$, either $A + B \leqslant_{init} B + A$ or $B + A \leqslant_{init} A + B$. Suppose without loss of generality that $A + B \leqslant_{init} B + A$. 

Symmetrically, $A + B \leqslant_{fin} Y$ and $B + A \leqslant_{fin} Y$ implies $A + B \leqslant_{fin} B + A$ or $B + A \leqslant_{fin} A + B$. In the former case, by Theorem \ref{BplusAinitandfinAplusB} we have $A + B \cong B + A$. In the latter case, we get $A + B \cong B + A$ by \ref{CSBLO}.
\end{proof}

\theoremstyle{definition}
\newtheorem{omegaABsumdefn}[lct]{Definition}
\begin{omegaABsumdefn}\label{omegaABsumdefn}
\phantom{.}
\begin{itemize}
\item[i.] An \textit{$\omega$-$AB$-sum} is an $\omega$-sum
\[
C_0 + C_1 + \cdots = \sum_{i \in \omega} C_i
\]
such that for all $i \in \omega$, $C_i = A$ or $C_i = B$.
\item[ii.] An \textit{$\omega^*$-$AB$-sum} is an $\omega^*$-sum
\[
\cdots + C_1 + C_0 = \sum_{i \in \omega^*} C_i 
\]
such that for all $i \in \omega^*$, $C_i = A$ or $C_i = B$.
\end{itemize}
\end{omegaABsumdefn}

Two $\omega$-$AB$-sums $\sum_{i \in \omega} C_i$ and $\sum_{i \in \omega} D_i$ are \textit{offset} if one of $C_0$ and $D_0$ is equal to $A$ and the other is equal to $B$. Offset $\omega^*$-$AB$-sums are defined symmetrically.

We will use the following theorem to prove that the isomorphisms between sums appearing in the Generalized Sum Problem imply $A + B \cong B + A$. 

\theoremstyle{definition}
\newtheorem{offsetomegaABsumthm}[lct]{Theorem}
\begin{offsetomegaABsumthm}\label{offsetomegaABsumthm}
\phantom{.}
\begin{itemize}
\item[i.] Suppose $\sum_{i \in \omega} C_i$ and $\sum_{i \in \omega} D_i$ are offset $\omega$-$AB$-sums, and $\sum_i C_i \leqslant_{init} \sum_i D_i$. Then $A + B \leqslant_{init} \sum_i D_i$ and $B + A \leqslant_{init} \sum_i D_i$.
\item[ii.] Suppose $\sum_{i \in \omega^*} C_i$ and $\sum_{i \in \omega^*} D_i$ are offset $\omega^*$-$AB$-sums, and $\sum_i C_i \leqslant_{fin} \sum_i D_i$. Then $A + B \leqslant_{fin} \sum_i D_i$ and $B + A \leqslant_{fin} \sum_i D_i$.
\end{itemize}
\end{offsetomegaABsumthm}

\begin{proof}
For (i.): By symmetry, we may assume $C_0 = A$ and $D_0 = B$. 

Let $X = \sum_i D_i$ and label $X = [y, z]$ by its endcuts. Our goal is to show $A + B$ and $B + A$ both embed initially in $X$. 

Fix a cut sequence
\[
y = d_0 < d_1 < d_2 < \ldots 
\]
witnessing $X = \sum_i D_i$, i.e. such that $[d_i, d_{i+1}] \cong D_i$ for all $i \in \omega$ and such that the cut at the right of this sequence is the right endcut $z$.

Also fix a cut sequence
\[
y = c_0 < c_1 < c_2 < \ldots 
\]
witnessing $\sum_i C_i \leqslant_{init} X$. One of the following holds:
\begin{itemize}
    \item[I.] $c_1 \leq d_1$;
    \item[II.] $c_1 > d_1$. 
\end{itemize}

Suppose (I.). Then $A \cong [c_0, c_1] \leqslant_{init} [d_0, d_1] \cong B$, which gives $A \leqslant_{init} B$.

We claim $B + A \leqslant_{init} X$. Indeed, if $D_1 = A$, then we have 
\[
B + A \cong D_0 + D_1 \leqslant_{init} \sum_{i \in \omega} D_i = X.
\]
If instead $D_1 = B$, then $B + B \cong D_0 + D_1 \leqslant_{init}X$. Since $A \leqslant_{init} B$ we have 
\[
B + A \leqslant_{init} B + B \leqslant_{init} X.
\]
Thus in either case $B + A \leqslant_{init} X$, as claimed. 

It remains to show $A + B \leqslant_{init} X$ in case (I.). One of the following conditions holds:
\begin{itemize}
    \item[a.] $C_i = A$ for all $i \in \omega$;
    \item[b.] There is a least $N \geq 1$ such that $C_N = B$. 
\end{itemize}

Suppose (I.a.). Then $\sum_i C_i \cong \omega A$. Let $c$ denote the cut at the right of the sequence $c_i$ in $X$. If $c \leq d_1$, then this witnesses $\omega A \leqslant_{init} [d_0, d_1] \cong B$. In this case, $A + B \cong B$. Hence in particular $A + B \leqslant_{init} B$. Since $B$ is initial in $X$, we have $A + B \leqslant_{init} X$. 

Otherwise let $M$ be least such that $d_1 \leq c_M$. Then $B \cong [d_0, d_1] \leqslant_{init} [c_0, c_M] \cong MA$. Hence $A + B \leqslant_{init} A + MA \cong (M+1)A$. Then since 
\[
(M+1)A \leqslant_{init} \omega A \cong \sum_{i\in \omega} C_i \leqslant_{init} X
\]
we have $A + B \leqslant_{init} X$, which finishes the argument for (I.a.) 

Suppose (I.b.). If $N = 1$, then $A + B = C_0 + C_1 \leqslant_{init} \sum_i C_i \leqslant_{init} X$, and we are done immediately. So suppose $N \geq 2$. One of the following holds:
\begin{itemize}
    \item[x.] $c_N \geq d_1$;
    \item[y.] $c_N < d_1$. 
\end{itemize}

Suppose (I.b.x.). Let $M \leq N$ be least such that $d_1 \leq c_M$. Then, as above, we have $B \cong [d_0, d_1]$ initial in $MA \cong [c_0, c_M]$, and hence $A + B \leqslant_{init} (M+1)A$. 

We claim that we again $(M+1)A \leqslant_{init} \sum_i C_i$. This is immediate if $M < N$, as then $C_0 + \cdots + C_M \cong (M+1)A$. Otherwise $M = N$, and so $C_0 + \cdots + C_{M-1} + C_M \cong MA + B$. Then since $A \leqslant_{init} B$, we have $MA + A$ initial in $MA + B$. So again we have $(M+1)A \leqslant_{init} \sum_i C_i$, which gives $A + B \leqslant_{init} X$, concluding (I.b.x.). 

Suppose (I.b.y.). We have $[d_0, d_1] = D_0 \cong B \cong C_N = [c_N, c_{N+1}]$. Fix an isomorphism $f: [d_0, d_1] \rightarrow [c_N, c_{N+1}]$. Then $f$ is a partial convex self-embedding of $X$, which is overlapping since $d_0 < c_N < d_1$. 

Observe that the jump $A_{d_0, f}$ equals $[d_0, c_N] = [c_0, c_N] \cong NA$. There are two final cases to consider:
\begin{itemize}
    \item[1.] $O_f(d_0)$ is an $\omega$-orbital;
    \item[2.] $O_f(d_0)$ is finitary. 
\end{itemize}

If (I.b.x.1.), then $\omega NA \cong \omega A$ is initial in $\textrm{dom}(f) = [d_0, d_1] \cong B$. Hence $A + B \cong B \leqslant_{init} X$, and we are done as above. 

Finally, suppose (I.b.x.2.). Then for some $K \geq 1$ (as $f$ is overlapping) and some order $C$ initial in $NA$, we have that $\textrm{dom}(f) = [d_0, d_1]$ is isomorphic to $K(NA) + C$ (see equation \ref{finitaryleftorbitaldecompdomain}), and hence $B \cong K(NA) + C$. 

Now, we know $NA + B \cong C_0 + \cdots + C_{N-1} + C_N \leqslant_{init} \sum_i C_i$, which gives $NA + KNA + C \leqslant_{init} \sum_i C_i$. We may rewrite this as 
\begin{equation}\label{marlboro}
A + KNA + (N-1)A + C \leqslant_{init} \sum_i C_i.
\end{equation}
We claim $C \leqslant_{init} (N-1)A + C$; if this holds, then combining with \ref{marlboro} we get the chain
\begin{equation}\label{lasso}
A + B \cong A + KNA + C \leqslant_{init} A + KNA + (N-1)A + C \leqslant_{init} \sum_i C_i \leqslant_{init} X,
\end{equation}
as desired. So we prove the claim. We know $C \leqslant_{init} NA$, i.e. 
\begin{equation}\label{cowboy}
C \leqslant_{init} (N-1)A + A.
\end{equation}
Since also $A \leqslant_{init} NA$, it follows $C \leqslant_{init} A$ or $A \leqslant_{init} C$. If $C \leqslant_{init} A$, then since $N \geq 2$ we immediately have $C \leqslant_{init} (N-1)A + C$. If $A \leqslant_{init} C$, we have $(N-1)A + A \leqslant_{init} (N-1)A + C$, which from \ref{cowboy} also gives $C \leqslant_{init} (N-1)A + C$, as claimed. 

Thus the chain \ref{lasso} holds, and we have $A + B \leqslant_{init} X$, concluding the proof for (I.b.x.2.), and hence the proof for (I.).

Now suppose (II.). The proof is similar to (I.), but with slightly less casework. We summarize it and leave the details. 

Since now $B \leqslant_{init} A$, we have easily $A + B \leqslant_{init} C_0 + C_1 \leqslant_{init} \sum_i C_i \leqslant_{init} \sum_i D_i$. So we must check $B + A \leqslant_{init} \sum_i D_i$. 

Since $\sum_i C_i$ is initial in $\sum_i D_i$, the first term $C_0 \cong A$ is initial in some finite sum $D_0 + \cdots + D_{N-1}$. If all of these terms are $B$, then $A$ is initial in $NB$. Then it is not hard to see $(N+1)B$ is initial in $\sum_i D_i$, and hence $B + A$ is as well. 

If instead the least copy of $A$ is $D_N$ and this term overlaps the initial term $C_0 \cong A$, we fix an isomorphism $f$ from $C_0$ to $D_N$. Then $f$ is an overlapping partial convex embedding of $X$, with initial jump some finite multiple of $B$. If its left orbital is an $\omega$-orbital, then it is isomorphic to $\omega B$ and we have $B + A \cong A$, which yields $B + A \leqslant_{init} \sum_i D_i$. If its left orbital is finitary, then an analogous argument to the finitary case for (I.) again shows $B + A \leqslant_{init} \sum_i D_i$. 

The proof for (ii.) is symmetric. 
\end{proof}

We may also consider finite $AB$-sums. 

\theoremstyle{definition}
\newtheorem{ABsumdefn}[lct]{Definition}
\begin{ABsumdefn}\label{ABsumdefn}
An \textit{$AB$-sum} is an $n$-sum for some $n \in \omega$,
\[
C_0 + C_1 + \cdots + C_{n-1} = \sum_{i < n} C_i,
\]
such that for all $i < n$, $C_i = A$ or $C_i = B$.
\medskip

Two $AB$-sums $\sum_{i < n} C_i$ and $\sum_{j < m} D_j$ are \textit{offset left} if one of $C_0$ and $D_0$ is equal to $A$ and the other to $B$, and \textit{offset right} if one of $C_{n-1}$ and $D_{m-1}$ is equal to $A$ and the other to $B$. They are \textit{offset} if they are offset left and right.
\end{ABsumdefn}

The following theorem solves the Generalized Sum Problem affirmatively. 

\theoremstyle{definition}
\newtheorem{generalizedsumproblemsoln}[lct]{Theorem}
\begin{generalizedsumproblemsoln}\label{generalizedsumproblemsoln}
Suppose that $\sum_{i < n} C_i \cong \sum_{j < m} D_j$ are isomorphic $AB$-sums.
\begin{itemize}
    \item[i.] If the sums are offset left, then either $A + B \leqslant_{init} B + A$ or $B + A \leqslant_{init} A + B$;
    \item[ii.] If the sums are offset right, then either $A + B \leqslant_{fin} B + A$ or $B + A \leqslant_{fin} A + B$;
    \item[iii.] If the sums are offset, then $A + B \cong B + A$. 
\end{itemize}
\end{generalizedsumproblemsoln} 

\begin{proof}
For (i.): Extend $\sum_{i < n} C_i$ and $\sum_{j < m} D_j$ to $\omega$-$AB$-sums by defining $C_i = D_j = A$ for all $i \geq n$ and $j \geq m$. Then the resulting sums $\sum_{i \in \omega} C_i$ and $\sum_{j \in \omega} D_j$ are (left) offset and remain isomorphic. In particular, $\sum_{i \in \omega} C_i \leqslant_{init} \sum_{j \in \omega} D_j$. By Theorem \ref{offsetomegaABsumthm}, $A + B$ and $B + A$ are both initial in $\sum_{j \in \omega} D_j$. It follows that either $A + B \leqslant_{init} B + A$ or $B + A \leqslant_{init} A + B$.

For (ii.): The argument is symmetric.

For (iii.): Since the sums are offset left, we have from (i.) that either $A + B \leqslant_{init} B + A$ or $B + A \leqslant_{init} A + B$. By symmetry, we may assume $B + A \leqslant_{init} A + B$. 

Since the sums are also offset right, we get from (ii.) that either $A + B \leqslant_{fin} B + A$ or $B + A \leqslant_{fin} A + B$. In the former case, Theorem \ref{CSBLO}, and in the latter, Theorem \ref{BplusAinitandfinAplusB}, now give $A + B \cong B + A$. 
\end{proof}

Since the sums in Tarski's Generalized Sum Problem are offset, Theorem \ref{generalizedsumproblemsoln}.(iii.) solves the problem in the affirmative. Theorem \ref{generalizedsumproblemsoln} is more general, since it applies to $AB$-sums $\sum_{i < n} C_i \cong \sum_{j < m} D_j$ in which $C_0 = C_{n-1}$ and $D_0 = D_{m-1}$, one of these pairs equaling $A$ and the other $B$. In particular, we have the following corollary, which was proved by Tarski in \cite{Tarski} by a different route. 

\theoremstyle{definition}
\newtheorem{nAcongmBimpliesABcommute}[lct]{Proposition}
\begin{nAcongmBimpliesABcommute}\label{nAcongmBimpliesABcommute} (Tarski; \cite[1.49]{Tarski})
If $nA \cong mB$ for some $n, m \in \omega$, then $A + B \cong B + A$.
\end{nAcongmBimpliesABcommute}

\begin{proof}
Since $nA$ and $mB$ are offset $AB$-sums, the proposition follows immediately from Theorem \ref{generalizedsumproblemsoln}. 
\end{proof}

If $nA \cong mB$, then $\omega A \cong \omega B$ since
\[
\omega A \cong \omega (nA) \cong \omega (mB) \cong \omega B.
\]
And symmetrically, we have $\omega^*A \cong \omega^*B$. Thus the following corollary, which is an immediate consequence of Theorem \ref{offsetomegaABsumthm}, may be viewed as a generalization of Proposition \ref{nAcongmBimpliesABcommute}. 

\theoremstyle{definition}
\newtheorem{omegaAomegastarAinitfinomegaBomegastarB}[lct]{Corollary}
\begin{omegaAomegastarAinitfinomegaBomegastarB}\label{omegaAomegastarAinitfinomegaBomegastarB}
\phantom{.}
\begin{itemize}
    \item[i.] If $\omega A \leqslant_{init} \omega B$, then $A + B \leqslant_{init} \omega B$ and $B + A \leqslant_{init} \omega B$. 
    \item[ii.] If $\omega^* A \leqslant_{fin} \omega^*B$, then $A + B \leqslant_{fin} \omega^*B$ and $B + A \leqslant_{fin} \omega^*B$.
    \item[iii.] If $\omega A \leqslant_{init} \omega B$ and $\omega^* A \leqslant_{fin} \omega^*B$, then $A + B \cong B + A$. 
\end{itemize}
\end{omegaAomegastarAinitfinomegaBomegastarB}

\begin{proof}
Since $\omega A$ and $\omega B$ are offset $\omega$-$AB$-sums, (i.) is immediate from Theorem \ref{offsetomegaABsumthm}; (ii.) is symmetric; (iii.) now follows from Corollary \ref{AplusBBplusAinitfinXandY}, taking $X = \omega B$ and $Y = \omega^*B$. 
\end{proof}

The symmetrized converse to Corollary \ref{omegaAomegastarAinitfinomegaBomegastarB}.(iii) is also true. 

\theoremstyle{definition}
\newtheorem{ABcommuteiffomegaAomegastarAinitfinomegaBomegastarBorvv}[lct]{Theorem}
\begin{ABcommuteiffomegaAomegastarAinitfinomegaBomegastarBorvv}\label{ABcommuteiffomegaAomegastarAinitfinomegaBomegastarBorvv}
$A + B \cong B + A$ if and only if one of the following conditions holds:
\begin{itemize}
    \item[i.] $\omega A \leqslant_{init} \omega B$ and $\omega^*A \leqslant_{fin} \omega^*B$;
    \item[ii.] $\omega B \leqslant_{init} \omega A$ and $\omega^* B \leqslant_{fin} \omega^* A$.
\end{itemize}
\end{ABcommuteiffomegaAomegastarAinitfinomegaBomegastarBorvv}

We will prove Theorem \ref{ABcommuteiffomegaAomegastarAinitfinomegaBomegastarBorvv} in Section \ref{section:continuationandvaluation}, and from it deduce the revised version of Tarski's conjecture \ref{revisedtarconj}. 

\section{A proof of Lindenbaum's Division Theorem}\label{section:linddivisthm}

In this section we use the Euclidean algorithm to give a proof of Lindenbaum's Division Theorem \ref{ldt}. It may be contrasted with Tarski's proof \cite[1.50]{Tarski} of the division theorem for a general ordinal algebra. As far as the authors are aware, Tarski's proof is the only previously published proof of the Division Theorem. 

Like the proofs of our other arithmetic results for $(LO, +)$, the proof here requires casework to handle the absorption and splitting phenomena that occur in sums of general linear orders but not in sums of natural (or real) numbers. However, the approach of the proof is to straightforwardly divide over an equation just as though working over $(\mathbb{N}, +)$ or $(\mathbb{R}, +)$: given a pair of orders $A, B \in LO$ that have a common finite multiple $nA \cong mB$, we will apply the Euclidean algorithm to find their greatest common divisor $C \in LO$.

Theorem \ref{divtheorem} below is the Division Theorem written in a slightly generalized form. Modulo the Cancellation Theorem \ref{lct}, this statement is equivalent to the original. On the other hand, as stated Theorem \ref{divtheorem} also implies \ref{lct}.

\theoremstyle{definition}
\newtheorem{divtheorem}[lct]{Theorem}
\begin{divtheorem}\label{divtheorem}
Suppose there are natural numbers $n, m \geq 1$ such that $nA \cong mB$. Let $d = \gcd{(n, m)}$, and let $m' = \frac{m}{d}$ and $n' = \frac{n}{d}$. Then there exists a linear order $C$ such that $A \cong m'C$ and $B \cong n'C$. 
\end{divtheorem}

The proof goes by first observing that if $A$ and $B$ are splitting, then the statement trivializes. In the non-splitting case, we show that the Euclidean algorithm, applied to a symmetric setup witnessing $A + B \cong B + A$, proceeds in exactly the same way as the classical Euclidean algorithm for dividing the greater of the natural numbers $n'$ and $m'$ by the lesser. Hence the algorithm terminates; the order $C$ will be the final remainder term.

\begin{proof}
Writing $nA \cong mB$ as $d(n'A) \cong d(m'B)$ and then applying cancellation \ref{lctintext} yields $n'A \cong m'B$. By Lemma \ref{nAcongmBimpliesABcommute} we have $A + B \cong B + A$.

If $n' = m' = 1$, then $A \cong B$; let $C = A$ and we are done. Otherwise, since $\gcd (n', m') = 1$, either $n' < m'$ or $m' < n'$. 

By symmetry, we may assume $n' < m'$. If $n' = 1$, then $A \cong m'B$. Let $C = B$ and we are again done. So assume $1 < n' < m'$. 

Now from $n'A \cong m'B$ and Lemma \ref{2AconvnB2BconvmAsplittingiff} it follows that $A$ is splitting if and only if $B$ is splitting. In the splitting case, $A \cong n'A \cong m'B \cong B$. Then also $A \cong m'A \cong n'A \cong n'B \cong B$. Letting $C = A$, we are again done. 

So assume $A$ and $B$ are non-splitting. Let $A_0 = A$, $A_1 = B$, $a_1 = n'$, and $a_0 = m'$ so that
\begin{equation}\label{a1A0equalsa0A1}
a_1A_0 \cong a_0A_1
\end{equation}
Since $a_0 > a_1$, it follows from \ref{a1A0equalsa0A1} that $a_1A_1 \leqslant_{init} a_1A_0$ and $a_1 A_1 \leqslant_{fin} a_1 A_0$. Then by Theorem \ref{leqslantinitfincancellationthm}, $A_1 \leqslant_{init} A_0$ and $A_1 \leqslant_{fin} A_0$. 

Fix a symmetric division setup $(I = [c, d]; f, g)$ witnessing $A_0 + A_1 \cong A_1 + A_0$ and run the (left, say) Euclidean algorithm. 

We claim that $A_0$ is the dividend and $A_1$ is the divisor in this setup. Otherwise, we have $A_0 \leqslant_{init} A_1$, which along with $A_1 \leqslant_{fin} A_0$ yields $A_1 \cong A_0$ by \ref{CSBLO}. But then \ref{a1A0equalsa0A1} along with the splitting dichotomy theorem \ref{splittingdichthm} implies $A_0$ and $A_1$ are splitting, contradicting our hypothesis. 

If the algorithm terminates at Stage $0$ in an absorbed factor, then $\omega A_1 \leqslant_{init} A_0$. Then since $a_1A_0 \cong a_0A_1 \leqslant_{init} \omega A_1 \leqslant A_0$, we then have $a_1A_0 \leqslant_{init} A_0$. Hence $A_0$ is splitting by $\ref{2XconvXiff2XcongX}$, a contradiction. 

Thus the algorithm does not terminate at Stage 0 in an absorbed factor and we may let $\{A_k \cong N_k A_{k+1} + A_{k+2}: k < M\}$ be the symmetric division system (of length $1 \leq M \leq \omega$) resulting from the run of the algorithm. As noted in Section \ref{subsect:runofsymmetricalgo}, we have $A_i + A_j \cong A_j + A_i$ for all terms $A_i$, $A_j$ appearing in the system. As noted in Section \ref{section:divisionsystems}, if the system is terminating, then by rewriting its final isomorphism if necessary we may assume that it terminates in a divisor. 

We will show that the system is necessarily terminating. Toward this, run the (classical) Euclidean algorithm for dividing $a_0$ by $a_1$, and let $\{a_k = M_k a_{k+1} + a_{k+2}\}$ be the resulting division system. We will show that $\{A_k \cong N_k A_{k+1} + A_{k+2}\}$ is obtained from this system by replacing each term $a_i$ by $A_i$. That is, these systems have the same finite length and the same coefficients.

Consider the system of equations $\{b_k = N_k b_{k+1} + b_{k+2}\}$ obtained from the division system $\{A_k \cong N_k A_{k+1} + A_{k+2}\}$ by replacing each term $A_k$ with a variable term $b_k$. 

Suppose first that our division system $\{A_k \cong N_k A_{k+1} + A_{k+2}\}$ is terminating (in a divisor), with final isomorphism $A_K \cong N_K A_{K+1}$. Set $b_{K+1} = 1$ and let $b_0 > b_1 > \ldots > b_{K}$ be the sequence of natural numbers determined by the equations $b_k = N_k b_{k+1} + b_{k+2}$ for $k \leq K$. 
\medskip

\underline{Claim 1}: $b_0 A_1 \cong b_1 A_0$. 
\begin{proof}
We prove the claim by induction on $K$. If $K = 0$ then our division system consists of the single isomorphism $A_0 \cong N_0 A_1$. In this case $b_1 = 1$ and $b_0 = N_0$, so that $A_0 \cong N_0 A_1$ witnesses the claim. 

For larger $K$, assume the claim holds for systems of shorter length. Observe that $\{A_k \cong N_k A_{k+1} + A_{k+2}: 1 \leq k \leq K\}$ is a division system of length one less than the original system. By induction, $b_1 A_2 \cong b_2 A_1$. From this, the isomorphism $A_0 \cong N_0 A_1 + A_2$, and the equation $b_0 = N_0b_1 + b_2$ we obtain
\[
b_1 A_0 \cong b_1(N_0A_1 + A_2) \cong b_1N_0A_1 + b_1A_2 \cong N_0 b_1 A_1 + b_2 A_1 \cong b_0 A_1, 
\]
as desired. Note here that $A_1 + A_2 \cong A_2 + A_1$ is needed to justify the distribution \[
b_1(N_0A_1 + A_2) \cong b_1N_0A_1 + b_1A_2.
\]
This completes the induction and proves the claim. 
\end{proof}

\underline{Claim 2}: $\frac{a_0}{a_1} = \frac{b_0}{b_1}$.

\begin{proof}
Suppose $\frac{a_0}{a_1} > \frac{b_0}{b_1}$, so that $a_0b_1 > b_0a_1$. From $a_0A_1 \cong a_1A_0$ and $b_0A_1 \cong b_1A_0$ we obtain
\[
a_0b_1 A_1 \cong a_1b_1A_0 \cong a_1b_0A_1.
\]
Since $a_0b_1 > b_0a_1$, this implies $A_1$ is splitting by \ref{splittingdichthm}, a contradiction. The argument for $\frac{a_0}{a_1} < \frac{b_0}{b_1}$ is symmetric, and the claim is proved. 
\end{proof}

Since $b_0$ and $b_1$ are coprime (as $\gcd (b_0, b_1) = b_K = 1$), and $a_0$ and $a_1$ are also coprime, Claim 2 implies $b_0 = a_0$ and $b_1 = a_1$. Hence the system $\{b_k = N_k b_{k+1} + b_{k+2}\}$ is the same as $\{a_k = M_ka_{k+1} + a_{k+2}\}$ up to relabeling each $b_k$ as $a_k$; in particular $N_k = M_k$ for all $k \leq K$. 

Since setting $a_{K+1} = 1$ in the system $\{a_k = N_ka_{k+1} + a_{k+2}\}$ determines $a_0 = m'$ and $a_1 = n'$, back-substituting $A_{K+1}$ in the isomorphisms $A_k \cong N_kA_{k+1} + A_{k+2}$ yields the isomorphisms $A_0 \cong m' A_{K+1}$ and $A_1 \cong n' A_{K+1}$. Letting $C = A_{K+1}$, we are done in the case when the system is terminating. 

Now suppose $\{A_k \cong N_k A_{k+1} + A_{k+2}\}$ is a non-terminating system, and let $\{b_k = N_kb_{k+1} + b_{k+2}\}$ be the corresponding non-terminating system of equations in the variables $b_k$. For a fixed $K$, introduce variables $c_k^K$ for $0 \leq k \leq K+1$ and consider the truncated system of equations
\[
\begin{array}{rcl}
c_0^K & = & N_0 c_1^K + c_2^K \\
c_1^K & = & N_1 c_2^K + c_3^K \\
& \vdots & \\
c_{K-1}^K & = & N_{K-1} c_K^K + c_{K+1}^K \\
c_K^K & = & N_K c_{K+1}^K.
\end{array}
\]

Set $b_0 = 1$ and let $\{b_k\}$ be the unique sequence of positive real numbers satisfying the equations from our system $b_k = N_k b_{k+1} + b_{k+2}$. From the divided through equations
\[
\frac{b_k}{b_{k+1}} = N_k + \frac{b_{k+2}}{b_{k+1}}
\]
each $b_k$ may be computed via its continued fraction expansion
\[
\frac{b_{k+1}}{b_k} = 1/(N_k + 1/(N_{k+1} + 1/(N_{k+2} + \cdots))).
\]
In particular,
\[
b_1 = 1/(N_0 + 1/(N_1 + 1/(N_2 + \cdots))).
\]

We quote some standard facts about continued fractions and the sequence $\{b_k\}$. 

Since the system $\{b_k = N_kb_{k+1} + b_{k+2}\}$ is non-terminating, the ratios $\frac{b_k}{b_{k+1}}$ are irrational. In particular, $\frac{b_0}{b_1} = \frac{1}{b_1}$ is irrational and hence so is $b_1$. 

For each $K$, let $c^K_{K+1} = 1$ and then let $c^K_k$ for $0 \leq k \leq K$ be the positive integers satisfying the truncated system $\{c_k^K = N_k c_{k+1}^K + c_{k+2}^K: 0 \leq k \leq K\}$. Then the ratios $\frac{c^K_0}{c^K_1}$ converge to $\frac{b_0}{b_1}$ as $K \rightarrow \infty$. 

More specifically, the subsequence of even terms $\frac{c^{2K}_0}{c^{2K}_1}$ increases to $\frac{b_0}{b_1}$, and the sequence of odd terms $\frac{c^{2K+1}_0}{c^{2K+1}_1}$ decreases to $\frac{b_0}{b_1}$:
\[
\frac{c^{0}_0}{c^{0}_1} < \frac{c^{2}_0}{c^{2}_1} < \frac{c^{4}_0}{c^{4}_1} < \ldots < \frac{b_0}{b_1} < \ldots < \frac{c^{5}_0}{c^{5}_1} < \frac{c^{3}_0}{c^{3}_1} < \frac{c^{1}_0}{c^{1}_1}.
\]
\medskip

\underline{Claim 3}: For $K$ even, we have
\[
c_1^K A_0 \cong c_0^K A_1 + A_{K+2}
\]

For $K$ odd, we have
\[
c_1^K A_0 + A_{K+2} \cong c_0^K A_1.
\]
\begin{proof}
By induction on $K$. If $K = 0$, the truncated system consists of a single equation $c^0_0 = N_0c^0_1$. Since $c^0_1 = 1$, we have $c^0_0 = N_0$. Thus
\[
c^0_1 A_0 = A_0 \cong N_0 A_1 + A_2 \cong c_0^0A_1 + A_2,
\]
which is of the desired form. 

Fix $K \geq 1$ and assume the claim for shorter truncations. Suppose first that $K$ is odd. Since $A_0 \cong N_0A_1 + A_2$, using commutativity of the terms $A_1$ and $A_2$ we have
\begin{equation}\label{c1expK}
c_1^K A_0 \cong c_1^KN_0A_1 + c_1^K A_2.
\end{equation}

We may apply our induction hypothesis to the shorter truncated system $\{c^K_k = N_kc^K_{k+1} + c^K_{k+2}: 1 \leq k \leq K\}$ corresponding to the division system $\{A_k \cong N_k A_{k+1} + A_{k+2}: 1 \leq k < \omega \}$ to conclude
\begin{equation}\label{c2expKA1}
c_2^K A_1 \cong c_1^K A_2 + A_{K+2}
\end{equation}
since, if we re-index this system using $k' = k - 1$ and $K' = K - 1$, then $K'$ is even. From \ref{c1expK} we have
\[
c_1^K A_0 + A_{K+2} \cong c_1^KN_0A_1 + c_1^K A_2 + A_{K+2}.
\]
Then from \ref{c2expKA1},
\[
c_1^K A_0 + A_{K+2} \cong c_1^KN_0A_1 + c_2^KA_1 \cong (c_1^KN_0 + c_2^K)A_1 = c_0^KA_1,
\]
as desired. The induction step when $K$ is even is similar. 
\end{proof}

Since $\frac{b_0}{b_1}$ is irrational, we have $\frac{b_0}{b_1} \neq \frac{a_0}{a_1}$. Suppose $\frac{b_0}{b_1} > \frac{a_0}{a_1}$. Let $K$ be even and large enough such that $\frac{c_0^K}{c_1^K} > \frac{a_0}{a_1}$. Then $a_1c_0^K > a_0c_1^K$. 

From $a_0A_1 \cong a_1A_0$ we get $a_0c_1^KA_1 \cong a_1c_1^KA_0$. Then since $K$ is even we obtain
\[
a_0c_1^KA_1 \cong a_1c_1^KA_0 \cong a_1(c_0^KA_1 + A_{K+2}).
\]
Since $A_{K+2}$ commutes with $A_1$, we may distribute the $a_1$ on the right to conclude 
\[
a_0c_1^KA_1 \cong a_1c_0^KA_1 + a_1A_{K+2}.
\]
In particular $a_1c_0^KA_1 \leqslant_{conv} a_0c_1^KA_1$. But since $a_1c_0^K > a_0 c_1^K$, this implies that $A_1$ is splitting by \ref{splittingdichthm}, contradicting our hypothesis. 

We obtain a similar contradiction in the case when $\frac{b_0}{b_1} < \frac{a_0}{a_1}$. Hence the system must be terminating, and we are done. 
\end{proof}

\section{Symbolic and dynamical representations for division systems}\label{section:symboldynamrepnsofdivissystems}

In this section we establish structural representation theorems for the orders $A_k$ from a given non-terminating division system $\{A_k: k < \omega\}$. These representations are of two types: symbolic representations $X(I_{[u]})$ obtained as replacements of a lex-ordered sequence space $X$ up to the eventual equality relation $E_0$ on $X$, and dynamical representations $\mathbb{R}(K_{[x]})$ obtained as replacements of $\mathbb{R}$ up to the orbit equivalence relation $E_G$ of a group of translations on $\mathbb{R}$.

Here, $X$ is the space of continued fraction sequences in the base $\{a_k\}$ of a collection of real numbers satisfying the equations from the division system, and $G$ is the group of translations generated by these numbers. In Section \ref{section:aronszajnrepnthm}, we will use these representations to generalize Aronszajn's representation theorem for commuting pairs $A + B \cong B + A$ from \cite{Aronszajn}. In Section \ref{section:commutesemigroupsinLOrepn}, we use them again to help identify the commutative semigroups representable in $(LO, +)$.

\subsection{Symbolic representations of left division systems}\label{subsect:leftdivissystemsymbolicrepn}

A \textit{finite sequence} is a function $r$ with $\textrm{dom}(r) = n = \{0, 1, \ldots, n-1\}$ for some $n \in \omega$. The \textit{length} of such an $r$, denoted $|r|$, is $n$. We write $r_i$ for $r(i)$, and also write $r = r_0r_1 \ldots r_{n-1}$. An \textit{infinite sequence} is a function $u$ with domain $\omega$. We write $u = u_0u_1 \ldots$ where $u_i = u(i)$. 

Given a finite or infinite sequence $u$ and a natural number $n$ with $n \leq \textrm{dom}(u)$, we write $u \upharpoonright n$ for the restriction of $u$ to $n$, i.e. $u \upharpoonright n = u_0u_1 \ldots u_{n-1}$. Given a finite sequence $r$, $ru$ denotes the sequence obtained by concatenating $r$ on the right by $u$. 

A \textit{tree} is a set $T$ of finite sequences that is closed under initial segments: if $r \in T$ then for all $n \in \textrm{dom}(r)$ we have $(r \upharpoonright n) \in T$. If $s, r \in T$ with $|s| \geq |r|$ and $s \upharpoonright |r| = r$, we say $s$ \textit{extends} $r$ or $s$ is \textit{below} $r$. If $|s| = |r| + 1$ and $s$ extends $r$, that is, if $s$ has the form $s = rs_{|r|}$, then $s$ is a \textit{successor} of $r$. A sequence $r \in T$ is also called a \textit{node}. A node is \textit{terminal} if it has no successors in $T$. The \textit{body} of $T$, denoted $[T]$, is the set of infinite sequences $u$ such that $(u \upharpoonright n) \in T$ for all $n \in \omega$. 

We write $\omega^{<\omega}$ for the tree of all finite sequences $r$ with entries $r_i \in \omega$. We write $\omega^{\omega}$ for the set of all infinite sequences $u$ with entries $u_i \in \omega$. Observe $[\omega^{<\omega}] = \omega^{\omega}$.

\theoremstyle{definition}
\newtheorem{setofdiviscoeffsdefn}[lct]{Definition}
\begin{setofdiviscoeffsdefn}\label{setofdiviscoeffsdefn}
A set of \textit{division coefficients} is a set $\mathcal{C} = \{N_k: k \in \omega\}$ of natural numbers $N_k \geq 1$. If $N_k \geq 2$ for all $k \in \omega$, we also call $\mathcal{C}$ a set of \textit{rational division coefficients}. 
\end{setofdiviscoeffsdefn}

Fix a set of division coefficients $\mathcal{C} = \{N_k: k \in \omega\}$. 

\theoremstyle{definition}
\newtheorem{leftdivistreedefn}[lct]{Definition}
\begin{leftdivistreedefn}\label{leftdivistreedefn}
The \textit{left division tree $T^{\mathcal{C}} = T$ on the coefficients $\mathcal{C}$} is defined recursively as follows. 
\begin{itemize}
    \item $T_0 = \{\emptyset\}$. 
    \item Given $T_n$ for $n \in \omega$, and a terminal node $r \in T_n$ with $|r| = k$, define $S_r$ to be the set of extensions $\{r0, r1, \ldots, r(N_k-1), rN_k, rN_k0\}$ of $r$, consisting of the successors $ri$ for $0 \leq i \leq N_k$ and the double successor $rN_k0$. 

    Define 
    \[ 
    T_{n+1} = T_n \cup \bigcup_{\textrm{$r \in T_n$ terminal}} S_r.
    \]
    \item Define
    \[
    T = \bigcup_{n \in \omega} T_n.
    \]
\end{itemize} 
\end{leftdivistreedefn}

Intuitively, the root node $\emptyset$ in $T$ represents the initial order $A_0$ from some non-terminating left division system $\{A_k\}$ on the coefficients $\mathcal{C} = \{N_k\}$. Each node $r \in T$ with $|r| = k$ represents an interval in $A_0$ isomorphic to $A_k$. 

Notice that if $r \in T$ is terminal in $T_n$ for some $n$, and $|r| = k$, then the extensions of $r$ that are terminal in $T_{n+1}$ are $r0, r1, \ldots, r(N_k-1)$, and $rN_k0$. These nodes are listed in lexicographical order and have lengths $k+1, k+1, \ldots, k+1$, and $k+2$, respectively, reflecting that $A_k \cong N_k A_{k+1} + A_{k+2}$. The successor $rN_k$ of $r$, which also has length $k+1$ but is not terminal in $T_{n+1}$, has only the single successor $rN_k0$. This node represents the initial copy of $A_{k+2}$ in $A_{k+1}$ (from the decomposition $A_{k+1} \cong N_{k+1}A_{k+2} + A_{k+3}$), instead of $A_{k+1}$ itself. 

The nodes $r \in T$ that are terminal in $T_n$ for some $n$ are called \textit{regular}. 

\theoremstyle{definition}
\newtheorem{leftdivisorderdefn}[lct]{Definition}
\begin{leftdivisorderdefn}\label{leftdivisorderdefn}
The \textit{left division order $X^{\mathcal{C}} = X$ on the coefficients $\mathcal{C}$} is the body $[T^{\mathcal{C}}]$ of the left division tree $T^{\mathcal{C}}$, equipped with the lexicographical ordering $<_{lex}$. 
\end{leftdivisorderdefn}

Until further notice, $T = T^{\mathcal{C}}$ denotes the left division tree and $X = X^{\mathcal{C}}$ denotes the corresponding left division order for the fixed set of coefficients $\mathcal{C}$. Observe that $T \subseteq \omega^{<\omega}$ and $X = [T] \subseteq \omega^{\omega}$. 

Given $r \in T$, let $T_r$ denote the set of sequences in $T$ beginning with $r$, that is, 
\[
T_r = \{s \in T: \textrm{$|s| \geq |r|$ and $s \upharpoonright |r| = r$}\}.
\] 
Let $X_r = \{u \in X: u \upharpoonright |r| = r\}$. 

Observe that $X_{\emptyset} = X$ and each $X_r$ is convex in $X$. We call the intervals $X_r$ \textit{basic intervals}. Observe that if $r$ is an irregular node then $X_r = X_{r0}$. For regular nodes $r$ and successors $ri$ of $r$, $X_{ri}$ is always a strict subinterval of $X_r$. 

Given two infinite sequences $u, v \in \omega^{\omega}$, we say $u$ is \textit{eventually equal} to $v$, and write $u E_0 v$, if there are finite sequences $r, s \in \omega^{<\omega}$ with $|r| = |s|$ and an infinite sequence $u' \in \omega^{\omega}$ such that $u = ru'$ and $v = su'$. We call $E_0$ the \textit{eventual equality relation}; it is easy to check that it is an equivalence relation on $\omega^{\omega}$. 

Let $E_0^X$ denote the restriction of $E_0$ to $X$. We will often write $E_0$ instead of $E_0^X$. We write $[u]_{E_0^X}$, or often just $[u]$, for the $E_0^X$-class of a given $u \in X$. 

\theoremstyle{definition}
\newtheorem{E0classesdenseinX}[lct]{Proposition}
\begin{E0classesdenseinX}\label{E0classesdenseinX}
For all $u \in X$ and $r \in T$, we have $[u] \cap X_r \neq \emptyset$.
\end{E0classesdenseinX}

\begin{proof}
Fix $r \in T$, and suppose $|r| = k$. We may assume $r$ is regular, since if not we may extend it to a regular successor. Fix $u \in X$, and let $s = u \upharpoonright k = u_0u_1 \ldots u_{k-1}$. We may assume $s$ is also regular: if not, then $su_k$ is regular, and since $r0$ is always regular, we may replace $r$ with $r0$ and $s$ with $su_k$ to get an extension of $r$ and initial sequence of $u$ both of length $k+1$ and regular. 

Then $u = su'$ for some infinite tail-sequence $u'$. By the recursive construction of $T$, we have that $ru' \in X$. Clearly $ru' \in [u] \cap X_r$. 
\end{proof}

Theorem \ref{leftsymbolicrepnthm} below is our symbolic representation theorem for non-terminating left division systems. It uses the notion of replacement up to an equivalence relation defined in Section \ref{subsection:replacementsuptoE}.

\theoremstyle{definition}
\newtheorem{leftsymbolicrepnthm}[lct]{Theorem}
\begin{leftsymbolicrepnthm}\label{leftsymbolicrepnthm}
Suppose $T$ is the left division tree for the set of division coefficients $\mathcal{C} = \{N_k\}$ and $X = [T]$ is the corresponding left division order. 

\begin{itemize}
    \item[i.] For every non-terminating left division system  $\{A_k: k \in \omega\}$ on the coefficients $\mathcal{C}$, there exists a collection of orders $\{I_{[u]}: u \in X\}$ indexed by the $E_0$-classes $[u]$ of $X$ such that for every regular node $r \in T$ with $|r| = k$ we have $A_k \cong X_r(I_{[u]})$. In particular, $A_0 \cong X(I_{[u]})$. 
    
    \item[ii.] Conversely, suppose $\{I_{[u]}: u \in X\}$ is a collection of orders indexed by the $E_0$-classes of $X$ and $X(I_{[u]})$ is the corresponding replacement of $X$ up to the eventually equality relation $E_0$. Then for every pair of regular nodes $r, s \in T$ with $|r| = |s|$ we have $X_r(I_{[u]}) \cong X_s(I_{[u]})$.
    
   Furthermore, if for each $k$ we let $A_k$ be a fixed order isomorphic to $X_r(I_{[u]})$ for some (equivalently, every) regular node $r$ with $|r| = k$, then $A_k \cong N_k A_{k+1} + A_{k+2}$ for all $k \in \omega$, that is, $\{A_k: k \in \omega\}$ is a non-terminating left division system on the coefficients $\mathcal{C}$. 
\end{itemize}
\end{leftsymbolicrepnthm}

\begin{proof}
For (i.): For each $k \in \omega$, fix a cut sequence in $A_k$
\[
a^k_0 < a^k_1 < \ldots < a^k_{N_k} < a^k_{N_k+1}
\]
that witnesses $A_k \cong N_k A_{k+1} + A_{k+2}$, i.e. such that $a^k_0, a^k_{N_k+1}$ are the endcuts of $A_k$, $[a^k_i, a^k_{i+1}] \cong A_{k+1}$ for $i < N_k$, and $[a^k_{N_k}, a^k_{N_k+1}] \cong A_{k+2}$.

Fix isomorphisms $f^k_i: A_{k+1} \rightarrow [a^k_i, a^k_{i+1}]$ for $i < N_k$, and an isomorphism $f^k_{N_k0}: A_{k+2} \rightarrow [a^k_{N_k}, a^k_{N_k+1}]$. (The subscript in the latter map is the $2$-sequence $N_k0$.)

We inductively define convex embeddings $f_r: A_{|r|} \rightarrow A_0$ for each regular node $r \in T$, as follows. 
\begin{itemize}
    \item $f_{\emptyset} = \textrm{id}_{A_0}$; 
    \item If $r$ is regular with $|r| = k$ and $f_r$ is defined, define 
    \[
    f_{ri} = f_rf^k_i
    \]
    for all $i \in \{0, 1, \ldots, N_k-1, N_k0\}$. 
\end{itemize}
Denote the image of $f_r$ by $A_r$. By induction, $A_r \cong A_{|r|}$ for all regular $r$. For $r$ irregular, define $A_r = A_{r0}$. 

Given a regular $r \in T$ with $|r| = k$, the intervals $A_{ri}$ for $i \in \{0, 1, \ldots, N_k-1, N_k0\}$ partition $A_r$ in the lexicographic order of their indices. It follows that for $r, s \in T$, we have $A_r \subseteq A_s$ if $r$ extends $s$, and $A_r < A_s$ if $r <_{lex} s$ (i.e. if $r_i < s_i$ where $i < \min(|r|, |s|)$ is least such that $r_i \neq s_i$).

For each $u \in X$, define
\[
I_u = \bigcap_{n \in \omega} A_{u \upharpoonright n}. 
\]

The intersection of a descending sequence of intervals is either an interval or cut. Thus if we view cuts as empty intervals, each $I_u$ is an interval in $A_0$. It follows from the preceding discussion that these intervals appear in $A_0$ in the lexicographic order of their indices $u$. Moreover, by a straightforward induction each point $a \in A_0$ belongs to some (unique) $I_u$, i.e. these intervals partition $A_0$. Thus $A_0 \cong X(I_u)$. 

Relativizing to $A_r$, we obtain $A_r \cong X_r(I_u) \cong A_{|r|}$ for every regular $r \in T$. 

We next show that $I_u \cong I_v$ whenever $u E_0 v$, so that in fact each $A_r$ has the form $X_r(I_{[u]})$. 

For regular nodes $r, s \in T$ with $|r| = |s|$, the map $f_s f_r^{-1}: A_r \rightarrow A_s$ is an isomorphism. Further, for any $i \in \{0, 1, \ldots, N_k-1, N_k0\}$, we have 
\[
\begin{array}{rcl}
f_s f_r^{-1} f_{ri} & = & f_sf_r^{-1}f_rf_i^k \\
& = & f_sf_i^k \\
& = & f_{si}.
\end{array}
\]

By induction, for any finite sequence $t$ such that $rt \in T$ and $rt$ is regular, we have $f_s f_r^{-1}f_{rt} = f_{st}$. It follows $f_s f_r^{-1}[A_{rt}] = A_{st}$. 

Now suppose $u, v \in X$ and $u E_0 v$ as witnessed by the equations $u = ru'$ and $v = su'$ for some $r, s \in T$ with $|r| = |s|$ and tail-sequence $u'$. Extending $r$ and $s$ by the first entry of $u'$ if necessary (which must be $0$ if one of $r, s$ is irregular) we may assume $r$ and $s$ are regular. Then we have
\[
\begin{array}{rcl}
f_s f_r^{-1}[I_u] & = & f_s f_r^{-1}[I_{ru'}] \\
& = & f_s f_r^{-1} \left[\bigcap_n A_{(ru') \upharpoonright n} \right] \\
& = & f_s f_r^{-1} \left[\bigcap_{n \geq \textrm{dom}(r)} A_{(ru') \upharpoonright n} \right] \\ 
& = & f_s f_r^{-1} \left[\bigcap_n A_{r(u' \upharpoonright n)} \right] \\
& = & \bigcap_n f_s f_r^{-1} [A_{r(u' \upharpoonright n)}] \\
& = & \bigcap_n A_{s(u' \upharpoonright n)} \\
& = & I_{su'} \\
& = & I_v.
\end{array}
\]
Hence $I_u \cong I_v$. 

Since $u$ and $v$ were arbitrary, for any fixed $u \in X$ the intervals $I_v$ for $v \in [u]$ have a common order type. For each class $[u]$, fix $I_{[u]}$ an order of this common type. Then $A_0 \cong X(I_{[u]})$. More generally, $A_r \cong X_r(I_{[u]}) \cong A_{|r|}$ for every regular $r$. This concludes the proof of (i.).

For (ii.): Fix $r, s \in T$ regular with $|r| = |s|$. We first check that the restricted replacements $X_r(I_{[u]})$ and $X_s(I_{[u]})$ are isomorphic. 

By the construction of $X$, for every infinite sequence $v$ we have $rv \in X$ if and only if $sv \in X$. Since $rv E_0 sv$ for any such $v$, and hence $I_{rv} = I_{sv}$, the rule $(rv, i) \mapsto (sv, i)$ is a well-defined mapping from $X_r(I_{[u]})$ to $X_s(I_{[u]})$. Moreover, it is clearly order-preserving and bijective, i.e., an order-isomorphism of $X_r(I_{[u]})$ and $X_s(I_{[u]})$. Thus $X_r(I_{[u]}) \cong X_s(I_{[u]})$, as claimed. 

Now, for each $k \in \omega$, fix an order $A_k$ of the same order type as every $X_r(I_{[u]})$ with $r$ regular and $|r| = k$. From the definition of $X$ we have
\[
X_r \cong X_{r0} + X_{r1} + \cdots + X_{r(N_k-1)} + X_{rN_k0}.
\]
Thus we have
\[
\begin{array}{rcl}
A_k & \cong & X_r(I_{[u]}) \\
& \cong & X_{r0}(I_{[u]}) + X_{r1}(I_{[u]}) + \cdots + X_{r(N_k-1)}(I_{[u]}) + X_{rN_k0}(I_{[u]}) \\
& \cong & A_{k+1} + A_{k+1} + \cdots + A_{k+1} + A_{k+2} \\
& \cong & N_k A_{k+1} + A_{k+2}.
\end{array}
\]
Hence $\{A_k\}$ is a non-terminating left division system on the coefficients $\mathcal{C}$. 
\end{proof}

\subsubsection{Labeling the division order $X$} \label{subsubsect:labellingdivistreeX}

The sequences $u$ from the division order $X$ associated to a non-terminating division system $\{A_k \cong N_k A_{k+1} + A_{k+2}\}$ may be naturally viewed as coordinates for real numbers. These  coordinates are in the base $\{a_0, a_1, \ldots\} \subseteq \mathbb{R}$, where the $a_i$ are reals satisfying the equations 
\[
a_k = N_k a_{k+1} + a_{k+2}
\]
from the division system. Said another way, the base sequence $\{a_i\}$ for the coordinate space $X$ is the remainder sequence from a Euclidean division of two reals $a_0, a_1$ whose ratio $\frac{a_0}{a_1}$ has the (simple) continued fraction expansion $[N_0, N_1, N_2, \ldots]$.

We will show that by identifying each coordinate sequence $u \in X$ with the corresponding real number $\sum_i u_i$, from the symbolic representations of the orders $A_k$ given by Theorem \ref{leftsymbolicrepnthm} (as replacements of the basic intervals $X_r \subseteq X$ up to $E_0$) we obtain dynamical representations of these orders as replacements of intervals in $\mathbb{R}$ up to the orbit equivalence relation of a group of translations. (More accurately, because our division system $\{A_k\}$ is a \textit{left} division system, these replacements will be up to the equivalence relation obtained by splitting the orbits of a group of translations into forward and backward parts.) We describe these dynamical representations in Section \ref{subsect:leftsystemrealrepn}. 

Toward defining these representations, in this section we analyze the structure of the left division order $X$ and introduce some notation. It can be inferred from the work below that $X$ is an order-theoretic Cantor set, that is, a separable, complete linear order with a countable dense collection of jumps. 

From the definition of the division tree $T$, the points in $X$ are precisely those sequences $u \in \omega^{\omega}$ that satisfy, for each $k \in \omega$, the following conditions:
\begin{itemize}
    \item[i.] $0 \leq u_k \leq N_k$;
    \item[ii.] if $u_k = N_k$ then $u_{k+1} = 0$. 
\end{itemize}

For a given $k$, let $0^k$ denote the $k$-length sequence $00 \ldots 0$. Notice $0^k \in T$ for every $k \in \omega$. Let $\overline{0}$ denote the identically zero sequence $000\ldots \in X$. Observe that $\overline{0}$ is the left endpoint of $X$. 

Let $e_k$ denote the sequence $0^kN_k0N_{k+2}0\ldots \in X$. Explicitly, $(e_k)_i = 0$ for $i < k$, and for $i \in \omega$, $(e_k)_{k+i} = N_{k+i}$ for $i$ even and $(e_k)_{k+i} = 0$ for $i$ odd. Observe $e_k$ is the right endpoint of the basic interval $X_{0^k}$. In particular, $e_0 = N_00N_20\ldots$ is the right endpoint of $X$. Also observe $e_k E_0 e_{k'}$ if and only if $k \equiv k' \pmod 2$. 

We say that a pair of points $a < b$ in a linear order $A$ is a \textit{jump pair} if $b$ is the successor of $a$ in $A$, i.e. if there is no $c \in A$ such that $a < c < b$. We say $a$ is the \textit{left point} and $b$ is the \textit{right point} of the jump pair.

The proposition below identifies the jump pairs in $X$. 

\theoremstyle{definition}
\newtheorem{jumppairpropn}[lct]{Proposition}
\begin{jumppairpropn}\label{jumppairpropn}
Fix $u < v$ in $X$. 
\begin{itemize}
    \item[i.] $u, v$ are a jump pair if and only if for some regular node $r \in T$ of length $k = |r|$ we have:
    \begin{itemize}
        \item $0 \leq u_k < N_k$;
        \item $v_k = u_k + 1$;
        \item $u = ru_kN_{k+1}0N_{k+3}0\ldots$;
        \item $v = r(u_k+1)000\ldots = rv_k\overline{0}$.
    \end{itemize}
    \item[ii.] If $u, v$ are not a jump pair, there is $s \in T$ such that $u < X_s < v$. 
\end{itemize} 
\end{jumppairpropn}

\begin{proof}
For (i.): It is not hard to check that if $u, v$ satisfy the conditions in (i.) then they form a jump pair.

Conversely, suppose $u, v$ are a jump pair. Since $u < v$ there is $r \in T$ of length $|r| = k$ such that $u, v$ both begin with $r$ and $u_k < v_k$. Since successors of irregular nodes can only take the value $0$, $u_k < v_k$ implies $r$ is regular and $0 \leq u_k < N_k$. 

If $u_k+1 < v_k$, then $u < X_{r(u_k+1)} < v$, contradicting that $u, v$ are a jump pair. Hence $u_k + 1 = v_k$. 

Since $u$ and $v$ are adjacent, we must have that for all $i \geq 1$, $u \upharpoonright (k+i+1)$ is the rightmost node below $u \upharpoonright (k+i)$ and $v \upharpoonright (k+i+1)$ is leftmost below $v \upharpoonright (k+i)$. Otherwise we could find a node $s \in T$ lexicographically between $u$ and $v$, which gives $u < X_s < v$, contradicting adjacency. It follows that $u = ru_kN_{k+1}0N_{k+3}0\ldots$ and $v = rv_k\overline{0}$. 

For (ii.): If $u, v$ are not adjacent in $X$ (i.e. not a jump pair) then there is a node $s \in T$ that is lexicographically between $u$ and $v$. Hence $u < X_s < v$. 
\end{proof}

It follows from Propositions \ref{E0classesdenseinX} and \ref{jumppairpropn} that for every pair $u < v$ in $X$ and $E_0$-class $[w]$, either $u, v$ form a jump pair or there is $w' \in [w]$ such that $u < w' < v$. 

From Proposition \ref{jumppairpropn}, if $v$ is the right point of a jump pair in $X$, then $v \in [\overline{0}]$. Conversely, if $v E_0 \overline{0}$ and $v > \overline{0}$, then $v$ is the right point of a jump pair. Indeed, if $k$ is maximal such that $v_k > 0$, so that $v = rv_k\overline{0}$ for some $r$ of length $k$, then $u = r(v_k-1)N_{k+1}0N_{k+3}0\ldots$ is the left point in a jump pair with $v$. 

If $u$ is the left point in a jump pair in $X$, then from Proposition \ref{jumppairpropn} either $u \in [e_0]$ or $u \in [e_1]$. Conversely, if $u \in [e_0]$ and $u < e_0$, then $u$ is the left point of a jump pair with some $v \in [\overline{0}]$. Namely, if $u_k$ is the last coordinate of $u$ that is not maximal below $u_{k-1}$, so that $u=ru_kN_{k+1}0N_{k+3}0\ldots$, then $v=r(u_k+1)\overline{0}$ is the successor of $u$. On the other hand, \textit{every} point $u \in [e_1]$ is the left point of a jump pair in $X$. 

If $u < v$ is a jump pair in $X$ and $u \in [e_0]$, then we say $u, v$ is a  \textit{backward orbit pair}. If $u \in [e_1]$ then it is a \textit{forward orbit pair}.

Since $[\overline{0}]$, $[e_0]$, and $[e_1]$ are disjoint, two jumps in $X$ are never consecutive. That is, no point in $X$ is at once the right point in a jump pair and the left point in another.

It will be useful to extend the symbolic representation from Theorem \ref{leftsymbolicrepnthm} of the initial order $A_0$ in a left division system to its $\mathbb{Z}$-product $\mathbb{Z}A_0$. To do this, we first consider the $\mathbb{Z}$-product $\mathbb{Z}X$ of the division order $X$. We extend $E_0$ to $\mathbb{Z}X$ by defining $(m, u) E_0 (n, v)$ in $\mathbb{Z}X$ whenever $u E_0 v$ in $X$.

The following proposition identifies the jump pairs in $\mathbb{Z}X$. Its proof is clear.

\theoremstyle{definition}
\newtheorem{jumppairsinZXpropn}[lct]{Proposition}
\begin{jumppairsinZXpropn}\label{jumppairsinZXpropn}
Fix points $(m, u) < (n, v)$ in $X$. Exactly one holds:
\begin{itemize}
    \item[i.] $n = m$ and $u < v$ is a jump pair in $X$;
    \item[ii.] $n = m+1$, $u = e_0$, and $v = \overline{0}$;
    \item[iii.] There is $s \in T$ and $k \in \mathbb{Z}$ with $m \leq k \leq n$ such that $(m, u) < X_{(k, s)} < (n, v)$. 
\end{itemize}
If (i.) or (ii.) holds, then $(m, u), (n, v)$ form a jump pair in $\mathbb{Z}X$. Conversely, every jump pair in $\mathbb{Z}X$ satisfies (i.) or (ii.). 
\end{jumppairsinZXpropn}

Here $X_{(k, s)}$ denotes the interval in $\mathbb{Z}X$ consisting of points $(k, u)$ with $s$ initial in $u$. 

It follows from Proposition \ref{jumppairsinZXpropn} that in every jump pair $(m, u) < (n, v)$ in $\mathbb{Z}X$ we have $(n, v) E_0 (0, \overline{0})$, and either $(m, u) E_0 (0, e_0)$ or $(m, u) E_0 (0, e_1)$. Conversely, since in $\mathbb{Z}X$ the left endpoint $(m, \overline{0})$ in every copy of $X$ forms a jump pair with the right endpoint $(m-1, e_0)$ of the previous copy, now every point in $[(0, \overline{0})]$ is the right point in a jump pair, and every point in $[(0, e_0)] \cup [(0, e_1)]$ is the left point of a jump pair. 

Copying over the terminology from $X$, we say that pairs for which $(m, u) E_0 (0, e_1)$ are \textit{forward orbit pairs}, and pairs for which $(m, u) E_0 (0, e_0)$ are \textit{backward orbit pairs}. 

Given a left division system $\{A_k\}$ on the division coefficients $\mathcal{C}$, the representation for $A_0 \cong X(I_{[u]})$ from Theorem \ref{leftsymbolicrepnthm} can be lifted to get a representation for $\mathbb{Z}A_{0}$.

\theoremstyle{definition}
\newtheorem{leftsymbolicrepnthmforZA0}[lct]{Theorem}
\begin{leftsymbolicrepnthmforZA0}\label{leftsymbolicrepnthmforZA0}
Suppose $\{A_k: k \in \omega\}$ is a non-terminating left division system on the coefficients $\mathcal{C}$. 

Then there is a collection of orders $I_{[(n,u)]}$ indexed by the $E_0$-classes of $\mathbb{Z}X$ such that $\mathbb{Z}A_0 \cong \mathbb{Z}X(I_{[(n, u)]})$. Moreover, for every $n_0 \in \mathbb{Z}$ and regular node $r \in T$ with $|r| = k$, we have $A_k \cong X_{(n_0,r)}(I_{[(n,u)]})$.
\end{leftsymbolicrepnthmforZA0}

\begin{proof}
Let $A_0 \cong X(I_{[u]})$ be the symbolic representation of $A_0$ from Theorem \ref{leftsymbolicrepnthm}. Define $I_{(n, u)} = I_u$ for every $n \in \mathbb{Z}$ and $u \in X$. Then $I_{(m, u)} = I_{(n, v)}$ whenever $(m, u)E_0(n,v)$. Hence the replacement $\mathbb{Z}X(I_{(n, u)})$ is a replacement of $\mathbb{Z}X$ up to $E_0$, which we denote $\mathbb{Z}X(I_{[(n, u)]})$. 

For a given $n_0 \in \mathbb{Z}$ and $r \in T$, it is straightforward to see from the definition of the replacements that 
\[
X_{(n_0, r)}(I_{[(n, u)]}) \cong X_r(I_{[u]}) \cong A_{|r|}.
\]
In particular, for every $n_0 \in \mathbb{Z}$ we have $X_{(n_0, \emptyset)}(I_{[(n, u)]}) \cong X(I_{[u]}) \cong A_0$, which yields
\[
\begin{array}{rcl}
\mathbb{Z}X(I_{[(n, u)]}) & \cong & \cdots + X_{(-1, \emptyset)}(I_{[(n, u)]}) + X_{(0, \emptyset)}(I_{[(n, u)]}) + X_{(1, \emptyset)}(I_{[(n, u)]}) + \cdots \\
& \cong & \mathbb{Z}A_0,
\end{array}
\]
as desired. 
\end{proof}

In the next section, we will push forward the extended representation of $\mathbb{Z}A_0$ from Theorem \ref{leftsymbolicrepnthmforZA0} to get a representation of $\mathbb{Z}A_0$ as a replacement of $\mathbb{R}$. 

\subsection{Real representations of left systems}\label{subsect:leftsystemrealrepn}

We now turn to describing our dynamical representation for left systems. 

First, we will show that after pasting together the points in each of its jump pairs, $\mathbb{Z}X$ equipped with $E_0$ is naturally order-isomorphic to $\mathbb{R}$ equipped with the orbit equivalence relation $E_G$ of a group of translations generated by real numbers satisfying the equations from a division system on $\mathcal{C}$.

Consider the system of equations $\{a_k = N_k a_{k+1} + a_{k+2}\}$ on the coefficients $\mathcal{C}$. Define $a_0 = 1$ and then let $\{a_k: k \in \omega\}$ be the unique sequence of real numbers satisfying these equations. As discussed in Section \ref{section:linddivisthm}, the $a_k$ may be computed via the continued fraction expansions determined by the equations in the system. Note that $a_k$ is irrational for $k \geq 1$, and the sequence $\{a_k\}$ decreases to $0$.

A \textit{translation} of $\mathbb{R}$ is a map of the form $f(x) = x + a$ for some $a \in \mathbb{R}$. We often identify such a translation $f$ with $a$, and by extension, the group of translations of $\mathbb{R}$ (under composition) with the additive group $(\mathbb{R}, +)$. 

Let $G = \langle a_0, a_1 \rangle$ be the group of translations of $\mathbb{R}$ generated by the translations $f(x) = x + a_0 = x + 1$ and $g(x) = x + a_1$, which we identify with $1$ and $a_1$.  

Let $E_G$ be the orbit equivalence relation of $G$: for $x, y \in \mathbb{R}$ we have $x E_G y$ if there are integers $m, n$ such that $y = x + m + na_1$. By the irrationality of $a_1$, each orbit equivalence class $[x]_{E_G} = [x]$ is dense in $\mathbb{R}$. 

When viewing $G$ as a subgroup of $(\mathbb{R}, +)$, we have $[0] = G$. 

Suppose $x \in [0]$, so that $x = m + na_1$ for some $m, n \in \mathbb{Z}$. We say that $x$ is in the \textit{forward orbit of $0$ under $G$} (or simply, the \textit{forward orbit of $G$}) if $n > 0$, and $x$ is in the \textit{backward orbit of $0$} if $n \leq 0$. 

We write $[0]_{E_G}^+ = [0]^+$ for the forward orbit of $0$ and $[0]_{E_G}^- = [0]^-$ for the backward orbit. Then $[0] = [0]^+ \cup [0]^-$. By the irrationality of $a_1$, $[0]^+ \cap [0]^- = \emptyset$. Both $[0]^+$ and $[0]^-$ are dense in $\mathbb{R}$. 

Write $E_G^{\pm}$ for the equivalence relation obtained from $E_G$ by splitting $[0]$ into its forward and backward parts, so that $[0]^+ = [a_1]_{E_G^{\pm}}$ and $[0]^- = [0]_{E_G^{\pm}}$. For $x \not\in [0]$, we have $[x]_{E_G^{\pm}} = [x]_{E_G}$. 

We adopt the convention that the notation $[x]$ without any subscript refers to $[x]_{E_G}$. When we wish to specify the $E_G^{\pm}$-class of $x$, we write $[x]_{E_G^{\pm}}$. 

There is a close connection between the forward and backward orbit distinction in $\mathbb{R}$ and $\mathbb{Z}X$ that explains the terminology in $\mathbb{Z}X$. 

\theoremstyle{definition}
\newtheorem{sigmadefn}[lct]{Definition}
\begin{sigmadefn}\label{sigmadefn}
Define a map $\sigma: \mathbb{Z}X \rightarrow \mathbb{R}$ by the rule
\[
\sigma(m, u) = m + \sum_{k \in \omega} u_k a_{k+1}.
\]
\end{sigmadefn}

We will use the map $\sigma$ to push forward the representation $\mathbb{Z}X(I_{[(n, u)]})$ of $\mathbb{Z}A_0$ (as a replacement of $\mathbb{Z}X$ up to $E_0$) to a representation $\mathbb{R}(K_{[x]})$ of $\mathbb{Z}A_0$ as a replacement of $\mathbb{R}$ up to $E_G^{\pm}$. 

Toward this, we now establish a series of lemmas whose content is that, up to condensing jump pairs to singletons and distinguishing the forward and backward orbits of the origin,  $\sigma$ is an order-isomorphism of $\mathbb{Z}X$ equipped with $E_0$ and $\mathbb{R}$ equipped with $E_G$.

\theoremstyle{definition}
\newtheorem{normalformsumdefn}[lct]{Definition}
\begin{normalformsumdefn}\label{normalformsumdefn}
A formal sum $\sum_{k \in \omega} x_k a_{k+1}$ with integer coefficients $x_k$ is a \textit{normal form sum} if for all $k \in \omega$ we have $0 \leq x_k \leq N_k$, and whenever $x_k = N_k$ then $x_{k+1} = 0$. 
\end{normalformsumdefn}

\theoremstyle{definition}
\newtheorem{sigmamekismplusaklem}[lct]{Lemma}
\begin{sigmamekismplusaklem}\label{sigmamekismplusaklem}
For all $m \in \mathbb{Z}$ and $k \in \omega$, we have $\sigma(m, e_k) = m + a_k$.
\end{sigmamekismplusaklem}

\begin{proof}
Fix $k \in \omega$. By repeated expansion of remainder terms, we have
\[
\begin{array}{rcl}
a_k & = & N_k a_{k+1} + a_{k+2} \\
& = & N_k a_{k+1} + N_{k+2} a_{k+3} + a_{k+4} \\
& \vdots & \\
& = & N_k a_{k+1} + N_{k+2} a_{k+3} + \cdots + N_{k+2i} a_{k+2i+1} + a_{k+2i+2}\\
& \vdots &.
\end{array}
\]

Since the terms $a_{k+2i+2}$ converge to $0$, the partial sums $\sum_{i=0}^m N_{k+2i} a_{k+2i+1}$ converge to $a_k$, i.e. $a_k$ may be expressed as a convergent series
\[
a_k = N_k a_{k+1} + N_{k+2} a_{k+3} + \cdots = \sum_{i\in\omega} N_{k+2i} a_{k+2i+1}.
\]
Now observe that
\[
\sigma(m, e_k) = m + 0a_0 + 0a_1 + \cdots + 0a_{k-1} +  \sum_{i\in\omega} N_{k+2i} a_{k+2i+1} = m + a_k,
\]
as claimed. 
\end{proof}

The next lemma will allow us to show that $\sigma$ is essentially $<$-preserving. 

\theoremstyle{definition}
\newtheorem{boundingnormalformsumlemma}[lct]{Lemma}
\begin{boundingnormalformsumlemma}\label{boundingnormalformsumlemma}
Fix $k \in \omega$. Suppose $\sum_i x_i a_{i+1}$ is a normal form sum such that $x_i = 0$ for all $i < k$. Then the sum converges to a real number $x \leq a_k$, and $x = a_k$ if and only if $x_{k+2i} = N_{k+2i}$ and $x_{k+2i+1} = 0$ for all $i \in \omega$. 
\end{boundingnormalformsumlemma}

\begin{proof}
Let $x$ denote the normal form sum $\sum_i x_i a_{i+1}$. Since its first $k$ terms are $0$, we have
\[
x = \sum_i x_i a_{i+1} = \sum_i x_{k+i} a_{k+i+1} = x_k a_{k+1} + x_{k+1}a_{k+2} + \cdots.
\]
Suppose first $x_k < N_k$. Separating the terms $x_{k+i}a_{k+i+1}$ by parity, we have
\[
\begin{array}{rcl}
x & = & x_ka_{k+1} + (x_{k+1} a_{k+2} + x_{k+3}a_{k+4} + \cdots) + (x_{k+2}a_{k+3} + x_{k+4}a_{k+5} + \cdots) \\
& \leq & x_ka_{k+1} + a_{k+2} + a_{k+3},
\end{array}
\]
where the inequality follows from the proof of Lemma \ref{sigmamekismplusaklem} and the fact that $x$ is a normal form sum. In particular, the sums in the above expression converge, and hence $x$ does as well. 

Since $a_{k+2} + a_{k+3} \leq N_{k+1}a_{k+2} + a_{k+3} = a_{k+1}$, we have 
\[
x \leq (x_k + 1)a_{k+1} \leq N_ka_{k+1} < a_k,
\]
where we are using $x_k < N_k$. Hence we are done in this case. 

Now suppose $x_k = N_k$. Then, since $x$ is in normal form, $x_{k+1} = 0$. Isolating the term $x_{k+2}a_{k+3}$ and separating the subsequent terms by parity we have
\[
x = N_ka_{k+1} + x_{k+2}a_{k+3} + (x_{k+3}a_{k+4} + x_{k+4}a_{k+5} + \cdots)
\]
If $x_{k+2} < N_{k+2}$, then if we separate the parenthesized terms above by parity and again apply the proof of Lemma \ref{sigmamekismplusaklem} we obtain
\[
x \leq N_ka_{k+1} + (x_{k+2} + 1)a_{k+3} \leq N_ka_{k+1} + N_{k+2}a_{k+3} < N_ka_{k+1} + a_{k+2} = a_k.
\]
Otherwise $x_{k+2} = N_{k+2}$ and we induct. If at some finite stage the induction terminates we prove $x < a_k$. If it never terminates, we show
\[
x = N_ka_{k+1} + N_{k+2}a_{k+3} + \cdots = a_k,
\]
as claimed. 
\end{proof}

\theoremstyle{definition}
\newtheorem{sigmaontoRpropn}[lct]{Proposition}
\begin{sigmaontoRpropn}\label{sigmaontoRpropn}
The following hold:
\begin{itemize}
    \item[i.] The map $\sigma: \mathbb{Z}X \rightarrow \mathbb{R}$ is well-defined, that is, for every $u \in X$ the sum $\sum_{k} u_k \alpha_{k+1}$ converges. 
    \item[ii.] Given points $(m, u) < (n, v)$ in $\mathbb{Z}X$, we have $\sigma(m, u) \leq \sigma(n, v)$ in $\mathbb{R}$, and $\sigma(m, u) = \sigma(n, v)$ if and only if $(m, u), (n, v)$ form a jump pair.  
    \item[iii.] $\sigma$ is surjective.
\end{itemize}
\end{sigmaontoRpropn}

\begin{proof}
For (i.): By definition of $X$, $\sum_k u_k a_{k+1}$ is a normal form sum. By Lemma \ref{boundingnormalformsumlemma}, this sum converges to a real number $x \leq a_0 = 1$. 

For (ii.): Writing $u_{-1} = m$ and $v_{-1} = n$, let $k$ be least such that $u_k \neq v_k$. Then $u_k < v_k$, and we have
\[
\begin{array}{rcl}
\sigma(m, u) & = & u_{k-1} + u_0a_1 + \cdots + u_{k-1}a_k + u_ka_{k+1} + \sum_i u_{k+i +1} a_{k+i+2}, \\
\sigma(n, v) & = & u_{k-1} + u_0a_1 + \cdots + u_{k-1}a_k + v_ka_{k+1} + \sum_i v_{k+i+1} a_{k+i+2}.
\end{array}
\]

By Lemma \ref{boundingnormalformsumlemma}, we have $\sum_i u_{k+i+1} a_{k+i+2} \leq a_{k+1}$. Moreover, it follows from Lemmas \ref{jumppairsinZXpropn} and \ref{boundingnormalformsumlemma} that $(m, u), (n, v)$ form a jump pair if and only if we have $v_k = u_k + 1$, $\sum_i u_{k+i+1} a_{k+i+2} = a_{k+1}$ and $\sum_i v_{k+i+1} a_{k+i+2} = 0$, which yields $\sigma(m, u) = \sigma(n, v)$. 

If they do not form a jump pair, then either $u_k + 1 < v_k$, or $u_k + 1 = v_k$ and either $\sum_i u_{k+i+1} a_{k+i+2} < a_{k+1}$ or $\sum_i v_{k+i+1} a_{k+i+2} > 0$. In either case we have $\sigma(m, u) < \sigma(n, v)$, as claimed. 

For (iii.): We sketch the idea, which is straightforward. Fix $x \in \mathbb{R}$. 

Let $m \in \mathbb{Z}$ be least such that $x \in [m, m+1)$. Now let $u_0$ be maximal such that $m + u_0a_1 \leq x$. Since $1 = a_0 = N_0a_1 + a_2$, we have $0 \leq u_0 \leq N_0$. 

If $u_0 < N_0$, then $x \in [m + u_0a_1, m + (u_0+1)a_1)$, which has length $a_1 = N_1a_2 + a_3$. In this case, let $u_1$ be maximal such that $m + u_0a_1 + u_1a_2 \leq x$, and continue. 

If instead $u_0 = N_0$, then $x \in [m + N_0a_1, m + N_0a_1 + a_2) = [m + N_0a_1, 1)$, which has length $a_2 = N_2a_3 + a_4$. In this case, let $u_1 = 0$, and let $u_2$ be maximal such that $m + u_0a_1 + u_2a_3 \leq x$, and continue. 

The partial sums $m + u_0a_1 + u_1a_2 + \cdots + u_ka_{k+1}$ constructed this way converge to $x = m + \sum_i u_ia_{i+1} = \sigma(m, u)$, where $u = (u_0, u_1, \ldots)$. By construction, the sum $\sum_i u_ia_{i+1}$ is in normal form and thus $u \in X$. Hence $\sigma$ is surjective. 
\end{proof}

We write $[(0, \overline{0})]^+$ for the set of $(n, v) \in [(0, \overline{0})]$ in $\mathbb{Z}X$ belonging to forward orbit jump pairs, and correspondingly $[(0, \overline{0})]^-$ for the set of $(n, v) \in [(0, \overline{0})]$ belonging to backward orbit pairs. 

\theoremstyle{definition}
\newtheorem{sigmaimagesoffwdbwdorbitsarefwdbwdorbits}[lct]{Lemma}
\begin{sigmaimagesoffwdbwdorbitsarefwdbwdorbits}\label{sigmaimagesoffwdbwdorbitsarefwdbwdorbits}
The images of $[(0, \overline{0})]^+$ and $[(0, \overline{0})]^-$ under $\sigma$ are $[0]^+$ and $[0]^-$ respectively. 
\end{sigmaimagesoffwdbwdorbitsarefwdbwdorbits}

\begin{proof}
We first claim that for $(n, u) \in [(0, \overline{0})]^+$ we have $\sigma(n, u) \in [0]^+$. Since $(n, u) \in [(0, \overline{0})]^+$, the last nonzero entry $u_{2k}$ of $u$ is of even index. Thus we have
\begin{equation}\label{kool}
\sigma(n, u) = n + u_0a_1 + u_1a_2 + \cdots + u_{2k-1}a_{2k} + u_{2k}a_{2k+1}
\end{equation}

Since $a_{k+2} = a_k - N_ka_{k+1}$ for all $k$ and $a_1 \equiv 1a_1 \pmod 1$, a simple induction yields that for all $k \geq 1$ there is an integer $M_k > 0$ such that $a_k \equiv M_k a_1 \pmod 1$ when $k$ is odd and $a_k \equiv -M_ka_1 \pmod 1$ when $k$ is even. That is, $a_k$ is a positive or negative integer multiple of $a_1$ modulo $1=a_0$, depending on the parity of $k$. 

Thus we have
\[
\begin{array}{rcl}
\sigma(n, u) & \equiv & u_0M_1a_1 - u_1M_2a_1 + \cdots - u_{2k-1}M_{2k}a_1 + u_{2k}M_{2k+1}a_1 \pmod 1 \\
& \equiv & (u_0M_1 - u_1M_2 + \cdots - u_{2k-1}M_{2k} + u_{2k}M_{2k+1})a_1 \pmod 1 \\
& \equiv & (u_0M_1 + u_2M_3 + \cdots + u_{2k}M_{2k+1} \\
& & \phantom{,,,,,} - u_1M_2 - u_3M_4 - \cdots - u_{2k-1}M_{2k})a_1 \pmod 1
\end{array}
\]

To show $\sigma(n, u) \in [0]^+$, we must show $\sigma(n, u)$ is a positive multiple of $a_1$ modulo $1$. That is, we must show that the coefficient sum
\[
u_0M_1 + u_2M_3 + \cdots + u_{2k}M_{2k+1} - u_1M_2 - u_3M_4 - \cdots - u_{2k-1}M_{2k}
\]
from the expression above is positive. By assumption, $u_{2k} > 0$. Since $\sigma(n, u)$ is a normal form sum, we have $0 \leq u_i \leq N_i$ for all $i$. We claim that 
\[
M_{2k+1} > N_1M_2 + N_3M_4 + \cdots + N_{2k-1}M_{2k},
\]
from which it then follows that the coefficient sum is positive. To see the claim, notice that by beginning with $a_1 = N_1a_2 + a_3$ and iteratively expanding remainder terms using $a_{2i +1} = N_{2i+1}a_{2i+2} + a_{2i+3}$ for $1 \leq i < k$, we have
\[
\begin{array}{rcl}
a_1 & = & N_1a_2 + N_3a_4 + \cdots + N_{2k-1}a_{2k} + a_{2k+1} \\
& \equiv & (- N_1M_2 - N_3M_4 - \cdots - N_{2k-1}M_{2k} + M_{2k+1})a_1 \pmod 1
\end{array}
\]
Since $a_1$ is irrational, the congruence
\[
1a_1 \equiv (- N_1M_2 - N_3M_4 - \cdots - N_{2k-1}M_{2k} + M_{2k+1})a_1 \pmod 1
\]
implies $1 = - N_1M_2 - N_3M_4 - \cdots - N_{2k-1}M_{2k} + M_{2k+1}$, which gives the claim. 

Hence $\sigma(n, u) \in [0]^+$. A symmetric argument shows that for $(n, u) \in [(0, \overline{0})]^-$ we have $\sigma(n, u) \in [0]^-$. 

To finish the proof of the lemma, it suffices to check that $\sigma$ maps $[(0, \overline{0})]$ onto $[0]$. For this, it is enough to show that for every integer $k$, we have that $x = ka_1$ can be written as a normal form sum with only finitely many nonzero coefficients, as then there is a corresponding $(n, u) \in [(0, \overline{0})]$ such that $\sigma(n, u) = x$. (It then immediately follows that an arbitrary element $x' = m + ka_1 \in [0]$ is the image of some $(n, u) \in [(0, \overline{0})]$.) 

We first show this is true for $k > 0$, by induction. Since $1a_1$ is a normal form sum, the claim holds for $k = 1$. Suppose it is true for $k$, so that we have
\[
ka_1 = m + x_0a_1 + \cdots + x_na_{n+1}
\]
in normal form. Then
\[
(k+1)a_1 = m + (x_0 + 1)a_1 + \cdots + x_na_{n+1}.
\]

If $x_0 < N_0$, then this sum is in normal form. Otherwise $x_0 = N_0$ and $x_1 = 0$. Expanding the second term in the sum as 
\[
N_0a_1 + a_1 = N_0a_1 + N_1a_2 + a_3 = 1 + 0a_1 + (N_1-1)a_2 + a_3
\]
we have
\[
(k+1)a_1 = (m+1) + 0a_1 + (N_1 - 1)a_2 + (x_2 + 1)a_3 + \cdots + x_n a_{n+1}.
\]

If $x_2 < N_2$, then this sum is in normal form. Otherwise, $x_2 = N_2$ and $x_3 = 0$. Now similarly, we expand the term $(x_2 + 1)a_3$ as 
\[
N_2a_3 + a_3 = N_2a_3 + N_3a_4 + a_5 = a_2 + (N_3-1)a_4 + a_5,
\]
and then write the sum as follows:
\[
(k+1)a_1 = (m+1) + 0a_1 + N_1a_2 + 0a_3 + (N_3 - 1)a_4 + (x_4+1)a_5 + \cdots + x_n a_{n+1}.
\]
And so on. This carrying process terminates in a sum in normal form, either before reaching the last nonzero term in the original sum, or as a sum of the form
\[
(k+1)a_1 = (m+1) + 0a_1 + N_1a_2 + 0a_3 + N_3a_4 + 0a_5 + \cdots + N_{2k-1}a_{2k} + a_{2k+1}.
\]
In either case, we finish having written $(k+1)a_1$ as a normal form sum with only finitely many nonzero coefficients, as desired. 

For $k < 0$, we proceed similarly. First we note 
\begin{equation}\label{look}
-a_1 = -1 + (1-a_1) = -1 + (N_0 - 1)a_1 + a_2,
\end{equation}
and this latter sum is in normal form. For a fixed $k < 0$, suppose we have
\[
ka_1 = m + x_0a_1 + \cdots + x_na_{n+1}
\]
in normal form. If $x_0 \geq 1$, then
\[
(k-1)a_1 = m + (x_0 - 1)a_1 + \cdots + x_n a_{n+1},
\]
which is in normal form. Otherwise, $x_0 = 0$ and then from \ref{look} we have
\[
(k-1)a_1 = ka_1 + (-a_1) =  m-1 + (N_0-1)a_1 + (x_1 + 1)a_2 + \cdots + x_na_{n+1}.
\]
If $x_1 < N_1$, then this sum is in normal form. Otherwise, $x_1 = N_1$ and $x_2 = 0$. Now by a similar argument as in the $k > 0$ case we continue, finishing either before reaching the last nonzero term, or as a sum in the form
\[
(k-1)a_1 = m-1 + N_0a_1 + 0a_2 + N_2a_3 + \cdots + N_{2(k-1)}a_{2k-1} + a_{2k}.
\]
In either case, we have written $(k-1)a_1$ as a sum in normal form. We are done. 
\end{proof}

\theoremstyle{definition}
\newtheorem{sigmapreservesE0lemma}[lct]{Lemma}
\begin{sigmapreservesE0lemma}\label{sigmapreservesE0lemma}
Fix $(m, u), (n, v) \in \mathbb{Z}X$, and suppose 
\[
(m, u), (n, v) \not\in [(0, \overline{0})] \cup [(0, e_0)] \cup [(0, e_1)],
\]
i.e. neither of $(m, u), (n, v)$ belong to a jump pair. Then we have
\[
(m, u) E_0 (n, v) \textrm{ if and only if } \sigma(m, u) E_G \sigma(n, v).
\]
\end{sigmapreservesE0lemma}

\begin{proof}
First suppose $(m, u) E_0 (n, v)$. Then some $k$ we have
\[
\begin{array}{rcl}
u & = & u_0u_1 \ldots u_ku_{k+1} u_{k+2} \ldots \\
v & = & v_0v_1 \ldots v_ku_{k+1} u_{k+2} \ldots
\end{array}
\]
Thus 
\[
\sigma(m, u) - \sigma(n, v) = (m - n) + \sum_{i \leq k} (u_i - v_i) a_i
\]
Since each $a_i$ can be written as a linear combination of $a_0 = 1$ and $a_1$, it follows the difference above can also be so written, and so $\sigma(m, u) - \sigma(n, v) \in G$. Hence $\sigma(m, u) E_G \sigma(n, v)$, as desired. 

Conversely, let $x = \sigma(m, u)$ and $y = \sigma(n, v)$ and suppose $x E_G y$. We have
\[
\begin{array}{rcl}
x & = & m + \sum_i u_i a_{i+1}, \\
y & = & n + \sum_i v_i a_{i+1}.
\end{array}
\]
Since $(m, u)$ and $(n, v)$ do not belong to jump pairs, these representations for $x$ and $y$ as normal form sums are unique. 

From $x E_G y$ we have $y = x + l + ka_1$ for some $l, k \in \mathbb{Z}$. We want to show $(m, u) E_0 (n, v)$. It suffices to check the cases when $y = x + a_1$ and $y = x - a_1$, as then the general case follows easily by induction. Suppose first $y = x + a_1$.

We have
\[
y = x + a_1 = m + (u_0 + 1)a_1 + u_1a_2 + \cdots.
\]
If $u_0 < N_0$, then this sum is in normal form. Otherwise, $u_0 = N_0$ and $u_1 = 0$. Now we proceed as in the proof of Lemma \ref{boundingnormalformsumlemma}. The carry and rewrite process must terminate at a finite stage: if it does not, we have $u_{2k} = N_{2k}$ and $u_{2k+1} = 0$ for all $k$, so that $(n, u) = (n, e_0)$ is the left point of a jump pair, a contradiction. Thus $y$ may be written (uniquely) as a normal form sum differing in only finitely many coefficients from the sum for $x$. It follows $(m, u) E_0 (n, v)$, as desired. The case when $y = x - a_1$ is similar. 
\end{proof}

Combining Proposition \ref{sigmaontoRpropn} with Lemmas \ref{sigmaimagesoffwdbwdorbitsarefwdbwdorbits} and \ref{sigmapreservesE0lemma} yields the following.

\theoremstyle{definition}
\newtheorem{sigmapreservesE0essentiallypropn}[lct]{Proposition}
\begin{sigmapreservesE0essentiallypropn}\label{sigmapreservesE0essentiallypropn}

The images under $\sigma$ of the $E_0$-classes of $\mathbb{Z}X$ are computed as follows:

\begin{itemize}
    \item[i.] $\sigma([(0, \overline{0})]^+) = \sigma([(0, e_1)]) = [0]^+$, 
    \item[ii.] $\sigma([(0, \overline{0})]^-) = \sigma([(0, e_0)]) = [0]^-$, 
    \item[iii.] If $(n, u) \not\in [(0, \overline{0})] \cup [(0, e_0)] \cup [(0, e_1)]$, then  $\sigma(n, u) \not\in [0]$ and
    \[
    \sigma([(n, u)]_{E_0}) = [\sigma(n, u)]_{E_G}.
    \]
\end{itemize}
\end{sigmapreservesE0essentiallypropn}

\begin{proof}
Statements (i.) and (ii.) follow from Proposition \ref{sigmaontoRpropn}.(ii.) and Lemma \ref{sigmaimagesoffwdbwdorbitsarefwdbwdorbits}. 

For (iii.), suppose to the contrary that $\sigma(n, u) \in [0]$, so that
\[
\sigma(n, u) = n + u_0a_1 + u_1a_2 + \cdots = m + ka_1
\]
for some $m, k \in \mathbb{Z}$. 

Suppose first that $k \leq 0$. Since $(n, u) \not\in [(0, \overline{0})] \cup [(0, e_0)] \cup [(0, e_1)]$, by adding $a_1$ $k$-times to each side of the equation using our carry and rewrite process as in the proof of Lemma \ref{sigmapreservesE0lemma}, and then subtracting $n$, we obtain an equation
\[
u_0'a_1 + u_1'a_2 + \cdots = (m - n)
\]
where the left sum is in normal form and $u_i' \neq u_i$ for only finitely many $i$. This is a contradiction: since $(n, u) \not\in [(0, \overline{0})] \cup [(0, e_0)]$, we must have $u' \neq \overline{0}$ and $u' \neq e_0$, where $u' = (u_0', u_1', \ldots)$; hence the left sum must be strictly greater than $0$ and strictly less than $1$. 

Thus $\sigma(n, u) \not\in [0]$, as claimed. Combined with (i.) and (ii.) and the fact that $\sigma$ is surjective (by Proposition \ref{sigmaontoRpropn}.(iii.)), this gives that 
\[
\sigma(\mathbb{Z}X \setminus ([(0, \overline{0})] \cup [(0, e_0)] \cup [(0, e_1)])) = \mathbb{R} \setminus [0].
\]

The second part of Statement (iii.) now follows from Lemma \ref{sigmapreservesE0lemma}.
\end{proof}

Theorem \ref{leftrealrepnthm} below is our dynamical representation theorem for left division systems on the coefficients $\mathcal{C}$.

\theoremstyle{definition}
\newtheorem{leftrealrepnthm}[lct]{Theorem}
\begin{leftrealrepnthm}\label{leftrealrepnthm}
Suppose that $\{A_k\}$ is a non-terminating left division system on the coefficients $\mathcal{C}$. Then there is a collection of orders $K_{[x]}$ indexed by the $E_G^{\pm}$-classes of $\mathbb{R}$ such that $\mathbb{Z}A_0 \cong \mathbb{R}(K_{[x]})$. 

Moreover, there are decompositions
\[
\begin{array}{rcl}
K_0 & \cong & L + R \\
K_{a_1} & \cong & L' + R
\end{array}
\]
such that for every $k \in \omega$, we have
\[
\begin{array}{rcl}
\textrm{$k$ even} & \Rightarrow & A_k \cong R + (0, a_k)(K_{[x]}) + L \\
\textrm{$k$ odd} & \Rightarrow & A_k \cong R + (0, a_k)(K_{[x]}) + L',
\end{array}
\]
where $(0, a_k)$ denotes the open interval in $\mathbb{R}$, and $(0, a_k)(K_{[x]})$ the restriction of the replacement $\mathbb{R}(K_{[x]})$ to this interval. 
\end{leftrealrepnthm}

\begin{proof}
Let $\mathbb{Z}X(I_{[(n, u)]})$ be the symbolic representation of $\mathbb{Z}A_0$ from Theorem \ref{leftsymbolicrepnthmforZA0}. 

Using $\sigma$, we push forward $\mathbb{Z}X(I_{[(n, u)]})$ to a replacement $\mathbb{R}(K_x)$ of $\mathbb{R}$. More precisely, for $x \in \mathbb{R}$ such that $\sigma^{-1}(x)$ is a singleton $(m, u)$, define 
\[
K_x = I_{(m, u)}.
\]
By Propositions \ref{sigmaontoRpropn}, \ref{sigmaimagesoffwdbwdorbitsarefwdbwdorbits} and Lemma \ref{sigmapreservesE0essentiallypropn}, this happens exactly for $x \not\in [0]$. 

For $x \in [0]$, $\sigma^{-1}(x)$ is a jump pair $\{(m, u), (n, v)\}$, say with $(m, u) < (n, v)$. For such an $x$, define
\[
K_x = I_{(m, u)} + I_{(n, v)}. 
\]

Consider the resulting replacement $\mathbb{R}(K_x)$. It follows from Lemmas \ref{sigmaimagesoffwdbwdorbitsarefwdbwdorbits}, \ref{sigmapreservesE0lemma}, and \ref{sigmapreservesE0essentiallypropn} and the fact that $\mathbb{Z}X(I_{[(n, u)]})$ is a replacement up to $E_0$ that $\mathbb{R}(K_x)$ is a replacement up to $E_G^{\pm}$. We denote it $\mathbb{R}(K_{[x]})$.

Given a jump pair $(m, u) < (n, v)$ in $\mathbb{Z}X$, let us view the sum $I_{(m, u)} + I_{(n, v)}$ set-theoretically as the disjoint union of $I_{(m, u)}$ and $I_{(n, v)}$. Then the rule
\begin{equation}
(n, u, i) \mapsto (\sigma(n, u), i)
\end{equation}
gives a well-defined map from $\mathbb{Z}X(I_{[(n, u)]})$ to $\mathbb{R}(K_{[x]})$, which for jump pairs $(m, u) < (n, v)$ takes the orders $I_{(m, u)}$ and $I_{(n, v)}$ replacing the adjacent points $(m, u), (n, v)$ onto the order $I_{(m, u)} + I_{(n, v)}$ replacing $x = \sigma(m, u) = \sigma(n, v)$. By the definition of the replacement $\mathbb{R}(K_{[x]})$, this map is clearly an order-isomorphism. Hence 
\[
\mathbb{Z}A_0 \cong \mathbb{Z}X(I_{[(n, u)]}) \cong \mathbb{R}(K_{[x]}).
\]

Next, from the definition of $\mathbb{R}(K_{[x]})$ we have the isomorphisms
\[
\begin{array}{rcccl}
K_0 & \cong & I_{(-1, e_0)} + I_{(0, \overline{0})} & \cong & L + R \\
K_{a_1} & \cong & I_{(0, e_1)} + I_{(0, a_1\overline{0})} & \cong & L' + R
\end{array}
\]
where $L = I_{(-1, e_0)}$, $L' = I_{(0, e_1)}$, and $R = I_{(0, \overline{0})} \cong I_{(0, a_1\overline{0})}$.

For a given $k$, we have $A_k \cong X_{(0, 0^k)}(I_{[(n, v)]})$ by Theorem \ref{leftsymbolicrepnthmforZA0}. The left endpoint of $X_{(0, 0^k)}$ is $(0, \overline{0})$ and the right endpoint is $(0, e_k)$. 

Since $\sigma(0, \overline{0}) = 0$ and, by Lemma \ref{sigmamekismplusaklem}, $\sigma(0, e_k) = a_k$, we have that the interval $X_{(0, 0^k)}(I_{[(n, v)]})$ in $\mathbb{Z}X(I_{[(n, v)]})$ pushes forward to an interval in $\mathbb{R}(K_{[x]})$ that decomposes as 
\[
I_{(0, \overline{0})} + (0, a_k)(K_{[x]}) + I_{(0, e_k)}
\]
Since $I_{(0, e_k)} \cong I_{(0, e_0)} \cong L$ when $k$ is even and $I_{(0, e_k)} \cong I_{(0, e_1)} \cong L'$ when $k$ is odd, the last claim in the statement of the theorem follows. 
\end{proof}

\theoremstyle{definition}
\newtheorem{symbolicrealrepndefn}[lct]{Definition}
\begin{symbolicrealrepndefn}\label{symbolicrealrepndefn}
Let $\mathcal{D} = \{A_k\}$ be a fixed non-terminating left division system on the coefficients $\mathcal{C}$. The representations
\[
\begin{array}{rcl}
A_{|r|} & \cong & X_r(I_{[u]}) \\
\mathbb{Z}A_0 & \cong & \mathbb{Z}X(I_{[(n, u)]})
\end{array}
\]
yielded by Theorems \ref{leftsymbolicrepnthm} and \ref{leftsymbolicrepnthmforZA0} are \textit{symbolic representations} of the orders $A_k$ and $\mathbb{Z}A_0$ for the system $\mathcal{D}$.

The representations
\[
\begin{array}{rcl}
A_{2k} & \cong & R + (0, a_{2k})(K_{[x]}) + L \\
A_{2k+1} & \cong & R + (0, a_{2k+1})(K_{[x]}) + L' \\
\mathbb{Z}A_0 & \cong & \mathbb{R}(K_{[x]})
\end{array}
\]
yielded by Theorem \ref{leftrealrepnthm} are \textit{real representations} of the orders $A_k$ and $\mathbb{Z}A_0$ for $\mathcal{D}$.
\end{symbolicrealrepndefn}

\theoremstyle{definition}
\newtheorem{sufficientsymmdefn}[lct]{Definition}
\begin{sufficientsymmdefn}\label{sufficientsymmdefn}
The real representation $\mathbb{Z}A_0 \cong \mathbb{R}(K_{[x]})$ from Theorem \ref{leftrealrepnthm} is \textit{sufficient} if $K_0 \cong K_{a_1}$. It is \textit{symmetric} if $L \cong L'$.
\end{sufficientsymmdefn}

Since $K_0 \cong L + R$ and $K_{a_1} \cong L' + R$, symmetry of the real representation in Definition \ref{sufficientsymmdefn} implies sufficiency.

The following proposition says that sufficiency is just the assertion $K_a \cong K_{a'}$ for all $a, a' \in G = [0]$. 

\theoremstyle{definition}
\newtheorem{sufficientimpliesrealrepnisuptoEG}[lct]{Proposition}
\begin{sufficientimpliesrealrepnisuptoEG}\label{sufficientimpliesrealrepnisuptoEG}
If the real representation $\mathbb{Z}A_0 \cong \mathbb{R}(K_{[x]})$ from Theorem \ref{leftrealrepnthm} is sufficient, then $\mathbb{R}(K_{[x]})$ is a replacement up to $E_G$. 
\end{sufficientimpliesrealrepnisuptoEG}

\begin{proof}
Immediate from Theorem \ref{leftrealrepnthm}: if $K_0 \cong K_{a_1}$ then $K_x \cong K_y$ for all $x, y \in [0]$. It follows $K_x \cong K_y$ for all $x, y$ such that $xE_Gy$. 
\end{proof}

The following lemma says that when its real representation is sufficient, every $a \in G$ induces an automorphism of $\mathbb{Z}A_0 \cong \mathbb{R}(K_{[x]})$.

\theoremstyle{definition}
\newtheorem{sufficientthenGactsbyautos}[lct]{Lemma}
\begin{sufficientthenGactsbyautos}\label{sufficientthenGactsbyautos}
Suppose the representation $\mathbb{Z}A_0 \cong \mathbb{R}(K_{[x]})$ from Theorem \ref{leftrealrepnthm} is sufficient and $\varphi: K_0 \rightarrow K_{a_1}$ is a fixed isomorphism. Then any $a \in G$ induces an automorphism of $\mathbb{R}(K_{[x]})$ via the map
\[
(x, k) \mapsto \left\{  \begin{array}{ll}
                            (x+a, \varphi(k)) & \textrm{if $x \in [0]^-, x+a \in [0]^+$}; \\
                            (x+a, \varphi^{-1}(k)) & \textrm{if $x \in [0]^+, x+a \in [0]^-$}; \\
                            (x+a, k) & \textrm{otherwise}.
                        \end{array} \right.
\]
If the representation is symmetric, then $\varphi$ may be taken to be the identity. 
\end{sufficientthenGactsbyautos}

\begin{proof}
If $x \in [0]^-$ and $x+a \in [0]^+$, then $K_x \cong K_0$ and $K_{x+a} \cong K_{a_1}$, hence the rule $(x, k) \mapsto (x+a, \varphi(k))$ is well-defined. 

Symmetrically, $(x, k) \mapsto (x+a, \varphi^{-1}(k))$ is well-defined if $x \in [0]^+, x+a \in [0]^-$. 

In all other cases we have $K_x \cong K_{x+a}$, and so $(x, k) \mapsto (x+a, k)$ is well-defined. Thus the map is well-defined globally, and clearly an automorphism of $\mathbb{R}(K_{[x]})$. 
\end{proof}

In Section \ref{subsect:leftalgorun}, we established certain isomorphisms between the infinite discrete sums $\omega A_k$ and $\omega^* A_k$ of the terms $A_k$ from a \textit{terminating} run of the left Euclidean algorithm. From that discussion it is not hard to see that these isomorphisms are derived purely from the identities $A_k \cong N_kA_{k+1} + A_{k+2}$ in the division system. Hence these isomorphisms hold for the terms from an arbitrary terminating left division system $\{A_k; M\}$. That is, for such a system we have, for all $k, k' < M$,
\[
\begin{array}{rccl}
\mathbb{Z} A_k & \cong & \mathbb{Z} A_{k'}; & \\
\omega A_k & \cong & \omega A_{k'}; & \\
\omega^* A_k & \cong & \omega^* A_{k'} & \textrm{if $k \equiv k' \pmod 2$}.
\end{array}
\]

The proposition below says that if a \textit{non-terminating} left system has a sufficient real representation, then the same isomorphisms hold. 

\theoremstyle{definition}
\newtheorem{sufficiencyimpliesIsosforomegaprods}[lct]{Proposition}
\begin{sufficiencyimpliesIsosforomegaprods}\label{sufficiencyimpliesIsosforomegaprods}
Consider the real representation $\mathbb{Z}A_0 \cong \mathbb{R}(K_{[x]})$ from Theorem \ref{leftrealrepnthm}.
\begin{itemize}
    \item[i.] If the representation is sufficient, then for all $k, k' \in \omega$, we have
    \[
    \begin{array}{rccl}
    \mathbb{Z} A_k & \cong & \mathbb{Z} A_{k'}; & \\
    \omega A_k & \cong & \omega A_{k'}; & \\
    \omega^* A_k & \cong & \omega^* A_{k'} & \textrm{if $k \equiv k' \pmod 2$}.
    \end{array}
    \]
    \item[ii.] If the representation is moreover symmetric, then $A_k + A_{k'} \cong A_{k'} + A_k$ and $\omega^* A_k \cong \omega^* A_{k'}$ for every $k, k' \in \omega$.
\end{itemize}
\end{sufficiencyimpliesIsosforomegaprods}

\begin{proof}
We have
\[
\omega A_k \cong R + (0, a_k)(K_{[x]}) + M + R + (0, a_k)(K_{[x]}) + M + R + \cdots
\]
where $M = L$ or $M = L'$ depending on the parity of $k$. In either case, by sufficiency we have $M + R \cong K_{a_k}$. Further, since the representation is a replacement up to $E_G$ and $a_k \in G$, we have $K_{a_k} \cong K_{na_k}$ and $(0, a_k)(K_{[x]}) \cong (na_k, (n+1)a_k)(K_{[x]})$ for every $n \in \omega$, via translation by $na_k \in G$. Hence
\[
\begin{array}{rcl}
\omega A_k & \cong & R + (0, a_k)(K_{[x]}) + M + R + (0, a_k)(K_{[x]}) + M + R + \cdots \\
& \cong & R + (0, a_k)(K_{[x]}) + K_{a_k} + (a_k, 2a_k)(K_{[x]}) + K_{2a_k} + \cdots \\
& \cong & R + (0, \infty)(K_{[x]}).
\end{array}
\]
Since $k$ was arbitrary, we have $\omega A_k \cong \omega A_{k'} \cong R + (0, \infty)(K_{[x]})$ for every $k, k' \in \omega$. 

Similar arguments show that
\[
\mathbb{Z} A_k \cong \mathbb{R}(K_{[x]}) \cong \mathbb{Z}A_0
\]
for every $k, k' \in \omega$, and on the other side,
\[
\omega^* A_k \cong (-\infty, 0)(K_{[x]}) + L
\]
when $k$ is even and 
\[
\omega^* A_k \cong (-\infty, 0)(K_{[x]}) + L'
\]
when $k$ is odd. This proves (i.).

For (ii.), if we have $L \cong L'$, then the above yields 
\[
\omega^*A_k \cong (-\infty, 0)(K_{[x]}) + L \cong \omega^*A_{k'}
\]
Further, 
\[
\begin{array}{rcl}
A_k + A_{k'} & \cong & R + (0, a_k)(K_{[x]}) + L + R + (0, a_{k'}) + L \\
& \cong & R + (0, a_k)(K_{[x]}) + K_{a_k} + (a_k, a_k + a_{k'})(K_{[x]}) + L \\
& \cong & R + (0, a_k + a_{k'})(K_{[x]}) + L
\end{array}
\]
for all $k, k' \in \omega$. Since $a_k + a_{k'} = a_{k'} + a_k$, this implies $A_k + A_{k'} \cong A_{k'} + A_k$.
\end{proof}

Though we do not include a construction here, it is possible to build examples showing that the sufficiency hypothesis from Proposition \ref{sufficiencyimpliesIsosforomegaprods} cannot in general be dropped.

\subsection{Real representations of right systems}\label{subsect:rightsystemrealrepn}

In this section we state the analogous real representation theorem for non-terminating right division systems. The (omitted) proof is symmetric with the left division proof.

Since it is only the real representations of such systems that will be of interest later, we leave out discussing their symbolic representations. 

The set of division coefficients $\mathcal{C} = \{N_k\}$, the sequence $\{a_k\}$, and the group $G$ are the same as in the previous section. 

There is one important difference of notation from the left-sided setting: we now view $G$ as generated by the translations $x \mapsto x - 1$ and $x \mapsto x - a_1$. Consequently, we redefine the forward and backward orbits of $[0] = [0]_{E_G}$ as follows:
\[
[0]^+ = \{x \in \mathbb{R}: \exists m, n \in \mathbb{Z}, n > 0, x = m + n(-a_1)\}
\]
and
\[
[0]^- = \{x \in \mathbb{R}: \exists m, n \in \mathbb{Z}, n \leq 0, x = m + n(-a_1)\}.
\]

We again write $E_G^{\pm}$ for the equivalence relation obtained from $E_G$ by splitting $[0]$ into the two classes $[0]^+$ and $[0]^-$. 

\theoremstyle{definition}
\newtheorem{rightrealrepnthm}[lct]{Theorem}
\begin{rightrealrepnthm}\label{rightrealrepnthm}
Suppose that $\{A_k\}$ is a non-terminating right division system on the coefficients $\mathcal{C}$. Then there is a collection of orders $K_{[x]}$ indexed by the $E_G^{\pm}$-classes of $\mathbb{R}$ such that $\mathbb{Z}A_0 \cong \mathbb{R}(K_{[x]})$. 

Moreover, there are decompositions
\[
\begin{array}{rcl}
K_0 & \cong & L + R \\
K_{-a_1} & \cong & L + R'
\end{array}
\]
such that for every $k \in \omega$, 
\[
\begin{array}{rcl}
\textrm{$k$ even} & \Rightarrow & A_k \cong R + (-a_k, 0)(K_{[x]}) + L \\
\textrm{$k$ odd} & \Rightarrow & A_k \cong R' + (-a_k, 0)(K_{[x]}) + L.
\end{array}
\]
\end{rightrealrepnthm}

Symmetrically with the left case, we say that the representation from Theorem \ref{rightrealrepnthm} is \textit{sufficient} if $L + R \cong L + R'$ and \textit{symmetric} if $R \cong R'$. 

If the representation is sufficient, then $K_a \cong K_{a'}$ for any $a, a' \in G = [0]$, in which case any $a \in G$ lifts to an automorphism of $\mathbb{R}(K_{[x]})$ via the right-sided version of Lemma \ref{sufficientthenGactsbyautos}. 

Here is the analogue of Proposition \ref{sufficiencyimpliesIsosforomegaprods}. 

\theoremstyle{definition}
\newtheorem{sufficiencyimpliesIsosforomegaprodsRIGHTVERSION}[lct]{Proposition}
\begin{sufficiencyimpliesIsosforomegaprodsRIGHTVERSION}\label{sufficiencyimpliesIsosforomegaprodsRIGHTVERSION}
Consider the real representation $\mathbb{Z}A_0 \cong \mathbb{R}(K_{[x]})$ from Theorem \ref{rightrealrepnthm}.
\begin{itemize}
    \item[i.] If the representation is sufficient, then for all $k, k' \in \omega$, we have
    \[
    \begin{array}{rccl}
    \mathbb{Z} A_k & \cong & \mathbb{Z} A_{k'}; & \\
    \omega^* A_k & \cong & \omega^* A_{k'}; & \\
    \omega A_k & \cong & \omega A_{k'} & \textrm{if $k \equiv k' \pmod 2$}.
    \end{array}
    \]
    \item[ii.] If the representation is moreover symmetric, then $A_k + A_{k'} \cong A_{k'} + A_k$ and $\omega A_k \cong \omega A_{k'}$ for every $k, k' \in \omega$.
\end{itemize}
\end{sufficiencyimpliesIsosforomegaprodsRIGHTVERSION}

\subsection{Real representations of symmetric systems}

We next state the real representation theorem for non-terminating symmetric division systems. 

There is a new wrinkle here. One might guess that the real representation of a symmetric system is automatically symmetric. However, in general we have to assume an extra hypothesis to guarantee symmetry of the representation, namely that the system is non-splitting.

\theoremstyle{definition}
\newtheorem{symmetricrealrepnthm}[lct]{Theorem}
\begin{symmetricrealrepnthm}\label{symmetricrealrepnthm}
Suppose that $\{A_k\}$ is a non-terminating, non-splitting symmetric division system on the coefficients $\mathcal{C}$. Then there is a collection of orders $K_{[x]}$ indexed by the $E_G$-classes of $\mathbb{R}$ such that $\mathbb{Z}A_0 \cong \mathbb{R}(K_{[x]})$. 

Moreover, there is a decomposition
\[
\begin{array}{rcl}
K_0 & \cong & L + R
\end{array}
\]
such that for every $k \in \omega$, 
\[
\begin{array}{rcl}
A_k & \cong & R + (0, a_k)(K_{[x]}) + L.
\end{array}
\]
\end{symmetricrealrepnthm}
 
\begin{proof}
View $\{A_k\} = \{A_k \cong N_k A_{k+1} + A_{k+2}\} = \mathcal{D}$ as a left division system.

Let $\mathbb{Z}X(I_{[(n, u)]})$ be the symbolic representation of $\mathbb{Z}A_0$ from Theorem \ref{leftsymbolicrepnthmforZA0}, and let $\mathbb{R}(K_{[x]})$ be its real representation from Theorem \ref{leftrealrepnthm}. Let 
\[
\begin{array}{rcl}
A_0 & \cong & R + (0, 1)(K_{[x]}) + L \\
A_1 & \cong & R + (0, a_1)(K_{[x]}) + L'
\end{array}
\]
denote the corresponding real representations of $A_0$ and $A_1$. 

Since the system is symmetric we have $A_0 + A_1 \cong A_1 + A_0$. Fix an isomorphism $f: A_0 + A_1 \rightarrow A_1 + A_0$. View $L'$ as a final segment of $A_1$, and $A_1$ as a final segment of $A_0 + A_1$. Likewise, view $L$ as a final segment of $A_0$, and $A_0$ of $A_1 + A_0$. 

From the proof of \ref{leftrealrepnthm} we have $L \cong I_{(0, e_0)}$ and $L' \cong I_{(0, e_1)}$. 

It follows from the proof of Theorem \ref{leftsymbolicrepnthm} that in the symbolic representation $A_0 \cong X(I_{[u]})$, the rightmost replacement order $I_{e_0}$, which is isomorphic to $I_{(0, e_0)}$, is obtained as an intersection of a rightward $\mathcal{D}$-branching $\{I_n\}$ in $A_{\emptyset} \cong A_0$: namely, by taking $I_0$ to be the final copy $A_{N_00}$ of $A_2$ within $A_{\emptyset} \cong N_0 A_1 + A_2$, taking $I_1$ to be the final copy $A_{N_00N_20}$ of $A_4$ within $A_{N_00} \cong N_2 A_3 + A_4$, and so on. 

Recall the definition of the branching condensation $\sim_{\mathcal{D}}$ from Section \ref{subsect:simDcondensation}. Since $\mathcal{D}$ is non-splitting, by Proposition \ref{Dbranchingintersectionpropn} the branching intersection $\bigcap_n I_n \cong L$ is an initial segment of the $\sim_{\mathcal{D}}$-class $\delta(x)$ for any given $x \in L$. Since $L$ is a final segment of $A_0$, it must in fact equal $\delta(x)$. Thus, if $L$ is non-empty, it is the rightmost $\sim_{\mathcal{D}}$-class in $A_0$. 

By a similar argument $L'$ is the rightmost $\sim_{\mathcal{D}}$-class in $A_1$ (if it is nonempty). By Proposition \ref{simDclassesinvarunderconvfcoro}, since $\mathcal{D}$ is non-splitting, $\sim_{\mathcal{D}}$-classes are preserved by convex embeddings. Since $f$ is an isomorphism (i.e., a surjective convex embedding), it follows one of $A_0 + A_1$ and $A_1 + A_0$ has a nonempty rightmost $\sim_{\mathcal{D}}$-class if and only if the other does, and in this case, $f$ sends one of these classes onto the other. That is, either $L = L' = \emptyset$ or $f[L'] = L$. In either case, $L \cong L'$. 

Thus the representation $\mathbb{R}(K_{[x]})$ is symmetric, and the theorem now follows from Theorem \ref{leftrealrepnthm}. 
\end{proof}

In Section \ref{subsect:invarofrepnsfornonsplitsys}, we will generalize the rigidity argument for $\sim_{\mathcal{D}}$-classes of non-splitting systems from the proof above. See Theorem \ref{simdequalssimDfornonsplittinsystems}.

\theoremstyle{definition}
\newtheorem{Isosofomegaprodsfornonsplsymmsystems}[lct]{Proposition}
\begin{Isosofomegaprodsfornonsplsymmsystems}\label{Isosofomegaprodsfornonsplsymmsystems}
Consider the real representation $\mathbb{Z}A_0 \cong \mathbb{R}(K_{[x]})$ from Theorem \ref{symmetricrealrepnthm}. Then for all $k, k' \in \omega$ we have the identities:
\[
    \begin{array}{rccl}
    \mathbb{Z} A_k & \cong & \mathbb{Z} A_{k'}; & \\
    \omega^* A_k & \cong & \omega^* A_{k'}; & \\
    \omega A_k & \cong & \omega A_{k'}. & 
    \end{array}
\]
\end{Isosofomegaprodsfornonsplsymmsystems}
\begin{proof}
Immediate from Theorems \ref{sufficiencyimpliesIsosforomegaprods} and \ref{symmetricrealrepnthm}.
\end{proof}

\subsection{Representations of rational systems}\label{subsect:repnsforrationalsystems}

In this section we present the analogous symbolic and real representation theorems for non-terminating rational division systems. We indicate the proofs only briefly: though they are not identical to the corresponding proofs from Section \ref{subsect:leftdivissystemsymbolicrepn} for left systems, they take the same approach and are in fact simpler. In particular, there is no distinction between forward and backward orbits of the origin, and the representations are automatically sufficient. 

Let $\mathcal{C} = \{N_k\}$ denote a fixed set of rational division coefficients. 

\theoremstyle{definition}
\newtheorem{rationaldivtreedefn}[lct]{Definition}
\begin{rationaldivtreedefn}\label{rationaldivtreedefn}
The \textit{rational division tree $T^{\mathcal{C}} = T$ on the coefficients $\mathcal{C}$} is defined recursively as follows. 
\begin{itemize}
    \item $T_0 = \{\emptyset\}$. 
    \item Given $T_k$ for $k \in \omega$ and a terminal node $r \in T_k$, we have by induction that $|r| = k$. Define the set of extensions $S_r = \{r0, r1, \ldots, r(N_k-1)\}$ of $r$, consisting of the successors $ri$ for $0 \leq i \leq N_k-1$.

    Define 
    \[ 
    T_{k+1} = T_k \cup \bigcup_{\textrm{$r \in T_k$ terminal}} S_r.
    \]
    \item Define
    \[
    T = \bigcup_{k \in \omega} T_k.
    \]
\end{itemize}
\end{rationaldivtreedefn}

\theoremstyle{definition}
\newtheorem{rationaldivorderdefn}[lct]{Definition}
\begin{rationaldivorderdefn}\label{rationaldivorderdefn}
The \textit{rational division order $X^{\mathcal{C}} = X$ on the coefficients $\mathcal{C}$} is the body $[T]$ of the rational division tree $T$ on the coefficients $\mathcal{C}$, equipped with the lexicographical ordering. 
\end{rationaldivorderdefn}

We copy over the notation and terminology from the one-sided representation setting. In particular $T_r$, $X_r$, and $E_0$ have their expected meanings for rational division trees $T$ and orders $X$.

\theoremstyle{definition}
\newtheorem{rationalsymbolicrepnthm}[lct]{Theorem}
\begin{rationalsymbolicrepnthm}\label{rationalsymbolicrepnthm}
Suppose $T$ is the rational division tree for the division coefficients $\mathcal{C}$ and $X$ is the corresponding rational division order.
\begin{itemize}
    \item[i.] For every non-terminating rational division system  $\{A_k: k \in \omega\}$ on the coefficients $\mathcal{C}$, there exists a collection of orders $\{I_{[u]}: u \in X\}$ indexed by the $E_0$-classes $[u]$ of $X$ such that for every node $r \in T$ with $|r| = k$, we have $A_k \cong X_r(I_{[u]})$. In particular, $A_0 \cong X(I_{[u]})$.

    \item[ii.] Conversely, suppose $\{I_{[u]}: u \in X\}$ is a collection of orders indexed by the $E_0$-classes of $X$ and $X(I_{[u]})$ is the corresponding replacement of $X$ up to $E_0$. Then for every pair of nodes $r, s \in T$ with $|r| = |s|$ we have $X_r(I_{[u]}) \cong X_s(I_{[u]})$. 
    
   Furthermore, if for each $k$ we fix $A_k$ isomorphic to $X_r(I_{[u]})$ for some (equivalently, any) node $r$ with $|r| = k$, then $A_k \cong N_k A_{k+1}$ for all $k \in \omega$, that is, $\{A_k: k \in \omega\}$ is a non-terminating rational division system on the coefficients $\mathcal{C}$. 
\end{itemize}
\end{rationalsymbolicrepnthm}

\begin{proof}
Analogous to the proof of Theorem \ref{leftsymbolicrepnthm}. Since the system is rational, the tree $T$ has no irregular nodes. For every $r \in T$, recursively define embeddings $f_r: A_{|r|} \rightarrow A_0$ via the rule 
\[
f_{ri} = f_r f_i^k
\]
for $0 \leq i \leq N_k-1$, where $f_i^k$ is a fixed embedding of $A_{k+1}$ onto its $i$th copy in $A_k$. Taking intersections of iterated images of the maps $f_r$ yields a representation of the form $A_0 \cong X(I_{[u]})$, and more generally $A_k \cong X_r(I_{[u]})$ for any $|r| = k$. This yields (i.); (ii.) is also a straightforward adaptation.
\end{proof}

Consider the sequence of rational numbers $a_k$ determined by the equations 
\[
a_k = N_k a_{k+1}
\]
(on the coefficients $N_k \in \mathcal{C}$) by setting $a_0 = 1$. Then for $k > 0$, we have 
\[
a_k = \frac{1}{N_0N_1 \cdots N_{k-1}}.
\]

Let $G$ denote the group of translations on $\mathbb{R}$ generated by the maps $x \mapsto x + a_k$, and let $E_G$ denote its orbit equivalence relation. Since $N_k \geq 2$ for all $k$, the sequence $a_k$ converges to $0$, and hence the orbits $[x]_{E_G} = [x]$ are each dense in $\mathbb{R}$. 

\theoremstyle{definition}
\newtheorem{rationalrealrepnthm}[lct]{Theorem}
\begin{rationalrealrepnthm}\label{rationalrealrepnthm}
Suppose that $\{A_k\}$ is a rational division system on the coefficients $\mathcal{C}$. Then there is a collection of orders $K_{[x]}$ indexed by the $E_G$-classes of $\mathbb{R}$ such that $\mathbb{Z}A_0 \cong \mathbb{R}(K_{[x]})$. 

Moreover, there is a decomposition
\[
\begin{array}{rcl}
K_0 & \cong & L + R
\end{array}
\]
such that for every $k \in \omega$, 
\[
\begin{array}{rcl}
A_k & \cong & R + (0, a_k)(K_{[x]}) + L.
\end{array}
\]
\end{rationalrealrepnthm}

\begin{proof}
From the representation $X(I_{[u]})$ of $A_0$ we may write down a representation $\mathbb{Z}X(I_{[(n, u)]})$ of $\mathbb{Z}A_0$ that we will push forward to get $\mathbb{R}(K_{[x]})$.

Define the map $\sigma: \mathbb{Z}X \rightarrow \mathbb{R}$ as before:
\[
\sigma(n, u) = n + \sum_{k \in \omega} u_k a_{k+1}.
\]
Then $\sigma$ is essentially an order-isomorphism of $\mathbb{Z}X$ and $\mathbb{R}$, up to condensing each jump pair in $\mathbb{Z}X$ onto a singleton in $[0]_{E_G} = [0]$. 

Let $e_0$ denote the ``always rightmost" sequence $(N_0-1)(N_1-1)\ldots \in X$. Then $e_0$ is the right endpoint of $X$; $\overline{0}$ is its left endpoint. 

Each jump pair in $X$ has its left point in the $E_0$-class of $e_0$ and its right point in the $E_0$-class of $\overline{0}$. Thus if we replace each $x \in \mathbb{R}$ by 
\[
K_x = I_{\sigma^{-1}(x)}
\]
whenever $\sigma^{-1}(x)$ is a singleton, and by 
\[
K_x = I_{e_0} + I_{\overline{0}}
\]
whenever $x$ is a condensed jump pair (i.e. when $x \in [0]$), the resulting replacement $\mathbb{R}(K_{[x]})$ is up to $E_G$ and naturally isomorphic (via $\sigma$) to $\mathbb{Z}X(I_{[(n, u)]}) \cong \mathbb{Z}A_0$. 
\end{proof}

Since the replacement $\mathbb{R}(K_{[x]})$ is up to $E_G$ in Theorem \ref{rationalrealrepnthm}, every $a \in G$ induces the automorphism $(x, k) \mapsto (x+a, k)$ of $\mathbb{R}(K_{[x]})$. 

\theoremstyle{definition}
\newtheorem{Isosofomegaprodsforrationalsystems}[lct]{Proposition}
\begin{Isosofomegaprodsforrationalsystems}\label{Isosofomegaprodsforrationalsystems}
Consider the real representation $\mathbb{Z}A_0 \cong \mathbb{R}(K_{[x]})$ from Theorem \ref{rationalrealrepnthm}. Then for all $k, k' \in \omega$, we have $A_k + A_{k'} \cong A_{k'} + A_k$ and we have the identities:
\[
    \begin{array}{rccl}
    \mathbb{Z} A_k & \cong & \mathbb{Z} A_{k'}; & \\
    \omega^* A_k & \cong & \omega^* A_{k'}; & \\
    \omega A_k & \cong & \omega A_{k'}. & 
    \end{array}
\]
\end{Isosofomegaprodsforrationalsystems}
\begin{proof}
Straightforward from Theorem \ref{rationalrealrepnthm}.
\end{proof}

\subsection{Invariance of representations for non-splitting systems}\label{subsect:invarofrepnsfornonsplitsys}

In this section we prove rigidity theorems for the real representations $\mathbb{R}(K_{[x]})$ for \textit{non-splitting} division systems $\mathcal{D}$. The non-splitting hypothesis is a strong one: we will show that when $\mathcal{D}$ is non-splitting the real representations of its terms $A_k$ are, in a strong sense, canonical. Our results here are the main tools we will use in our classification of the commutative semigroups representable in $(LO, +)$ in Section \ref{section:commutesemigroupsinLOrepn}. 

Fix a set of division coefficients $\mathcal{C}$, and also fix $\mathcal{D} = \{A_k: k \in \omega\}$ a non-terminating (left, right, or rational) division system on these coefficients. Let $X(I_{[u]})$ and $\mathbb{Z}X(I_{[u]}) \cong \mathbb{R}(K_{[x]})$ be the symbolic and real representations for $A_0$ and $\mathbb{Z}A_0$ obtained in the previous sections. 

A priori, these representations are not invariants of the order type of $A_0$: they depend on the other orders $A_k$ appearing in the system $\mathcal{D}$, the particular convex embeddings $f_r$ used to embed these orders in $A_0$, and the intervals $I_u$ that these maps uncover. 

The intervals $I_u$ are obtained by iteratively decomposing $A_0$ via the isomorphisms in $\mathcal{D}$ (or more precisely, via the maps $f_r$ witnessing these isomorphisms), and taking an intersection along a descending sequence of intervals $\{I_n\}$, the $n$th appearing at the $n$th stage of the decomposition, and each isomorphic to some $A_k$. 

Such a sequence is a $\mathcal{D}$-branching (see Definition \ref{Dbranchingdefn}). In the case when the system is non-splitting, this observation combined with Proposition \ref{Dbranchingintersectionpropn} may be used to show that the representations we obtain by this process, written in terms of the intervals $I_u$, are canonical. 

Define a condensation $\sim_d$ on the symbolic representation $X(I_{[u]})$ by the rule 
\[
\textrm{$(u, i) \sim_d (v, j)$ if $u = v$ or if $u, v$ form a jump pair in $X$.}
\]
This condensation condenses each replacement order $I_u$ in $X(I_{[u]})$ back into a point, as well as the intervals $I_u + I_v$ corresponding to a replaced jump pair. 

Let $d$ denote the condensation map for $\sim_d$, so that $d(u,i)$ denotes the $\sim_d$-class of any given $(u,i) \in X(I_{[u]})$.

We also write $\sim_d$ for the condensation defined on the representation $\mathbb{Z}X(I_{[u]})$ by the same rule, as well as the corresponding condensation on the real representation $\mathbb{R}(K_{[x]})$ (i.e. the relation obtained by pushing forward $\sim_d$ via $\sigma$). Observe that on $\mathbb{R}(K_{[x]})$, $\sim_d$ is just the replacement condensation defined by 
\[
\textrm{$(x, k) \sim_d (x', l)$ if $x = x'$,} 
\]
and the $d$-classes are just the replacement orders $K_x$. 

The condensation $\sim_d$ is defined on the \textit{representations} $X(I_{[u]})$, $\mathbb{Z}X(I_{[u]})$, and $\mathbb{R}(K_{[x]})$. A priori, to view $\sim_d$ as defined on the represented orders $A_0$ and $\mathbb{Z}A_0$ requires fixing isomorphisms between these orders and the representations. 

Recall the definition \ref{simDdefn} of the condensation scheme $\sim_{\mathcal{D}}$ associated to $\mathcal{D}$. This is a global condensation scheme on $LO$, invariant under convex embeddings, which depends only on the order types of the orders $A_k$ appearing in $\mathcal{D}$. As before we write $\delta$ for the condensation map of this condensation. 

\theoremstyle{definition}
\newtheorem{simdequalssimDfornonsplittinsystems}[lct]{Theorem}
\begin{simdequalssimDfornonsplittinsystems}\label{simdequalssimDfornonsplittinsystems}
Suppose $\mathcal{D}$ is a non-splitting system. Then $\sim_d$ coincides with $\sim_{\mathcal{D}}$ on each of the representations $X(I_{[u]})$, $\mathbb{Z}X(I_{[u]})$, and $\mathbb{R}(K_{[x]})$.
\end{simdequalssimDfornonsplittinsystems}

\begin{proof}
Assume $\mathcal{D}$ is a left division system. The proofs for right systems and rational systems are similar. 

We prove the statement first for $X(I_{[u]})$. From the proof of Theorem \ref{leftsymbolicrepnthm}, for each $u \in X$ we have 
\[
I_u = \bigcap_{k \in \omega} A_{u \upharpoonright k}.
\]
When $u \upharpoonright k$ is a regular initial sequence of $u$ we have
\[
\begin{array}{rcl}
A_{u \upharpoonright k} & \cong & A_{(u \upharpoonright k)0} + A_{(u \upharpoonright k)1} + \cdots + A_{(u \upharpoonright k)(N_k-1)} + A_{(u \upharpoonright k)N_k0} \\
& \cong & A_{k+1} + A_{k+1} + \cdots + A_{k+1} + A_{k+2}.
\end{array}
\]

Let $\{k_n\}$ be the strictly increasing sequence of indices such that $u \upharpoonright k_n$ is the $n$th regular initial sequence of $u$. Let $I_n^u = A_{u \upharpoonright k_n}$. Then $\{I_n^u\}$ is clearly a $\mathcal{D}$-branching. Note that this branching is rightward when $u \in [e_0] \cup [e_1]$, leftward when $u \in [\overline{0}]$, and middle exactly when $u \not\in [e_0] \cup [e_1] \cup [\overline{0}]$.

Fix $x \in X(I_{[u]})$. Then $x \in I_u$ for a unique $u \in X$. We claim that the $\sim_{\mathcal{D}}$-class of $x$ coincides with its $\sim_d$-class in $X(I_{[u]})$, i.e. $d(x) = \delta(x)$. 

If $u \not\in [e_0] \cup [e_1] \cup [\overline{0}]$, then $u$ is not in a jump pair and $I_u = d(x)$. It is also the intersection of the middle branching $\{I^u_n\}$, which by Proposition \ref{Dbranchingintersectionpropn} is the $\sim_{\mathcal{D}}$-class of $x$. That is, $d(x) = \delta(x) = I_u$ and we are done in this case. 

If $u \in [e_0] \cup [e_1]$ and $u < e_0$, then $u$ is the left point of a jump pair $u < v$ in $X$ with $v \in [\overline{0}]$. Thus $d(x) = I_u + I_v$. By Proposition \ref{Dbranchingintersectionpropn}, $I_u$ is initial in $\delta(x)$, since it is obtained as the intersection of the rightward branching $\{I_n^u\}$. Symmetrically $I_v$, which is right adjacent to $I_u$, is obtained as the intersection of a leftward branching, and hence final in the $\sim_{\mathcal{D}}$-class $\delta(y)$ of any $y \in I_v$. 

Since $I_u$ and $I_v$ are adjacent, they must be intervals in the same $\sim_{\mathcal{D}}$-class $C$, as otherwise there would exist two adjacent $\sim_{\mathcal{D}}$-classes in $X$, contradicting Proposition \ref{simDisacondpropn}.(iii.). It follows that $C$ (in which $I_u$ is initial and $I_v$ is final) is $I_u + I_v$, i.e. $d(x) = \delta(x) = C = I_u + I_v$. Thus again, the $\sim_d$ and $\sim_{\mathcal{D}}$-classes of $x$ coincide.

The argument is symmetric when $u \in [\overline{0}]$ and $u > \overline{0}$. It remains to check the claim in the cases when $u = \overline{0}$ and $u = e_0$. If $u = \overline{0}$, then $I_u = d(x)$, and it is final $\delta(x)$, being obtained as the intersection of a leftward branching. But then since $I_u$ is initial in $X(I_{[u]})$, it must in fact coincide with $\delta(x)$. Thus we are done again in this case, and the case when $u = e_0$ is symmetric. 

Since $x$ was arbitrary, $\sim_d$ and $\sim_{\mathcal{D}}$ coincide on $X(I_{[u]})$, as desired. The proofs for $\mathbb{Z}A_0$ and $\mathbb{R}(K_{[x]})$ are similar. 
\end{proof}

We now show that Theorem \ref{simdequalssimDfornonsplittinsystems} implies that every automorphism of $\mathbb{R}(K_{[x]})$ determines an automorphism of $\mathbb{R}$. 

\theoremstyle{definition}
\newtheorem{autosofRKxlifttoautosofRfornonsplittinsystems}[lct]{Proposition}
\begin{autosofRKxlifttoautosofRfornonsplittinsystems}\label{autosofRKxlifttoautosofRfornonsplittinsystems}
Suppose $\mathcal{D}$ is a non-splitting system. For an automorphism $f: \mathbb{R}(K_{[x]}) \rightarrow \mathbb{R}(K_{[x]})$, define $\hat{f}: \mathbb{R} \rightarrow \mathbb{R}$ by the rule 
\[
\hat{f}(x) = y \Leftrightarrow f[K_x] = K_y.
\]
Then $\hat{f}$ is an automorphism of $\mathbb{R}$.
\end{autosofRKxlifttoautosofRfornonsplittinsystems}

\begin{proof}
We first check that $f$ is well-defined. Since $f$ preserves $\sim_{\mathcal{D}}$ by Proposition \ref{simDclassesinvarunderconvfcoro}, $f$ preserves $\sim_d$ by Theorem \ref{simdequalssimDfornonsplittinsystems}. That is, 
\[
f[d(x, k)] = d(f(x, k))
\]
for any $(x, k) \in \mathbb{R}(K_{[x]})$. But $d(x, k) = K_x$ and $d(f(x, k)) = d(y, k') = K_y$ where $(y, k') = f(x, k)$, which shows there is $y \in \mathbb{R}$ such that $f[K_x] = K_y$. Thus $\hat{f}$ is well-defined. 

It is now straightforward to see that $\hat{f}$ is order-preserving (since $f$ is order-preserving and preserves $\sim_d$) and onto (since $f$ is onto), and hence an automorphism of $\mathbb{R}$, as claimed.
\end{proof}

When $\mathcal{D}$ is non-splitting, we will continue to use the notation from Proposition \ref{autosofRKxlifttoautosofRfornonsplittinsystems} and write $\hat{f}$ for the automorphism of $\mathbb{R}$ condensed from an automorphism $f: \mathbb{R}(K_{[x]}) \rightarrow \mathbb{R}(K_{[x]})$.

For a linear order $A$, we write $\textrm{Aut}(A) = \textrm{Aut}(A, <)$ for the group (under composition) of automorphisms of $A$. 

\theoremstyle{definition}
\newtheorem{ftofhatishomo}[lct]{Proposition}
\begin{ftofhatishomo}\label{ftofhatishomo}
Suppose $\mathcal{D}$ is a non-splitting system. The map
\[
\begin{array}{rcl}
\varphi: \textrm{Aut}(\mathbb{R}(K_{[x]})) & \rightarrow & \textrm{Aut}(\mathbb{R}) \\
\varphi(f) & = & \hat{f}
\end{array}
\]
is a group homomorphism with kernel 
\[
N = \{f \in \textrm{Aut}(\mathbb{R}(K_{[x]}): \forall x \in \mathbb{R}, f[K_x] = K_x\}.
\]
\end{ftofhatishomo}

\begin{proof}
Straightforward.
\end{proof}

For the next theorem, recall the archimedean comparability relation $\sim_{conv}$ from Definition \ref{simconvdefn}: $A \sim_{conv} B$ if $2A \leqslant_{conv} mB$ and $2B \leqslant_{conv} nA$ for some integers $n, m \geq 1$. Theorem \ref{AconvAprimesamecondensation} below says that any non-terminating system $\mathcal{D}'$ on an initial order $A_0'$ with $A_0' \sim_{conv} A_0$ determines the same condensation as $\mathcal{D}$.

\theoremstyle{definition}
\newtheorem{AconvAprimesamecondensation}[lct]{Theorem}
\begin{AconvAprimesamecondensation}\label{AconvAprimesamecondensation}
Suppose $\mathcal{D}' = \{A_k'\}$ is a non-terminating division system (on some set of division coefficients $\mathcal{C}')$ with $A_0' \sim_{conv} A_0$. Then $\sim_{\mathcal{D}}$ and $\sim_{\mathcal{D'}}$ coincide (on every linear order $A$).
\end{AconvAprimesamecondensation}

\begin{proof}
Fix a linear order $A$, and suppose there exist points $x < y$ in $A$ such that $x \sim_{\mathcal{D}} y$ but $x \not\sim_{\mathcal{D}'} y$. Then for some $k$ we have $A_k' \leqslant_{conv} [x, y]$. 

By hypothesis we have $2A_0 \leqslant_{conv} mA_0'$ for some $m \geq 1$. By iteratively splitting terms using the isomorphisms from $\mathcal{D}'$, we may write $mA_0'$ as a sum
\[
mA_0' \cong \sum_{i < N} A_{k_i}'
\]
such that $k_i \geq k$ for all $i < N$. Then $A_{k_i}' \leqslant_{conv} A_k'$ for all $i < N$. 

Now by iteratively splitting terms via the isomorphisms from $\mathcal{D}$, we may decompose $2A_0$ as an $M$-sum of terms $A_{k_j}$ from $\mathcal{D}$, with $M \geq N$:
\[
2A_0 \cong \sum_{j < M} A_{k_j}.
\]
Then since $M \geq N$ and $2A_0 \leqslant_{conv} mA_0'$ we have
\[
\sum_{j < N} A_{k_j} \leqslant_{init} \sum_{j < M} A_{k_j} \leqslant_{conv} \sum_{i < N} A_{k_i}'.
\] 
By Corollary \ref{sumAiconvsumBi} we deduce $A_{k_i} \leqslant_{conv} A_{k_i}'$ for some $i < N$. But then 
\[
A_{k_i} \leqslant_{conv} A_{k_i}' \leqslant_{conv} A_k' \leqslant_{conv} [x, y],
\]
contradicting $x \sim_{\mathcal{D}} y$. 

A symmetric argument shows $x \sim_{\mathcal{D}'} y$ implies $x \sim_{\mathcal{D}} y$. Thus $\sim_{\mathcal{D}}$ and $\sim_{\mathcal{D}'}$ coincide on $A$, as claimed. 
\end{proof}

We note that Theorem \ref{AconvAprimesamecondensation} does not assume that $\mathcal{D}$ is non-splitting.

Suppose that $A_0$ is a non-splitting order. Then any non-terminating division system $\mathcal{D}$ with initial order $A_0$ is non-splitting. Suppose we have two such systems $\mathcal{D}$ and $\mathcal{D}'$, and let $\mathbb{R}(K_{[x]})$ and $\mathbb{R}(K_{[x]}')$ denote the real representations of $\mathbb{Z}A_0$ obtained from $\mathcal{D}$ and $\mathcal{D}'$, respectively. For the moment we identify these orders with $\mathbb{Z}A_0$, and view the replacement orders $K_x$ and $K_x'$ from these representations as intervals in $\mathbb{Z}A_0$. 

Taken together, Theorems \ref{simdequalssimDfornonsplittinsystems} and \ref{AconvAprimesamecondensation} imply that these representations are ``the same" in the sense that the partitions of $\mathbb{Z}A_0$ into the intervals $K_x$ and $K_x'$ are the same. That is, for every $x \in \mathbb{R}$ there is $y \in \mathbb{R}$ such that $K_x = K_y'$. It will follow from our work below that since both representations are ``normalized" in that they represent $A_0$ as a replacement of the unit interval, we in fact have $K_x = K_x'$ for all $x \in \mathbb{R}$. 

However, the groups $G$ and $G'$ associated to $\mathcal{D}$ and $\mathcal{D}'$ depend on the isomorphisms in these systems, and may differ, as may their associated orbit equivalence relations $E_G$ and $E_{G'}$. We view $\mathbb{R}(K_{[x]})$ as being a replacement up to $E_G$ and $\mathbb{R}(K_{[x]}')$ as being up to $E_{G'}$, and so in this sense these representations may still carry slightly different information.

In fact, since $K_x = K_x'$ for all $x$, both representations are up to $E_H$, where $H$ is the group generated by $G \cup G'$. They may be up to a coarser equivalence relation still, depending on whether there exist other translations of $\mathbb{R}$ not in $H$ that induce automorphisms of $\mathbb{R}(K_{[x]})$. Nonetheless, we will continue somewhat informally to refer to ``the" real representation $\mathbb{Z}A_0 \cong \mathbb{R}(K_{[x]})$ when $A_0$ is non-splitting. 

\subsubsection{Length invariance in non-splitting systems} \label{subsubsect:lengthinvarnonsplittinsystems} For this section, we assume again that $\mathcal{D}$ is non-terminating and $\mathbb{R}(K_{[x]}) \cong \mathbb{Z}A_0$ is the associated real representation. Also as before, $d$ is the condensation map for $\sim_d$, and $\delta$ for $\sim_{\mathcal{D}}$. By Theorem \ref{simdequalssimDfornonsplittinsystems}, the two relations coincide on $\mathbb{Z}A_0$ when $\mathcal{D}$ is non-splitting.

\theoremstyle{definition}
\newtheorem{elldefn}[lct]{Definition}
\begin{elldefn}\label{elldefn}
Given an interval $I \subseteq \mathbb{R}(K_{[x]})$, define $\ell(I)$ to be the length in $\mathbb{R}$ of the condensed image $d[I]$ of $I$ under the replacement condensation map. We call $\ell(I)$ the \textit{length} of $I$ in $\mathbb{R}(K_{[x]})$.
\end{elldefn}

Since $\mathbb{R}$ is Dedekind complete, every closed and bounded interval in $\mathbb{R}$ is of the form $[a, b]$ for some $a \leq b$ in $\mathbb{R}$. This allows us to easily describe the possible forms of intervals in $\mathbb{R}(K_{[x]})$.  

\theoremstyle{definition}
\newtheorem{ellforintervalspropn}[lct]{Proposition}
\begin{ellforintervalspropn}\label{ellforintervalspropn}
Suppose $I \subseteq \mathbb{R}(K_{[x]})$ is an interval. 
\begin{itemize}
    \item[i.] $\ell(I) = 0$ if and only if $I \subseteq K_a$ for some $a \in \mathbb{R}$;
    \item[ii.] $0 < \ell(I) < \infty$ if and only if $I$ is of the form
    \[
    I \cong R + (a, b)(K_{[x]}) + L
    \]
    where $R$ is final (possibly empty) in $K_a$ and $L$ is initial (possibly empty) in $K_b$.
    \item[iii.] $\ell(I) = \infty$ if and only if $I$ is of one of the following forms
    \[
    \begin{array}{rcl}
    I & \cong & R + (a, \infty)(K_{[x]}) \\
    I & \cong & I \cong (-\infty, b)(K_{[x]}) + L \\
    I & \cong & \mathbb{R}(K_{[x]})
    \end{array}
    \]
    for some $a, b \in \mathbb{R}$ and some (possibly empty) intervals $R$ final $K_a$ and $L$ initial in $K_b$.
\end{itemize}
\end{ellforintervalspropn}

\begin{proof}
Straightforward. 
\end{proof}

In case (ii.) of Proposition \ref{ellforintervalspropn} we have $\ell(I) = b-a$. 

As previously, let $a_0 = 1$ and let $\{a_k: k \in \omega\}$ be the positive real numbers satisfying the equational versions of the isomorphisms in $\mathcal{D}$. Let $G = \langle a_k: k \in \omega \rangle$ denote the group generated by the translations $x \mapsto x + a_k$, viewed as a subgroup of $(\mathbb{R}, +)$. Since $\mathcal{D}$ is non-terminating, $G$ is dense in $\mathbb{R}$. 

The following definition of sufficiency captures the previous definitions for left, right, and rational systems. 

\theoremstyle{definition}
\newtheorem{sufficientgeneraldefn}[lct]{Definition}
\begin{sufficientgeneraldefn}\label{sufficientgeneraldefn}
The representation $\mathbb{R}(K_{[x]})$ is \textit{sufficient} if it is a replacement up to $E_G$, i.e., if $K_a \cong K_{a'}$ for all $a, a' \in G$ (and hence $K_x \cong K_{x + a}$ for all $x \in \mathbb{R}$ and $a \in G$). 
\end{sufficientgeneraldefn}

\theoremstyle{definition}
\newtheorem{sufficiencyimpliesGactsonrepnlemma}[lct]{Lemma}
\begin{sufficiencyimpliesGactsonrepnlemma}\label{sufficiencyimpliesGactsonrepnlemma}
Suppose $\mathbb{R}(K_{[x]})$ is sufficient. For each $x \in \mathbb{R}$ and $a \in G$, fix an isomorphism $\varphi_x^a: K_x \rightarrow K_{x+a}$. Then every $a \in G$ induces an automorphism of $R(K_{[x]})$ via the map
\[
(x, k) \mapsto (x+a, \varphi_x^a(k)).
\]
\end{sufficiencyimpliesGactsonrepnlemma}

\begin{proof}
Clear.
\end{proof}

\theoremstyle{definition}
\newtheorem{ellIlessellJmeansIconvJforsuffrepns}[lct]{Lemma}
\begin{ellIlessellJmeansIconvJforsuffrepns}\label{ellIlessellJmeansIconvJforsuffrepns}
Suppose $\mathbb{R}(K_{[x]})$ is sufficient and $I, J \subseteq \mathbb{R}(K_{[x]})$ are intervals. If $\ell(I) < \ell(J)$ then $I \leqslant_{conv} J$. 
\end{ellIlessellJmeansIconvJforsuffrepns}

\begin{proof}
Assume first that $0 < \ell(I) < \ell(J) < \infty$. By Proposition \ref{ellforintervalspropn}, we have
\[
\begin{array}{rcl}
I & \cong & R + (a, b)(K_{[x]}) + L \\
J & \cong & R' + (c, d)(K_{[x]}) + L'
\end{array}
\]
where $R$ is final in $K_a$, $L$ is initial in $K_b$, $R'$ is final in $K_c$, $L'$ is initial in $K_d$, and $b - a = \ell(I) < \ell(J) = d - c$. In particular, we have
\[
\textrm{$I \leqslant_{conv} [a, b](K_{[x]})$ and $(c, d)(K_{[x]}) \leqslant_{conv} J$.}
\]

Since $b - a < d - c$ and $G$ is dense in $\mathbb{R}$, there is $x_0 \in G$ such that 
\[
c < a+x_0 < b+x_0 < d.
\]
Since the representation is sufficient, $x \mapsto x + x_0$ lifts to an automorphism of $\mathbb{R}(K_{[x]})$. In particular, we have the chain
\[
I \leqslant_{conv} [a, b](K_{[x]}) \cong [a+x_0, b+x_0](K_{[x]}) \subseteq (c, d)(K_{[x]}) \leqslant_{conv} J. 
\]
Hence $I \leqslant_{conv} J$, as claimed. 

If $0 < \ell(I) < \ell(J) = \infty$, a similar proof works. 

Finally, if $\ell(I) = 0$, then $I \subseteq K_x$ for some $x \in \mathbb{R}$. Then $x + a$ is in the interior of the condensed image $d[J]$ of $J$ for some $a \in G$, by density of $G$. This implies $K_{x+a} \subseteq J$ and hence $K_x \leqslant_{conv} J$, yielding $I \leqslant_{conv} J$. 
\end{proof}

Each term $A_k$ from $\mathcal{D}$ is isomorphic to an interval of the form $R + (0, a_k)(K_{[x]}) + L$ in $\mathbb{R}(K_{[x]})$ of length $a_k$. It follows from Lemma \ref{ellIlessellJmeansIconvJforsuffrepns} that if $\mathbb{R}(K_{[x]})$ is sufficient and $I \subseteq \mathbb{R}(K_{[x]})$ is an interval with $\ell(I) > a_k$, then $A_k \leqslant_{conv} I$. Conversely, if $\ell(I) < a_k$, then $I \leqslant_{conv} A_k$. 

The next sequence of results are rigidity theorems for the representation $\mathbb{R}(K_{[x]})$ when $\mathcal{D}$ is non-splitting. From a long view, they are consequences of the Splitting Dichotomy Theorem \ref{splittingdichthm}. In particular, Proposition \ref{lengthisinvariantwhenDnonsplitting} below, which says that for non-splitting $\mathcal{D}$, length is an order type invariant for intervals $I \subseteq \mathbb{R}(K_{[x]})$, implies Corollary \ref{splittingdichcor} of the splitting dichotomy for the terms $A_k$, which says that finite multiples $nA$ of a non-splitting order $A$ are pairwise distinct. 

\theoremstyle{definition}
\newtheorem{lengthisinvariantwhenDnonsplitting}[lct]{Proposition}
\begin{lengthisinvariantwhenDnonsplitting}\label{lengthisinvariantwhenDnonsplitting}
Suppose $\mathcal{D}$ is non-splitting, $\mathbb{R}(K_{[x]})$ is sufficient, and $I, J \subseteq \mathbb{R}(K_{[x]})$ are intervals. If $I \cong J$, then $\ell(I) = \ell(J)$. 
\end{lengthisinvariantwhenDnonsplitting}

\begin{proof}
Suppose toward a contradiction that $I \cong J$ but $\ell(I) < \ell(J)$. We again assume first that $0 < \ell(I)$ and $\ell(J) < \infty$.

By Proposition \ref{ellforintervalspropn}, $I$ and $J$ have the forms
\[
\begin{array}{rcl}
I & \cong & R + (a, b)(K_{[x]}) + L \\
J & \cong & R' + (c, d)(K_{[x]}) + L'
\end{array}
\]
with $b-a < d-c$. By the proof of \ref{ellIlessellJmeansIconvJforsuffrepns}, sufficiency implies there is a translated copy of $I$ of the form $R + (a', b')(K_{[x]}) + L$ with $c < a' < b' < d$. Replacing $I$ with this copy, we may assume $a = a'$ and $b = b'$. 

Since $I \cong J$ there is an an isomorphism 
\begin{equation}\label{turkishgold}
f: R' + (c, d)(K_{[x]}) + L' \rightarrow R + (a, b)(K_{[x]}) + L. 
\end{equation}

By Corollary \ref{simDclassesinvarunderconvfcoro}, $f$ sends $\sim_{\mathcal{D}}$-classes to $\sim_{\mathcal{D}}$-classes. As $\mathcal{D}$ is non-splitting, by Theorem \ref{simdequalssimDfornonsplittinsystems} these classes are precisely the replacement orders $K_x$ and the segments $R, R', L, L'$ in the expression above. Since $R'$ is the initial $\sim_{\mathcal{D}}$-class in the domain of $f$, it must map onto the initial class $R$ in its image (and one of $R, R'$ is empty if and only if the other is as well). Likewise $f[L'] = L$, and it follows that $f$ restricts to an isomorphism
\[
f: (c, d)(K_{[x]}) \rightarrow (a, b)(K_{[x]}). 
\]
Identify $f$ with this restriction, and label $(c, d)(K_{[x]}) = [l, r]$ by its endcuts. Then $f$ is a convex self-embedding of $[l, r]$ (i.e., compression) with an initial $\omega$-orbital $O_f(l)$ whose initial jump $A_{l, f}$ equals $(c, a](K_{[x]})$. 

Label this jump $C = (c, a](K_{[x]})$. Then $O_f(l) \cong \omega C$. Since $\ell(C) = a - c > 0$, there is $k \in \omega$ such that $a_k < \ell(C)$. By Lemma \ref{ellIlessellJmeansIconvJforsuffrepns}, $A_k \leqslant_{conv} C$. Since every non-empty final segment of $\omega C$ embeds a convex copy of  $C$, it follows $A_k$ convexly embeds in every non-empty final segment of $O_f(l)$. 

As $O_f(l)$ is bounded in $\mathbb{R}(K_{[x]})$ (specifically, it is bounded above by $r$ which belongs to the interval $(a, b)(K_{[x]})$), we may find a non-empty final segment $O'$ of $O_f(l)$ such that $\ell(O') < a_{k+2}$. By Lemma \ref{ellIlessellJmeansIconvJforsuffrepns}, we have $O' \leqslant_{conv} A_{k+2}$. With the above, we get the chain
\[
A_k \leqslant_{conv} O' \leqslant_{conv} A_{k+2},
\]
which gives $A_k \leqslant_{conv} A_{k+2}$. Since $2A_{k+2} \leqslant_{conv} A_k$ (by Lemma \ref{splittinglemmadivissystems} when $\mathcal{D}$ is one-sided and by definition when $\mathcal{D}$ is rational), this implies $A_{k+2}$ is splitting, contradicting that $\mathcal{D}$ is non-splitting. 

If $0 < \ell(I) < \ell(J) = \infty$, then a similar argument works taking one or both of $c = -\infty$ and $d = \infty$. 

If $\ell(I) = 0 < \ell(J)$, then $I \subseteq K_x$ for some $x$. Then any isomorphism from $J$ onto $I$ embeds a copy of some $A_k$ into $I$. By splitting $A_k$ into a sum of terms if need be, we may assume this copy is bounded below by some $p \in K_x$ and above by some $q \in K_x$, contradicting $p \sim_{\mathcal{D}} q$. 
\end{proof}

\theoremstyle{definition}
\newtheorem{lengthforembeddedIforDnonsplittingdefn}[lct]{Definition}
\begin{lengthforembeddedIforDnonsplittingdefn}\label{lengthforembeddedIforDnonsplittingdefn}
Suppose $\mathcal{D}$ is non-splitting, $\mathbb{R}(K_{[x]})$ is sufficient, and $I$ is a linear order such that $I \leqslant_{conv} \mathbb{R}(K_{[x]})$. Define $\ell(I)$ as $\ell(f[I])$, where $f$ is a fixed convex embedding $f: I \rightarrow \mathbb{R}(K_{[x]})$.
\end{lengthforembeddedIforDnonsplittingdefn}

By Proposition \ref{lengthisinvariantwhenDnonsplitting},  $\ell(I)$ is well-defined in Definition \ref{lengthforembeddedIforDnonsplittingdefn}, as the value $\ell(f[I])$ does not depend on the particular embedding $f$. 

\theoremstyle{definition}
\newtheorem{ellIisadditivewhenDnonsplittingpropn}[lct]{Proposition}
\begin{ellIisadditivewhenDnonsplittingpropn}\label{ellIisadditivewhenDnonsplittingpropn}
Suppose $\mathcal{D}$ is non-splitting, $\mathbb{R}(K_{[x]})$ is sufficient, $I$ is a linear order such that $I \leqslant_{conv} \mathbb{R}(K_{[x]})$, and $I \cong I_0 + I_1$. Then $\ell(I) = \ell(I_0) + \ell(I_1)$. 
\end{ellIisadditivewhenDnonsplittingpropn}

\begin{proof}
Label $I = [l, r]$ by its endcuts, and let $l < m < r$ be a cut decomposition witnessing $I \cong I_0 + I_1$. Fix a convex embedding $f: I \rightarrow \mathbb{R}(K_{[x]})$. Then $f(l) \in K_a$, $f(m) \in K_b$, and $f(r) \in K_c$ for some $a, b, c \in \mathbb{R}$ with $a \leq b \leq c$. Then
\[
\ell(I) = \ell(f[I]) = c - a = (c - b) + (b - a) = \ell(f[I_0]) + \ell(f[I_1]) = \ell(I_0) + \ell(I_1),
\]
as claimed. 
\end{proof}

As before, we identify the additive group of real numbers $(\mathbb{R}, +)$ with the subgroup of $\textrm{Aut}(\mathbb{R}, <)$ consisting of translations, and call subgroups of $(\mathbb{R}, +)$ \textit{groups of translations}.

An automorphism $f \in \textrm{Aut}(\mathbb{R})$ is called \textit{irreducible} if $f$ has no fixed points. Equivalently, since any bounded orbital in $\mathbb{R}$ approaches a fixed point on the side it is bounded (by the Dedekind completeness of $\mathbb{R}$), $f$ is irreducible if $\mathbb{R}$ consists of a single $\mathbb{Z}$-orbital of $f$, i.e. $O_f(x) = \mathbb{R}$ for all points or cuts $x \in \mathbb{R}$ (excluding endcuts, which are fixed). The identity map is also irreducible, by convention.

For the next two propositions, let $\varphi$ denote the group homomorphism from Lemma \ref{ftofhatishomo}: 
\[
\begin{array}{rcl}
\varphi: \textrm{Aut}(\mathbb{R}(K_{[x]})) & \rightarrow & \textrm{Aut}(\mathbb{R}) \\
\varphi(f) & = & \hat{f}.
\end{array}
\]

Both of the propositions below are rigidity theorems: they are each expressions of the fact that for non-splitting systems $\mathcal{D}$, automorphisms of $\mathbb{R}(K_{[x]})$ cannot compress intervals (except infinitesimally). 

\theoremstyle{definition}
\newtheorem{fhatalwaysirredwhenDnonsplitting}[lct]{Proposition}
\begin{fhatalwaysirredwhenDnonsplitting}\label{fhatalwaysirredwhenDnonsplitting}
Suppose $\mathcal{D}$ is non-splitting and $\mathbb{R}(K_{[x]})$ is sufficient. Then for every $f \in \textrm{Aut}(\mathbb{R}(K_{[x]})$, the condensed automorphism $\hat{f} \in \textrm{Aut}(\mathbb{R})$ is irreducible. 
\end{fhatalwaysirredwhenDnonsplitting}

\begin{proof}
Suppose $\hat{f}$ is not irreducible. Then in particular $f$ is not the identity. Let $c$ be a cut in $\mathbb{R}$ not fixed by $\hat{f}$. Since $\hat{f}$ is an automorphism, $O_{\hat{f}}(c)$ is a $\mathbb{Z}$-orbital, which must be bounded on at least one side in $\mathbb{R}$ since $\hat{f}$ is not irreducible. 

Suppose without loss of generality $O_{\hat{f}}(c)$ is bounded on the right, and that ${\hat{f}}$ is increasing on $O_{\hat{f}}(c)$. Let $m$ denote the right endcut of $O_{\hat{f}}(c)$. Then $[c, m]$ is isomorphic to $\omega A_{c, \hat{f}}$. 

Let $l$ denote the length of the interval $A_{c, \hat{f}}$ in $\mathbb{R}$. As the lengths of the intervals $\hat{f}^n[A_{c, \hat{f}}]$ go to $0$ as $\hat{f}^n(c)$ approaches $m$, there is some $n$ such that $l$ is strictly greater than the length $l'$ of $\hat{f}^n[A_{c, \hat{f}}]$.

Viewing $c$ also as a cut in $\mathbb{R}(K_{[x]})$, we have
\[
A_{c, f} = A_{c, \hat{f}}(K_{[x]}).
\]
Hence $\ell(A_{c, f}) = l$, and likewise $\ell(f^n[A_{c, f}]) = l' < l$. But $A_{c, f} \cong f^n[A_{c, f}]$, contradicting Proposition \ref{lengthisinvariantwhenDnonsplitting}. 
\end{proof}

The H\"older-Conrad theorem (cf. \cite[Chapter 3]{DNR}) states that groups of irreducible order-automorphisms are isomorphic to subgroups $H \leq (\mathbb{R}, +)$. The following proposition is a rigidity theorem in the same spirit, deduced directly from our arithmetic theory for $(LO, +)$. The fact that the group $H$ appearing in the proposition is strict in $(\mathbb{R}, +)$ will feature in our classification of $LO$-representable commutative semigroups in Section \ref{section:commutesemigroupsinLOrepn}.

\theoremstyle{definition}
\newtheorem{autosofRKxaretranslationsDnonsplitting}[lct]{Proposition}
\begin{autosofRKxaretranslationsDnonsplitting}\label{autosofRKxaretranslationsDnonsplitting}
Suppose $\mathcal{D}$ is non-splitting and $\mathbb{R}(K_{[x]})$ is sufficient. Let 
\[
H = \varphi[\textrm{Aut}(\mathbb{R}(K_{[x]})]
\]
be the image in $\textrm{Aut}(\mathbb{R})$ of the homomorphism $\varphi$, so that $H \cong \textrm{Aut}(\mathbb{R}(K_{[x]})/\textrm{ker}(\varphi)$. Then $H$ is a strict subgroup of $(\mathbb{R}, +)$. 
\end{autosofRKxaretranslationsDnonsplitting}

\begin{proof}
We first check that $\hat{f}$ is a translation for every $f \in \textrm{Aut}(\mathbb{R}(K_{[x]}))$. By Lemma \ref{fhatalwaysirredwhenDnonsplitting}, $\hat{f}$ is irreducible. If $\hat{f}$ is the identity, we are done, so suppose $\hat{f}$ is not the identity. Since $\mathbb{R}$ consists of a single $\mathbb{Z}$-orbital of $\hat{f}$, $\hat{f}$ is either increasing or decreasing on $\mathbb{R}$. Suppose without loss of generality it is increasing. 

Suppose toward a contradiction that $\hat{f}$ is not a translation. Then we can find cuts $c, c' \in \mathbb{R}$ and positive reals $r, r' \in \mathbb{R}_{>0}$ such that $\hat{f}(c) = c + r$, $\hat{f}(c') = c' + r'$, and $r < r'$. 

Since $O_{\hat{f}}(c) = O_{\hat{f}}(c') = \mathbb{R}$ there is a unique $N \in \mathbb{Z}$ such that 
\[
c' \leq \hat{f}^N(c) < \hat{f}(c') \leq \hat{f}^{N+1}(c).
\]
Let $\hat{B}' = [c', \hat{f}^N(c)]$ and $\hat{C}' = [\hat{f}^N(c), \hat{f}(c')]$. Thus the interval
\[
[c', \hat{f}(c')] = [c', c' + r']
\]
decomposes into the initial segment $\hat{B}'$ and final segment $\hat{C}'$, and the interval
\[
[c, \hat{f}(c)] = [c, c + r]
\]
decomposes into the initial segment $\hat{C} = \hat{f}^{-N}[\hat{C}']$ and final segment $\hat{B} = \hat{f}^{-N+1}[\hat{B}']$. Since $r' < r$, either the length of $\hat{B}'$ is less than that of $\hat{B}$ or the length of $\hat{C}'$ is less than that of $\hat{C}$. Suppose the former, and let $B = \hat{B}(K_{[x]})$ and $B' = \hat{B}'(K_{[x]})$ be the corresponding intervals in $\mathbb{R}(K_{[x]})$. Then $\ell(B) > \ell(B')$ and also $B \cong B'$ (as witnessed by $f^{N-1}$), contradicting Lemma \ref{lengthisinvariantwhenDnonsplitting}. The case when the lengths of $\hat{C}'$ and $\hat{C}$ differ is similar. 

Thus $\hat{f}$ is a translation. Since $\hat{f}$ was arbitrary, $H$ is a subgroup of $(\mathbb{R}, +)$. 

It follows that $\mathbb{R}(K_{[x]})$ is in fact a replacement up to the orbit equivalence relation $E_H$ of $H$, i.e. $K_x \cong K_{h(x)}$ for all $x \in \mathbb{R}$ and $h \in H$: indeed, we have $h = \hat{f}$ for some $f \in \textrm{Aut}(\mathbb{R}(K_{[x]})$, and $K_x \cong f[K_x] = K_{\hat{f}(x)} = K_{h(x)}$. 

If $H = (\mathbb{R}, +)$, then it follows $K_x \cong K_y$ for every $x, y \in \mathbb{R}$. Letting $K$ be a fixed order of the common order type of the orders $K_x$, we have $\mathbb{R}(K_{[x]})$ is (isomorphic to) the lexicographic product $\mathbb{R}K$. Moreover, $A_0$ is of the form
\[
R + (0, 1)K + L, 
\]
for some $R$, $L$ such that $L + R \cong K$. But then 
\[
A_0 \cong R + (0, \frac{1}{2})K + L + R + (\frac{1}{2}, 1)K + L.
\]
Since the interval $(0, \frac{1}{2})$ and $(\frac{1}{2}, 1)$ are each isomorphic to $(0, 1)$ it follows $(0, \frac{1}{2})K$ and $(\frac{1}{2}, 1)K$ are each isomorphic to $(0, 1)K$. Hence the isomorphism above witnesses $A_0 \cong A_0 + A_0$, contradicting that $A_0$ is non-splitting. Thus $H$ is strict in $(\mathbb{R}, +)$, as claimed. 
\end{proof}

For most of the results of this section, we assume sufficiency of the real representation $\mathbb{R}(K_{[x]})$ as a hypothesis. However, in Section \ref{section:continuationandvaluation} we will need to \textit{prove} the sufficiency of the representation for a certain non-splitting system $\mathcal{D}$ by way of a length invariance argument. In order to do this, in Proposition \ref{ellIlessellJnonisoinsufficientsetting} below we establish a weak version of Proposition \ref{lengthisinvariantwhenDnonsplitting} in which we do not assume sufficiency of the representation, but assume instead that the intervals $I$ and $J$ satisfy $I \subseteq J$. In order to adapt the proof, we first show the following weak analogue of Proposition \ref{ellIlessellJmeansIconvJforsuffrepns}, which again works in the absence of sufficiency. 

\theoremstyle{definition}
\newtheorem{embeddingAkwheninsufficientlemma}[lct]{Lemma}
\begin{embeddingAkwheninsufficientlemma}\label{embeddingAkwheninsufficientlemma}
Let $\mathbb{R}(K_{[x]})$ be the real representation of $\mathcal{D}$. Fix $x \in \mathbb{R}$. 
\begin{itemize}
    \item[i.] For all $\varepsilon > 0$, there exists an integer $k \geq 0$ such that 
    \[
    A_k \leqslant_{conv} [x, x + \varepsilon](K_{[x]}).
    \]
    \item[ii.] For every integer $k \geq 0$, there exists $\varepsilon > 0$ such that 
    \[
    [x, x + \varepsilon](K_{[x]}) \leqslant_{conv} A_k.
    \]
\end{itemize}
\end{embeddingAkwheninsufficientlemma}

\begin{proof}
For (1.): Identifying $K_x$ with the corresponding interval in the symbolic representation $\mathbb{Z}A_0$, we have that $K_x$ is either isomorphic to (if $x \not \in [0]$), or has a final segment isomorphic to (if $x \in [0]$), some $I_{(n, u)} \cong I_u$ obtained as an intersection of the leftward branching
\[
A_{u_0} \supseteq A_{u_0 u_1} \supseteq \cdots \supseteq A_{u \upharpoonright n} \supseteq \cdots. 
\]

Thus for any interval $I \subseteq \mathbb{R}(K_{[x]})$ extending $K_x$ to the right, by the leftwardness of the branching there exists $k$ such that $A_{u \upharpoonright k} \subseteq I$. Passing to $k+1$ if need be, we may assume $u \upharpoonright k$ is a regular node, and hence $A_{u \upharpoonright k} \cong A_k$. Thus $A_k \leqslant_{conv} I$, and (i.) follows. 

For (ii.): Since the branching is leftward, $A_{u \upharpoonright k}$ strictly extends $K_x$ to the right. There must exist some $\varepsilon > 0$ such that $I = [x, x+\varepsilon](K_{[x]}) \subseteq A_{u \upharpoonright k}$. Thus $I \leqslant_{conv} A_{u \upharpoonright k}$. We have that $A_{u \upharpoonright k}$ is either isomorphic to $A_k$ or $A_{k+1}$, depending on the regularity of $u \upharpoonright k$, but since $A_{k+1} \leqslant_{conv} A_k$, in either case we have $I \leqslant_{conv} A_k$, which gives (ii.).
\end{proof}

By symmetric arguments, the symmetrized versions of (i.) and (ii.), obtained by replacing $[x, x+\varepsilon](K_{[x]})$ with $[x - \varepsilon, x](K_{[x]})$, also hold. 

\theoremstyle{definition}
\newtheorem{ellIlessellJnonisoinsufficientsetting}[lct]{Proposition}
\begin{ellIlessellJnonisoinsufficientsetting}\label{ellIlessellJnonisoinsufficientsetting}
Suppose $\mathcal{D}$ is non-splitting and we have intervals $I \subseteq J \subseteq \mathbb{R}(K_{[x]})$ such that $\ell(I) < \ell(J)$. Then $I \not\cong J$. 
\end{ellIlessellJnonisoinsufficientsetting}

\begin{proof}
Follow the proof of Proposition \ref{lengthisinvariantwhenDnonsplitting}. Assume first as in that proof that $I$ and $J$ have the forms
\[
\begin{array}{rcl}
I & \cong & R + (a, b)(K_{[x]}) + L \\
J & \cong & R' + (c, d)(K_{[x]}) + L'.
\end{array}
\]

Since $I \subseteq J$ by assumption, we do not need to invoke the sufficiency of the representation to get a translation sending $I$ into $J$. Thus we have $c \leq a < b \leq d$. Since $\ell(I) < \ell(J)$, at least one of $c < a$ and $b < d$ holds. By symmetry we may assume $c < a$. Since $I \cong J$ we have an isomorphism $f: J \rightarrow I$ as in the proof of \ref{lengthisinvariantwhenDnonsplitting}; we follow the proof from there. 

Now, instead of invoking Lemma \ref{ellIlessellJmeansIconvJforsuffrepns} to get $A_k \leqslant_{conv} C$, observe that since $\ell(C) > 0$, $C$ it contains an interval of the form $[x, x + \varepsilon](K_{[x]})$ for some $\varepsilon > 0$, and hence convexly embeds some $A_k$ by Lemma \ref{embeddingAkwheninsufficientlemma}. Recall that $C$ embeds in every non-empty final segment of $O_f(l)$. 

To get around the second use of Lemma \ref{ellIlessellJmeansIconvJforsuffrepns}, we argue as follows. Since the orbital $O_f(l)$ is bounded in $\mathbb{R}(K_{[x]})$ and has positive length, it has the form $R + (y, x)(K_{[x]}) + L$ by \ref{ellforintervalspropn} for some $y < x$ in $\mathbb{R}$. 

Choose $\varepsilon > 0$ sufficiently small so that $x - \varepsilon > y$ and $[x - \varepsilon, x](K_{[x]}) \leqslant_{conv} A_{k+2}$, which is possible by the (symmetrized version of) Lemma \ref{embeddingAkwheninsufficientlemma}. Since $[x - \varepsilon, x](K_{[x]})$ contains a final segment of $O_f(l)$, and all final segments of $O_f(l)$ convexly embed $C$ and hence $A_k$, we have the chain
\[
A_k \leqslant_{conv} [x - \varepsilon, x](K_{[x]}) \leqslant_{conv} A_{k+2},
\]
which implies that $\mathcal{D}$ is splitting, a contradiction. The cases when $\ell(I) = 0$ and $\ell(J) = \infty$ can be handled by similar adaptations of the proof of \ref{lengthisinvariantwhenDnonsplitting}.
\end{proof}

\section{Generalizing Aronszajn's representation theorem}\label{section:aronszajnrepnthm} 

In this section we use the results of Section \ref{section:symboldynamrepnsofdivissystems} to deduce Aronszajn's representation theorem from \cite{Aronszajn} for commuting pairs $A + B \cong B + A$ in $LO$, as well as two generalizations. As with several of our other results, these generalizations are one-sided versions of a previously established theorem. 

This sidedness is reflected in differences in the proofs: whereas Aronszajn proved his representation theorem using a natural automorphism of $\mathbb{Z}(A+B)$ induced by the identity $A + B \cong B + A$, our results rely on the division systems yielded by the one-sided Euclidean algorithms from Section \ref{section:euclideanalgos} for dividing one of $A$ and $B$ by the other. As discussed there, runs of such algorithms yield identities between the infinite discrete products of $A$ and $B$. However, they do not in general imply the existence of the kind of automorphism used in Aronszajn's proof. 

\theoremstyle{definition}
\newtheorem{leftrightsymmdividesdefn}[lct]{Definition}
\begin{leftrightsymmdividesdefn}\label{leftrightsymmdividesdefn}
\phantom{.}
\begin{itemize}
    \item[i.] $B$ \textit{left divides} $A$ if either $\omega B \leqslant_{init} A$, or there is a left division system $\{A_k \cong N_k A_{k+1} + A_{k+2}: k < M\}$ of some length $M \leq \omega$ such that $A_0 \cong A$ and $A_1 \cong B$.
    \item[ii.] $B$ \textit{right divides} $A$ if either $\omega^* B \leqslant_{fin} A$, or there is a right division system $\{A_k \cong A_{k+2} + N_k A_{k+1}: k < M\}$ of some length $M \leq \omega$ such that $A_0 \cong A$ and $A_1 \cong B$.
    \item[iii.] $B$ \textit{symmetrically divides} $A$ if either $\omega B \leqslant_{init} A$ and $\omega^* B \leqslant_{fin} A$, or there is a symmetric division system $\{A_k \cong N_k A_{k+1} + A_{k+2}: k < M\}$ of some length $M \leq \omega$ such that $A_0 \cong A$ and $A_1 \cong B$.
\end{itemize}
\end{leftrightsymmdividesdefn}

In Definition \ref{leftrightsymmdividesdefn}, we say that $B$ is a \textit{left divisor}, \textit{right divisor}, and \textit{symmetric divisor} of $A$, respectively. 

As in Section \ref{section:symboldynamrepnsofdivissystems}, if $G = \langle 1, a_1 \rangle$ is the group of translations of generated by $x \mapsto x + 1$ and $x \mapsto x + a_1$, we write $E_G$ for the orbit equivalence relation of $G$ and $E_G^{\pm}$ for the equivalence relation obtained from $E_G$ by splitting $[0]_{E_G}$ into its forward and backward orbits, switching conventions for these orbits when we view $G$ as $\langle -1, -a_1 \rangle$ (i.e., when working with right division systems) as in Section \ref{subsect:rightsystemrealrepn}. 

Here are our one-sided generalizations of Aronszajn's representation. In the interest of generality, we state these theorems assuming the hypothesis of a one-sided division system, but we have in mind the relations $B + A \leqslant_{init} A + B$ and $B + A \leqslant_{fin} A + B$ (which imply the existence of corresponding one-sided division systems) as the arithmetic hypotheses analogous to the hypothesis $A + B \cong B + A$ in Aronszajn's theorem. 

\theoremstyle{definition}
\newtheorem{leftaronszajnrepnthm}[lct]{Theorem}
\begin{leftaronszajnrepnthm}\label{leftaronszajnrepnthm}
Suppose that $B$ left divides $A$. One of the following holds:
\begin{itemize}
    \item[i.] $B + A \cong A$.
    \item[ii.] There are orders $C$ and $D$ and natural numbers $n, m \geq 1$ such that $A \cong nC$, $B \cong mC + D$, and $D + C \cong C$. 
    \item[iii.] There are orders $C$ and $D$ and natural numbers $n, m \geq 1$ such that $A \cong nC + D$, $B \cong mC$, and $D + C \cong C$. 
    \item[iv.] There is an irrational number $a_1 \in (0, 1)$, a replacement $\mathbb{R}(K_{[x]})$ up to the equivalence relation $E_G^{\pm}$ from the group of translations $G$ generated by $x \mapsto x + 1$ and $x \mapsto x + a_1$, and decompositions
    \[
    \begin{array}{rcl}
    K_0 & \cong & L + R \\
    K_{a_1} & \cong & L' + R
    \end{array}
    \]
    such that 
    \[
    \begin{array}{rcl}
    A & \cong & R + (0, 1)(K_{[x]}) + L \\
    B & \cong & R + (0, a_1)(K_{[x]}) + L'.
    \end{array}
    \]
\end{itemize}
\end{leftaronszajnrepnthm}

\begin{proof}
If $\omega B \leqslant_{init} A$, then $B + A \cong A$, i.e. (i.) holds. Otherwise, we have the left division system $\{A_k: k < M\}$.

If $M < \omega$, then by back substitution from the last isomorphism in the system (see the discussion in Section \ref{subsect:leftalgorun}), either (ii.) or (iii.) holds, depending on the parity of $M$, with $C = A_M$ and $D = A_{M+1}$ (which is possibly empty).  

Finally, if $M = \omega$, then (iv.) holds by Theorem \ref{leftrealrepnthm}.
\end{proof}

\theoremstyle{definition}
\newtheorem{rightaronszajnrepnthm}[lct]{Theorem}
\begin{rightaronszajnrepnthm}\label{rightaronszajnrepnthm}
Suppose that $B$ right divides $A$. One of the following holds:
\begin{itemize}
    \item[i.] $A + B \cong A$.
    \item[ii.] There are orders $C$ and $D$ and natural numbers $n, m \geq 1$ such that $A \cong nC$, $B \cong D + mC$, and $C + D \cong C$. 
    \item[iii.] There are orders $C$ and $D$ and natural numbers $n, m \geq 1$ such that $A \cong D + nC$, $B \cong mC$, and $C + D \cong C$. 
    \item[iv.] There is an irrational number $a_1 \in (0, 1)$, a replacement $\mathbb{R}(K_{[x]})$ up to the equivalence relation $E_G^{\pm}$ from the group of translations $G$ generated by $x \mapsto x + 1$ and $x \mapsto x - a_1$, and decompositions
    \[
    \begin{array}{rcl}
    K_0 & \cong & L + R \\
    K_{-a_1} & \cong & L + R'
    \end{array}
    \]
    such that 
    \[
    \begin{array}{rcl}
    A & \cong & R + (-1, 0)(K_{[x]}) + L \\
    B & \cong & R' + (-a_1, 0)(K_{[x]}) + L.
    \end{array}
    \]
\end{itemize}
\end{rightaronszajnrepnthm}

\begin{proof}
Symmetric. 
\end{proof}

We conclude with a proof Aronszajn's representation theorem, stated below in a slightly reformulated way to fit the language of this paper. The proof here differs from the original: in particular, it cases out on whether the division system in question is splitting or non-splitting, so as to leverage the rigidity theorems from Section \ref{subsubsect:lengthinvarnonsplittinsystems}. 

\theoremstyle{definition}
\newtheorem{symmetricaronszajnrepnthm}[lct]{Theorem}
\begin{symmetricaronszajnrepnthm}\label{symmetricaronszajnrepnthm}
Suppose that $B$ symmetrically divides $A$. One of the following holds:
\begin{itemize}
    \item[i.] $A + B \cong B + A \cong A$.
    \item[ii.] There is an order $C$ and natural numbers $n, m \geq 1$ such that $A \cong nC$ and $B \cong mC$. 
    \item[iii.] There is an irrational number $a_1 \in (0, 1)$, a replacement $\mathbb{R}(K_{[x]})$ up to the orbit equivalence relation $E_G$ of the group of translations $G$ generated by $x \mapsto x + 1$ and $x \mapsto x + a_1$, and decompositions
    \[
    \begin{array}{rcl}
    K_0 & \cong & L + R \\
    K_{a_1} & \cong & L + R
    \end{array}
    \]
    such that 
    \[
    \begin{array}{rcl}
    A & \cong & R + (0, 1)(K_{[x]}) + L \\
    B & \cong & R + (0, a_1)(K_{[x]}) + L.
    \end{array}
    \]
\end{itemize}
\end{symmetricaronszajnrepnthm}

\begin{proof}
If $\omega B \leqslant_{init} A$ and $\omega^* B \leqslant_{conv} A$ then $B + A \cong A + B \cong A$, i.e. (i.) holds. Otherwise, we have the symmetric division system $\{A_k: k < M\}$.

If $M < \omega$ we may assume the system terminates in a divisor. Then by back substitution, we have (ii.) with $C = A_M$. 

So suppose $M = \omega$. Consider first the case when the system is splitting. Then by Proposition \ref{isosoftermsinsplittingsystems} we have $A \cong B$. Thus 
\[
A + B \cong B + A \cong A + A \cong A,
\]
so we again have (i.). 

Now suppose the system is non-splitting. Then we have (iii.) by Theorem \ref{symmetricrealrepnthm}. 
\end{proof}

The resemblance between conditions (i.) and (ii.) in Theorem \ref{symmetricaronszajnrepnthm} and the corresponding conditions in Tarski's Conjecture \ref{tarconj} indicates a real connection. Indeed, we will show in Section \ref{section:revisedtarskiconj} below that $A + B \cong B + A$ holds if and only if one of $A$ and $B$ symmetrically divides the other.

Thus Tarski's Conjecture amounts to the assertion: one of $A$ and $B$ symmetrically divides the other if and only if (the symmetrized version of) condition (i.) holds, or (ii.) holds. But there is a third possibility, namely condition (iii.), and there exist examples of pairs satisfying condition (iii.) but neither (i.) nor (ii.), as Lindenbaum showed; see Aronszajn's discussion in \cite{Aronszajn}. 

\section{Tarski's conjecture}\label{section:revisedtarskiconj}

In this section we prove the revised version \ref{revisedtarconj} of Tarski's Conjecture \ref{tarconj}, which gives an arithmetic characterization of the additively commuting pairs $A, B \in LO$ (Theorem \ref{revisedtarconjintext} below). That is to say, this characterization is written in terms of (infinite) discrete sums of the orders $A$ and $B$, and not in terms of replacements of $\mathbb{R}$ like the representation theorems from Section \ref{section:aronszajnrepnthm}. In the companion paper \cite{ErvinPaul}, we show that this characterization holds in an arbitrary ordinal algebra. 

We also prove the closely related ``symmetrized" characterization stated first as Theorem \ref{ABcommuteiffomegaAomegastarAinitfinomegaBomegastarBorvv} (Theorem \ref{ABcommuteifOmegasumsinitfin}). Finally, we justify the discussion at the end of Section \ref{section:aronszajnrepnthm}, by observing that $A + B \cong B + A$ if and only if one of $A, B$ symmetrically divides the other. All of the proofs follow quickly from our previous work. 

\theoremstyle{definition}
\newtheorem{ABcommuteifOmegasumsinitfin}[lct]{Theorem}
\begin{ABcommuteifOmegasumsinitfin}\label{ABcommuteifOmegasumsinitfin}
$A + B \cong B + A$ if and only if one of the following conditions holds:
\begin{itemize}
    \item[i.] $\omega A \leqslant_{init} \omega B$ and $\omega^* A \leqslant_{fin} \omega^*B$; 
    \item[ii.] $\omega B \leqslant_{init} \omega A$ and $\omega^* B \leqslant_{fin} \omega^* A$. 
\end{itemize}
\end{ABcommuteifOmegasumsinitfin}

\begin{proof}
If either (i.) or (ii.) holds, then $A + B \cong B + A$ by Corollary \ref{omegaAomegastarAinitfinomegaBomegastarB}.(iii.). 

Conversely, suppose $A + B \cong B + A$. Let $(I; \rho, \tau)$ be a symmetric setup witnessing the isomorphism, and run the Euclidean algorithm on this setup. Assume first that $A$ is the dividend and $B$ the divisor for this setup. 

If the algorithm terminates at Stage $0$ in an absorbed factor, then $B + A \cong A + B \cong A$. This gives $\omega B \leqslant_{init} A$ and $\omega^* B \leqslant_{fin} A$, which certainly implies $\omega B \leqslant_{init} \omega A$ and $\omega^* B \leqslant_{fin} \omega^*A$, i.e. (ii.) holds. Otherwise the algorithm yields a symmetric division system $\{A_k\}$ of some length $M \leq \omega$. 

If $M < \omega$, then we have $\omega A \cong \omega B$ and $\omega^* A \cong \omega^* B$ by the discussion in Section \ref{subsect:runofsymmetricalgo}, which also implies (ii.).

If $M = \omega$, then the system $\{A_k\}$ is non-terminating. If it is splitting, we have $A \cong B$ by Proposition \ref{isosoftermsinsplittingsystems}, which implies (ii.). If it is non-splitting, then we have $\omega A \cong \omega B$ and $\omega^* A \cong \omega^* B$ by Proposition \ref{Isosofomegaprodsfornonsplsymmsystems}, which again gives (ii.). 

Thus in all cases we have (ii.). If instead $B$ is the dividend and $A$ the divisor of the setup, then in all cases we have (i.). 
\end{proof}

\theoremstyle{definition}
\newtheorem{revisedtarconjintext}[lct]{Theorem}
\begin{revisedtarconjintext}\label{revisedtarconjintext}
$A + B \cong B + A$ if and only if one of the following conditions holds:
\begin{itemize}
    \item[i.] $A + B \cong B + A \cong A$;
    \item[ii.] $A + B \cong B + A \cong B$;
    \item[iii.] $\omega A \cong \omega B$ and $\omega^*A \cong \omega^*B$. 
\end{itemize}
\end{revisedtarconjintext}

\begin{proof}
The proof Theorem \ref{ABcommuteifOmegasumsinitfin} shows in fact that if $A + B \cong B + A$, then one of (i.), (ii.), or (iii.) holds. 

Conversely, (i.) and (ii.) immediately imply $A + B \cong B + A$. If (iii.) holds, then $A + B \cong B + A$ follows from Corollary \ref{omegaAomegastarAinitfinomegaBomegastarB}.(iii.).
\end{proof}

\theoremstyle{definition}
\newtheorem{ABcommuteiffonesymmdividesother}[lct]{Theorem}
\begin{ABcommuteiffonesymmdividesother}\label{ABcommuteiffonesymmdividesother}
$A + B \cong B + A$ if and only if one of $A$ and $B$ symmetrically divides the other. 
\end{ABcommuteiffonesymmdividesother}
\begin{proof}
The proof of Theorem \ref{ABcommuteifOmegasumsinitfin} shows that if $A + B \cong B + A$, then either $A + B \cong B + A \cong A$ (in which case $B$ symmetrically divides $A$), or $B + A \cong A + B \cong B$ (so $A$ symmetrically divides $B$), or there is a symmetric division system with initial terms $A$ and $B$ (in which case one of $A$ and $B$ symmetrically divides the other). 

Conversely, suppose that $B$ symmetrically divides $A$. Then either $A + B \cong B + A \cong A$, or there is a symmetric division system with initial terms $A_0 = A$ and $A_1 = B$. In this second case, again by the proof of Theorem \ref{ABcommuteifOmegasumsinitfin}, we have $A + B \cong B + A$. And likewise, if instead $A$ symmetrically divides $B$.  
\end{proof}

\section{Continuation and valuation}\label{section:continuationandvaluation}

In this section we introduce and study several archimedean dominance relations in $(LO, +)$, as well as the corresponding archimedean equivalence relations. Like the notions of divisibility considered in previous sections, these relations have both one-sided and symmetric versions. We then investigate the relationship between these relations and commutativity in $(LO, +)$, and show that the symmetric relations induce a global valuation on $(LO, +)$ akin to valuations of ordered groups and ordered fields in terms of their archimedean classes. 

This valuation structure is reflected in our classification of the commutative semigroups representable in $(LO, +)$ in Section \ref{section:commutesemigroupsinLOrepn}. It can be viewed as an abstract generalization of McCleary's structure theory for group actions $G \curvearrowright X$ by lattice-ordered groups $G$ on linear orders $X$ from \cite{McCleary}.

\subsection{Archimedean orderings on $LO$}\label{subsect:archimedeanorderingsonLO} \phantom{.}

\theoremstyle{definition}
\newtheorem{lesssimrelnsdefn}[lct]{Definition}
\begin{lesssimrelnsdefn}\label{lesssimrelnsdefn}
Define binary relations $\lesssim_{init}$, $\lesssim_{fin}$, and $\lesssim_{conv}$ on $LO$ as follows:
\begin{itemize}
    \item[i.] $A \lesssim_{init} B$ if $\omega A \leqslant_{init} \omega B$. 
    \item[ii.] $A \lesssim_{fin} B$ if $\omega^* A \leqslant_{fin} \omega^* B$. 
    \item[iii.] $A \lesssim_{conv} B$ if $\mathbb{Z}A \leqslant_{conv} \mathbb{Z}B$. 
\end{itemize}
Also define a binary relation $\lessapprox$ on $LO$ by the rule:
\[
\textrm{$A \lessapprox B$ if $A \lesssim_{init} B$ and $A \lesssim_{fin} B$.}
\]
\end{lesssimrelnsdefn}

Observe that each of the relations $\lesssim_{init}$, $\lesssim_{fin}$, and $\lesssim_{conv}$ is transitive on $LO$. Hence $\lessapprox$ is also transitive on $LO$. 

In this notation, Theorem \ref{ABcommuteifOmegasumsinitfin} says exactly that $A + B \cong B + A$ if and only if either $A \lessapprox B$ or $B \lessapprox A$. 

\theoremstyle{definition}
\newtheorem{gplusfembeddingdefn}[lct]{Definition}
\begin{gplusfembeddingdefn}\label{gplusfembeddingdefn}
If $g: X \rightarrow Y$ and $f: X' \rightarrow Y'$ are embeddings, the \textit{concatenation} of $g$ and $f$ is the embedding $g + f: X + X' \rightarrow Y + Y'$ satisfying $(g + f) \upharpoonright X = g$ and $(g + f) \upharpoonright X' = f$.
\end{gplusfembeddingdefn}

Observe that if $g$ is a final embedding and $f$ is an initial embedding, then $g + f$ is a convex embedding. 

Suppose $A \lesssim_{init} B$ as witnessed by an initial embedding $f: \omega A \rightarrow \omega B$, and $A \lesssim_{fin} B$ as witnessed by a final embedding $g: \omega^*A \rightarrow \omega^* B$. Then $g + f$ witnesses $A \lesssim_{conv} B$. Thus $A \lessapprox B$ implies $A \lesssim_{conv} B$. (The converse is false in general.)

\theoremstyle{definition}
\newtheorem{ZAcentZBdefn}[lct]{Definition}
\begin{ZAcentZBdefn}\label{ZAcentZBdefn}
Identify $\mathbb{Z}A$ with $\omega^*A + \omega A$ and $\mathbb{Z}B$ with  $\omega^*B + \omega B$. Define
\[
\mathbb{Z}A \leqslant_{cent} \mathbb{Z}B
\]
if there is a convex embedding $F: \mathbb{Z}A \rightarrow \mathbb{Z}B$ such that $F = g + f$ for some final embedding $g: \omega^* A \rightarrow \omega^* B$ and initial embedding $f: \omega A \rightarrow \omega B$. 

We call $F$ a \textit{centered convex embedding} of $\mathbb{Z}A$ into $\mathbb{Z}B$. 
\end{ZAcentZBdefn}

Analogously, we write $\mathbb{Z}A \cong_{cent} \mathbb{Z}B$ if there is a centered isomorphism of $\mathbb{Z}A$ and $\mathbb{Z}B$, i.e. an isomorphism $F: \mathbb{Z}A \rightarrow \mathbb{Z}B$ that can be decomposed as the concatenation $g + f$ of isomorphisms $g: \omega^*A \rightarrow \omega^*B$ and $f: \omega A \rightarrow \omega B$. 

\theoremstyle{definition}
\newtheorem{AlessapproxBiffZAcentZBpropn}[lct]{Proposition}
\begin{AlessapproxBiffZAcentZBpropn}\label{AlessapproxBiffZAcentZBpropn}
$A \lessapprox B$ if and only if $\mathbb{Z}A \leqslant_{cent} \mathbb{Z}B$. 
\end{AlessapproxBiffZAcentZBpropn}
 
\begin{proof}
Immediate from Definition \ref{ZAcentZBdefn} and the preceding discussion.  
\end{proof}

The following proposition illustrates a sense in which the relations $\lesssim_{\star}$ may be viewed as notions of archimedean dominance on $LO$.

\theoremstyle{definition}
\newtheorem{AlessapproxBiffnAlesssomemB}[lct]{Proposition}
\begin{AlessapproxBiffnAlesssomemB}\label{AlessapproxBiffnAlesssomemB}
For $\star \in \{init, fin, conv\}$, if $A \lesssim_{\star} B$ then for every $n \in \omega$, there exists $m \in \omega$ such that $nA \leqslant_{\star} mB$. 
\end{AlessapproxBiffnAlesssomemB}

\begin{proof}
Clear. 
\end{proof}

For $\star \in \{init, fin\}$, the converse to Proposition \ref{AlessapproxBiffnAlesssomemB} is also true; see Proposition \ref{archimequivalencetolesssimstar} below.  

We will use the lemma below to refine the notions $\lesssim_{init}$ and $\lesssim_{fin}$ in the following Proposition \ref{lesssiminitandlesssimfinrefinements}. 

\theoremstyle{definition}
\newtheorem{AinitBinitAandsplimpliesomegasiso}[lct]{Lemma}
\begin{AinitBinitAandsplimpliesomegasiso}\label{AinitBinitAandsplimpliesomegasiso}
Suppose that $A$ and $B$ are splitting. 
\begin{itemize}
    \item[i.] If $A \leqslant_{init} B$ and $B \leqslant_{init} A$, then $\omega A \cong \omega B$.
    \item[ii.] If $A \leqslant_{fin} B$ and $B \leqslant_{fin} A$, then $\omega^*A \cong \omega^*B$. 
    \item[iii.] If $A \leqslant_{conv} B$ and $B \leqslant_{conv} A$, then $\mathbb{Z}A \cong \mathbb{Z}B$. 
\end{itemize}
\end{AinitBinitAandsplimpliesomegasiso}

\begin{proof}
For (i.): Since $A \leqslant_{init} B$ and $A$ is splitting we have $2A \leqslant_{init} B$, and so $B \cong 2A + X$ for some $X$. Hence $2B \cong 2A + X + 2A + X$, and since $B$ is splitting, $B \cong 2A + X + 2A + X$.

Of course $A \leqslant_{init} A + X + A$ and $A \leqslant_{fin} A + X + A$. On the other hand, since 
\[
B \cong 2A + X + 2A + X \cong A + (A + X + A) + A + X
\]
we have $A + X + A \leqslant_{conv} B$. Thus 
\[
A + X + A \leqslant_{conv} B \leqslant_{init} A,
\]
which gives $A + X + A \leqslant_{conv} A$. By \ref{XinitXfinYconv}, we deduce $A \cong A + X + A$. 

Now observe
\[
\begin{array}{rcl}
\omega B & \cong & B + B + B + \cdots \\
& \cong & 2A + X + 2A + X + \cdots \\
& \cong & A + (A + X + A) + (A + X + A) + \cdots \\
& \cong & \omega A, 
\end{array}
\]
as claimed. The proof for (ii.) is symmetric.

For (iii.): a similar argument works. We have $A \leqslant_{conv} B$ and hence $2A \leqslant_{conv} B$, which gives $B \cong X + 2A + Y$ for some $X, Y$. Thus
\[
X + A + A + Y + X + A + A + Y \cong 2B \cong B \leqslant_{conv} A, 
\]
which, reasoning as above, gives $A \cong A + Y + X + A$. Now notice
\[
\mathbb{Z}B \cong \mathbb{Z}(A + Y + X + A),
\]
and we are done. 
\end{proof}

The relations $\lesssim_{init}$ and $\lesssim_{fin}$ admit the following refinements. 

\theoremstyle{definition}
\newtheorem{lesssiminitandlesssimfinrefinements}[lct]{Proposition}
\begin{lesssiminitandlesssimfinrefinements}\label{lesssiminitandlesssimfinrefinements}
The following equivalences hold:
\begin{itemize}
    \item[i.] $A \lesssim_{init} B$ if and only if $\omega A \leqslant_{init} B$ or $\omega A \cong \omega B$;
    \item[ii.] $A \lesssim_{fin} B$ if and only if $\omega^*A \leqslant_{fin} B$ or $\omega^*A \cong \omega^*B$. 
\end{itemize}
\end{lesssiminitandlesssimfinrefinements}

\begin{proof}
For (i.): The backward direction is clear. So suppose $A \lesssim_{init} B$, and in $\omega B$ fix a cut sequence
\[
b_0 < b_1 < b_2 < \ldots 
\]
witnessing the decomposition
\[
\omega B \cong B + B + B + \cdots
\]
i.e., so that $b_0$ is the left endcut of $\omega B$ and $b_i$ is the cut at the $i$th $+$ sign for $i \geq 1$. Also in $\omega B$, fix a cut sequence
\[
b_0 = a_0 < a_1 < a_2 < \ldots 
\]
witnessing $\omega A \leqslant_{init} \omega B$. 

If $a_i < b_1$ for all $i \in \omega$, then the cut sequence above witnesses $\omega A \leqslant_{init} B$. If instead the sequence $a_i$ converges to the right endcut of $\omega B$, then the sequence witnesses $\omega A \cong \omega B$. 

Otherwise, there is an integer $N \geq 1$ such that the right endcut $c$ of the sequence $a_i$ satisfies $b_N < c \leq b_{N+1}$. Let $M \geq 1$ be least such that $b_N \leq a_M$. Then since $[b_0, b_N] \cong NB$, $[b_0, b_{N+1}] \cong (N+1)B$, and $[a_0, a_M] \cong MA$, the sequence $a_i$ witnesses
\[
NB \leqslant_{init} MA \leqslant_{init} \omega A \leqslant_{init} (N+1)B.
\]
By Corollary \ref{directedrefinementforomegaA}, $\omega A \leqslant_{init} (N+1)B$ implies $\omega A \leqslant_{conv} B$. Since $B \leqslant_{conv} NB \leqslant_{conv} MA$, we have $\omega A \leqslant_{conv} MA$. Again by \ref{directedrefinementforomegaA}, this yields $\omega A \leqslant_{conv} A$, and in particular $2A \leqslant_{conv} A$. Thus $A$ is splitting. 

We claim $B$ is also splitting. Let $Q = [c, b_{N+1}]$ be the final segment of $[b_N, b_{N+1}]$ above $\omega A$. Then we have
\[
(N+1)B \cong [b_0, b_{N+1}] \cong [b_0, c] + [c, b_{N+1}] \cong \omega A + Q.
\]
On the other hand, since the tail segment $[a_M, c]$ of $\omega A$ is also isomorphic to $\omega A$, we have
\[
[a_M, b_{N+1}] \cong [a_M, c] + [c, b_{N+1}] \cong \omega A + Q \cong (N+1)B.
\]
Since $[a_M, b_{N+1}]$ is final in $[b_N, b_{N+1}] \cong B$, it follows $(N+1)B \leqslant_{fin} B$. As $N \geq 1$, we deduce $B$ is splitting from \ref{2XconvXiff2XcongX}, as claimed. 

Now $NB \leqslant_{init} MA \leqslant_{init} (N+1)B$ yields $B \leqslant_{init} A \leqslant_{init} B$, so (i.) follows by Lemma \ref{AinitBinitAandsplimpliesomegasiso}, and the argument for (ii.) is symmetric. 
\end{proof}

Since $X + Y \cong Y$ if and only if $\omega X \leqslant_{init} Y$, and $Y + X \cong Y$ if and only if $\omega^*X \leqslant_{fin} Y$, Proposition \ref{lesssiminitandlesssimfinrefinements} can be reformulated as follows.

\theoremstyle{definition}
\newtheorem{lessimninitlesssimfinreformulations}[lct]{Corollary}
\begin{lessimninitlesssimfinreformulations}\label{lessimninitlesssimfinreformulations}
The following equivalences hold:
\begin{itemize}
    \item[i.] $A \lesssim_{init} B$ if and only if $A + B \cong B$ or $\omega A \cong \omega B$;
    \item[ii.] $A \lesssim_{fin} B$ if and only if $B + A \cong B$ or $\omega^*A \cong \omega^* B$. 
\end{itemize}
\end{lessimninitlesssimfinreformulations}

We also get analogs of Proposition \ref{lesssiminitandlesssimfinrefinements} and Corollary \ref{lessimninitlesssimfinreformulations} for $\lessapprox$ by extending the notion of centered convex embedding. 

\theoremstyle{definition}
\newtheorem{ZAcent2Bdefn}[lct]{Definition}
\begin{ZAcent2Bdefn}\label{ZAcent2Bdefn}
Identify $\mathbb{Z}A$ with $\omega^*A + \omega A$ and $2B$ with  $B + B$. Define
\[
\mathbb{Z}A \leqslant_{cent} 2B
\]
if there is a convex embedding $F: \mathbb{Z}A \rightarrow 2B$ such that $F = g + f$ for some final embedding $g: \omega^* A \rightarrow B$ and initial embedding $f: \omega A \rightarrow B$. 

We call $F$ a \textit{centered convex embedding} of $\mathbb{Z}A$ into $2B$. 
\end{ZAcent2Bdefn}

\theoremstyle{definition}
\newtheorem{AlessapproxBreformulations}[lct]{Proposition}
\begin{AlessapproxBreformulations}\label{AlessapproxBreformulations}
The following are equivalent:
\begin{itemize}
    \item[i.] $A \lessapprox B$;
    \item[ii.] $\mathbb{Z}A \leqslant_{cent} 2B$, or $\mathbb{Z}A \cong_{cent} \mathbb{Z}B$; 
    \item[iii.] $A + B \cong B + A \cong B$, or $\omega A \cong \omega B$ and $\omega^*A \cong \omega^*B$. 
\end{itemize}
\end{AlessapproxBreformulations}

\begin{proof}
Assume (i.). Then $A \lesssim_{init} B$, and so $\omega A \leqslant_{init} B$ or $\omega A \cong \omega B$ by Proposition \ref{lesssiminitandlesssimfinrefinements}. Also $A \lesssim_{fin} B$, and so either $\omega^*A \leqslant_{fin} B$ or $\omega^* A \cong \omega^* B$. 

If $\omega A \leqslant_{init} B$ as witnessed by $f$ and $\omega^* A \leqslant_{fin} B$ as witnessed by $g$, then $g + f$ witnesses $\mathbb{Z}A \leqslant_{cent} 2B$. 

If $\omega A \cong \omega B$ as witnessed by $f$ and $\omega^* A \cong \omega^* B$ as witnessed by $g$, then $g + f$ witnesses $\mathbb{Z}A \cong_{cent} \mathbb{Z}B$. 

Suppose $\omega A \leqslant_{init} B$ and $\omega^* A \cong \omega^* B$. The latter implies $B \leqslant_{fin} mA$ for some $m \geq 1$. Then
\[
\omega A \leqslant_{init} B \leqslant_{fin} mA
\]
which yields $\omega A \leqslant_{conv} mA$, and hence $\omega A \leqslant_{conv} A$ by \ref{directedrefinementforomegaA}. Hence $A$ is splitting. 

Since $\omega A \leqslant_{init} B$ we have $A \leqslant_{init} B$. On the other hand, $B \leqslant_{fin} mA \cong A$. Thus $A \cong B$ by \ref{CSBLO}, which implies $\mathbb{Z}A \cong_{cent} \mathbb{Z}B$. 

Symmetrically, if $\omega A \cong \omega B$ and $\omega^*A \leqslant_{fin} B$, then $\mathbb{Z}A \cong_{cent} \mathbb{Z}B$. 

Thus in all cases, either $\mathbb{Z}A \leqslant_{cent} 2B$ or $\mathbb{Z}A \cong_{cent} \mathbb{Z}B$ holds, which concludes the proof that (i.). implies (ii.). 

The proofs the (ii.) implies (iii.) and (iii.) implies (i.) are straightforward. 
\end{proof}

We define the equivalence relations corresponding to the orderings $\lesssim_{\star}$ and $\lessapprox$. 

\theoremstyle{definition}
\newtheorem{simstarandapproxdefns}[lct]{Definition}
\begin{simstarandapproxdefns}
For $\star \in \{init, fin, conv\}$, define
\[
\textrm{$A \sim_{\star} B$ if $A \lesssim_{\star} B$ and $B \lesssim_{\star} A$.}
\]
Also define
\[
\textrm{$A \approx B$ if $A \lessapprox B$ and $B \lessapprox A$.} 
\]
\end{simstarandapproxdefns}

Notice that $A \approx B$ if and only if $A \sim_{init} B$ and $A \sim_{fin} B$: both assertions are equivalent to the conjunction
\[
(\omega A \leqslant_{init} \omega B) \wedge (\omega^* A \leqslant_{fin} \omega^*B) \wedge (\omega B \leqslant_{init} \omega A) \wedge (\omega^* B \leqslant_{fin} \omega^* A).
\]

The relations $\sim_{\star}$ and $\approx$ are equivalence relations on $LO$. The following propositions illustrates a sense in which they may be viewed as notions of archimedean equivalence. 

\theoremstyle{definition}
\newtheorem{AsimstarBiffnAleqstarmBleqstarnApropn}[lct]{Proposition}
\begin{AsimstarBiffnAleqstarmBleqstarnApropn}\label{AsimstarBiffnAleqstarmBleqstarnApropn}
For $\star \in \{init, fin, conv\}$, if $A \sim_{\star} B$ then the following conditions hold:
\begin{itemize}
    \item[i.] For every $n \in \omega$, there is $m \in \omega$ such that $nA \leqslant_{\star} mB$;
    \item[ii.] For every $m' \in \omega$, there is $n' \in \omega$ such that $m'B \leqslant_{\star} n'A$. 
\end{itemize}
\end{AsimstarBiffnAleqstarmBleqstarnApropn}

\begin{proof}
Follows from Proposition \ref{AlessapproxBiffnAlesssomemB}. 
\end{proof}

Proposition \ref{AsimstarBiffnAleqstarmBleqstarnApropn} immediately yields the following.

\theoremstyle{definition}
\newtheorem{AsimstarBthenAspliffBspl}[lct]{Proposition}
\begin{AsimstarBthenAspliffBspl}\label{AsimstarBthenAspliffBspl}
Suppose $A \sim_{\star} B$ for some $\star \in \{init, fin, conv\}$ or $A \approx B$. Then $A$ is splitting if and only if $B$ is splitting.     
\end{AsimstarBthenAspliffBspl}
\begin{proof}
By Proposition \ref{AsimstarBiffnAleqstarmBleqstarnApropn}, $A$ and $B$ satisfy the hypotheses of Lemma \ref{2AconvnB2BconvmAsplittingiff}. 
\end{proof}

The relations $\sim_{\star}$ and $\approx$ are cleanly characterized arithmetically in terms of isomorphisms between infinite discrete products.

\theoremstyle{definition}
\newtheorem{simstarandisosofinfdiscreteproducts}[lct]{Proposition}
\begin{simstarandisosofinfdiscreteproducts}\label{simstarandisosofinfdiscreteproducts}
The following equivalences hold:
\begin{itemize}
\item[i.] $A \sim_{init} B$ if and only if $\omega A \cong \omega B$;
\item[ii.] $A \sim_{fin} B$ if and only if $\omega^* A \cong \omega^* B$; 
\item[iii.] $A \sim_{conv} B$ if and only if $\mathbb{Z}A \cong \mathbb{Z}B$;
\item[iv.] $A \approx B$ if and only if $\mathbb{Z}A \cong_{cent} \mathbb{Z}B$.
\end{itemize}
\end{simstarandisosofinfdiscreteproducts}
\begin{proof}
The backward implications are immediate from the definitions, so we prove the forward ones. 

For (i.): By Proposition \ref{lesssiminitandlesssimfinrefinements}, $A \sim_{init} B$ implies $\omega A \cong \omega B$, or $\omega A \leqslant_{init} B$ and $\omega B \leqslant_{init} A$. In the first case we are done immediately. 

In the second case, since 
\[
\omega A \leqslant_{init} B \leqslant_{init} \omega B \leqslant_{init} A,
\]
we have that $A$, and symmetrically $B$, is splitting. Since clearly $A \leqslant_{init} B \leqslant_{init} A$, we have $\omega A \cong \omega B$ in this case as well, by Lemma \ref{AinitBinitAandsplimpliesomegasiso}. 

The argument for (ii.) is symmetric, and (iv.) follows from (i.) and (ii.). 

For (iii.): Fix convex embeddings $f$ witnessing $\mathbb{Z}A \leqslant_{conv} \mathbb{Z}B$ and $g$ witnessing $\mathbb{Z}B \leqslant_{conv} \mathbb{Z}A$. If at least one of $f$ and $g$ is an isomorphism, then we are done. 

So suppose both embeddings are strict. Then the image of $\mathbb{Z}A$ is bounded on at least one side in $\mathbb{Z}B$. Since the cut at this bounded side falls either within or at an endcut of some copy of $B$, it follows that either $\omega A \leqslant_{conv} B$ or $\omega^* A \leqslant_{conv} B$.

Since $\mathbb{Z}B \leqslant_{conv} \mathbb{Z}A$, we have that $B \leqslant_{conv} mA$ for some $m \geq 1$. Combining with the above, we conclude $\omega A$ or $\omega^*A$ is convexly embeddable in $mA$, and hence $A$ is splitting. Hence $B \leqslant_{conv} mA \cong A$. 

By the symmetric argument applied to $g$, $B$ is splitting and $A \leqslant_{conv} B$. Now we get $\mathbb{Z}A \cong \mathbb{Z}B$ from Lemma \ref{AinitBinitAandsplimpliesomegasiso}.(iii.).
\end{proof}

For $\star \in \{init, fin, conv\}$, define the strict relations $\lnsim_{\star}$ by
\[
\textrm{$A \lnsim_{\star} B$ if $A \lesssim_{\star} B$ and $A \not \sim_{\star} B$.}
\]

Likewise, $A \lnapprox B$ means $A \lessapprox B$ and $A \not \approx B$.

\theoremstyle{definition}
\newtheorem{AlnsimBcharintermsofinfdiscreteprods}[lct]{Proposition}
\begin{AlnsimBcharintermsofinfdiscreteprods}\label{AlnsimBcharintermsofinfdiscreteprods} The following equivalences hold:
\begin{itemize}
    \item[i.] $A \lnsim_{init} B$ if and only if $\omega A \leqslant_{init} B$ and $\omega A \not\cong \omega B$; 
    \item[ii.] $A \lnsim_{fin} B$ if and only if $\omega^* A \leqslant_{fin} B$ and $\omega^* A \not\cong \omega^* B$;
    \item[iii.] $A \lnapprox B$ if and only if $\mathbb{Z}A \leqslant_{cent} 2B$ and $\mathbb{Z}A \not\cong \mathbb{Z}B$. 
\end{itemize}
\end{AlnsimBcharintermsofinfdiscreteprods}
\begin{proof}
Immediate from Propositions \ref{lesssiminitandlesssimfinrefinements} and \ref{simstarandisosofinfdiscreteproducts}. 
\end{proof}

We may also arithmetically characterize the situation when both of the conditions from Proposition \ref{lesssiminitandlesssimfinrefinements} implying $A \lesssim_{\star} B$ hold at once.

\theoremstyle{definition}
\newtheorem{bothomegaAleqBandomegaAconjomegaB}[lct]{Proposition}
\begin{bothomegaAleqBandomegaAconjomegaB}\label{bothomegaAleqBandomegaAconjomegaB}
The following equivalences hold:
\begin{itemize}
    \item[i.] $\omega A \leqslant_{init} B$ and $\omega A \cong \omega B$ if and only if $A$ and $B$ are splitting and $A \leqslant_{init} B \leqslant_{init} A$. 
    \item[ii.] $\omega^* A \leqslant_{fin} B$ and $\omega^* A \cong \omega^* B$ if and only if $A$ and $B$ are splitting and $A \leqslant_{fin} B \leqslant_{fin} A$. 
    \item[iii.] $\mathbb{Z}A \leqslant_{cent} 2B$ and $\mathbb{Z}A \cong_{cent} \mathbb{Z}B$ if and only if $A$ and $B$ are splitting and $A \cong B$. 
\end{itemize}
\end{bothomegaAleqBandomegaAconjomegaB}

\begin{proof}
For (i.): The forward direction is similar to previous arguments. For the backward direction, by Lemma \ref{AinitBinitAandsplimpliesomegasiso} it remains only to check that $\omega A \leqslant_{init} B$. Label $A = [c, d]$ by its endcuts. Since $A$ is splitting we may iteratively pick a cut sequence witnessing the following sequence of decompositions, in which in passing from each line to the next, we split the rightmost copy of $A$:
\[
\begin{array}{rcl}
A & \cong & A + A \\
& \cong & A + (A + A) \\
& \cong & A + (A + (A + A))) \\
& \vdots &
\end{array}
\]
i.e. a sequence $a_1 < a_2 < \ldots$ corresponding to the cuts at the $+$ signs in the above expression. Letting $a$ denote the cut at the right of this sequence, we have $[c, a] \cong \omega A$. Hence $\omega A \leqslant_{init} A$. Since $A \leqslant_{init} B$, we are done. 

A symmetric argument gives (ii.). And (iii.) follows from (i.) and (ii.) once we observe that $A \leqslant_{init} B$ (from (i.)) and $B \leqslant_{fin} A$ (from (ii.)) imply $A \cong B$. 
\end{proof}

Proposition \ref{omegaAcongomegaBthenomegaAplusBtoo} below implies that for $\star \in \{init, fin\}$, $\sim_{\star}$-equivalence classes are closed under sums. 

\theoremstyle{definition}
\newtheorem{omegaAcongomegaBthenomegaAplusBtoo}[lct]{Proposition}
\begin{omegaAcongomegaBthenomegaAplusBtoo}\label{omegaAcongomegaBthenomegaAplusBtoo}
\phantom{.}
\begin{itemize}
    \item[i.] If $\omega A \cong \omega B$, then $\omega A \cong \omega(A + B) \cong \omega(B + A) \cong \omega B$;
    \item[ii.] If $\omega^* A \cong \omega^* B$, then $\omega^* A \cong \omega^*(B+A) \cong \omega^*(A + B) \cong \omega^*B$.
\end{itemize}
\end{omegaAcongomegaBthenomegaAplusBtoo}

\begin{proof}
For (i.): Label $\omega A = [a_0, c]$ by its endcuts. Choose a cut $a_1 > a_0$ such that $[a_0, a_1] \cong A$ and $[a_1, c] \cong \omega A$. Since $\omega A \cong \omega B$, we have $[a_1, c] \cong \omega B$. Thus we may choose a cut $b_1 > a_1$ such that $[a_1, b_1] \cong B$ and $[b_1, c] \cong \omega B$. Then likewise $[b_1, c] \cong \omega A$ and so there is a cut $a_2 > b_1$ such that $[b_1, a_2] \cong A$ and $[a_2, c] \cong \omega A$; and so on. Continuing in this way we obtain a cut sequence 
\[
a_0 < a_1 < b_1 < a_2 < b_2 < \ldots
\]
witnessing $\omega(A + B) \leqslant_{init} \omega A$. If the sequence converges to $c$, then it witnesses $\omega(A+B) \cong \omega A$ and we are done. 

Otherwise, there is a tail of the sequence that witnesses $\omega(A + B) \leqslant_{conv} A$. Since $A \leqslant_{init} A + B$, we get $\omega(A+B) \leqslant_{conv} A + B$ and deduce $A + B$ is splitting. 

We claim in this case that $A$ is also splitting. We first check that $A \leqslant_{init} B + A$. Indeed, $B + A$ is initial in $\omega A$ by Corollary \ref{omegaAomegastarAinitfinomegaBomegastarB}, and since $A$ is also initial in $\omega A$, one of $A \leqslant_{init} B + A$ and $B + A \leqslant_{init} A$ holds. In the first case we are done. In the second case, since $A \leqslant_{fin} B + A$ we have $B + A \cong A$; hence $A \leqslant_{init} B + A$ in this case as well. 

Since $A$ is initial in $B + A$, we have $2A \leqslant_{init} A + (B + A) \leqslant_{init} 2(A+B)$, which gives $2A \leqslant_{init} 2(A+B)$. On the other hand, $\omega(A + B) \leqslant_{conv} A$ implies in particular $2(A + B) \leqslant_{conv} A$. Thus $A$ and $A + B$ satisfy the hypotheses of Lemma \ref{2AconvnB2BconvmAsplittingiff}. Since $A + B$ is splitting, $A$ is also splitting, as claimed. 

Since $\omega (A + B)$ is initial in $\omega A$, we have $A + B \leqslant_{init} mA \cong A$ for some $m$. Since also $A \leqslant_{init} A + B$, Proposition \ref{AinitBinitAandsplimpliesomegasiso} applies, and we get $\omega A \cong \omega (A + B)$, as desired. 

Proving $\omega B \cong \omega (B+A)$ is symmetric, which concludes (i.). 

The proof for (ii.) is symmetric. 
\end{proof}

\subsection{Continuation theorems}\label{subsect:continuationthms}

In this section we consider bounded versions of the archimedean relations $\lesssim_{\star}$ and $\sim_{\star}$. We show that for the equivalence relations at least, the bounded versions are equivalent to their unbounded counterparts. 

It will follow that the converse to Proposition \ref{AsimstarBiffnAleqstarmBleqstarnApropn} holds, i.e. that $\sim_{\star}$ and $\approx$ may be \textit{characterized} as notions of archimedean equivalence. 

We also show that for $\star \in \{init, fin\}$ the converse to Proposition \ref{AlessapproxBiffnAlesssomemB} holds, i.e., that the relations $\lesssim_{\star}$ are characterized as notions of archimedean dominance.

\theoremstyle{definition}
\newtheorem{lesssiminitfinbdefn}[lct]{Definition}
\begin{lesssiminitfinbdefn}\label{lesssiminitfinbdefn}
For $\star \in \{init, fin\}$ define:
\begin{itemize}
    \item[i.] $A \lesssim_{\star}^b B$ if there exists a natural number $m$ such that $2A \leqslant_{\star} mB$;
    \item[ii.] $A \sim_{\star}^b B$ if $A \lesssim_{\star}^b B$ and $B \lesssim_{\star}^b A$. 
\end{itemize}
\end{lesssiminitfinbdefn}  

The exponent $b$ stands for \textit{bounded}, referring to the bound of $2$ on the coefficient of the embedded order $A$ in the definition.

We also introduce bounded versions of $\lesssim_{conv}$ and $\sim_{conv}$. To get our proofs below to go through, the coefficient of the embedded order for these relations changes from $2$ to $3$. 

\theoremstyle{definition}
\newtheorem{lesssimconvbdefn}[lct]{Definition}
\begin{lesssimconvbdefn}\label{lesssimconvbdefn}
Define:
\begin{itemize}
    \item[i.] $A \lesssim_{conv}^b B$ if there is a natural number $m$ such that $3A \leqslant_{conv} mB$. 
    \item[ii.] $A \sim_{conv}^b B$ if $A \lesssim_{conv}^b B$ and $B \lesssim_{conv}^b A$.
\end{itemize}

\end{lesssimconvbdefn}

While the bounded relations $\lesssim_{\star}^b$ are not equivalent to their unbounded analogs, we show below that the corresponding equivalence relations $\sim^b_{\star}$ are equivalent to the unbounded versions.

Finally, we define the bounded versions of $\lessapprox$ and $\approx$. 

\theoremstyle{definition}
\newtheorem{AlessapproxbBdefn}[lct]{Definition}
\begin{AlessapproxbBdefn}\label{AlessapproxbBdefn}
Define:
\begin{itemize}
    \item[i.] $A \lessapprox^b B$ if $A \lesssim_{init}^b B$ and $A \lesssim_{fin}^b B$. 
    \item[ii.] $A \approx^b B$ if $A \lessapprox^b B$ and $B \lessapprox^b A$.
\end{itemize}
\end{AlessapproxbBdefn}

The following results are the main results of this section. Theorem \ref{continuationthm} below says that the bounded notions of archimedean equivalence introduced above are equivalent to the corresponding unbounded notions. To prove it, we will need the following lemma.

\theoremstyle{definition}
\newtheorem{continuationlemma}[lct]{Lemma}
\begin{continuationlemma}\label{continuationlemma}
(Continuation lemma): 
\begin{itemize}
\item[i.] Suppose that $\mathcal{D} = \{A_k: k \in \omega\}$ is a non-terminating and non-splitting left division system, and for some $m \in \omega$ one of the following relations holds:
\[
\begin{array}{rcl}
A_0 + A_1 + A_m & \leqslant_{init} & A_1 + A_0 + A_m \\
A_1 + A_0 + A_m & \leqslant_{init} & A_0 + A_1 + A_m.
\end{array}
\]
Then the corresponding real representation $\mathbb{Z}A_0 \cong \mathbb{R}(K_{[x]})$ is sufficient. 
\item[ii.] Suppose that $\mathcal{D} = \{A_k: k \in \omega\}$ is a non-terminating and non-splitting right division system, and for some $m \in \omega$ one of the following relations holds:
\[
\begin{array}{rcl}
A_m + A_1 + A_0 & \leqslant_{fin} & A_m + A_0 + A_1 \\
A_m + A_0 + A_1 & \leqslant_{fin} & A_m + A_1 + A_0.
\end{array}
\]
Then the corresponding real representation $\mathbb{Z}A_0 \cong \mathbb{R}(K_{[x]})$ is sufficient. 
\end{itemize}
\end{continuationlemma}

Before proving the lemma, let us motivate it by observing that the hypothesis
\[
\textrm{$A + B \leqslant_{init} B + A$ or $B + A \leqslant_{init} A + B$}
\]
is enough in many cases to guarantee that $\omega A \cong \omega B$. But not in all cases: if we fix a left division setup witnessing the hypothesis and run the Euclidean algorithm, then if the resulting division system is non-terminating and non-splitting, sufficiency of the corresponding real representation is in general necessary to conclude $\omega A \cong \omega B$. 

The hypothesis in Lemma \ref{continuationlemma}.(i.) may be read as
\[
\textrm{$A + B + \varepsilon \leqslant_{init} B + A + \varepsilon$ or $B + A + \varepsilon \leqslant_{init} A + B + \varepsilon$}
\]
where $A_0 = A$, $A_1 = B$, and $\varepsilon$ is a small but non-negligible initial segment in $A$. In the case when the resulting division system is non-splitting and non-terminating, under this hypothesis the lemma guarantees that the corresponding real representation \textit{is} sufficient. Hence in particular we have $\omega A \cong \omega B$ by Proposition \ref{sufficiencyimpliesIsosforomegaprods}. We will need this fact to show $A \sim_{init}^b B$ implies $A \sim_{init} B$ in Theorem \ref{continuationthm} below. 

Here is the proof of Lemma \ref{continuationlemma}. 

\begin{proof}
For (i.): Let $A_0 = A$, $A_1 = B$, and $A_m = C$. Suppose we have the second of the two relations in the hypothesis:
\[
B + A + C \leqslant_{init} A + B + C.
\]
The proof when the first relation holds is similar.

By the representation theorem \ref{leftrealrepnthm}, there is an irrational $a_1 \in (0, 1)$ such that we have the following real representations:
\[
\begin{array}{rcl}
\mathbb{Z}A & \cong & \mathbb{R}(K_{[x]}) \\
A & \cong & R + (0, 1)(K_{[x]}) + L \\
B & \cong & R + (0, a_1)(K_{[x]}) + L',
\end{array}
\]
where $K_0 \cong L + R$ and $K_{a_1} \cong L' + R$. Our goal is to show that the representation $\mathbb{R}(K_{[x]})$ is sufficient, i.e. that $K_0 \cong K_{a_1}$, i.e. that $L + R \cong L' + R$.

By the discussion following Proposition \ref{AkplusAkplusone}, there are final segments $M$ and $M'$ of $B + A$ and $A + B$ respectively, such that 
\[
\begin{array}{rcl}
B + A & \cong & \sum_{i \in \omega} N_i A_{i+1} + M \\
A + B & \cong & \sum_{i \in \omega} N_i A_{i+1} + M'.
\end{array}
\]
We claim that $M \cong L$ and $M' \cong L'$, i.e. that $M$ and $M'$ are the final $\sim_{\mathcal{D}}$-classes in $B + A$ and $A + B$ respectively. We first prove $M \cong L$. 

Label $B + A = [c, d]$ by its endcuts, and let $m$ be the cut at the $+$ sign preceding $M$ in the above expression for $B+A$, so that $c < m \leq d$ and 
\[
\begin{array}{rcl}
\textrm{$[c, m]$} & \cong & \sum_{i \in \omega} N_i A_{i+1} \\
\textrm{$[m,d]$} & \cong & M.
\end{array}
\]

Let $d' \leq d$ denote the cut at the left of the rightmost $\sim_{\mathcal{D}}$-class in $B + A$. More precisely, if there is a non-empty rightmost $\sim_{\mathcal{D}}$-class in $B + A$, let $d'$ denote its left endcut, and if there is no such class let $d' = d$.

We prove $d' = m$. Suppose first $d' < m$. It follows from the decomposition above that every final segment of $B + A$ strictly containing $[m, d]$ contains a tail-sum of the form $\sum_{i \geq k} N_i A_{i+1}$ for some large enough $k$; in particular $[d', d]$ contains such a segment. By the discussion after Proposition \ref{AkplusAkplusone}, such a tail-sum is isomorphic to either $A_k + A_{k+1}$ or $A_{k+1} + A_k$, depending on the parity of $k$. Let us suppose the former; the argument in the second case is similar.  

Writing $A_k$ as $N_kA_{k+1} + A_{k+2}$, from the fact that $A_{k+3}$ is final in $A_{k+1}$ it follows we may find points $x < y$ in $[d', d]$ such that $[x, y]$ contains a convex copy of $A_{k+3}$, contradicting that $x \sim_{\mathcal{D}} y$. Hence $d' \geq m$. 

If $d' > m$, then $[m, d]$ contains a pair of points $x < y$ such that $x \not\sim_{\mathcal{D}} y$. Hence $A_k \leqslant_{conv} [x, y] \leqslant_{conv} [m, d]$ for some $k$. Since $M$ is a final segment of every tail-sum $\sum_{i \geq n} N_i A_{i+1}$, there are arbitrarily large $n$ such that $[m, d] \leqslant_{fin} A_n + A_{n+1}$. In particular we may find $n \geq k + 3$ such that $[m, d] \leqslant_{fin} A_n + A_{n+1}$. Writing $A_k \cong N_k A_{k+1} + A_{k+2}$, we have the chain
\[
A_{k+1} + A_{k+2} \leqslant_{conv} N_k A_{k+1} + A_{k+2} \cong A_k \leqslant_{conv} [m, d] \leqslant_{fin} A_n + A_{n+1}. 
\]
Thus $A_{k+1} + A_{k+2} \leqslant_{conv} A_n + A_{n+1}$, which yields $A_{k+1} \leqslant_{conv} A_n \leqslant_{conv} A_{k+3}$ or $A_{k+2} \leqslant_{conv} A_{n+1} \leqslant_{conv} A_{k+4}$. Since $2A_{k+3} \leqslant_{conv} A_{k+1}$ and $2A_{k+4} \leqslant_{conv} A_{k+2}$ by Lemma \ref{splittinglemmadivissystems}, we deduce that either $A_{k+3}$ or $A_{k+4}$ is splitting, a contradiction, since $\mathcal{D}$ is non-splitting. Hence $m = d'$ as claimed. 

We now prove $M \cong L$. Identify $A$ with its representation $R + (0, 1)(K_{[x]}) + L$, and fix a final embedding $f: A \rightarrow B + A$. Since $L$ is the final $\sim_{\mathcal{D}}$-class in $A$ and $\sim_{\mathcal{D}}$-classes are invariant under convex embeddings, $f$ must map $L$ onto the final $\sim_{\mathcal{D}}$-class in $B+A$. Thus $f[L] = M$. More precisely, one of $B + A$ and $A$ has a non-empty final $\sim_{\mathcal{D}}$-class if and only if the other does, hence $L$ is non-empty if and only if $M$ is non-empty, in which case $f[L] = M$. Thus $M \cong L$, as claimed. 

A symmetric argument for $A + B$ shows $L' \cong M'$.

Now, as $C = A_m$ there is an irrational $a_m \in (0, a_1)$ such that we have the real representation
\[
C \cong R + (0, a_m)(K_{[x]}) + L_0,
\]
where either $L_0 = L$ or $L_0 = L'$ depending on the parity of $m$, and $K_{a_m} \cong L_0 + R$. We obtain the following expressions for $B + A + C$ and $A + B + C$:
\begin{equation}\label{rubbadubba}
\begin{array}{rcl}
B + A + C & \cong & \sum_i N_i A_{i+1} + L + R + (0, a_m)(K_{[x]}) + L_0 \\
A + B + C & \cong & \sum_i N_i A_{i+1} + L' + R + (0, a_m)(K_{[x]}) + L_0
\end{array}
\end{equation}

By assumption we have $B + A + C \leqslant_{init} A + B + C$. We will use an embedding witnessing this relation to show $L + R \cong L' + R$. The argument for the alternate case when $A + B + C \leqslant_{init} B + A + C$ is essentially symmetric. 

Label $B + A + C = [x, y]$ and $A + B + C = [x', y']$ by their endcuts. Fix cuts $p \leq q$ in $B + A + C$ corresponding to the cuts at the $+$ signs to the left of $L$ and right of $R$ in the above expression \ref{rubbadubba}, so that $[x, p] \cong \sum_i N_iA_{i+1}$, $[p, q] \cong L + R$, and $[q, y] \cong (0, a_m)(K_{[x]}) + L_0$. Fix corresponding cuts $p' \leq q'$ in $A + B + C$, and let $f: B + A + C \rightarrow A + B + C$ be an initial embedding. 

We claim $f(p) = p'$. Since $[x, p] \cong [x', p'] \cong \sum_i N_i A_{i+1}$, if $f(p) < p'$ then $f \upharpoonright [x, p]$ is a strict initial embedding of $\sum_i N_i A_{i+1}$ into itself. And if $f(p) > p'$ then $f^{-1} \upharpoonright [x', p']$ is such an embedding. Thus to show $f(p) = p'$ it suffices to show no such embedding exists. We give a length invariance argument that makes use of Proposition \ref{ellIlessellJnonisoinsufficientsetting}.

Since
\[
\sum_{i \in \omega} N_i A_{i+1} + L' \cong A + B \leqslant_{init} 2A \leqslant_{conv} \mathbb{Z}A \cong \mathbb{R}(K_{[x]}),
\]
we have that $\sum_{i \in \omega} N_i A_{i+1}$ is (isomorphic to) an interval in $\mathbb{R}(K_{[x]})$. Since this interval contains convex copies of the orders $A_k$, its length in $\mathbb{R}(K_{[x]})$ must be strictly positive. Thus by Proposition \ref{ellforintervalspropn}, we have 
\[
\sum_{i \in \omega} N_i A_{i+1} + L' \cong M + (a, b)(K_{[x]}) + N
\]
for some $a < b$ in $\mathbb{R}$, $M$ final in $K_a$, and $N$ initial in $K_b$. (In fact, it is not hard to show this isomorphism holds with $a = 0$ and $b = 1 + a_1$.) Since $L'$ and $N$ both label the final $\sim_{\mathcal{D}}$-class in $A + B$, these segments must coincide, so that
\[
\sum_{i \in \omega} N_i A_{i+1} \cong M + (a, b)(K_{[x]}).
\]

Since $M + (a, b)(K_{[x]})$ has no final $\sim_{\mathcal{D}}$-class, any strict initial segment $I$ of this order must be contained in $M + (a, b')(K_{[x]})$ for some $b' < b$, and hence $\ell(I) \leq b' - a < b - a = \ell(M + (a, b)(K_{[x]}))$. By Proposition \ref{ellIlessellJnonisoinsufficientsetting}, since $\mathcal{D}$ is non-splitting, there is no isomorphism from $\sum_i N_i A_{i+1} \cong M + (a, b)(K_{[x]})$ onto such a strict initial segment, as claimed. (It is exactly here we need the generality of Proposition \ref{ellIlessellJnonisoinsufficientsetting}, since we have not yet established the sufficiency of the representation $\mathbb{R}(K_{[x]})$.) By the above discussion, we have $f(p) = p'$. 

Thus the restriction $f \upharpoonright [p, y]$ witnesses
\[
L + R + (0, a_m)(K_{[x]}) + L_0 \leqslant_{init} L' + R + (0, a_m)(K_{[x]}) + L_0
\]

We next claim that $f(q) = q'$, which implies $f[L + R] = f[[p, q]] = [p', q'] = L' + R$. It suffices to check that $L + R$ and $L' + R$ are the initial $\sim_{\mathcal{D}}$-classes in these orders.

Let $q^*$ denote the right endcut of the initial $\sim_{\mathcal{D}}$-class in $L + R + (0, a_m)(K_{[x]}) + L_0$. More precisely, if there is no initial $\sim_{\mathcal{D}}$-class let $q^* = p$, otherwise let $q^*$ denote the right endcut of this class, which is then equal to $[p, q^*]$. 

We claim $q^* = q$. If $q < q^*$, then $[p, q^*]$ contains an interval of the form $(0, a)(K_{[x]})$, where $0 < a < a_m$. By Lemma \ref{embeddingAkwheninsufficientlemma}, such an interval convexly embeds $A_k$ for some $k$, and by increasing $k$ if necessary we may assume that the image of this embedding lies between two points in $[p, q^*]$, contradicting that $[p, q^*]$ is a $\sim_{\mathcal{D}}$-class. If $q > q^*$, then $[p, q] \cong L + R$ convexly embeds some $A_k$, a contradiction, since $L + R \cong K_0$, which is a $\sim_{\mathcal{D}}$-class in $\mathbb{R}(K_{[x]})$ and thus convexly embeds no $A_k$. Hence $q = q^*$ as claimed.

It follows $L + R$ is the initial $\sim_{\mathcal{D}}$-class in $L + R + (0, a_m)(K_{[x]}) + L_0$. A similar argument shows $L' + R$ is the initial $\sim_{\mathcal{D}}$-class in $L' + R + (0, a_m)(K_{[x]}) + L_0$. It follows $f[L + R] = L' + R$, and so $L + R \cong L' + R$. Thus the real representation 
\[
\mathbb{Z}A \cong \mathbb{R}(K_{[x]})
\]
is sufficient, and we are done. 

The proof for (ii.) is symmetric. 
\end{proof}

With Lemma \ref{continuationlemma} in hand, we may prove the equivalence of the bounded relations $\sim_{\star}^b$ and $\approx^b$ with their unbounded counterparts. 

\theoremstyle{definition}
\newtheorem{continuationthm}[lct]{Theorem}
\begin{continuationthm}\label{continuationthm}
(Continuation):
\begin{itemize}
    \item[i.] For $\star \in \{init, fin, conv\}$, $A \sim_{\star}^b B$ if and only if $A \sim_{\star} B$;
    \item[ii.] $A \approx^b B$ if and only if $A \approx B$. 
\end{itemize}
\end{continuationthm}

Before the proof, let us again give some motivating discussion. Taking $\star = init$, (i.) says that $\omega A \cong \omega B$ if and only if there exist natural numbers $n$ and $m$ such that $2A \leqslant_{init} mB$ and $2B \leqslant_{init} nA$. Thus not only is the converse to Proposition \ref{AsimstarBiffnAleqstarmBleqstarnApropn} true (see Proposition \ref{boundedarchimequivisequivtounboundedarchimeequiv} below), but to guarantee $\omega A \cong \omega B$ it suffices to check that \textit{two} copies of $A$ (instead of any $n$ copies of $A$) can be extended rightward to some $m$ copies of $B$, and vice versa. 

One might ask if this bounded condition could be optimized further while maintaining equivalence with $\omega A \cong \omega B$. Replacing $2$ with $1$ leads to a strictly weaker condition, that is, $A \leqslant_{init} mB$ and $B \leqslant_{init} kA$ does not guarantee $\omega A \cong \omega B$ in general. Indeed, it is not hard to produce examples with $A \leqslant_{init} B \leqslant_{init} A$ but $\omega A \not\cong \omega B$. 

Suppose $\mathcal{D} = \{A_k\}$ is a left division system with $\{A, B\} = \{A_0, A_1\}$. The existence of such a system is implied by $\omega A \cong \omega B$ by Corollary \ref{omegaAomegastarAinitfinomegaBomegastarB}. In the other direction, if the system is terminating, then we have that $\omega A \cong \omega B$. Even when the system is non-terminating, then (as discussed also before Lemma \ref{continuationlemma}) if the corresponding representation is sufficient then the isomorphism still holds. However, there are examples showing that if $\mathcal{D}$ is non-splitting and non-terminating and the corresponding real representation $\mathbb{R}(K_{[x]})$ is insufficient, then it can happen that $\omega A \not\cong \omega B$. 

The proof of the Theorem \ref{continuationthm} below shows that $A \sim_{init}^b B$ implies the existence of such a division system, and in the non-terminating and non-splitting case, the $2$ appearing in the definition of $\sim^b_{init}$ is enough (via Lemma \ref{continuationlemma}) to guarantee that the system is sufficient. 

\begin{proof}
For (i.): Suppose $\star = init$. It is easy to see that $A \sim_{init} B$ implies $A \sim_{init}^b B$. 

So suppose $A \sim_{init}^b B$. We must show $A \sim_{init} B$, or equivalently by Proposition \ref{simstarandisosofinfdiscreteproducts}, that $\omega A \cong \omega B$. 

Fix $m \geq 1$ such that $2A \leqslant_{init} mB$. Then $2A \leqslant_{init} (m+1)B$ as well; hence we may assume $m \geq 2$. Then since both $2A \leqslant_{init} mB$ and $2B \leqslant_{init} mB$, either $2A \leqslant_{init} 2B$ or $2B \leqslant_{init} 2A$. 

Symmetrically, if we fix $n \geq 2$ such that $2B \leqslant_{init} nA$, we may deduce that either $2B \leqslant_{init} 2A$ or $2A \leqslant_{init} 2B$. Since the situation is symmetric in $A$ and $B$, we may assume $2B \leqslant_{init} 2A$. 

Then also $B \leqslant_{init} A$ by Proposition \ref{leqslantinitfincancellationthm}. Thus we have the chain:
\begin{equation}\label{turkishtopaz}
B \leqslant_{init} A \leqslant_{init} 2A \leqslant_{init} mB \leqslant_{init} \omega B.  
\end{equation}

From $2A \leqslant_{init} \omega B$ we get $B + 2A \leqslant_{init} B + \omega B \cong \omega B$. For consideration later, we separate out these relations:
\begin{equation}\label{turkishsilver}
\begin{array}{rcl}
B + 2A & \leqslant_{init} & \omega B, \\
2A & \leqslant_{init} & \omega B. 
\end{array}
\end{equation}

Since $B \leqslant_{init} A$, we have $A + B \leqslant_{init} A + A \cong 2A$. Then since $B + A \leqslant_{init} B + 2A$, we deduce from \ref{turkishsilver}:
\[
\begin{array}{rcl}
B + A & \leqslant_{init} & \omega B, \\
A + B & \leqslant_{init} & \omega B. 
\end{array}
\]
Thus either $B + A \leqslant_{init} A + B$, or vice versa. In either case, we may fix a left setup witnessing the relation, and run the left Euclidean algorithm. 

Suppose the algorithm terminates at Stage 0 in an absorbed factor. If the factor is $B$, then $\omega B \leqslant_{init} A$. Then since $A$ is initial in $mB$, it follows (by a now routine argument) that $B$ is splitting and $A$ is splitting, which yields $A \leqslant_{init} B \leqslant_{init} A$. Then Proposition \ref{AinitBinitAandsplimpliesomegasiso} gives that $\omega A \cong \omega B$, and we are done. The argument is similar if the absorbed factor is $A$. 

If the algorithm terminates at a finite stage, then by the discussion in Section \ref{subsect:leftalgorun} we have $\omega A \cong \omega B$, and we are done. 

If the algorithm is non-terminating but the resulting division system is splitting, then $A$ and $B$ are splitting and $A \leqslant_{init} B \leqslant_{init} A$, so that again we have $\omega A \cong \omega B$, and we are done. 

Thus we may assume that the algorithm is non-terminating, and that the resulting left division system $\mathcal{D} = \{A_k \cong N_k A_{k+1} + A_{k+2}\}$ is non-splitting. 

We have $\{A_0, A_1\} = \{A, B\}$; we claim that in fact $A_0 = A$ and $A_1 = B$. If not, then $A_0 = B$ and $A_1 = A$. By Theorem \ref{leftrealrepnthm}, there is an irrational number $a_1 \in (0, 1)$ such that we have the following real representations for $\mathbb{Z}B = \mathbb{Z}A_0$, $B = A_0$, and $A = A_1$:
\[
\begin{array}{rcl}
\mathbb{Z}B & \cong & \mathbb{R}(K_{[x]}) \\
B & \cong & R + (0, 1)(K_{[x]}) + L \\
A & \cong & R + (0, a_1)(K_{[x]}) + L'.
\end{array}
\]

Let $I$ and $J$ denote the intervals $R + (0, 1)(K_{[x]}) + L$ and $R + (0, a_1)(K_{[x]}) + L'$ in $\mathbb{R}(K_{[x]})$ respectively. Since $B \leqslant_{init} A$, there is an initial segment $I'$ of $J$ such that $I' \cong B$. Then $\ell(I') \leq \ell(J) = a_1 < \ell(I) = 1$. Since $B \cong I' \cong I$ and $I' \subseteq I$, this contradicts Theorem \ref{ellIlessellJnonisoinsufficientsetting}. Hence $A_0 = A$ and $A_1 = B$, as claimed. 

Thus, instead, there is an irrational $a_1 \in (0, 1)$ such that we have the following representations:
\[
\begin{array}{rcl}
\mathbb{Z}A & \cong & \mathbb{R}(K_{[x]}) \\
A & \cong & R + (0, 1)(K_{[x]}) + L \\
B & \cong & R + (0, a_1)(K_{[x]}) + L',
\end{array}
\]
where $K_0 \cong L + R$ and $K_{a_1} \cong L' + R$. Our goal is to show that the representation $\mathbb{R}(K_{[x]})$ is sufficient, i.e. that $K_0 \cong K_{a_1}$, i.e. that $L + R \cong L' + R$.  

Let $C = A_2$. Then there is an irrational $a_2 \in (0, a_1)$ such that we have the real representation
\[
C \cong R + (0, a_2)(K_{[x]}) + L,
\]
where $K_{a_2} \cong K_0 \cong L + R$. 

Since $A \cong N_0 B + C$ and $C$ is initial in $B$, we have $B + C \leqslant_{init} A$ and hence $A + B + C \leqslant_{init} 2A$. Since $C$ is also initial in $A$, we have $B + A + C \leqslant_{init} B + 2A$. Since both $2A$ and $B + 2A$ are initial in $\omega B$ by equation \ref{turkishsilver}, both $A + B + C$ and $B + A + C$ are initial in $\omega B$. It follows that one of $A + B + C$ and $B + A + C$ is initial in the other. 

Hence $A + B + C$ and $B + A + C$ satisfy the hypothesis of the continuation lemma \ref{continuationlemma}.(i.). It follows by that lemma that the real representation 
\[
\mathbb{Z}A \cong \mathbb{R}(K_{[x]})
\]
is sufficient. By Proposition \ref{sufficiencyimpliesIsosforomegaprods}, we have $\omega A_0 \cong \omega A_1$, i.e. $\omega A \cong \omega B$. Thus $A \sim_{init} B$, as desired. 

If $\star = fin$, a symmetric argument yields $A \sim_{fin} B$.

The remaining case for part (i.) of the theorem is when $\star = conv$. Again, it suffices to check that $A \sim_{conv}^b B$ implies $A \sim_{conv} B$. 

Suppose $A \sim_{conv}^b B$, and fix $m$ such that $3A \leqslant_{conv} mB$. Then also $2A \leqslant_{conv} mB$. Let $f: 2A \rightarrow mB$ be a convex embedding witnessing this relation. 

Fix a cut sequence $a_0 < a_1 < a_2$ in $2A$ witnessing its decomposition into two copies of $A$, i.e. such that $a_0, a_2$ are the endcuts of $2A$ and $[a_0, a_1] \cong [a_1, a_2] \cong A$. Likewise fix a cut sequence $b_0 < b_1 < \ldots < b_m$ in $mB$ witnessing its decomposition into $m$ copies of $B$. 

By removing any copies of $B$ at the left of $mB$ that are disjoint from the image of $f$, we may assume that $b_0 \leq f(a_0) < b_1$. Let $C = [b_0, f(a_0)]$ and $D = [f(a_0), b_1]$. Then we have
\[
B \cong [b_0, b_1] \cong [b_0, f(a_0)] + [f(a_0), b_1] \cong C + D.
\]
Thus
\[
\begin{array}{rcl}
[f(a_0), b_m] & \cong & [f(a_0), b_1] + [b_1, b_2] + \cdots + [b_{m-1}, b_m] \\
& \cong & D + \underbrace{B + B + \cdots + B}_{\textrm{$m-1$ times}} \\ 
& \cong & D + (C + D) + (C + D) + \cdots + (C + D) \\
& \cong & (D + C) + (D + C) + \cdots + (D + C) + D \\
& \cong & (m-1)B' + D,
\end{array}
\]
where $B' = D + C$. 

Note that $f$ is an initial embedding of $2A$ in $[f(a_0), b_m]$, thus witnessing $2A \leqslant_{init} (m-1)B' + D$. Since $(m-1)B' + D \leqslant_{init} (m-1)B' + D + C \cong mB'$, we have $2A \leqslant_{init} mB'$. Adding an extra copy of $B'$ at the right if necessary, we may assume $m \geq 2$. Then $2B' \leqslant_{init} mB'$, and so either $2B' \leqslant_{init} 2A$ or $2A \leqslant_{init} 2B'$. 

Suppose first that $2B' \leqslant_{init} 2A$. Then $B' \leqslant_{init} A$ by \ref{BinitCimpliesABinitAC}, and we have the chain
\begin{equation}\label{turkishopal}
B' \leqslant_{init} A \leqslant_{init} 2A \leqslant_{init} mB' \leqslant_{init} \omega B'.  
\end{equation}
It is from exactly this setup (see equation \ref{turkishtopaz}), with $B$ instead of $B'$, that we argued when $\star = init$ that, in all possible cases, we have $\omega A \cong \omega B$. Those same arguments applied here thus yield $\omega A \cong \omega B'$. 

In fact, the arguments yield $\mathbb{Z}A \cong \mathbb{Z}B'$: in the case when the corresponding Euclidean algorithm is terminating, we get not only $\omega A \cong \omega B'$ but also $\mathbb{Z}A \cong \mathbb{Z}B'$ by the discussion in Section \ref{subsect:leftalgorun}; when the algorithm is non-terminating and the resulting division system is splitting we get $\mathbb{Z}A \cong \mathbb{Z}B'$ by Proposition \ref{AinitBinitAandsplimpliesomegasiso}.(iii.), whose hypotheses are weaker than Proposition \ref{AinitBinitAandsplimpliesomegasiso}.(i.) used to show $\omega A \cong \omega B'$; and when the algorithm is non-terminating and the system is non-splitting, we get $\mathbb{Z}A \cong \mathbb{Z}B'$ from Proposition \ref{sufficiencyimpliesIsosforomegaprods} by the sufficiency of the resulting real representation, established above.

Now observe
\[
\begin{array}{rcl}
\mathbb{Z}B & \cong & \cdots + B + B + B + \cdots \\
& \cong & \cdots + (C + D) + (C + D) + (C + D) + \cdots \\
& \cong & \cdots + C) + (D + C) + (D + C) + (D + \cdots \\
& \cong & \mathbb{Z}B'.
\end{array}
\]
Hence $\mathbb{Z}A \cong \mathbb{Z}B$ as well, which gives $A \sim_{conv} B$ by Proposition \ref{simstarandisosofinfdiscreteproducts}.

It remains to show $A \sim_{conv} B$ in the case when $2A \leqslant_{init} 2B'$. Fix a cut sequence $b_0' < b_1' < b_2'$ in $2B'$ witnessing its decomposition into two copies of $B'$, likewise a cut sequence $a_0 < a_1 < a_2$ in $2A$. Also fix an initial embedding $g: 2A \rightarrow 2B'$. 

Observe that
\[
C + 2B' + D \cong C + 2(D+C) + D \cong C + D + C + D + C + D \cong 3B.
\]
Thus $2B' \leqslant_{conv} 3B$, and by hypothesis there is $n$ such that $3B \leqslant_{conv} nA$. Hence $2B' \leqslant_{conv} nA$, and it follows $2B' \leqslant_{conv} \omega A$.

Fix a cut sequence $a_0' < a_1' < \ldots$ in $\omega A$ witnessing its decomposition into $\omega$ copies of $A$, as well as a convex embedding $h: 2B' \rightarrow \omega A$. Again, by discarding some initial copies of $A$ if need be, we may assume $a_0' \leq h(b_0') < a_1'$. 

Let $P = [a_0', h(b_0')]$ and $Q = [h(b_0'), a_1']$. Thus
\[
A \cong [a_0', a_1'] \cong [a_0', h(b_0')] + [h(b_0'), a_1'] \cong P + Q.
\]

Let $A' = Q + P$. We claim that $\omega A \cong \omega A'$, i.e. that $\omega (P + Q) \cong \omega (Q + P)$. Before proving it, we check that the claim suffices to finish the argument.

We have $2A \leqslant_{init} 2B'$ by hypothesis, and hence $A \leqslant_{init} B'$. The embedding $h$ witnesses
\[
\begin{array}{rcl}
2B' & \leqslant_{init} & [h(b_0'), a_1'] + [a_1', a_2'] + [a_2', a_3'] + \cdots \\
& \cong & Q + A + A + \cdots \\
& \cong & Q + (P + Q) + (P + Q) + \cdots \\
& \cong & (Q + P) + (Q + P) + \cdots \\
& \cong & \omega A'.
\end{array}
\]
Thus if $\omega A' \cong \omega A$, then we have the chain
\[
A \leqslant_{init} B' \leqslant_{init} 2B' \leqslant_{init} \omega A
\]
which is the setup from which we may prove $\mathbb{Z}A \cong \mathbb{Z}B' \cong \mathbb{Z}B$, as above. Thus $A \sim_{conv} B$. 

So we prove $\omega A \cong \omega A'$, i.e. $\omega (P + Q) \cong \omega (Q + P)$. First note that if we view $h$ as an initial embedding of $2B'$ into $\omega A'$ as above, then its composition $hg$ with the initial embedding $g: 2A \rightarrow 2B'$ witnesses $2A \leqslant_{init} \omega A'$, i.e. $2(P + Q) \leqslant_{init} \omega (Q + P)$. 

It follows $P + Q \leqslant_{init} \omega (Q + P)$. Since $Q + P$ is also initial in $\omega (Q + P)$, one of $P + Q$ and $Q + P$ is initial in the other. Either way, we may fix a left setup witnessing the relation and run the Euclidean algorithm. 

If the algorithm terminates at Stage 0 in an absorbed factor, then one of $P + Q \cong Q$ and $Q + P \cong P$ holds. In the first case we have
\[
\begin{array}{rcl}
\omega(P + Q) & \cong & P + Q + P + Q + \cdots \\
& \cong & Q + P + Q + P + \cdots \\
& \cong & \omega(Q + P),
\end{array}
\]
where in passing from the first to second line we have absorbed the initial $P$ into the adjacent copy of $Q$ at its right. Likewise, in the second case we also have $\omega(P + Q) \cong \omega(Q + P)$. 

If the algorithm terminates at a finite stage, we have routinely that $\omega P \cong \omega Q$. Then by Proposition \ref{omegaAcongomegaBthenomegaAplusBtoo} we have $\omega(P + Q) \cong \omega (Q + P)$. 

If the algorithm is non-terminating and the resulting division system is splitting, then routinely we have $\omega P \cong \omega Q$, and hence again $\omega(P + Q) \cong \omega(Q+P)$. 

Finally, suppose the algorithm is non-terminating and the resulting left division system $\mathcal{D} = \{A_k\}$ is non-splitting. Suppose $P = A_0$ and $Q = A_1$; the argument is similar in the reverse case. Then in particular $Q \leqslant_{init} P$. 

Since $2(P + Q) \leqslant_{init} \omega (Q + P)$, we have $P + Q + P \leqslant_{init} \omega (Q + P)$. Since $Q \leqslant_{init} P$, we have $P + Q + Q \leqslant_{init} P + Q + P$, and it follows $P + Q + Q \leqslant_{init} \omega (Q + P)$. Since also $(Q + P) + Q \leqslant_{init} \omega (Q + P)$, one of $P + Q + Q$ and $Q + P + Q$ is initial in the other. 

Thus the hypothesis of Lemma \ref{continuationlemma}.(i.) is satisfied relative to $P$ and $Q$. We conclude that the resulting real representation of $\mathbb{Z}P$ is sufficient, and then by \ref{sufficiencyimpliesIsosforomegaprods} that $\omega P \cong \omega Q$. Again by Proposition \ref{omegaAcongomegaBthenomegaAplusBtoo}, we have $\omega(P + Q) \cong \omega (Q+P)$. 

Thus in all cases we have $\omega A \cong \omega (P + Q) \cong \omega (Q + P) \cong \omega A'$. By the preceding discussion, this finishes the argument showing $A \sim_{conv}^b B$ implies $A \sim_{conv} B$, and hence the proof of (i.) in the statement of the theorem.

For (ii.): By (i.), $A \approx^b B$ implies both $A \sim_{init} B$ and $A \sim_{fin} B$, and hence $A \approx B$; the reverse implication is again clear. 
\end{proof}

From Theorem \ref{continuationthm} we can show that the converse to Proposition \ref{AsimstarBiffnAleqstarmBleqstarnApropn} holds.

\theoremstyle{definition}
\newtheorem{boundedarchimequivisequivtounboundedarchimeequiv}[lct]{Proposition}
\begin{boundedarchimequivisequivtounboundedarchimeequiv}\label{boundedarchimequivisequivtounboundedarchimeequiv}
For $\star \in \{init, fin, conv\}$, suppose the following conditions hold:
\begin{itemize}
\item[i.] For every $n \in \omega$, there is $m \in \omega$ such that $nA \leqslant_{\star} mB$; 
\item[ii.] For every $m' \in \omega$, there is $n' \in \omega$ such that $m'B \leqslant_{\star} n' A$. 
\end{itemize}
Then $A \sim_{\star} B$. 
\end{boundedarchimequivisequivtounboundedarchimeequiv}

\begin{proof}
For a fixed $\star \in \{init, fin, conv\}$, (i.) and (ii.) in particular imply that there are $m$ and $n'$ such that $3A \leqslant_{\star} mB$ and $3B \leqslant_{\star} n'A$. It follows that $A \lesssim_{\star}^b B \lesssim_{\star}^b A$, i.e. $A \sim_{\star}^b B$. Hence $A \sim_{\star} B$ by Theorem \ref{continuationthm}, as claimed. 
\end{proof}

Proposition \ref{archimequivalencetolesssimstar} below says that for $\star \in \{init, fin\}$, the converse to Proposition \ref{AlessapproxBiffnAlesssomemB} holds. To prove it, we need the following continuation result due to Tarski. 

\theoremstyle{definition}
\newtheorem{nAinitallAiffomegaAinitTarski}[lct]{Lemma}
\begin{nAinitallAiffomegaAinitTarski}\label{nAinitallAiffomegaAinitTarski}
(Tarski; \cite[1.44]{Tarski}) \phantom{.}
\begin{itemize}
    \item[i.] Suppose that for all integers $k \geq 1$, $kA \leqslant_{init} B$. Then $\omega A \leqslant_{init} B$;
    \item[ii.] Suppose that for all integers $k \geq 1$, $kA \leqslant_{fin} B$. Then $\omega^* A \leqslant_{fin} B$;
\end{itemize}

\begin{proof}
For (i.): If $A$ is splitting, then $\omega A \leqslant_{init} A$ (see e.g. the proof to Proposition \ref{bothomegaAleqBandomegaAconjomegaB}). In this case, $kA \leqslant_{init} B$ implies $\omega A \leqslant_{init} B$ already for $k = 1$. 

So assume $A$ is non-splitting. Label $B = [c, d]$ by its endcuts. We inductively build a cut sequence
\[
c = a_0 < a_1 < a_2 < \ldots
\]
in $B$ witnessing $\omega A \leqslant_{init} B$. At the first stage, since $A \leqslant_{init} B$ there are cuts $c = a_0 < a_1$ such that $[a_0, a_1] \cong A$. 

Suppose that after the $k$th stage we have a cut sequence
\[
c = a_0 < a_1 < \ldots < a_k
\]
in $B$ witnessing $kA \leqslant_{init} B$. Fix a cut sequence
\[
c = a_0' < a_1' < \ldots < a_k' < a_{k+1}'
\]
witnessing $(k+1)A \leqslant_{init} B$. It must be that $a_{k+1}' > a_k$. Otherwise, this sequence witnesses $(k+1)A \leqslant_{init} kA$, which implies $A$ is splitting by Theorem \ref{splittingdichthm}, a contradiction. 

If $a_k' = a_k$, then letting $a_{k+1} = a_{k+1}'$ we have that the sequence
\[
a_0 < a_1 < \ldots < a_k < a_{k+1}
\]
extends the sequence $a_0 < \ldots < a_k$ and witnesses $(k+1)A \leqslant_{init} B$. 

The other possibilities are that $a_k < a_k'$ and $a_k' < a_k$. Suppose first $a_k < a_k'$. Since $[c, a_k'] \cong [c, a_k] \cong kA$, we may fix an isomorphism $f: [c, a_k'] \rightarrow [c, a_k]$. Then $f$ is a compression of $[c, a_k']$ whose rightmost orbital $O_f(a_k')$ is therefore a decreasing $\omega^*$-orbital with initial jump $A_{a_k', f}$. Let $C = A_{a_k', f}$, so that $O_f(a_k') = \omega^*C$. 

We claim that $f^n(a_k') > a_{k-1}$ for all $n \in \omega$. If not, then for some large enough $N$ we have $f^N(a_k') \leq a_{k-1}$. But then $f^N \upharpoonright [c, a_k']$ is an initial embedding of $kA$ in $(k-1)A$, contradicting that $A$ is non-splitting. 

Let $l$ be the left endcut of the orbital $O_f(a_k')$. We have just shown $l \geq a_{k-1}$. Observe $[l, a_k] = [l, f(a_k')] \cong \omega^* C$ so that $\omega^*C$ is final in $[a_{k-1}, a_k]$. It follows that
\[
[a_{k-1}, a_k'] \cong [a_{k-1}, a_k] + [a_k, a_k'] \cong [a_{k-1}, a_k] + C \cong [a_{k-1}, a_k] \cong A.
\]
Thus the cut sequence
\[
a_0 < a_1 < \ldots < a_{k-1} < a_k' < a_{k+1}'
\]
witnesses $(k+1)A \leqslant_{init} B$ and extends original sequence $a_0 < \ldots < a_k$ except at $a_k$. We relabel $a_k'$ as $a_k$ and let $a_{k+1} = a_{k+1}'$ and move onto the next stage. 

The case when $a_k' < a_k$ is similar. 

Since each cut $a_k$ stabilizes after stage $k+2$ at the latest, the limit sequence $a_0 < a_1 < \ldots$ is well-defined and witnesses $\omega A \leqslant_{init} B$, as desired. 

The proof for (ii.) is symmetric.
\end{proof}
\end{nAinitallAiffomegaAinitTarski}

\theoremstyle{definition}
\newtheorem{archimequivalencetolesssimstar}[lct]{Proposition}
\begin{archimequivalencetolesssimstar}\label{archimequivalencetolesssimstar}
Fix $\star \in \{init, fin\}$ and suppose that for every $n \in \omega$, there is $m \in \omega$ such that $nA \leqslant_{\star} mB$. Then $A \lesssim_{\star} B$.
\end{archimequivalencetolesssimstar}

\begin{proof}
Suppose $\star = init$. Then since $mB \leqslant_{init} \omega B$ for all $m \in \omega$, the hypothesis implies that for every $n \in \omega$ we have $nA \leqslant_{init} \omega B$. By Lemma \ref{nAinitallAiffomegaAinitTarski} we have $\omega A \leqslant_{init} \omega B$, i.e. $A \lesssim_{\star} B$. 

The proof for $\star = fin$ is symmetric. 
\end{proof}

Since initial and final embeddings are in particular convex embeddings, the conditions $A \lesssim_{init}^b B$ and $A \lesssim_{fin}^b B$ are each stronger than $A \lesssim_{conv}^b B$. Hence $A \sim_{init}^b B$ and $A \sim_{fin}^b B$ each individually imply $A \sim_{conv}^b B$. By Theorem \ref{continuationthm}, the unbounded equivalence relations $A \sim_{init} B$ and $A \sim_{fin} B$ each individually imply $A \sim_{conv} B$, which is not immediate from the definition of these relations. 

Thus the implications in the following diagram hold:
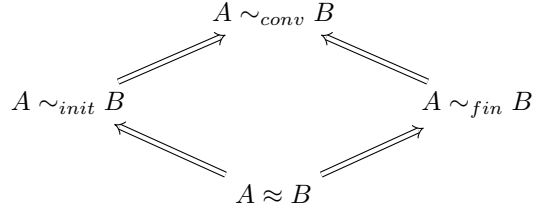
\[
\begin{tikzcd}
                                       & A \sim_{conv} B                                           &                                       \\
A \sim_{init} B \arrow[ru, Rightarrow] &                                                           & A \sim_{fin} B \arrow[lu, Rightarrow] \\
                                       & A \approx B \arrow[lu, Rightarrow] \arrow[ru, Rightarrow] &                                      
\end{tikzcd}
\]

By Proposition \ref{simstarandisosofinfdiscreteproducts}, the top two implications may be phrased arithmetically as follows.

\theoremstyle{definition}
\newtheorem{ericsthm}[lct]{Theorem}
\begin{ericsthm}\label{ericsthm}
\phantom{.}
\begin{itemize}
    \item[i.] $\omega A \cong \omega B$ implies $\mathbb{Z}A \cong \mathbb{Z}B$;
    \item[ii.] $\omega^* A \cong \omega^* B$ implies $\mathbb{Z}A \cong \mathbb{Z}B$.
\end{itemize}  
\end{ericsthm}

As noted above, Theorem \ref{ABcommuteifOmegasumsinitfin} is exactly the statement that $A \approx B$ is equivalent to $A + B \cong B + A$. Collecting our work above, we conclude this section with the following omnibus list of arithmetic equivalences to $A + B \cong B + A$. 

\theoremstyle{definition}
\newtheorem{ABcommutearithmeticequivalenceslist}[lct]{Theorem}
\begin{ABcommutearithmeticequivalenceslist}\label{ABcommutearithmeticequivalenceslist}
The following are equivalent:
\begin{itemize}
    \item[i.] $A + B \cong B + A$;
    \item[ii.] $\omega A \leqslant_{init} \omega B$ and $\omega^*A \leqslant_{fin} \omega^*B$, or $\omega B \leqslant_{init} \omega A$ and $\omega^* B \leqslant_{fin} \omega^* A$; 
    \item[iii.] $A + B \cong B + A \cong A$, or $A + B \cong B + A \cong B$, or $\omega A \cong \omega B$ and $\omega^* A \cong \omega^*B$;
    \item[iv.] $\mathbb{Z}A \leqslant_{cent} \mathbb{Z}B$ or $\mathbb{Z}B \leqslant_{cent} \mathbb{Z}A$;
    \item[v.] $\mathbb{Z}A \leqslant_{cent} 2B$, or $\mathbb{Z}B \leqslant_{cent} 2A$, or $\mathbb{Z}A \cong_{cent} \mathbb{Z}B$;
    \item[vi.] For every natural number $n$ there is $m$ such that $nA \leqslant_{init} mB$ and $nA \leqslant_{fin} mB$, or for every $n$ there is $m$ such that $ nB \leqslant_{init} mA$ and $nB \leqslant_{fin} mA$.
\end{itemize}
\end{ABcommutearithmeticequivalenceslist}

\begin{proof}
(i.) $\Leftrightarrow$ (ii.) is \ref{ABcommuteifOmegasumsinitfin}; (i.) $\Leftrightarrow$ (iii.) is \ref{revisedtarconjintext}; (iv.) $\Leftrightarrow$ (ii.) is given by \ref{AlessapproxBiffZAcentZBpropn}; (i.) $\Leftrightarrow$ (v.) is given by $\ref{AlessapproxBreformulations}$; by Proposition \ref{archimequivalencetolesssimstar}, (vi.) is equivalent to the assertion that $A \lesssim_{init} B$ and $A \lesssim_{fin} B$ or vice versa, which is equivalent to (ii.).
\end{proof}

\subsection{An archimedean valuation on $LO$} \label{subsect:archimedeanvaluationonLO} In preparation for our representability results in Section \ref{section:commutesemigroupsinLOrepn}, in this section we consider the map 
\[
A \mapsto [A]_{\approx}
\]
that associates to every $A \in LO$ its $\approx$-class $[A]_{\approx} = \{B \in LO: B \approx A\}$. This map may be viewed as a global valuation on $LO$ analogous to the valuation on an abelian orderable group $G$ that associates to each $g \in G$ the set of group elements $h \in G$ to which $g$ is archimedean comparable. We will show that the values $[B]_{\approx}$ lying $\lessapprox$-below a fixed value $[A]_{\approx}$ are \textit{linearly} ordered by $\lessapprox$. 

A binary relation $\leq$ on a class $K$ is a \textit{quasi-order} if it is reflexive and transitive. Given a quasi-order $\leq$ on $K$, the relation $\equiv$ on $K$ defined by $a \equiv b$ if $a \leq b$ and $b \leq a$ is an equivalence relation. We denote the $\equiv$-class of a given $a \in K$ by $[a]_{\equiv}$, or simply $[a]$. The relation induced by $\leq$ on the set of equivalence classes $K /{\equiv}$, which we also denote by $\leq$, is a partial order on $K/{\equiv}$. 

A quasi-order $\leq$ on $K$ is a \textit{quasi-linear-order} if $\leq$ is total, i.e. if $a \leq b$ or $b \leq a$ for all $a, b \in K$; equivalently, if $\leq$ is a linear order on $K/{\equiv}$.  

A quasi-order $\leq$ on $K$ is a \textit{quasi-tree-order} if for every $a \in K$, $\{b \in K: b \leq a\}$ is quasi-linearly-ordered by $\leq$; equivalently, if $\{[b] \in K / {\equiv}: [b] \leq [a]\}$ is linearly ordered by $\leq$. 

\theoremstyle{definition}
\newtheorem{lesssimandlessapproxaretreeorders}[lct]{Theorem}
\begin{lesssimandlessapproxaretreeorders}\label{lesssimandlessapproxaretreeorders}
The relations $\lesssim_{init}, \lesssim_{fin}$, and $\lessapprox$ are quasi-tree-orders on $LO$.
\end{lesssimandlessapproxaretreeorders}

\begin{proof}
Fix $A \in LO$, and suppose $B \lesssim_{init} A$ and $C \lesssim_{init} A$. Then $\omega B \leqslant_{init} \omega A$ and $\omega C \leqslant_{init} \omega A$. It follows that either $\omega B \leqslant_{init} \omega C$ or $\omega C \leqslant_{init} \omega B$, i.e. $B \lesssim_{init} C$ or $C \lesssim_{init} B$. Since $B$ and $C$ were arbitrary, $\lesssim_{init}$ is total below $A$. Since $A$ was arbitrary, $\lesssim_{init}$ is a quasi-tree-order on $LO$, as claimed. 

The proof for $\lesssim_{fin}$ is symmetric. 

Now suppose $B \lessapprox A$ and $C \lessapprox A$. We want to show that one of $B$ and $C$ is $\lessapprox$ below the other. 

Since $B \lesssim_{init} A$ and $C \lesssim_{init} A$, one of $B$ and $C$ is $\lesssim_{init}$ below the other. Without loss of generality, suppose $C \lesssim_{init} B$, so that $\omega C \leqslant_{init} \omega B$. 

Symmetrically, one of $B$ and $C$ is $\lesssim_{fin}$ below the other. If $C \lesssim_{fin} B$ then $C \lessapprox B$, and we are done.

So suppose $B \lesssim_{fin} C$. Then $\omega^*B \leqslant_{fin} \omega^* C$. By Proposition \ref{lesssiminitandlesssimfinrefinements}, either $\omega^* B = \omega^* C$ or $\omega^* B \leqslant_{fin} C$. The first case implies $C \lesssim_{fin} B$, which again gives $C \lessapprox B$. 

So suppose $\omega^* B \leqslant_{fin} C$. Since $\omega C \leqslant_{init} \omega B$ gives that $C$ embeds initially in a finite multiple of $B$, it follows from $\omega^* B \leqslant_{fin} C$ that $B$ is splitting. It moreover follows from these relations that $B$ and $C$ satisfy the hypotheses of Lemma \ref{2AconvnB2BconvmAsplittingiff}, hence $C$ is splitting as well. Now it is clear that $C \leqslant_{init} B$ and $B \leqslant_{fin} C$. Hence $B \cong C$ by \ref{CSBLO}, which gives $C \lessapprox B$, and we are again done. 

It follows $\lessapprox$ is a quasi-tree-order on $LO$, as claimed. 
\end{proof}

\section{Colorings of linear orders}\label{section:coloringsoflinearorders}

In this section we define the concept of a colored linear order and introduce some notation for handling such orders. Nearly all of the results in this paper for uncolored linear orders hold for colored linear orders as well, and by the same proofs, up to replacing every convex embedding by a colored one (see Definition \ref{coloredconvembeddingdefn} below). We will use a coloring in Section \ref{section:commutesemigroupsinLOrepn} to make a construction presented there as transparent as possible. Afterwards, we describe how the coloring may be dispensed with via a slightly more complicated construction. 

Given a linear order $X$, a \textit{coloring} of $X$ is a function $\mathsf{c}$ such that $\textrm{dom}(\mathsf{c}) = X$. For each $x \in X$, $\mathsf{c}(x)$ is the \textit{color} of $x$. A \textit{colored linear order} is a pair $(X, \mathsf{c})$ such that $X$ is a linear order and $\mathsf{c}$ is a coloring of $X$. When $\mathsf{c}$ is understood, we write $X$ instead of $(X, \mathsf{c})$ for a colored linear order. The class of colored linear orders is denoted $\mathsf{c}LO$.

Given a linear order $X$ and a collection $\{(I_x, \mathsf{c}_x): x \in X\}$ of colored linear orders indexed by $X$, the \textit{colored replacement} of $X$ by the orders $(I_x, \mathsf{c}_x)$ is the replacement $X(I_x)$ equipped with the coloring $\mathsf{c}$ defined by $\mathsf{c}(x, i) = \mathsf{c}_x(i)$. 

As with uncolored orders, we use summation notation for colored replacements of discrete orders. In particular, if $X = (X, \mathsf{c})$ and $Y = (Y, \mathsf{d})$ are colored orders, then $X + Y$ denotes the colored order $(X + Y, \mathsf{e})$ where $\mathsf{e}$ is defined by the rule $\mathsf{e}(z) = \mathsf{c}(z)$ if $z \in X$ and $\mathsf{e}(z) = \mathsf{d}(z)$ if $z \in Y$. 

\theoremstyle{definition}
\newtheorem{coloredconvembeddingdefn}[lct]{Definition}
\begin{coloredconvembeddingdefn}\label{coloredconvembeddingdefn}
A \textit{colored convex embedding} between $(X, \mathsf{c})$ and $(Y, \mathsf{d})$ is a convex embedding $f: X \rightarrow Y$ such that $\mathsf{c}(x) = \mathsf{d}(f(x))$ for all $x \in X$.  

We write $(X, \mathsf{c}) \leqslant_{conv} (Y, \mathsf{d})$ if there exists a colored convex embedding between $X$ and $Y$. 
\end{coloredconvembeddingdefn}

We emphasize that colored convex embeddings $f: (X, \mathsf{c}) \rightarrow (Y, \mathsf{d})$ are required to be color-preserving, not merely color-class preserving. That is, they satisfy 
\[
\mathsf{c}(x) = \mathsf{d}(f(x))
\]
which is stronger than the condition
\[
\mathsf{c}(x) = \mathsf{c}(x') \Leftrightarrow \mathsf{d}(f(x)) = \mathsf{d}(f(x')).
\]

When $X$ and $Y$ are understood to carry colorings, we similarly copy the notation $X \leqslant_{init} Y$, $X \leqslant_{fin} Y$, and $X \cong Y$ from the uncolored setting to mean that there exists a colored initial embedding, colored final embedding, and colored isomorphism from $X$ to $Y$, respectively.

We will need the following rigidity result for the construction in Section \ref{section:commutesemigroupsinLOrepn} below. It says that if $G$ is a group of translations of $\mathbb{R}$ and we color each point in $\mathbb{R}$ by its orbit equivalence class under the action by $G$, then colored intervals $I \subseteq \mathbb{R}$ cannot be compressed by colored convex self-embeddings of $\mathbb{R}$.

\theoremstyle{definition}
\newtheorem{groupsoftranslationsarefull}[lct]{Theorem}
\begin{groupsoftranslationsarefull}\label{groupsoftranslationsarefull}
Fix a strict subgroup $G \subsetneq (\mathbb{R}, +)$. Let $G$ act on $\mathbb{R}$ by translations, and let $E_G$ denote the orbit equivalence relation of this action. 

Define a coloring $\mathsf{c}$ of $\mathbb{R}$ by the rule $\mathsf{c}(x) = [x]_{E_G}$. View $\mathbb{R}$ as equipped with this coloring. 

Suppose $I \subseteq \mathbb{R}$ is an interval and $f: I \rightarrow \mathbb{R}$ is a colored convex-embedding. Then $f$ is the restriction to $I$ of some translation in $G$. That is, there is $t \in G$ such that for all $x \in I$, $f(x) = x + t$. 
\end{groupsoftranslationsarefull}

\begin{proof}
Fix $x_0 \in I$ and let $t = f(x_0) - x_0$. Then $t \in G$ since $f$ is a colored convex embedding of $I$. Suppose there is $x_1 \in I$ such that $f(x_1) - x_1 \neq t$. Let $t' = f(x_1) - x_1$. 

We assume that $x_0 < x_1$ and $t < t'$; the other cases are handled similarly. Let $g(x) = f(x) - x$. Since $f$ is continuous (as it is increasing and onto an interval), $g$ is continuous as well. Thus for every $r \in [t, t']$ there is $z \in [x_0, x_1]$ such that $g(z) = r$. Since $G$ is a strict subgroup of $(\mathbb{R}, +)$, its complement $\mathbb{R} \setminus G$ is dense in $\mathbb{R}$. In particular, there is $r_0 \in [t, t']$ such that $r_0 \not\in G$. Fix $z_0 \in [x_0, x_1]$ such that $g(z_0) = r_0$. Then $f(z_0) = z_0 + r_0$ and we have 
\[
\mathsf{c}(z_0) = [z_0]_{E_G} \neq [z_0 + r_0]_{E_G} = [f(z_0)]_{E_G} = \mathsf{c}(f(z_0)),\]
a contradiction, as $f$ is color-preserving. 
\end{proof}

\section{Representing commutative semigroups in $(LO, +)$}\label{section:commutesemigroupsinLOrepn}

In this section we characterize the commutative semigroups representable in $(LO, +)$; see Theorem \ref{representablesemigroupstheorem} below. 

\subsection{Value semigroups}\label{subsect:valuesemigroups}

\theoremstyle{definition}
\newtheorem{semigroupisrepresentableinLOdefn}[lct]{Definition}
\begin{semigroupisrepresentableinLOdefn}\label{semigroupisrepresentableinLOdefn}
A semigroup $(S, \oplus)$ is \textit{representable} in $(LO, +)$ if there is a map $\iota: S \rightarrow LO$ such that for all $a, b \in S$:
\begin{itemize}
\item[i.)] $a \neq b$ implies $\iota(a) \not\cong \iota(b)$,
\item[ii.)] $\iota(a \oplus b) \cong \iota(a) + \iota(b)$. 
\end{itemize}
The map $\iota$ is a \textit{representation} of $S$ in $LO$. 
\end{semigroupisrepresentableinLOdefn}

The ordered sum $+$ on $LO$ naturally lifts to the class $LO/{\cong}$ of linear order types, by defining 
\[
[A]_{\cong} + [B]_{\cong} = [A + B]_{\cong}.
\]
Under this operation $(LO/{\cong}, +)$ is a class semigroup. 

It follows from Definition \ref{semigroupisrepresentableinLOdefn} that if $\iota: S \rightarrow LO$ is a representation, then the induced map $\iota: S \rightarrow LO/{\cong}$ is a semigroup embedding. 

A subclass $K \subseteq LO$ is a \textit{class of order types} if it is closed under isomorphism.

\theoremstyle{definition}
\newtheorem{classofordertypesrepresentsdefn}[lct]{Definition}
\begin{classofordertypesrepresentsdefn}\label{classofordertypesrepresentsdefn}
A class of order types $K$ \textit{represents} a semigroup $(S, \oplus)$ if there is a representation $\iota: S \rightarrow K$ such that each $X \in K$ is isomorphic to $\iota(s)$ for some $s \in S$.
\end{classofordertypesrepresentsdefn}

If $\iota: S \rightarrow K$ is a representation witnessing that $K$ represents $S$, then the induced map $\iota: S \rightarrow K/{\cong}$ is bijective, and hence a semigroup isomorphism. Thus $K$ represents $S$ if and only if $(S, \oplus)$ and $(K/{\cong}, +)$ are isomorphic. 

The \textit{trivial semigroup} is the semigroup on one element. Observe that a class $K$ represents the trivial semigroup if and only if $K$ is the order type of a splitting linear order, i.e. if there is $A \in LO$ such that $A \cong A + A$ and $K = \{B \in LO: B \cong A\}$.

Let $\mathbb{R}_{> 0} = (\mathbb{R}_{> 0}, +)$ denote the additive semigroup of positive real numbers.

\theoremstyle{definition}
\newtheorem{strictlyfulldefn}[lct]{Definition}
\begin{strictlyfulldefn}\label{strictlyfulldefn}
Suppose $S$ is a subsemigroup of $(\mathbb{R}_{>0}, +)$. 
\begin{itemize}
    \item[i.] $S$ is \textit{strict} if $S \neq \mathbb{R}_{>0}$. 
    \item[ii.] $S$ is a \textit{halfgroup} if $-S \cup \{0\} \cup S$ is a subgroup of $(\mathbb{R}, +)$, where $-S = \{-x: x \in S\}$.
\end{itemize}
\end{strictlyfulldefn}

Thus $S$ is a strict halfgroup if it is the (strictly) positive cone of a strict subgroup of $(\mathbb{R}, +)$. More generally, a semigroup $S$ is said to be a strict halfgroup if it is isomorphic to a strict halfgroup $S'$ of $\mathbb{R}_{> 0}$.

For $A \in LO$, observe that the $\approx$-class $[A]_{\approx}$ is a class of order types. 

\theoremstyle{definition}
\newtheorem{approxclassesclosedundersum}[lct]{Lemma}
\begin{approxclassesclosedundersum}\label{approxclassesclosedundersum}
$[A]_{\approx}$ is closed under $+$. 
\end{approxclassesclosedundersum}
\begin{proof}
Given $B, C \in [A]_{\approx}$, we have $\omega A \cong \omega B \cong \omega C$ and $\omega^* A \cong \omega^* B \cong \omega^* C$ by Proposition \ref{simstarandisosofinfdiscreteproducts}. Hence $\omega(B+C) \cong \omega B \cong \omega A$ and $\omega^*(B+C) \cong \omega^*B \cong \omega^*A$ by Proposition \ref{omegaAcongomegaBthenomegaAplusBtoo}. Thus $B + C \approx A$, i.e. $B + C \in [A]_{\approx}$. The lemma follows. 
\end{proof}

\theoremstyle{definition}
\newtheorem{Aapproxclasseitherfullystrictortrivialrepn}[lct]{Theorem}
\begin{Aapproxclasseitherfullystrictortrivialrepn}\label{Aapproxclasseitherfullystrictortrivialrepn}
Fix $A \in LO$.
\begin{itemize}
    \item[i.] If $A$ is non-splitting, then $[A]_{\approx}$ represents a strict halfgroup $S \subseteq \mathbb{R}_{> 0}$;
    \item[ii.] If $A$ is splitting, then $[A]_{\approx}$ represents the trivial semigroup. 
\end{itemize}
\end{Aapproxclasseitherfullystrictortrivialrepn}

\begin{proof}
Suppose first that $A$ is splitting. To prove (ii.) it suffices to check that $[A]_{\approx} = [A]_{\cong}$, i.e. $B \approx A$ if and only if $B \cong A$. 

The backward direction is immediate, so suppose $B \approx A$. Then Proposition \ref{simstarandisosofinfdiscreteproducts} gives that $\omega A \cong \omega B$ and $\omega^* A \cong \omega^* B$. From these isomorphisms it follows that any finite multiple $nA$ of $A$ is initial and final in some finite multiple $mB$ of $B$, and vice versa. Hence $A$ and $B$ satisfy the hypotheses of Lemma \ref{2AconvnB2BconvmAsplittingiff}, and so $B$ is splitting. Now we have that $A$ is initial and final in some $mB \cong B$, and $B$ is initial and final in some $nA \cong A$. Hence $A \cong B$ by \ref{CSBLO}, which finishes the proof for (ii.).

Now suppose $A$ is non-splitting. There are two cases to consider. First, suppose there is a non-terminating division system $\mathcal{D} = \{A_k: k \in \omega\}$ with $A = A_0$ that is either symmetric or rational. Since $A$ is non-splitting, $\mathcal{D}$ is non-splitting. We have the associated representations:
\[
\begin{array}{rcl}
\mathbb{Z}A & \cong & \mathbb{R}(K_{[x]}) \\
K_0 & \cong & L + R \\
A & \cong & R + (0, 1)(K_{[x]}) + L.
\end{array}
\]
These representations are sufficient, since $\mathcal{D}$ is either symmetric or rational. 

Expanding, we have:
\[
\begin{array}{rcl}
\mathbb{R}(K_{[x]}) & \cong & \cdots + K_{-1} + (-1, 0)(K_{[x]}) + K_0 + (0, 1)(K_{[x]}) + K_1 + \cdots \\
 & \cong & \cdots + R + (-1, 0)(K_{[x]}) + L + R + (0, 1)(K_{[x]}) + L + \cdots \\
& \cong & \cdots + A + A + \cdots \\
& \cong & \mathbb{Z}A. 
\end{array}
\]

Let $c$ denote the cut at the central $+$ sign between in $K_0 = L + R$.

By Propositions \ref{fhatalwaysirredwhenDnonsplitting} and \ref{autosofRKxaretranslationsDnonsplitting}, for each automorphism $f$ of $\mathbb{R}(K_{[x]})$, the condensed automorphism $\hat{f}$ of $\mathbb{R}$ is a translation. 

Let $G = \{\hat{f}: f \in \textrm{Aut}(\mathbb{R}(K_{[x]})\}$ be the group of these translations. Viewing $G$ as a subgroup of $(\mathbb{R}, +)$, let $S = \{s \in G: s > 0\}$ be the corresponding halfgroup. By Proposition \ref{autosofRKxaretranslationsDnonsplitting}, $G$ is a strict subgroup of $(\mathbb{R}, +)$, and hence $S$ is a strict halfgroup. We are going to show that $[A]_{\approx}$ represents $S$. 

Suppose $s \in S$ and $f \in \textrm{Aut}(\mathbb{R}(K_{[x]})$ is a representative of $s$, i.e. $\hat{f} = s$. Let $B_f = [c, f(c)]$. We have
\[
B_f \cong R' + (0, s)(K_{[x]}) + L'
\]
Where $R'$ is the final segment of $K_0$ after the cut $c$, and $L'$ is the initial segment of $K_s$ preceding $f(c)$. Then $R' = R$. Since $f[K_0] = K_s$, it follows $L'$ is the image under $f$ of the initial segment of $K_0$ preceding $c$, i.e. $L' = f[L]$. Thus $L' \cong L$, and we have
\[
B_f \cong R + (0, s)(K_{[x]}) + L.
\]
Note that the expression on the right depends only on $s$. Thus for any other $g \in \textrm{Aut}(\mathbb{R}(K_{[s]})$ with $\hat{g} = \hat{f} = s$, we have
\[
B_g \cong B_f \cong R + (0, s)(K_{[x]}) + L.
\]

Let $d$ denote the cut at the right of $\mathbb{R}(K_{[x]})$. The final segment $[c, d](K_{[x]}) \cong \omega A$ of $\mathbb{R}(K_{[x]})$ decomposes as 
\[
\begin{array}{rcl}
[c, d](K_{[x]}) & \cong & [c, f(c)] + [f(c), f^2(c)] + \cdots \\
& \cong & \omega B_f.
\end{array}
\]
Symmetrically, we have $\omega^*B_f \cong \omega^*A$. Hence $B_f \approx A$. 

Combining these observations shows that the rule
\begin{equation}\label{dunhill}
s \mapsto [B_f]_{\cong}
\end{equation}
that associates to each $s \in S$, the isomorphism class $[B_f]_{\cong}$ of any automorphism $f$ with $\hat{f} = s$, is a well-defined map from $S$ into the class $[A]_{\approx}/{\cong}$ of order types $[B]_{\cong}$ of orders $B \approx A$. 

If $\hat{f} = s < s' = \hat{f'}$, then $\ell(B_f) = s < s' = \ell(B_{f'})$. Hence $B_f \not\cong B_{f'}$ by Proposition \ref{lengthisinvariantwhenDnonsplitting}. Thus the map \ref{dunhill} is injective. 

To show it is surjective, fix $B \approx A$. Let $F: \mathbb{Z}B \rightarrow \mathbb{Z}A$ be an embedding of $\mathbb{Z}B$ into $\mathbb{Z}A$ that is centered at $c$. Such an embedding exists by Proposition \ref{simstarandisosofinfdiscreteproducts}. Identify $\mathbb{Z}B$ with its image, and let $f$ be the automorphism of $\mathbb{Z}B$ that sends each copy of $B$ onto the copy to its right. 

Then in particular $[c, f(c)] \cong B$. Since $\mathbb{Z}B \cong \mathbb{Z}A \cong \mathbb{R}(K_{[x]})$ we may identify $f$ with an automorphism of $\mathbb{R}(K_{[x]})$. Then $f$'s condensed form $\hat{f}$ is increasing and irreducible on $\mathbb{R}$. Hence $\hat{f} = s \in S$, and clearly $B_f \cong B$. Thus the map \ref{dunhill} is surjective. 

Finally, we show the rule \ref{dunhill} is a semigroup isomorphism of $S$ with the class of types $[A]_{\approx}/{\cong}$ under the ordered sum. Fix $s = \hat{f}$ and $t = \hat{g}$ in $S$, and fix $h$ such that $\hat{h} = s+t$. Then by above
\[
\begin{array}{rcl}
B_h & \cong & R + (0, s + t)(K_{[x]}) + L \\
& \cong & R + (0, s)(K_{[x]}) + K_s + (s, s+t)(K_{[x]}) + L \\
& \cong & R + (0, s)(K_{[x]}) + L + R + (s, s+t)(K_{[x]}) + L \\
& \cong & R + (0, s)(K_{[x]}) + L + R + (0, t)(K_{[x]}) + L \\
& \cong & B_f + B_g,
\end{array}
\]
where the isomorphisms $K_s \cong K_0 \cong L+R$ and 
\[
R + (s, s+t)(K_{[x]}) + L \cong R + (0, t)(K_{[x]}) + L
\]
follow from the fact that $s \in S$. Thus the rule \ref{dunhill} respects the semigroup operations of $+$ $(\mathbb{R}$ sum) on $S$ and $+$ (ordered sum) on $[A]_{\approx}/{\cong}$, and hence is an isomorphism of $S$ with $[A]_{\approx}/{\cong}$. It follows $[A]_{\approx}$ represents $S$. 

It remains to consider the case when there is no symmetric or rational non-terminating system $\mathcal{D}$ on $A$. We check in this case that $[A]_{\approx}$ represents the discrete halfgroup $\mathbb{Z}_{>0}$. 

We claim that there is a linear order $C \approx A$ such that for every $A' \approx A$ we have $A' \cong nC$ for some $n \in \mathbb{Z}_{>0}$ (i.e. every element of $[A]_{\approx}$ is isomorphic to a finite multiple of $C$). 

If this is true with $C = A$, we are done. Otherwise, there is $B \approx A$ that is not a multiple of $A$. Fix a symmetric division setup witnessing $A + B \cong B + A$ and run the Euclidean algorithm. 

This algorithm must terminate since there is no non-terminating symmetric system on $A$. It cannot terminate at the $0$th stage in an absorbed factor, as this would give that $A$ and $B$ are splitting by Proposition \ref{bothomegaAleqBandomegaAconjomegaB}, contradicting our hypothesis. Thus it terminates at some finite stage, and we may assume since the setup was symmetric that it terminates in a divisor $C$. We have $A \cong nC$ and $B \cong mC$. Notice that we must have $n > 1$, since by assumption $B$ is not a finite multiple of $A$. Notice $C \approx A$. 

If every $A' \approx A$ is a finite multiple of $C$, then we are done. Otherwise, write $C = A_1$, $n = n_1$ and find $B \approx A_1 \approx A$ that is not a multiple of $A_1$. Run the algorithm on a symmetric setup witnessing $A_1 + B \cong B + A_1$ to get $C$ and $n_2 > 1$ such that $A_1 \cong n_2C$. And continue. 

If this process does not terminate, then we obtain a sequence
\[
\begin{array}{rcl}
A & \cong & n_1 A_1 \\
A_1 & \cong & n_2 A_2 \\
A_2 & \cong & n_3 A_3 \\
& \vdots &
\end{array}
\]
which yields a non-terminating rational system on $A$. Hence this process must terminate, and thus yield the desired $C$. 

Since $C$ is non-splitting the map
\[
n \mapsto nC
\]
clearly witnesses that $[A]_{\approx}$ represents $\mathbb{Z}_{>0}$. We are done.
\end{proof}

\subsection{Representing replacement semigroups}\label{subsect:representingreplacementsemigroups}

\theoremstyle{definition}
\newtheorem{replacementsemigroupdefn}[lct]{Definition}
\begin{replacementsemigroupdefn}\label{replacementsemigroupdefn}
Suppose that $X$ is a linear order and $\{(S_x, \oplus_x): x \in X\}$ is a collection of semigroups indexed by $X$. The \textit{replacement semigroup} $X(S_x)$ is the semigroup with underlying set $\{(x, s): x \in X, s \in S_x\}$ and semigroup operation $\oplus$ defined by:
\[
\begin{array}{rcl}
    (x, s) \oplus (y, t) & = & \left\{   \begin{array}{ll}
                                        (x, s) & x > y \\
                                        (y, t) & x < y \\
                                        (x, s \oplus_x t) & x = y.
                                        \end{array}  
                                        \right.
\end{array}
\]
\end{replacementsemigroupdefn}

Notice that if $S_x$ is commutative for every $x \in X$, then $X(S_x)$ is commutative.

\theoremstyle{definition}
\newtheorem{LOsemigroupdefn}[lct]{Definition}
\begin{LOsemigroupdefn}\label{LOsemigroupdefn}
A semigroup $(S, \oplus)$ is an \textit{$LO$ semigroup} if $S$ is isomorphic to a replacement semigroup $X(S_x)$ such that for every $x \in X$, either $S_x$ is a strict halfgroup of $\mathbb{R}_{>0}$, or $S_x$ is trivial. 
\end{LOsemigroupdefn}

A commutative semigroup $(S, \oplus)$ is called \textit{naturally totally ordered} if the non-strict partial ordering relation $\leq$ defined by the rule
\[
\textrm{$s \leq t$ if $s = t$ or there is $r \in S$ such that $s \oplus r = t$}
\]
is total, i.e. a non-strict linear ordering of $S$. Naturally totally ordered semigroups were defined by Clifford in \cite{Clifford}. $LO$ semigroups are easily seen to be naturally totally ordered. 

The representation theorem \ref{representablesemigroupstheorem} asserts that the commutative semigroups representable in $(LO, +)$ are precisely the subsemigroups of $LO$ semigroups. We first establish the backward direction of the theorem. Since subsemigroups of representable semigroups are representable, it suffices to show that an arbitrary $LO$ semigroup $X(S_x)$ may be represented in $(LO, +)$. The proof is via a Hahn-embedding-theorem-type construction, which we now describe. 

\subsubsection{A Hahn construction for representing $LO$ semigroups} Suppose that $X$ is a linear order, and for every $x \in X$ we are given a linear order $A_x$ with distinguished element $0_x$. 

Let $\prod_{x \in X} A_x$ denote the $X$-indexed cartesian product of the orders $A_x$, i.e. the set of all functions $u: X \rightarrow \bigcup_x A_x$ satisfying $u(x) \in A_x$. If there is a linear order $A$ such that $A_x = A$ for all $x \in X$, then $\prod_x A_x = A^X$.

For $u \in \prod_x A_x$, the \textit{support of $u$} is $\{x \in X: u(x) \neq 0_x\}$, denoted $\textrm{spt}(u)$. 

Recall that a linear order $Y$ is \textit{well-ordered} (or \textit{a well-order}) if every non-empty suborder of $Y$ has a least element. Recall also that suborders of well-orders are well-orders, and if $Y$ and $Z$ are well-ordered suborders of a (not necessarily well-ordered) linear order $X$, then $Y \cup Z$ is also well-ordered. 

\theoremstyle{definition}
\newtheorem{AsubxtoXexponentreplacementdefn}[lct]{Definition}
\begin{AsubxtoXexponentreplacementdefn}\label{AsubxtoXexponentreplacementdefn}
Define 
\[
\textrm{$(A_x)^X = \{u \in \prod_x A_x: \textrm{spt$(u)$ is well-ordered}\}$}.
\]
We call $(A_x)^X$ the \textit{$X$-power} of the orders $\{A_x\}$.
\end{AsubxtoXexponentreplacementdefn}

If $u, v \in (A_x)^X$, then $\{x \in X: u(x) \neq v(x) \}$ is a subset of $\textrm{spt}(u) \cup \textrm{spt}(v)$ and hence is also well-ordered. Thus the relation $<_{lex}$ on $(A_x)^X$, defined by 
\[
\textrm{$u <_{lex} v$ when $x_0$ is least in $\{x \in X: u(x) \neq v(x)\}$ and $u(x_0) < v(x_0)$}
\]
is well-defined and total. And it is straightforward to check that in fact $<_{lex}$ linearly orders $(A_x)^X$. We view $(A_x)^X$ as equipped with this ordering. If $X$ is well-ordered, then $(A_x)^X = \prod_{x \in X} A_x$ equipped with the usual lexicographical ordering.

We will view an element $u \in (A_x)^X$ as a generalized sequence, think of $u(x)$ as the $x$th entry of $u$, and denote it by $u_x$. 

In the same spirit, we may also consider initial sequences of elements $u \in (A_x)^X$ and the intervals they determine. 

Suppose $(I, J)$ is a cut in $X$. Given $r \in (A_x)^I$ (where the index variable $x$ in this expression is understood to now range over $I$) and $u \in (A_x)^J$, we write $ru$ for the function with domain $X$ satisfying $ru \upharpoonright I = r$ and $ru \upharpoonright J = u$. Observe that $ru \in (A_x)^X$. 

Conversely, any $v \in (A_x)^X$ is decomposed uniquely as $ru$ for some $r \in (A_x)^I$ and $u \in (A_x)^J$. Said another way, for any cut $(I, J)$ in $X$, we have that $(A_x)^X$ is isomorphic to $(A_x)^I \times (A_x)^J$ with its usual $2$-fold lexicographical ordering. 

Given $r \in (A_x)^I$, define
\[
(A_x)_r^X = \{v \in (A_x)^X: \,\, \textrm{$v = ru$ for some $u \in (A_x)^J$}\},
\]
i.e. $(A_x)_r^X$ is the set of elements in $(A_x)^X$ that have $r$ as an initial sequence. Observe that $(A_x)_r^X$ is convex in $(A_x)^X$. Observe also that for any $r, s \in (A_x)^I$ we have $(A_x)_r^X \cong (A_x)_s^X$ via the map $ru \mapsto su$ for $u \in (A_x)^J$.

For $S$ a strict halfgroup of $\mathbb{R}_{> 0}$, we write $E_S$ for the orbit equivalence relation on $\mathbb{R}$ of the corresponding group $G = -S \cup \{0\} \cup S$, i.e. the relation defined by the rule $x E_S y$ if $y = x \pm s$ for some $s \in S \cup \{0\}$. We write $E_{\mathbb{R}}$ for the equivalence relation on $\mathbb{R}$ with a single equivalence class, i.e. the relation defined by the rule $x E_{\mathbb{R}} y$ for all $x, y \in \mathbb{R}$. 

For the remainder of this section, let $S = X(S_x)$ be a fixed $LO$ semigroup. 

For every $x \in X$ such that $S_x$ is a strict halfgroup, by replacing $S_x$ with a rescaled isomorphic copy if necessary, we may assume $1 \in S_x$. 

To show that $X(S_x)$ can be represented in $(LO, +)$, we first construct a linear order $Z$ by modifying the $(A_x)^X$ construction. The orders $I_{(x, s)} \in LO$ that will represent the elements $(x, s) \in X(S_x)$ in our representation theorem will appear as a $\lessapprox$-chain of intervals in $Z$. 

Given $x \in X$, if $S_x$ is a strict halfgroup, let $E_x$ denote $E_{S_x}$; if $S_x$ is trivial, let $E_x = E_{\mathbb{R}}$. Define
\[
\begin{array}{rcccl}
    A_x & = & [0]_{E_x} & = & \left\{   \begin{array}{ll}
                                        -S_x \cup \{0\} \cup S_x & \textrm{if $S_x$ a strict halfgroup;} \\
                                        \mathbb{R} & \textrm{if $S_x$ trivial,} \\
                                        \end{array}  
                                        \right.
\end{array}
\]
and let $0_x = 0$ (i.e., the distinguished element $0_x$ of every $A_x$ is the $0$ of $\mathbb{R}$). Since $1 \in S_x$ whenever $S_x$ is a strict halfgroup, we have $1 \in A_x$ for every $x \in X$. 

Let $Y = (A_x)^{X^*}$. Notice here that we are taking the power over the reverse $X^*$ of $X$.

Label $X^* = [c, d]$ by its endcuts. For a given $y \in X^*$ and initial sequence $r \in (A_x)^{[c, y)}$, the interval $Y_r$ decomposes as the set of intervals $\{Y_{ra}: a \in [0]_{E_y}\}$. These intervals are lexicographically ordered in order type $[0]_{E_y} = A_y$. 

To construct the order $Z$ from $Y$, for every such $r$ we insert points between the intervals $Y_{ra}$ corresponding to points in $\mathbb{R} \setminus A_y$, and then color these points in a way that will allow us to show that the orders $I_{(y, s)}$ and $I_{(z, t)}$ defined below, representing the points $(y, s)$ and $(z, t)$ in $X(S_x)$, are non-isomorphic whenever $(y, s) \neq (z, t)$. 

More precisely, define 
\[
\begin{array}{rcl}
T & = & \{r: \,\, \textrm{$r$ is a function with $\textrm{dom}(r) = [c, y]$ for some $y \in X^*$}, \\
& & \textrm{$r \upharpoonright [c, y) \in (A_x)^{[c, y)}$}, \\
& & \textrm{$r(y) \in \mathbb{R} \setminus A_y$}\}. 
\end{array}
\]

Equip $T$ with a coloring $\mathsf{c}$ such that for all $r, r' \in T$ we have:
\begin{itemize}
    \item[i.] if $\textrm{dom}(r) \neq \textrm{dom}(r')$ then $\mathsf{c}(r) \neq \mathsf{c}(r')$;
    \item[ii.] if $\textrm{dom}(r) = \textrm{dom}(r') = [c, y]$, then $\mathsf{c}(r) = \mathsf{c}(r')$ if and only if $r(y) E_{S_y} r'(y)$. 
\end{itemize}
We may define such a $\mathsf{c}$ explicitly by letting $\mathsf{c}(r)$ be the ordered pair $(y, [r(y)]_{E_{S_y}})$ whenever $\textrm{dom}(r) = [c, y]$. 

Let $Z = Y \cup T$. Observe that any pair of distinct sequences $u, v \in Z$ have a coordinate of least difference, and moreover the lexicographic ordering on $Z$ extends the lexicographic ordering on $Y$ and linearly orders $Z$. 

We extend the coloring $\mathsf{c}$ from $T$ to $Z$ by coloring every point in $Y$ black, where black is a color different from the colors assigned to points in $T$. For $I$ an initial segment of $X^*$ and $r \in (A_x)^I$, let $T_r$ denote the set of $u \in T$ such that $I$ is initial in $\textrm{dom}(u)$ and $u \upharpoonright I = r$. Let $Z_r$ denote $Y_r \cup T_r$. 

Given $y \in X^*$ and an initial sequence $r \in (A_x)^{[c, y)}$, the interval $Z_r$ decomposes as the set of intervals $\{Z_{ra}: a \in A_y\}$, which are lexicographically ordered in order type $[0]_{E_y} = A_y$, and the set of points $\{rb: b \in \mathbb{R} \setminus A_y\}$ which are interspersed between the intervals $Z_{ra}$ in order type $\mathbb{R} \setminus [0]_{E_y}$:
\[
Z_r = \{Z_{ra}: a \in A_y\} \cup \{rb: b \in \mathbb{R} \setminus A_y\}.
\]

We may also explicitly describe the elements of $Z_r$ as sequences. 

\theoremstyle{definition}
\newtheorem{sequencesinZlemma}[lct]{Lemma}
\begin{sequencesinZlemma}\label{sequencesinZlemma}
Fix $y \in X^*$ and $r \in (A_x)^{[c, y)}$. The elements of $Z_r$  have exactly one of the following forms:
\begin{itemize}
    \item[i.] $rb$, where $b \in \mathbb{R} \setminus A_y$;
    \item[ii.] $ramb$, where $a \in A_y$, $m \in (A_x)^{(y, y')}$ for some $y' > y$ in $X^*$, and $b \in \mathbb{R} \setminus A_{y'}$;
    \item[iii.] $rau$, where $a \in A_y$ and $u \in (A_x)^{(y, d]}$. 
\end{itemize}
\end{sequencesinZlemma}

\begin{proof}
Sequences in $T_r$ have exactly one of the forms (i.) and (ii.), and sequences in $Y_r$ have form (iii.). 
\end{proof}

If $S_y$ is trivial, then $\mathbb{R} \setminus A_y = \emptyset$. Thus only when $S_y$ is a strict halfgroup are there are elements in $Z_r$ of form (i.).

Let $\overline{0}$ denote the identically zero $X^*$-sequence. For $I$ an interval in $X^*$, we write $\overline{0}^I$ for the identically zero sequence in $(A_x)^I$. For a given $y \in X^*$ and $s \in A_y$, we write 
\[
\overline{0}^{[c, y)}s\overline{0}^{(y, d]}
\]
for the sequence with $s$ in the $y$th coordinate and $0$s elsewhere. Note that all such sequences belong to $(A_x)^{X^*} = Y$, and hence also to $Z$.

We may now define the orders $I_{(y, s)}$ that appear in our representation. They are all half-open intervals in $Z$ of the form $[\overline{0}, u)$ for some $u$ of the form $\overline{0}^{[c, y)}s\overline{0}^{(y, d]}$.

\theoremstyle{definition}
\newtheorem{Ixsrepresentingordersdefn}[lct]{Definition}
\begin{Ixsrepresentingordersdefn}\label{Ixsrepresentingordersdefn}
Given $(y, s) \in X(S_x)$, define an interval $I_{(y, s)}$ of $Z$ as follows:
\begin{itemize}
\item[i.] if $S_y$ is a strict halfgroup, let $I_{(y, s)} = [\overline{0}, \overline{0}^{[c, y)}s\overline{0}^{(y, d]})$;
\item[ii.] if $S_y = \{s\}$ is trivial, let $I_{(y, s)} = [\overline{0}, \overline{0}^{[c, y)}1\overline{0}^{(y, d]})$. 
\end{itemize}
\end{Ixsrepresentingordersdefn}

Note that these intervals are well-defined. This is clear in case (ii.). In case (i.), observe we have $S_y \subseteq A_y = [0]_{E_y}$, and for every $s \in S_y$, we have $s > 0$. Thus $\overline{0}^{[c, y)}s\overline{0}^{(y, d]} \in Y \subseteq Z$ and $\overline{0}^{[c, y)}s\overline{0}^{(y, d]} > \overline{0}$. 

Theorem \ref{ZrepresentsXSxincLOthm} below says that if the semigroups $S_x$ are nontrivial for densely many $x \in X$, then \ref{Ixsrepresentingordersdefn} defines a representation of $X(S_x)$ in the class of colored linear orders. After we prove it, we show in Corollary \ref{representsXSxinLOcorollarynodensity} that the density hypothesis can be dispensed with. 

\theoremstyle{definition}
\newtheorem{ZrepresentsXSxincLOthm}[lct]{Theorem}
\begin{ZrepresentsXSxincLOthm}\label{ZrepresentsXSxincLOthm}
Suppose that $\{x \in X: \textrm{$S_x$ is a strict halfgroup}\}$
is dense in $X$. Define 
\[
\begin{array}{rcl}
\iota: X(S_x) & \rightarrow & \mathsf{c}LO \\
\iota(y, s) & = & I_{(y, s)}
\end{array}
\]
Then $\iota$ is a representation of $X(S_x)$ in $\mathsf{c}LO$. 
\end{ZrepresentsXSxincLOthm}

\begin{proof}

We begin with two lemmas.

\theoremstyle{definition}
\newtheorem*{Lemma0}{Lemma 0}
\begin{Lemma0}\label{Lemma0}
Fix $y \in X^*$ and $r \in (A_x)^{[c, y)}$. Then fix $s \in [0]_{E_y} = A_y$ with $s > 0$. Using the description of sequences given in Lemma \ref{sequencesinZlemma}, define a map $Z_r \rightarrow Z_r$ by the rules
\[
\begin{array}{rcl}
rb & \mapsto & r(b+s) \\
ramb & \mapsto & r(a+s)mb \\
rau & \mapsto & r(a+s)u.
\end{array}
\]
Then this map is a colored automorphism of $Z_r$. 
\end{Lemma0}
\begin{proof}
The lemma asserts that translation by a positive $s \in [0]_{E_y}$ in the $y$th coordinate induces a colored automorphism of $Z_r$. 

The set of entries in the $y$th coordinate of sequences $u \in Z_r$ is $\mathbb{R}$. Since $s \in [0]_{E_y}$, the sequences $r(b+s)$ and $r(a+s)mb$ are in $T_r$ if and only if $rb$ and $rsmb$ are in $T_r$; and $r(a+s)u$ is in $Y_r$ if and only if $rau$ is in $Y_r$. Hence the map is well-defined. Since translation by $s$ is an order-automorphism of $\mathbb{R}$, it follows the map is an order-automorphism of $Z_r$. 

Moreover, it is color-preserving: we have $\mathsf{c}(rau) = \mathsf{c}(r(a+s)u)$ since both sequences belong to $Y$; $\mathsf{c}(ramb) = \mathsf{c}(r(a+s)mb)$ since both sequences have the same domain and same final coordinate; and, in the case when $S_y$ is a strict halfgroup and there is something to check, we have $\mathsf{c}(rb) = \mathsf{c}(r(b+s))$, since both sequences have the same domain and $b E_y (b+s)$. The lemma is proved.
\end{proof}

\theoremstyle{definition}
\newtheorem*{Lemma1}{Lemma 1}
\begin{Lemma1}\label{Lemma1}
Fix $y \in X^*$ and $r \in (A_x)^{[c, y)}$. Suppose $S_y$ is trivial. Define a map $Z_r \rightarrow Z_r$ by the rules
\[
\begin{array}{rcl}
ramb & \mapsto & r(2a)mb \\
rau & \mapsto & r(2a)u.
\end{array}
\]
Then this map is a colored automorphism of $Z_r$. 
\end{Lemma1}

\begin{proof}
Since $a \mapsto 2a$ defines an order-automorphism of $\mathbb{R} = E_{S_y}$, the map defined above is an order-automorphism of $Z_r$. It is clearly color-preserving.  
\end{proof}

The remainder of the proof splits into four claims.

In all of these claims and their proofs, the symbol $\cong$ indicates color-isomorphism. Likewise, $\leqslant_{init}$ and $\leqslant_{fin}$ indicate a colored initial embedding and colored final embedding. 

\theoremstyle{definition}
\newtheorem*{ClaimA}{Claim A}
\begin{ClaimA}\label{ClaimA}
If $y < z$ in $X$, so that $z < y$ in $X^*$, then for any $s \in S_y$ and $t \in S_z$ we have 
\begin{equation*}
I_{(y, s)} + I_{(z, t)} \cong I_{(z, t)} + I_{(y, s)} \cong I_{(z, t)}.
\end{equation*}
\end{ClaimA}

\begin{proof}
We show $\omega I_{(y, s)} \leqslant_{init} I_{(z, t)}$ and $\omega^* I_{(y, s)} \leqslant_{fin} I_{(z, t)}$. 

Let $r = \overline{0}^{[c, y)}$. Label $Z_r = [p, q]$ by its endcuts. We have $Z_r \cong [p, \overline{0}) + [\overline{0}, q]$.

Notice $I_{(x, s)} = [\overline{0}, rs\overline{0}^{(y, d]})$ is initial in the final segment $[\overline{0}, q]$ of $Z_r$. For $k \in \mathbb{Z}$, define
\[
I_k = [r(ks)\overline{0}^{(y, d]}, r((k+1)s)\overline{0}^{(y, d]}).
\]

Thus $I_0 = I_{(y, s)}$. Since the left endpoints of these intervals traverse $Z_r$, we have the decompositions:

\begin{tabular}{rcl}
    $Z_r$ & $\cong$ & $\cdots + I_{-1} + I_0 + I_1 + I_2 + \cdots$; \\
    $[\overline{0}, q]$ & $\cong$ & $I_0 + I_1 + I_2 + \cdots$; \\
    $[p, \overline{0})$ & $\cong$ & $\cdots + I_{-2} + I_{-1}$.
\end{tabular}

Observe that the color-automorphism of $Z_r$ from Lemma 0 with respect to the pair $(s, y)$ (translation by $s$ in coordinate $y$) maps each $I_k$ onto $I_{k+1}$. Thus $I_k \cong I_0 \cong I_{(y, s)}$ for all $k \in \mathbb{Z}$, and we have:

\begin{tabular}{rcl}
    $[\overline{0}, q]$ & $\cong$ & $\omega I_{(y, s)}$; \\
    $[p, \overline{0})$ & $\cong$ & $\omega^* I_{(y, s)}$.
\end{tabular}

Now let $r' = \overline{0}^{[c, z)}$. Label $Z_{r'} = [p', q']$ by its endcuts. 

Observe that $I_{(z, t)} = [\overline{0}, r't\overline{0}^{(z, d]})$ is initial in $[\overline{0}, q']$. Further, since $z < y$ in $X^*$, we have $[\overline{0}, q]$ is initial in $I_{(z, t)}$, which gives $\omega I_{(y, s)} \leqslant_{init} I_{(z, t)}$. 

For $k \in \mathbb{Z}$, define
\[
I_k' = [r'(kt)\overline{0}^{(z, d]}, r'((k+1)t)\overline{0}^{(z, d]}).
\]
Then $I_0' = I_{(z, t)}$. Observe $[p, \overline{0}) \cong \omega^* I_{(y, s)}$ is final in $I_{-1}'$. Now the map on $Z_{r'}$ from Lemma 0 with respect to $(z, t)$ (translation by $t$ in coordinate $z$) witnesses $I_{-1}' \cong I_0'$. Thus we have $\omega^* I_{(y, s)} \leqslant_{fin} I_0' \cong I_{(z, t)}$, which completes the proof of the claim. 
\end{proof}

\theoremstyle{definition}
\newtheorem*{ClaimB}{Claim B}
\begin{ClaimB}\label{ClaimB}
Given $y \in X$ such that $S_y$ is nontrivial, for any $s, t \in S_y$ we have 
\begin{equation*}
    I_{(y, s)} + I_{(y, t)} \cong I_{(y, s + t)}.
\end{equation*}
\end{ClaimB}

\begin{proof}
By cutting $I_{(y, s+t)}$ at the right endpoint $\overline{0}^{[c, y)}s\overline{0}^{(y, d]}$ of $I_{(y, s)}$, we have the decomposition
\[
I_{(y, s+t)} \cong I_{(y, s)} + [\overline{0}^{[c, y)}s\overline{0}^{(y, d]}, \overline{0}^{[c, y)}(s+t)\overline{0}^{(y, d]})
\]

Translating by $s$ in the $y$th coordinate (in $Z_{\overline{0}^{[c, y)}}$, using Lemma 0) sends $I_{(y, t)}$ onto the interval to the right of the $+$ sign in the expression above. Thus 
\[
I_{(y, s+t)} \cong I_{(y, s)} + I_{(y, t)}
\]
as claimed. 
\end{proof}

\theoremstyle{definition}
\newtheorem*{ClaimC}{Claim C}
\begin{ClaimC}\label{ClaimC}
Given $y \in X$, if $S_y = \{s\}$ is trivial, then 
\begin{equation*}
    I_{(y, s)} + I_{(y, s)} \cong I_{(y, s)}.
\end{equation*}
\end{ClaimC}

\begin{proof}
Using Lemma 1 to multiply by $2$ in the $y$th coordinate witnesses that 
\[
I_{(y, s)} = [\overline{0}, \overline{0}^{[c, y)}1\overline{0}^{(y, d]}) \cong [\overline{0}, \overline{0}^{[c, y)}2\overline{0}^{(y, d]}).
\]

By inspection we have
\[
[\overline{0}, \overline{0}^{[c, y)}2\overline{0}^{(y, d]}) \cong I_{(y, s)} + [\overline{0}^{[c, y)}1\overline{0}^{(y, d]}, \overline{0}^{[c, y)}2\overline{0}^{(y, d]}).
\]

Translation by $1$ witnesses that the interval on the right is isomorphic to $I_{(y, s)}$. Combining the above then yields
\[
I_{(y, s)} \cong [\overline{0}, \overline{0}^{[c, y)}2\overline{0}^{(y, d]}) \cong I_{(y, s)} + I_{(y, s)},
\]
as claimed. 
\end{proof}

\theoremstyle{definition}
\newtheorem*{ClaimD}{Claim D}
\begin{ClaimD}\label{ClaimD}
If $(y, s)$ and $(z, t)$ are distinct elements of $X(S_x)$, then $I_{(y, s)} \not\cong I_{(z, t)}$. 
\end{ClaimD}

\begin{proof}
Since $(y, s) \neq (z, t)$, we have that either $y \neq z$, or $y = z$ and $s \neq t$. 

Suppose first $y \neq z$. Without loss of generality, we may assume $y < z$ in $X$, so that $z < y$ in $X^*$. Since strict halfgroups appear densely often among the semigroups $S_x$, there is some $z' \in X^*$ with $z \leq z' < y$ such that $S_{z'}$ is a strict halfgroup. 

Choose $b \in \mathbb{R} \setminus A_{z'}$ with $b > 0$. Then $rb \in Z_r$ where $r = \overline{0}^{[c, z')}$. Observe that if $z < z'$ (in $X^*$), then $Z_r \subseteq I_{(z, t)}$ and hence $rb \in I_{(z, t)}$. If $z = z'$, then we make sure to choose $b$ so that $rb \in I_{(z, t)}$, i.e. so that $0 < b < t$, which is possible since $\mathbb{R} \setminus A_{z'}$ is dense in $\mathbb{R}$ by the strictness of $S_{z'}$. 

We have $I_{(y, s)} \subseteq Z_{r'} \subseteq I_{(z, t)}$, where $r' = \overline{0}^{[c, y)}$. Observe that since $\textrm{dom}(rb) = [c, z']$ is shorter domain than the domain of any element of $Z_{r'}$, its color $\mathsf{c}(rb)$ is different from the color of any element in $Z_{r'}$. In particular there is no point in $I_{(y, s)}$ of color $\mathsf{c}(rb)$. Hence there is no color isomorphism between $I_{(y, s)}$ and $I_{(z, t)}$, as claimed. 

Now suppose $y = z$ and $s \neq t$. Since it contains distinct points we must have that $S_y$ is nontrivial, i.e. a strict halfgroup. 

Without loss of generality, suppose $s < t$. Let $r = \overline{0}^{[c, y)}$. 

Suppose toward a contradiction there is a colored isomorphism $f: I_{(y, t)} \rightarrow I_{(y, s)}$. Since the color assigned to a given $rb$ with $b \in \mathbb{R} \setminus A_y$ is distinct from the color of all points belonging to one of the intervals $Z_{ra}$ with $a \in A_y$, and between any two such intervals such a point appears, we must have that $f$ maps each such point $rb \in I_{(y, t)}$ to a corresponding point in $I_{(y, s)}$, and each such interval $Z_{ra} \subseteq I_{(y, t)}$ onto a corresponding interval in $I_{(y, s)}$. 

That is, for all $0 < b < t$ with $b \not\in A_y$ there is $0 < b' < s$ with $b' \not\in A_y$ and $\mathsf{c}(rb) = \mathsf{c}(rb')$ (i.e. $b E_{S_y} b'$) such that $f(rb) = f(rb')$; and for all $0 < a < t$ with $a \in A_y = [0]_{E_y}$, there is $0 < a' < s$ with $a \in A_y$ such that $f[Z_{ra}] = Z_{ra'}$. 

Condense each such $Z_{ra}$ to a black point. After condensing, observe that the intervals $I_{(y, t)} = [\overline{0}, \overline{0}^{[c, y)}t\overline{0}^{(y, d)})$ and $I_{(y, s)} = [\overline{0}, \overline{0}^{[c, y)}s\overline{0}^{(y, d)})$ are isomorphic to the $\mathbb{R}$ intervals $[0, t)$ and $[0, s)$ respectively. Identify $I_{(y, t)}$ with $[0, t)$ and $I_{(y, s)}$ with $[0, s)$, colored in the same way. Then $[0, t)$ and $[0, s)$ are colored by the orbit equivalence relation $E_G$ as in Theorem \ref{groupsoftranslationsarefull}, where $G = -S_y \cup \{0\} \cup S_y$. The discussion in the previous paragraph shows $f$ is a colored isomorphism of these intervals. But $f$ is not a translation, since $f(0) = 0$ and $f(t) = s < t$, contradicting Theorem \ref{groupsoftranslationsarefull}. 

It follows $I_{(y, s)} \not\cong I_{(z, t)}$. 
\end{proof}

The conjunction of Claims (A.), (B.), (C.), and (D.) gives that $\iota$ is a representation of $X(S_x)$ in $\mathsf{c}LO$. 
\end{proof}

Corollary \ref{representsXSxinLOcorollarynodensity} shows that we can remove the density hypothesis from Theorem \ref{ZrepresentsXSxincLOthm}.

\theoremstyle{definition}
\newtheorem{representsXSxinLOcorollarynodensity}[lct]{Corollary}
\begin{representsXSxinLOcorollarynodensity}\label{representsXSxinLOcorollarynodensity}
Suppose that $X(S_x)$ is an $LO$ semigroup. Then $X(S_x)$ is representable in $\mathsf{c}LO$. 
\end{representsXSxinLOcorollarynodensity}

\begin{proof}
We pass to a larger $LO$ semigroup satisfying the density hypothesis from Theorem \ref{ZrepresentsXSxincLOthm}.

Consider the lexicographic product $X \mathbb{Q} = \{(x, q): x \in X, q \in \mathbb{Q}\}$ obtained by replacing every point in $X$ with a copy of the rationals. Identify $X$ with the set of pairs $(x, 0) \in X \mathbb{Q}$. By the density of $\mathbb{Q}$, for any $x < x'$ in $X$ there is $y \in X\mathbb{Q} \setminus X$ with $x < y < x'$. 

For each $x \in X$, let $S_x' = S_x$. For $x \in X\mathbb{Q} \setminus X$, let $S_x'$ be some fixed strict halfgroup of $\mathbb{R}_{>0}$. Then $X\mathbb{Q}(S_x')$ is an $LO$ semigroup satisfying the hypothesis of Theorem \ref{ZrepresentsXSxincLOthm}. Hence $X\mathbb{Q}(S_x')$ is representable in $\mathsf{c}LO$. Since $X(S_x)$ is a subsemigroup of $X\mathbb{Q}(S_x')$, it follows $X(S_x)$ is representable in $\mathsf{c}LO$. 
\end{proof}

The proof of Theorem \ref{XSxrepninLOnocolorsthm} below shows that the representation of $X(S_x)$ in $\mathsf{c}LO$ from \ref{ZrepresentsXSxincLOthm} may be adapted to get a representation in $LO$. 

\theoremstyle{definition}
\newtheorem{XSxrepninLOnocolorsthm}[lct]{Theorem}
\begin{XSxrepninLOnocolorsthm}\label{XSxrepninLOnocolorsthm}
Suppose $X(S_x)$ is an $LO$ semigroup. Then $X(S_x)$ is representable in $LO$. 
\end{XSxrepninLOnocolorsthm}

\begin{proof}
By the proof of Corollary \ref{representsXSxinLOcorollarynodensity}, we may (replacing $X$ with $X\mathbb{Q}$ if necessary) assume $X$ is dense without endpoints, and $X(S_x)$ satisfies the hypothesis of Theorem \ref{ZrepresentsXSxincLOthm}. 

We encode the coloring $\mathsf{c}$ of the order $Z$ constructed above in an uncolored replacement of $Z$. Then we show that the corresponding replacements of the orders $I_{(y, s)}$ give an uncolored representation of $X(S_x)$. 

Enumerate the colors $\{c_i: i < \kappa\}$ that appear in $Z$ (i.e. $\{c_i\}$ enumerates the outputs of $\mathsf{c}$). Let $\aleph_i$ denote the $i$th infinite cardinal number. 

For each $i < \kappa$, fix a linear order $K_i$ that is $\aleph_i$-dense, that is, an order $K_i$ with the property that for any $k < k'$ in $K_i$, the interval $[k, k']$ has cardinality $\aleph_i$. Then any non-singleton interval $I \subseteq K_i$ is also $\aleph_i$ dense. Consequently, if $i \neq j$ there is no convex embedding of $I$ into $K_j$. 

(We note that $\aleph_i$-dense orders exist for all $i$: concretely one may take $K_i$ to be the suborder consisting of all eventually $0$ sequences in the $\omega$-power $\omega_i^{\omega}$, where $\omega_i$ denotes the least ordinal of cardinality $\aleph_i$.)

Let $Z'$ be the order obtained by replacing every point $u \in Z$ such that $\mathsf{c}(u) = c_i$ by $K_i$. That is, $Z' = Z(K_{[u]})$ is the replacement up to the equivalence relation on $Z$ determined by the color classes of $\mathsf{c}$, satisfying $K_u = K_i$ whenever $\mathsf{c}(u) = c_i$. We view the orders $K_i$ as uncolored, and $Z'$ as uncolored as well. 

If $I$ is an interval in $Z$, then any colored convex embedding $f: I \rightarrow Z$ determines a convex embedding of the restricted replacement $I(K_{[u]})$ into $Z'$, namely the embedding defined by the rule
\[
(u, k) \mapsto (f(u), k).
\]
This map is well-defined since $\mathsf{c}(u) = \mathsf{c}(f(u))$ and hence $K_u = K_{f(u)}$. 

Conversely, suppose $I \subseteq Z$ is an interval and consider the restricted replacement $I(K_{[u]})$, which is an interval in $Z'$. Suppose $F: I(K_{[u]}) \rightarrow Z'$ is a convex embedding whose image is also a replaced interval $I'(K_{[u]})$. 

\underline{Claim ($\dagger$)}: For every $u \in I$, the image $F[K_u] = K_v$ for some $v \in Z$ with $\mathsf{c}(u) = \mathsf{c}(v)$. That is, $F$ condenses to a colored isomorphism $f: I \rightarrow I'$ defined by $f(u) = v$ whenever $F[K_u] = K_v$. 

To prove the claim, we first note that from our assumptions on $X$ it is not difficult to check that $Z$ is densely ordered, and $T$ is dense in $Z$. Further, for any $u < v$ in $Z$ and color $c_i$, there is $w$ in $Z$ with $u < w < v$ and $\mathsf{c}(w) \neq c_i$.

Fix $u \in I$, and suppose $\mathsf{c}(u) = c_i$. Consider $F[K_u]$. If this image intersects two distinct replacement factors $K_v < K_{v'}$, then it contains $K_w$ for some $w$ with $v < w < v'$ and $\mathsf{c}(w) \neq c_i$. Since $K_w \leqslant_{conv} F[K_u]$, we have $K_w \leqslant_{conv} K_u$, contradicting the difference in these intervals' densities. 

Thus $K_u \subseteq K_v$ for some $v \in Z$. If the containment is strict, then the image $F^{-1}(K_v)$ of $K_v$ under the inverse embedding $F^{-1}$ intersects both $K_u$ and $K_{u'}$ for some $u' \neq u$ in $I$, and we get the same contradiction. Claim ($\dagger$) follows.  

For each $(y, s) \in X(S_x)$, let $I_{(y, s)}' = I_{(y, s)}(K_{[u]})$ be the restriction of the replacement $Z' = Z(K_{[u]})$ to $I_{(y, s)}$. 

Let Claims (a.), (b.), (c.), and (d.) be the statements obtained from Claims (A.), (B.), (C.), and (D.) by replacing the orders $I_{(y, s)}$ and $I_{(z, t)}$ in their statements by $I_{(y, s)}'$ and $I_{(z, t)}'$. We may prove Claims (a.), (b.), (c.) using the same proofs for (A.), (B.), and (C.), substituting each interval $I \subseteq Z$ that appears in these proofs by its replaced version $I' = I(K_{[u]})$, and each convex embedding $f$ on such an interval $I$ by its lift to $I'$. 

We may also prove Claim (d.): by ($\dagger$), if we were able to find an isomorphism $F: I_{(y, s)}' \rightarrow I_{(z, t)}'$, then the corresponding condensed isomorphism $f$ would be a color-isomorphism of $I_{(y, s)}$ and $I_{(z, t)}$, contradicting Claim (D.). 

It follows from Claims (a.), (b.), (c.), and (d.) that the rule
\[
(y, s) \mapsto I_{(y, s)}'
\]
defines an uncolored representation of $X(S_x)$, i.e. a representation of $X(S_x)$ in $LO$. 
\end{proof}

\subsection{The representation theorem}\label{subsect:therepnthm}

\theoremstyle{definition}
\newtheorem{representablesemigroupstheorem}[lct]{Theorem}
\begin{representablesemigroupstheorem} \label{representablesemigroupstheorem}
A commutative semigroup $(S, \oplus)$ is representable in $(LO, +)$ if and only if $S$ is isomorphic to a subsemigroup of an $LO$ semigroup.  
\end{representablesemigroupstheorem}

\begin{proof}
Since subsemigroups of representable semigroups are representable, the backward direction of the theorem is given by Theorem \ref{XSxrepninLOnocolorsthm}. 

For the forward direction, suppose $S$ is a commutative semigroup and there is a representation $\iota: S \rightarrow LO$. 

For simplicity of notation, for the remainder of the proof we view $LO$ as the class of linear order types $LO/{\cong}$, and identify each linear order $A$ with its order type $[A]_{\cong}$. Under this identification $(LO, +)$ is a class semigroup.

Then $\iota$ is a semigroup isomorphism of $S$ with its image. Identify $S$ with its image, so that $(S, +)$ is a semigroup of order types. 

For each $A \in S$, we write $[A]$ for $[A]_{\approx} \cap S$, i.e.
\[
[A] = \{B \in S: B \approx A\}.
\]

Since $+$ is commutative on $S$, we have $\lessapprox$ is total on $S$ by Theorem \ref{ABcommutearithmeticequivalenceslist}. Thus $\lessapprox$ quasi-linearly orders $S$, and linearly orders $\{[A]: A \in S\}$. 

Let $X = \{[A]: A \in S\}$. For each $x \in X$, define $S_x = x$. Then $S_x = [A]$ for some $A \in S$. Note that since $S$ is closed under $+$, by Lemma \ref{approxclassesclosedundersum} each $S_x$ is closed under $+$ as well, and hence a semigroup under $\oplus_x = +$.

Consider the replacement semigroup $X(S_x)$, and let $\boxplus$ denote its semigroup operation as defined in \ref{replacementsemigroupdefn}. Observe that $X(S_x)$ consists of all pairs of the form $([A], A)$ for $A \in S$. 

Define $F: S \rightarrow X(S_x)$ by the rule $F(A) = ([A], A)$. By the above discussion, $F$ is clearly injective and surjective. Further, for any $A, B \in S$ we have
\[
\begin{array}{rclcl}
    F(A + B) & = &  ([A+B], A+B) & = &  \left\{ \begin{array}{ll}
                                        ([A], A) & \textrm{if $[B] \lnapprox [A]$} \\
                                        ([B], B) & \textrm{if $[A] \lnapprox [B]$}  \\
                                        ([A], A + B) & \textrm{if $[A] = [B]$} .
                                        \end{array}  
                                        \right. \\
\end{array}
\]
By the definition of $X(S_x)$, the expression on the right equals $([A], A) \boxplus ([B], B)$. Thus $F$ is a semigroup isomorphism. 

Now observe that each factor $S_x = [A] = [A]_{\approx} \cap S$ is a subsemigroup of the semigroup $S_x' = [A]_{\approx}$. Thus $X(S_x)$ is a subsemigroup of $X(S_x')$. By Theorem \ref{Aapproxclasseitherfullystrictortrivialrepn}, each $S_x'$ is either a copy of the trivial semigroup, or isomorphic to a strict halfgroup of $\mathbb{R}_{>0}$. Thus $X(S_x')$ is an $LO$ semigroup, and we are done. 
\end{proof}

\bibliographystyle{amsplain}

\begin{thebibliography}{00}

\bibitem{Aronszajn} N. Aronszajn,
{\it Characterisation of types of order satisfying $\alpha_0 + \alpha_1= \alpha_1+ \alpha_0$},
Fundamenta Mathematicae 39.1 (1952): 65-96.

\bibitem{Clifford} A.H. Clifford, 
{\it Naturally totally ordered commutative semigroups},
American Journal of Mathematics 76.3 (1954): 631-646.

\bibitem{DNR} B. Deroin, A. Navas, and C. Rivas,
{\it Groups, orders, and dynamics}, 
arXiv preprint arXiv:1408.5805 (2014).

\bibitem{ErvinPaul} G. Ervin and E. Paul,
{\it Commutativity laws for ordinal algebras},
forthcoming. 

\bibitem{Jonsson} B. J\'onsson,
{\it Arithmetic of ordered sets},
Ordered Sets: Proceedings of the NATO Advanced Study Institute held at Banff, Canada, August 28 to September 12, 1981. Dordrecht: Springer Netherlands (1982).

\bibitem{LindenbaumTarski} A. Lindenbaum and A. Tarski, 
{\it Communication sur les recherches de le théorie des ensembles},
(1926).

\bibitem{McCleary} S.H. McCleary, 
{\it The structure of intransitive ordered permutation groups},
Algebra Universalis 6.1 (1976): 229-255.

\bibitem{TarskiCard} A. Tarski,
{\it Cardinal Algebras. With an Appendix: Cardinal Products of Isomorphism Types, by Bjarni Jónsson and Alfred Tarski.} Oxford University Press, New York, N. Y., 1949.

\bibitem{Tarski} A. Tarski, 
{\it Ordinal algebras},
North-Holland Publishing Co., Amsterdam (1956), i+133 pp.

\end{thebibliography}

\end{document}